\documentclass[sn-mathphys-num]{sn-jnl}

\usepackage{graphicx}%
\usepackage{multirow}%
\usepackage{amsmath,amssymb,amsfonts}%
\usepackage{amsthm}%
\usepackage{mathrsfs}%
\usepackage[title]{appendix}%
\usepackage{xcolor}%
\usepackage{textcomp}%
\usepackage{manyfoot}%
\usepackage{booktabs}%
\usepackage{algorithm}%
\usepackage{algorithmicx}%
\usepackage{algpseudocode}%
\usepackage{listings}%
\usepackage{setspace}%

\theoremstyle{thmstyleone}%
\newtheorem{theorem}{Theorem}
\newtheorem{proposition}[theorem]{Proposition}%
\newtheorem{lemma}[theorem]{Lemma}
\newtheorem{corollary}[theorem]{Corollary}

\theoremstyle{thmstyletwo}%
\newtheorem{remark}{Remark}%
\theoremstyle{thmstylethree}%

\begin{document}

\title[Dewetting Regime]{Motion of Islands of Elastic Thin Films in the Dewetting Regime}


\author[1]{\fnm{Gianni} \sur{Dal Maso}}\email{dalmaso@sissa.it}
\equalcont{These authors contributed equally to this work.}

\author[2]{\fnm{Irene} \sur{Fonseca}}\email{fonseca@andrew.cmu.edu}
\equalcont{These authors contributed equally to this work.}

\author*[3]{\fnm{Giovanni} \sur{Leoni}}\email{giovanni@andrew.cmu.edu}

\affil*[1]{\orgname{SISSA}, \orgaddress{\street{Via Bonomea 265}, \city{Trieste}, \postcode{ 34136},  \country{Italy}}}

\affil[2]{\orgdiv{Department of Mathematical Sciences}, \orgname{Carnegie Mellon University}, \orgaddress{\street{5000 Forbes Avenue}, \city{Pittsburgh}, \postcode{15217}, \state{PA}, \country{USA}}}

\affil[3]{\orgdiv{Department of Mathematical Sciences}, \orgname{Carnegie Mellon University}, \orgaddress{\street{5000 Forbes Avenue}, \city{Pittsburgh}, \postcode{15217}, \state{PA}, \country{USA}}}


\abstract{This paper addresses a two-dimensional sharp interface variational model for
 solid-state dewetting of thin films with surface energies, introduced 
by Wang, Jiang, Bao, and Srolovitz in \cite{jiang2016solid}.
Using the $H^{-1}$-gradient 
flow structure of the evolution
law, short-time existence for a surface diffusion evolution
equation with curvature regularization is established in the context of epitaxially strained two-dimensional films. The main novelty, as compared to the study of the wetting regime, is the presence of moving contact lines.}

\keywords{Minimizing Movements, Moving Contact Angles, Elastic Thin Films}


\pacs[MSC Classification]{35K25, 35B65, 74K35}

\maketitle

\section{Introduction}\label{sec1}

Understanding the dewetting process and its underlying mechanisms is essential for controlling the morphology and properties of thin films, which have many applications in microelectronics, optics, and other fields (see \cite{freund2004thin}).
This type of dewetting is distinct from liquid dewetting and is primarily driven by surface diffusion-controlled mass transport at temperatures below the melting point of the film. The movement of the contact line, where the film, substrate, and vapor phases meet, plays a crucial role in this process.
Mathematically modeling the morphology evolution of solid-state dewetting involves treating it as a surface-tracking problem. The minimization of interfacial energy guides the process and consists of a combination of surface diffusion-controlled mass transport and the movement of the contact line. While moving contact line problems have been extensively studied in fluid mechanics, incorporating surface diffusion-based geometric evolution equations with moving contact lines presents significant challenges in materials science, applied mathematics, and scientific computing.

This paper studies a two-dimensional sharp interface variational model for
simulating solid-state dewetting of thin films with surface energies, proposed
by Wang, Jiang, Bao, and Srolovitz in \cite{jiang2016solid}, (see also
\cite{jiang2018}). The morphology evolution is driven by surface
diffusion and contact points migration, coupled with elastic deformation. We restrict our consideration to a
single island whose profile is given by a function $h:[\alpha,\beta
]\rightarrow\lbrack0,\infty)$, where $\alpha$ and $\beta$ are the contact
points, $h(\alpha)=h(\beta)=0$,  and $h(x)>0$ for every $x\in(\alpha,\beta)$.
The region occupied by the island is%
\[
\Omega_{h}:=\{(x,y)\in\mathbb{R}^{2}:\,\alpha<x<\beta,\,0<y<h(x)\}\,.
\]
We assume that the region $\mathbb{R}\times(-\infty,0]$ is filled by a rigid
substrate and $(\mathbb{R}\times(0,\infty))\setminus\Omega_{h}$ by a vapor.

The elastic displacement within the island is described by a function
$u:\Omega_{h}\rightarrow\mathbb{R}^{2}$, which satisfies the boundary
condition
\begin{equation}
u(x,0)=(e_{0}x,0)\,\text{ for }\,x\in(\alpha,\beta)\,.
\label{boundary condition u}%
\end{equation}
The parameter $e_{0}\neq0$ reflects the mismatch between the crystalline
structures of the thin film and the substrate. The underlying energy is 
\begin{equation}
\mathcal{W}(u,\Omega_{h})+\gamma\operatorname*{length}(\Gamma_{h})-\gamma
_{0}(\beta-\alpha)+\frac{\nu_{0}}{2}\int_{\Gamma_{h}}\kappa^{2}ds\,, \label{energy}%
\end{equation}
where $\mathcal{W}(u,\Omega_{h})$ is the linearized elastic energy of the
displacement $u$, $\Gamma_{h}$ is the graph of $h$, $\kappa$ is the curvature,
and $s$ is the arclength parameter on $\Gamma_{h}$. The constant
$\gamma=\gamma_{FV}>0$ is the surface energy density between film and vapor,
while $\gamma_{0}=\gamma_{VS}-\gamma_{FS}$, where $\gamma_{VS}$ is the surface
energy density between vapor and substrate, and $\gamma_{FS}$ between film and
substrate. The constant $\nu_{0}>0$ is a small parameter in the curvature
regularization, which is commonly used in the literature
(\cite{di1992regularized}, \cite{gurtin2002interface}). The dewetting regime
(Volmer--Weber) is characterized by the inequality $\gamma>\gamma_{0}$, which
favors the exposure of the substrate.

We will assume that at each time the displacement $u$ satisfies the elastic
equilibrium problem on $\Omega_{h}$ with natural boundary conditions on
$\Gamma_{h}$ and the Dirichlet boundary condition \eqref{boundary condition u}
in the rest of the boundary.

The time evolution of $(\alpha,\beta,h)$ is obtained as gradient flow of the
energy \eqref{energy} with the area constraint $|\Omega_{h}|=A_{0}>0$, where
for the dynamics of $h$ we use a type of $H^{-1}(\Gamma_{h})$ norm (see
\cite{ambrosio2008gradient}, \cite{cahn1994overview}). To describe the
equations of this evolution system, we introduce the chemical potential
$\zeta$ defined in terms of the arclength parameter $s$ of $\Gamma_{h}$ by%
\begin{equation}
\zeta=-\gamma\kappa+\nu_{0}\Big(  \partial_{ss}\kappa+\frac{\kappa^{3}}{2}\Big)
+\widetilde{W}\,, \label{evolution h}%
\end{equation}
 where $\widetilde{W}$
is the value of the elastic energy density of $u$ at the point of $\Gamma_{h}$
corresponding to $s$.

The equation for $h$ is given by%
\begin{equation}
\widetilde{V}=\rho_{0}\partial_{ss}\zeta\quad\text{on }\Gamma_{h}\,,
\label{normal equation}%
\end{equation}
where $\widetilde{V}$ denotes the normal velocity of the time dependent curve
$\Gamma_{h}$ at the point corresponding to $s$ and $\rho_{0}>0$ is a material constant.

The equations for the contact points migration are%
\begin{align}
&  \sigma_{0}\dot{\alpha}=\gamma\cos\theta_{\alpha}-\gamma_{0}+\nu_{0}\partial_{s}
\kappa_{\alpha}\sin\theta_{\alpha}\,,\label{young law}\\
&  \sigma_{0}\dot{\beta}=-\gamma\cos\theta_{\beta}+ \gamma_{0}-\nu_{0}
\partial_{s}\kappa_{\beta}\sin\theta_{\beta}\,,\nonumber
\end{align}
 where $\sigma_{0}>0$ is a material constant, the dot denotes the time
derivative, $\theta_{\alpha}$ and $\theta_{\beta}$ are the oriented angles
between the $x$-axis and the tangent  to  $\Gamma_{h}$ at
$(\alpha,0)$ and $(\beta,0)$,  both oriented with increasing values of $x$, while $\partial_{s}\kappa_{\alpha}$ and
$\partial_{s}\kappa_{\beta}$ are the derivatives with respect to $s$ of the curvature
$\kappa$ of $\Gamma_{h}$ at the values of $s$ corresponding to $(\alpha,0)$
and $(\beta,0)$. Observe that if we neglect the curvature regularization, that
is, we take $\nu_{0}=0$, then the resulting equations reduce to the usual
Young's law (\cite{davoli-piovano}, \cite{dephilippis-maggi2015}, \cite{freund2004thin}).

The main result of this paper is that for every initial condition $(\alpha
_{0},\beta_{0},h_{0})$ there exists a small time $T>0$ such that the evolution
equations \eqref{normal equation} and \eqref{young law} admit a weak solution
on $[0,T]$.
Our approach does not allow us to obtain the uniqueness of solutions.

The existence proof relies on a minimizing movements argument: we consider a
time discretization and construct an approximate solution via incremental
minimum problems involving the energy \eqref{energy}. While this approach is
not novel for epitaxial growth (see \cite{fonseca2012motion},
\ \cite{fonseca2015motion}, \cite{piovano2014evolution}), the major challenge
here is that the domain $\Omega_{h}$ has evolving corners. This requires a
delicate $W^{2,p}(\Omega_{h})$ estimate of the solution for the Lam\'{e}
system, with a precise dependence on the time step of the discretization. This
estimate plays a central role in the study of the convergence of the
discretized solutions for the generalized Young's law \eqref{young law}.
 We refer to \cite{wilkening-borucki-sethian} for a stress-driven grain boundary diffusion problem, where the analysis of the singularities of the solutions to the Lam\'{e} system near triple junctions plays a crucial role.

There is an extensive body of literature in two and three dimensions for the
static problem both in the wetting ($\gamma<\gamma_{0}$) and dewetting
($\gamma>\gamma_{0})$ regimes. We refer to \cite{bella2014study},
\cite{bonacini2013epitaxially}, \cite{bonacini2015stability},
\cite{braides2007relaxation}, \cite{capriani2013quantitative},
\cite{chambolle2002computing}, \cite{chambolle2007interaction},
\cite{crismale-friedrich}, \cite{de2012regularity},
\cite{fonseca2007equilibrium},\cite{fonseca2011material},
\cite{fonseca2014shapes}, \cite{fusco2012equilibrium},
\cite{goldman2014scaling}, \cite{kukta1997minimum},
\cite{spencer1997equilibrium}, and the references therein for the wetting
regime;\ and \cite{davoli-piovano}, \cite{piovano-velcic2022}  for the dewetting regime. We also refer to
 \cite{almi-lucardesi} for a study of the
Lam\'{e} system in the presence of cracks.

For the evolution case in the wetting regime, we refer to
\cite{fonseca2012motion}, \ \cite{fonseca2015motion},
\cite{fusco-julin-morini2d}, \cite{fusco-julin-morini3d},
\cite{piovano2014evolution}, \cite{siegel2004evolution}. Observe that in the
papers \cite{fusco-julin-morini2d}, \cite{fusco-julin-morini3d} the curvature
regularization is omitted.

While moving contact line problems have been studied in the fluid mechanics
community (see, e.g., \cite{guo-tice},\cite{guo-tice-ns}, \cite{tice-wu} and
the references therein), to our knowledge our work is the first to prove the
generalized Young's law in a problem involving elasticity.

\section{Preliminaries}

Throughout this paper, we fix the physical parameters $\gamma$, $\gamma_{0}$,
$\sigma_{0}$, $A_{0}$, $e_{0}\in\mathbb{R}$, with $\gamma>0$, $\gamma
>\gamma_{0}$, $\sigma_{0}>0$, $A_{0}>0$, $e_{0}\neq0$, the Lam\'{e}
coefficients $\lambda$ and $\mu$, with $\mu>0$ and $\lambda+\mu>0$, and the
regularizing parameters $L_{0}\geq1$  and  $\nu_{0}>0$. We renormalize the
parameter $\rho_{0}$ in \eqref{normal equation} to be one.
 We use standard notation for Lebesgue and Sobolev spaces, as well as for spaces of
H\"older continuous and differentiable functions. 

We introduce the class of admissible surface profiles, $\mathcal{A}_{s}$, as
the set of all $(\alpha,\beta,h)$ such that $\alpha<\beta$, $h\in
H^{2}((\alpha,\beta))\cap H_{0}^{1}((\alpha,\beta))$, with $h\geq0$ in
$(\alpha,\beta)$, $\operatorname*{Lip}h\leq L_{0}$, and%
\begin{equation}
\int_{\alpha}^{\beta}h(x)\,dx=A_{0}\,. \label{area constraint}%
\end{equation}
Moreover, if $(\alpha,\beta,h)\in\mathcal{A}_{s}$ then $\check{h}$ is the
extension of $h$ by zero outside of $[\alpha,\beta]$, and we set
\begin{equation}
H(x;\alpha,\beta,h):=\int_{-\infty}^{x}\check{h}(\rho)\,d\rho=\int_{\alpha
}^{x}\check{h}(\rho)\,d\rho\,,\quad x\in\mathbb{R}\,, \label{Functional H}%
\end{equation}
and%
\begin{equation}
\Omega_{h}:=\{(x,y)\in\mathbb{R}^{2}:\,\alpha<x<\beta,\,0<y<h(x)\}\,.
\label{Omega h}%
\end{equation}
Furthermore, the admissible class $\mathcal{A}_{e}(\alpha,\beta,h)$ of elastic
displacements in $\Omega_{h}$ is defined as%
\begin{equation}
\mathcal{A}_{e}(\alpha,\beta,h):=\{u\in H^{1}(\Omega_{h};\mathbb{R}%
^{2}):\,u(x,0)=(e_{0}x,0)\text{ for a.e. }x\in(\alpha,\beta)\}\,.
\label{class A h}%
\end{equation}
Finally the admissible class $\mathcal{A}$ for the total energy is%
\begin{equation}
\mathcal{A}:=\{(\alpha,\beta,h,u):\,(\alpha,\beta,h)\in\mathcal{A}_{s}%
,\,u\in\mathcal{A}_{e}(\alpha,\beta,h)\}\,. \label{class A}%
\end{equation}

In what follows we will use the result below.

\begin{lemma}
\label{lemma endpoints}We have
\begin{equation}
\beta-\alpha\geq\sqrt{\frac{2A_{0}}{L_{0}}}%
\label{interval lower estimate}%
\end{equation}
for every $(\alpha,\beta,h)\in\mathcal{A}_{s}$.
\end{lemma}

\begin{proof}
Since $h(\alpha)=h(\beta)=0$ and $\operatorname*{Lip}h\leq L_{0}$, we have
$h(x)\leq L_{0}\frac{\beta-\alpha}{2}$ for every $x\in(\alpha,\beta)$. Hence,
by \eqref{area constraint},%
\[
A_{0}=\int_{\alpha}^{\beta}h(x)\,dx\leq\frac{L_{0}}{2}(\beta-\alpha)^{2},
\]
which concludes the proof.
\end{proof}

For $(\alpha,\beta,h)\in\mathcal{A}_{s}$ we define the surface energy as%
\begin{equation}
\mathcal{S}(\alpha,\beta,h):=\ \gamma\int_{\alpha}^{\beta}\sqrt{1+(h^{\prime
}(x))^{2}}dx-\gamma_{0}(\beta-\alpha)+\frac{\nu_{0}}{2}\int_{\alpha}^{\beta}%
\frac{(h^{\prime\prime}(x))^{2}}{(1+(h^{\prime}(x))^{2})^{5/2}}%
dx\,.\label{surface energy}%
\end{equation}
 Note that since $\gamma>0$ and $\gamma>\gamma_{0}$, we have%
\begin{equation}
\gamma\int_{\alpha}^{\beta}\sqrt{1+(h^{\prime}(x))^{2}}dx-\gamma_{0}%
(\beta-\alpha)\geq(\gamma-\gamma_{0})(\beta-\alpha)\geq
0\label{surface energy 2}%
\end{equation}
and so
\begin{equation}
\mathcal{S}(\alpha,\beta,h)\geq0\,.\label{surface energy nonnegative}%
\end{equation}

For $(\alpha,\beta,h)\in\mathcal{A}_{s}$ and $u\in\mathcal{A}_{e}(\alpha
,\beta,h)$ we define the elastic energy as%
\begin{equation}
\mathcal{E}(\alpha,\beta,h,u):=\int_{\Omega_{h}}W(Eu(x,y))\,dxdy\,,
\label{elastic energy}%
\end{equation}
where $W:\mathbb{R}^{2\times2}\rightarrow\lbrack0,\infty)$ is given by
\begin{equation}
W(\xi):=\frac{1}{2}\mathbb{C}\xi\cdot\xi\,,\quad\text{with}\quad\mathbb{C}%
\xi:=\mu(\xi+\xi^{T})+\lambda(\operatorname*{tr}\xi)I\,, \label{W}%
\end{equation}
where $\mu$, $\lambda\in\mathbb{R}$ are the Lam\'{e} coefficients and $I$ is
the $2\times2$ identity matrix. Note that $\mathbb{C}\xi=\mathbb{C}%
\xi_{\operatorname*{sym}}\in\mathbb{R}_{\operatorname*{sym}}^{2\times2}$ for
every $\xi\in\mathbb{R}^{2\times2}$, where $\xi_{\operatorname*{sym}}%
:=(\xi+\xi^{T})/2$.

We assume that $\mu>0$ and $\lambda+\mu>0$ so that there exists a constant
$C_{W}>0$ such that
\begin{equation}
\frac{1}{C_{W}}|\xi|^{2}\leq W(\xi)\leq C_{W}|\xi|^{2} \label{W convex}%
\end{equation}
for all $\xi\in\mathbb{R}_{\operatorname*{sym}}^{2\times2}$.

In order to study the incremental problem, we introduce the following
functionals. Given $\tau>0$ and $(h^{0},\alpha^{0},\beta^{0})\in
\mathcal{A}_{s}$, for every $(h,\alpha,\beta)\in\mathcal{A}_{s}$ we define
\begin{equation}
\mathcal{T}_{\tau}(\alpha,\beta,h;\alpha^{0},\beta^{0},h^{0}):=\frac{1}{2\tau
}\int_{\mathbb{R}}(H-H^{0})^{2}\sqrt{1+((\check{h}^{0})^{\prime})^{2}}%
dx+\frac{\sigma_{0}}{2\tau}(\alpha-\alpha^{0})^{2}+\frac{\sigma_{0}}{2\tau
}(\beta-\beta^{0})^{2}, \label{tau energy}%
\end{equation}
where we abbreviate%
\begin{equation}
H(x):=H(x;\alpha,\beta,h)\,,\quad H^{0}(x):=H(x;\alpha^{0},\beta^{0},h^{0})\,,
\label{H and H0}%
\end{equation}
and $H$ is given in \eqref{Functional H}. Observe that
\begin{equation}
\frac{1}{2\tau}\int_{\mathbb{R}}(H-H^{0})^{2}\sqrt{1+((\check{h}^{0})^{\prime
})^{2}}dx=\frac{1}{2\tau}\int_{\min\{\alpha,\alpha^{0}\}}^{\max\{\beta
,\beta^{0}\}}(H-H^{0})^{2}\sqrt{1+((\check{h}^{0})^{\prime})^{2}}dx
\label{integral on interval}%
\end{equation}
since by construction and \eqref{area constraint}, $H-H^{0}=0$ for
$x\notin(\min\{\alpha,\alpha^{0}\},\max\{\beta,\beta^{0}\})$.

The existence of a minimizer for the incremental problem will be a consequence
of the following result.

\begin{theorem}
\label{theorem existence}For every $\tau>0$ and every $(\alpha^{0},\beta
^{0},h^{0})\in\mathcal{A}_{s}$ there exists a minimizer $(\alpha,\beta
,h,u)\in\mathcal{A}$ of the total energy functional
\begin{equation}
\mathcal{F}^{0}(\alpha,\beta,h,u):=\mathcal{S}(\alpha,\beta,h)+\mathcal{E}%
(\alpha,\beta,h,u)+\mathcal{T}_{\tau}(\alpha,\beta,h;\alpha^{0},\beta
^{0},h^{0})\,. \label{total energy}%
\end{equation}
 Moreover, there exists a constant $C>0$, depending only on the structural parameters $A_0$, $e_0$, $\lambda$, $\mu$, and $L_0$, such that
\begin{equation}\label{estimate in H1}
\Vert u\Vert_{H^{1}(\Omega_{h})}\le C
\end{equation}
for every minimizer $(\alpha,\beta
,h,u)$ of $\mathcal{F}^{0}$ in $\mathcal{A}$.

\end{theorem}

The proof of the theorem relies on the following Korn's inequality.

\begin{lemma}
\label{lemma korn}Let $(\alpha,\beta,h)\in\mathcal{A}_{s}$, 
let   $\Omega_{h}$ be as in \eqref{Omega h}, and let $1<p<\infty$.
Then there exists a constant $C>0$,  depending only on $p$ and $L_{0}$,  such
that
\[
\int_{\Omega_{h}}\left\vert \nabla u\right\vert ^{p}dxdy\leq C\int_{\Omega
_{h}}\left\vert Eu\right\vert ^{p}dxdy+Ce_{0}^{p}A_{0}%
\]
for every $u\in W^{1,p}(\Omega_{h};\mathbb{R}^{2})$ such that $u(x,0)=(e_{0}%
x,0)$ for $x\in(\alpha,\beta)$ (in the sense of traces).
\end{lemma}

\begin{proof}
By a translation we can assume that $\alpha=0$. Define $v(x,y):=u(x,y)-(e_{0}%
x,0)$ in $\Omega_{h}$ and $v:=0$ in $(0,\beta)\times(0,-\infty)$,  and
\[
w(x,y)=v(\beta x,\beta y)
\]
for $(x,y)\in D_{h_{\beta}}$, where%
\[
D_{h_{\beta}}:=\{(x,y)\in\mathbb{R}^{2}:\,0<x<1\,,\,y<h_{\beta}(x)\}
\]
and $h_{\beta}(x):=h(\beta x)/\beta$. Note that
\[
|h_{\beta}(x_{1})-h_{\beta}(x_{2})|\leq\frac{1}{\beta}|h(\beta x_{1})-h(\beta
x_{2})|\leq L_{0}|x_{1}-x_{2}|
\]
for all $x_{1},x_{2}\in(0,1)$. Applying Theorem 4.2 in
\cite{fonseca2007equilibrium} to $w$ we can find a constant $C$ depending only
on $p$ and $L_{0}$ such that%
\[
\int_{D_{h_{\beta}}}\left\vert \nabla w\right\vert ^{p}dxdy\leq C\int_{D_{h_{\beta}}%
}\left\vert Ew\right\vert ^{p}dxdy\,.
\]
By the change of variables $(x^{\prime},y^{\prime})=(\beta x,\beta y)$ we have%
\[
\int_{\Omega_{h}}\left\vert \nabla v\right\vert ^{p}dxdy\leq C\int_{\Omega
_{h}}\left\vert Ev\right\vert ^{p}dxdy\,,
\]
where we used the fact that $v=0$ in $(0,\beta)\times(-\infty,0)$. Recalling
that $v=u-w_{0}$, it follows that%
\[
\int_{\Omega_{h}}\left\vert \nabla u\right\vert ^{p}dxdy\leq C\int_{\Omega
_{h}}\left\vert E u\right\vert ^{p}dxdy+Ce_{0}^{p}\mathcal{L}^{2}(\Omega
_{h})=C\int_{\Omega_{h}}\left\vert E u\right\vert ^{p}dxdy+Ce_{0}^{p}A_{0}\,
\]
by \eqref{area constraint}.
\end{proof}

\begin{proof}
[Proof of Theorem \ref{theorem existence}]Let $\{(\alpha_{n},\beta_{n}%
,h_{n},u_{n})\}_{n}$ be a minimizing sequence in $\mathcal{A}$ for
\eqref{total energy}. Then there exists a constant $M>0$ such that%
\begin{equation}
\mathcal{F}^{0}(\alpha_{n},\beta_{n},h_{n},u_{n})\leq M \label{m1}%
\end{equation}
for all $n$. Since all the terms in $\mathcal{F}^{0}$ are nonnegative (see
\eqref{surface energy nonnegative}), by \eqref{tau energy} and \eqref{m1}
there exist two constants $M_{l}\in\mathbb{R}$ and $M_{r}\in\mathbb{R}$ such
that%
\begin{equation}
M_{l}\leq\alpha_{n}<\beta_{n}\leq M_{r} \label{m2}%
\end{equation}
for every $n$. Hence, up to a subsequence, not relabeled, we can assume that
\begin{equation}
\alpha_{n}\rightarrow\alpha\quad\text{and\quad}\beta_{n}\rightarrow
\beta\label{m3}%
\end{equation}
as $n\rightarrow\infty$, with $\alpha\leq\beta$. Since $\operatorname*{Lip}%
h_{n}\leq L_{0}$ for every $n$, the extension $\check{h}_{n}$ by zero of
$h_{n}$ outside of $[\alpha_{n},\beta_{n}]$ satisfies $\operatorname*{Lip}%
\check{h}_{n}\leq L_{0}$. Using \eqref{m3}, up to a further subsequence, not
relabeled, there exists a nonnegative Lipschitz continuous function $\check
{h}:\mathbb{R}\rightarrow\mathbb{R}$ such that $\check{h}_{n}\rightarrow
\check{h}$ uniformly and $\check{h}=0$ outside $(\alpha,\beta)$. In particular
this implies that the restriction $h$ of $\check{h}$ to $(\alpha,\beta)$
belongs to $H_{0}^{1}((\alpha,\beta))$ and that \eqref{area constraint} is
satisfied. Since $A_{0}>0$ we deduce that $\alpha<\beta$. By
\eqref{surface energy}, \eqref{m1}, and the fact that $\operatorname*{Lip}%
\check{h}_{n}\leq L_{0}$ for every $n$,
\begin{equation}
\sup_{n}\int_{\alpha_{n}}^{\beta_{n}}| h_{n}^{\prime\prime}|^{2}dx<\infty\,.
\label{m3a}%
\end{equation}
By \eqref{m3} this implies that $h\in H^{2}((\alpha,\beta))$, and so
$(h,\alpha,\beta)\in\mathcal{A}_{s}$ and $h\in C^{1}([\alpha,\beta])$.
Moreover, since $\check{h}=0$ outside $(\alpha,\beta)$, from \eqref{m3} and
\eqref{m3a} we deduce that
\begin{equation}
\check{h}_{n}^{\prime}(x)\rightarrow\check{h}^{\prime}(x) \label{m3b}%
\end{equation}
for every $x\in\mathbb{R}\setminus\{\alpha,\beta\}$.

By \eqref{elastic energy}, \eqref{W convex}, and \eqref{m1},%
\begin{equation}
\int_{\Omega_{h_{n}}}|Eu_{n}(x,y)|^{2}dxdy\leq MC_{W}\, \label{m4}%
\end{equation}
for every $n$. By Korn's inequality (see Lemma \ref{lemma korn})%
\begin{equation}
\int_{\Omega_{h_{n}}}\left\vert \nabla u_{n}\right\vert ^{2}dxdy\leq
C\int_{\Omega_{h_{n}}}\left\vert Eu_{n}\right\vert ^{2}dxdy+Ce_{0}^{2}%
A_{0}\leq CMC_{W}+Ce_{0}^{2}A_{0}\,. \label{m5}%
\end{equation}
Since $u_{n}(x,0)=(e_{0}x,0)$ for $x\in(\alpha,\beta)$, by Poincar\'{e}'s
inequality we find that
\begin{equation}
\int_{\Omega_{h_{n}}}|u_{n}(x,y)|^{2}\,dxdy\leq C \label{m5aa}%
\end{equation}
where $C>0$ is independent of $n$. By \eqref{m5}, \eqref{m5aa}, and a standard
diagonal argument, using the increasing sequence of sets $\Omega
_{h^{\varepsilon}}$, where $h^{\varepsilon}:=(h-\varepsilon)\vee0$, we
conclude that there exist a subsequence, not relabeled, and a function $u\in
H^{1}(\Omega_{h};\mathbb{R}^{2})$ such that $u_{n}\rightharpoonup u$ weakly in
$H^{1}(\Omega_{h^{\varepsilon}};\mathbb{R}^{2})$ for every $\varepsilon$. Note
that the trace of $u$ satisfies $u(x,0)=(e_{0}x,0)$ for a.e. $x\in\{h>0\}$. In
conclusion, we have shown that $(h,u,\alpha,\beta)\in\mathcal{A}$, where
$\mathcal{A}$ is given in \eqref{class A}.

Next we claim that%
\begin{equation}
\liminf_{n\rightarrow\infty}\mathcal{S}(\alpha_{n},\beta_{n},h_{n}%
)\geq\mathcal{S}(\alpha,\beta,h)\,, \label{m5a}%
\end{equation}
where $\mathcal{S}$ is given in \eqref{surface energy}. It is convenient to
write%
\begin{align}
\mathcal{S}(\alpha,\beta,h)&+\gamma_{0}(M_{r}-M_{l})\nonumber\\  &=  \int_{M_{l}}^{M_{r}%
}g(\check{h}(x))\sqrt{1+(\check{h}^{\prime}(x))^{2}}dx
+\frac{\nu_{0}}{2}\int_{\alpha}^{\beta}\frac{(h^{\prime\prime}(x))^{2}}%
{(1+(h^{\prime}(x))^{2})^{5/2}}dx\,,\label{m6}
\end{align}
where
\[
g(y):=\left\{
{\setstretch{1.3}
\begin{array}
[c]{ll}%
\gamma & \text{if }y>0\,,\\
\gamma_{0} & \text{if }y=0\,,
\end{array}
}
\right.
\]
and similarly for $\mathcal{S}(\alpha_{n},\beta_{n},h_{n})+\gamma_{0}%
(M_{r}-M_{l})$.

By \eqref{m3b}, $\check{h}_{n}^{\prime}(x)\rightarrow\check{h}^{\prime}%
(x)$\ for a.e. $x\in\mathbb{R}$. Moreover,
\[
\liminf_{n\rightarrow\infty}g(\check{h}_{n}(x))\geq g(\check{h}(x))
\]
for every $x\in\mathbb{R}$, since $\check{h}_{n}\rightarrow\check{h}$
uniformly and $g$ is lower semicontinuous in view of the inequality
$\gamma>\gamma_{0}$. Therefore, by Fatou's lemma we have%
\begin{equation}
\liminf_{n\rightarrow\infty}\int_{M_{l}}^{M_{r}}g(\check{h}_{n}(x))\sqrt
{1+(\check{h}_{n}^{\prime}(x))^{2}}dx\geq\int_{M_{l}}^{M_{r}}g(\check
{h}(x))\sqrt{1+(\check{h}^{\prime}(x))^{2}}dx\,. \label{m7}%
\end{equation}
On the other hand for every $[a,b]\subset(\alpha,\beta)$, by \eqref{m3},
\eqref{m3a}, and \eqref{m3b} we can apply a weak-strong lower semicontinuity
theorem (see  \cite[Theorem 2.3.1]{Butt}) to get
\begin{align*}
\liminf_{n\rightarrow\infty}\int_{\alpha_{n}}^{\beta_{n}}\frac{(h_{n}%
^{\prime\prime}(x))^{2}}{(1+(h_{n}^{\prime}(x))^{2})^{5/2}}dx&\geq
\liminf_{n\rightarrow\infty}\int_{a}^{b}\frac{(h_{n}^{\prime\prime}(x))^{2}%
}{(1+(h_{n}^{\prime}(x))^{2})^{5/2}}dx\\
&\geq\int_{a}^{b}\frac{(h^{\prime\prime
}(x))^{2}}{(1+(h^{\prime}(x))^{2})^{5/2}}dx\,.
\end{align*}
Taking the supremum over all $[a,b]\subset(\alpha,\beta)$, we obtain
\[
\liminf_{n\rightarrow\infty}\int_{\alpha_{n}}^{\beta_{n}}\frac{(h_{n}%
^{\prime\prime}(x))^{2}}{(1+(h_{n}^{\prime}(x))^{2})^{5/2}}dx\geq\int_{\alpha
}^{\beta}\frac{(h^{\prime\prime}(x))^{2}}{(1+(h^{\prime}(x))^{2})^{5/2}}dx\,.
\]
Therefore,  from  \eqref{m6}  and  \eqref{m7} we deduce \eqref{m5a}.

Now we prove that%
\begin{equation}
\liminf_{n\rightarrow\infty}\mathcal{E}(\alpha_{n},\beta_{n},h_{n},u_{n}%
)\geq\mathcal{E}(\alpha,\beta,h,u)\,, \label{m8}%
\end{equation}
where $\mathcal{E}$ is defined in \eqref{elastic energy}. Fix an open set
$U\Subset\Omega_{h}$. Since $u_{n}\rightharpoonup u$ weakly in $H^{1}%
(U;\mathbb{R}^{2})$, by \eqref{W} and \eqref{W convex} we have%
\begin{align*}
\liminf_{n\rightarrow\infty}\int_{\Omega_{h_{n}}}W(Eu_{n}(x,y))\,dxdy&\geq
\liminf_{n\rightarrow\infty}\int_{U}W(Eu_{n}(x,y))\,dxdy\\&\geq\int%
_{U}W(Eu(x,y))\,dxdy
\end{align*}
and letting $U\nearrow\Omega_{h}$ we obtain \eqref{m8}.

Finally, by \eqref{Functional H}, \eqref{H and H0},
\eqref{integral on interval}, and \eqref{m3}, the fact that $h_{n}\rightarrow
h$ uniformly, it follows from Lebesgue's dominated convergence theorem that%
\begin{equation}
\lim_{n\rightarrow\infty}\mathcal{T}_{\tau}(\alpha_{n},\beta_{n},h_{n}%
;\alpha^{0},\beta^{0},h^{0})=\mathcal{T}_{\tau}(\alpha,\beta,h;\alpha
^{0},\beta^{0},h^{0})\,.\nonumber
\end{equation}
This, together with \eqref{total energy}, \eqref{m5a},  and \eqref{m8}, allows
us to conclude that%
\[
\liminf_{n\rightarrow\infty}\mathcal{F}^{0}(\alpha_{n},\beta_{n},h_{n}%
,u_{n})\geq\mathcal{F}^{0}(\alpha,\beta,h,u)\,.
\]
Since $(\alpha,\beta,h,u)\in\mathcal{A}$ and $\{(\alpha_{n},\beta_{n}%
,h_{n},u_{n})\}_{n}$ is a minimizing sequence, we deduce that $(\alpha
,\beta,h,u)$ is a minimizer.

 The proof of \eqref{estimate in H1} can be obtained from Korn's and Poincar\'e's inequalities, arguing as in the estimates for the minimizing sequence.
\end{proof}

\section{Euler--Lagrange Equations}

 Given $(\alpha,\beta,h)\in\mathcal{A}_{s}$ and $\alpha\leq a%
<b\leq\beta$, we define%
\begin{equation}
\Omega_{h}^{a,b} :=\{(x,y)\in\mathbb{R}%
^{2}:\,a<x<b%
\,,\,0<y<h(x)\}\,,\label{Omega alpha beta}
\end{equation}

\begin{theorem}
[Euler-Lagrange equations]\label{theorem euler-lagrange}Let $\tau>0$, let
$(\alpha^{0},\beta^{0},h^{0})\in\mathcal{A}_{s}$,  and let $(\alpha,\beta,h,u)\in\mathcal{A}$ be a
minimizer of the total energy functional
\begin{equation}
\mathcal{F}^{0}(\alpha,\beta,h,u):=\mathcal{S}(\alpha,\beta,h)+\mathcal{E}%
(\alpha,\beta,h,u)+\mathcal{T}_{\tau}(\alpha,\beta,h;\alpha^{0},\beta
^{0},h^{0})\,. \label{total energy 2}%
\end{equation}
Assume that 
\begin{equation}\label{rtyu2}
\operatorname*{Lip}h<L_{0}
\ \text{ and }\ 
h(x)>0\ \text{ for all }x\in(\alpha,\beta) 
\,.
\end{equation}
Then 
\begin{align}
h &\in  C^{4,1}((\alpha,\beta))\cap C^{5}((\alpha,\beta)\setminus\{\alpha^{0},\beta^{0}\})\,,
\label{regularity h}
\\
u &\in    C^{3,1/2}(\overline{\Omega}_{h}^{a,b};\mathbb{R}^{2}) \quad\text{for every } \alpha<a<b <\beta\,,
\label{regularity u}
\end{align}
and $u$ satisfies the elliptic boundary value problem
\begin{equation}
\left\{
 {\setstretch{1.1}
\begin{array}
[c]{ll}%
-\operatorname{div}\mathbb{C}Eu(x,y)
=0 & \text{in }\Omega_{h}%
\,,\\
\mathbb{C}Eu(x,h(x)) \nu^h(x)=0 & \text{for }x\in(\alpha,\beta
)\,,\\
u(x,0)=(e_{0}x,0) & \text{for }x\in(\alpha,\beta)\,,
\end{array}
}
\right.
\label{dirichlet boundary}%
\end{equation}
where $\nu^{h}(x)$ denotes the outer unit normal to $\partial\Omega_{h}$ at
$(x,h(x))$.  Moreover, 
 
\begin{equation}
\frac{1}{\tau}(h-h^{0})=\Big[-\gamma\frac{1}{J^{0}}\Big(\frac{h^{\prime}}%
{J}\Big)^{\prime\prime}+\nu_{0}\frac{1}{J^{0}}\Big(\frac{h^{\prime\prime}%
}{J^{5}}\Big)^{\prime\prime\prime}+\frac{5}{2}\nu_{0}\frac{1}{J^{0}}\Big(\frac
{h^{\prime}(h^{\prime\prime})^{2}}{J^{7}}\Big)^{\prime\prime}+\frac{1}{J^{0}%
}\overline{W}{}^{\prime}\Big]^{\prime}\,, \label{equation h}%
\end{equation}
for every $x\in (\alpha,\beta)\setminus\{\alpha^{0},\beta^{0}\}$, where
\begin{align}
J(x):=&\sqrt{1+(h^{\prime}(x))^{2}}\,,\quad J^{0}(x):=\sqrt{1+((\check{h}%
^{0})^{\prime}(x))^{2}}\,,\nonumber\\&\text{and}\quad
\overline{W}(x)  :=W(Eu(x,h(x)))\,.
\label{J and J0}
\end{align}
\end{theorem}

\begin{proof}
\textbf{Step 1:} We first observe that standard variations with respect to $u$
in $\Omega_{h}$ lead to the weak form of \eqref{dirichlet boundary}.
Since $h\in C^{1,1/2}((\alpha,\beta))$ and $h>0$ in $(\alpha,\beta)$, by  elliptic regularity (see \cite[Theorem 9.3]{agmon-douglis-nirenbergII}),  we have that
\[
 u\in C^{1,1/2}(\overline{\Omega}_{h}^{ a , b };\mathbb{R}^{2}) \quad\text{for every } \alpha<a<b <\beta\,.
\]
Fix $ \alpha<a<b <\beta$ and extend $u$ to a function  defined on $\Omega_{h}\cup (\lbrack a,b \rbrack\times\mathbb{R})$, still denoted by $u$, such that $u\in C^{1}(\lbrack a,b \rbrack\times\mathbb{R}{;\mathbb{R}^2})$. 
Let $\varphi\in C_{c}^{\infty}((\alpha,\beta))$ be such that
$\operatorname*{supp}\varphi\subset\lbrack a,b \rbrack$ and
\begin{equation}
\int_{\alpha}^{\beta}\varphi(x)\,dx=0\,. \label{varphi zero average}%
\end{equation}
Since $\operatorname*{Lip}h<L_{0}$ and $h$ is bounded away from zero on
$\lbrack a,b \rbrack$, we have that $(\alpha,\beta,h+\varepsilon
\varphi)\in\mathcal{A}_{s}$ and 
 $h(x)+\varepsilon\varphi(x)>0$ for all $x\in(\alpha,\beta)$, if
$|\varepsilon|$ is sufficiently small. In turn, $(\alpha,\beta,h+\varepsilon
\varphi,\left.  u\right\vert _{\Omega_{h+\varepsilon\varphi}})\in\mathcal{A}$.
Taking the derivative with respect to $\varepsilon$ of $\mathcal{F}^{0}%
(\alpha,\beta,h+\varepsilon\varphi,\left.  u\right\vert _{\Omega
_{h+\varepsilon\varphi}})$ at $\varepsilon=0$, we obtain%
\begin{align}
&  \gamma\int_{\alpha}^{\beta}\frac{h^{\prime}(x)\varphi^{\prime}%
(x)}{(1+(h^{\prime}(x))^{2})^{1/2}}dx+\nu_{0}\int_{\alpha}^{\beta}%
\frac{h^{\prime\prime}(x)\varphi^{\prime\prime}(x)}{(1+(h^{\prime}%
(x))^{2})^{5/2}}dx\nonumber\\&-\frac{5}{2}\nu_{0}\int_{\alpha}^{\beta}\frac{h^{\prime}%
(x)(h^{\prime\prime}(x))^{2}\varphi^{\prime}(x)}{(1+(h^{\prime}(x))^{2}%
)^{7/2}}dx\label{e-l1}+\int_{\alpha}^{\beta}W(Eu(x,h(x)))\varphi(x)\,dx\\
&  +\frac{1}{\tau}%
\int_{\mathbb{R}}(H(x)-H^{0}(x))\sqrt{1+((\check{h}^{0})^{\prime}(x))^{2}%
}\Big(\int_{-\infty}^{x}\varphi(s)\,ds\Big)dx=0\,.\nonumber
\end{align}
Since $\varphi$ belongs to $C_{c}^{\infty}(\mathbb{R})$  and satisfies \eqref{varphi zero average}, using integration  by
parts the integral over $\mathbb{R}$ becomes
\begin{equation}\label{3.10 bis}
-\frac{1}{\tau}\int_{\alpha}^{\beta}\Big(\int_{-\infty}^{x}(H(s)-H^{0}%
(s))\sqrt{1+((\check{h}^{0})^{\prime}(s))^{2}}\,ds\Big)\varphi(x)\,dx\,.
\end{equation}
This allows us to rewrite \eqref{e-l1} as%
\[
\int_{\alpha}^{\beta}( A h^{\prime\prime}\varphi^{\prime\prime}+ B h^{\prime
}\varphi^{\prime}+f\varphi)\,dx=0
\]
for all $\varphi\in C_{c}^{\infty}((\alpha,\beta))$ with $\operatorname*{supp}%
\varphi\subset\lbrack a,b \rbrack$ satisfying
\eqref{varphi zero average}, where
\begin{align}
 A(x)  &  :=\frac{\nu_{0}}{(1+(h^{\prime}(x))^{2})^{5/2}}\,,
\label{function a}\\
 B(x)  &  :=\gamma\frac{1}{(1+(h^{\prime}(x))^{2})^{1/2}}-\frac{5}{2}\nu_{0}%
\frac{(h^{\prime\prime}(x))^{2}}{(1+(h^{\prime})^{2})^{7/2}}%
\,,\label{function b}\\
f(x)  &  :=W(Eu(x,h(x)))-\frac{1}{\tau}\int_{\alpha}^{x}(H(s)-H^{0}%
(s))\sqrt{1+((\check{h}^{0})^{\prime}(s))^{2}}ds\,. \label{function f}%
\end{align}
By introducing a Lagrange multiplier $ m $ for \eqref{varphi zero average}, %
 we obtain
\[
\int_{ a}^{b } ( A h^{\prime\prime}\varphi^{\prime\prime
}+ B h^{\prime}\varphi^{\prime}+f\varphi)\,dx= m  \int_{ a}%
^{b }\varphi\,dx
\]
for all $\varphi\in C_{c}^{\infty}((  a, b))$.

Note that $ A\in C^{0,1/2}([\alpha,\beta])$, $ A\geq a_{0}$ for some constant
$a_{0}>0$, $ B\in L^{1}((\alpha,\beta))$, and $f\in C^{0}({[ a,b]})$
since $u\in C^{1}(\lbrack a,b\rbrack\times\mathbb{R})$.  Let us fix $x_{0}\in(a,b)$.
 Integrating by parts once we get%
\begin{equation}
\int_{a}^{b} ( A h^{\prime\prime}\varphi^{\prime\prime
}+( B h^{\prime}-F_{ m  })\varphi^{\prime})\,dx=0\,, \label{eul 1}%
\end{equation}
where
\begin{equation}
\,F_{ m  }(x):=\int_{x_{0}}^{x}(f(s)- m  )\,ds\,.
\label{function F}%
\end{equation}
For every $\psi\in C_{c}^{\infty}(( a, b))$ with
$\int_{a}^{b} \psi\,dx=0$, setting $\varphi(x):=\int%
_{a}^{x}\psi\,ds$, from \eqref{eul 1} we obtain%
\[
\int_{a}^{b} ( A h^{\prime\prime}\psi^{\prime}+( B h^{\prime
}-\,F_{ m  })\psi)\,dx=0\,.
\]
By the arbitrariness of $\psi$ we obtain
\begin{equation}
 A(x)h^{\prime\prime}(x)=\int_{x_{0}}^{x}( B(s)h^{\prime
}(s)-\,F_{ m  }(s))\,ds+{ p_{1}}  \label{eul 2}%
\end{equation}
for all $x\in\lbrack a, b]$, where { $ p_{1} $ is a polynomial of degree one}  associated with the constraint $\int_{a}^{b} %
\psi\,dx=0$. Observe that since $A$, $B$, and $\,F_{ m  }$ are independent
of $a$ and $b$, with $\alpha<a\le x_{0}\leq b
<\beta$, so is {$ p_{1} $}, hence \eqref{eul 2} holds for all $x\in (\alpha,\beta)$. Since $ A\in C^{0,1/2}([\alpha,\beta])$, $ A\geq
a_{0}$, $ B\in L^{1}((\alpha,\beta))$, and $F_{ m  }\in C^{1}%
( (\alpha,\beta))$, \eqref{eul 2} implies that $h^{\prime\prime}\in
 C^{0}((\alpha,\beta))$. In turn, this gives $ A\in C^{1}((\alpha,\beta))$ and $ B\in  C^{0}((\alpha,\beta))$. Hence, the
right-hand side of \eqref{eul 2} is $C^{1}((\alpha,\beta))$,
therefore $h^{\prime\prime}\in C^{1}((\alpha,\beta))$. In turn,
$ A\in C^{2}((\alpha,\beta))$ and $ B\in C^{1}((\alpha,\beta))$ by \eqref{function a} and \eqref{function b}, and so $h^{\prime\prime}\in C^{2}((\alpha,\beta))$ by
\eqref{eul 2}.

By elliptic regularity (\cite[Theorem 9.3]{agmon-douglis-nirenbergII}) it
follows that $u\in C^{3,1/2}(\overline{\Omega}_{h}^{a,b})$  for every $\alpha<a<b<\beta$,
 which gives \eqref{regularity u}. In turn, \eqref{dirichlet boundary} hold in the classical sense.

By \eqref{regularity u} and \eqref{function f} we have 
$f\in C^{0,1}((\alpha,\beta))\cap C^{1}((\alpha,\beta)\setminus\{\alpha^{0},\beta^{0}\})$. Hence, $\,F_{ m  }\in
C^{1,1}((\alpha,\beta))\cap C^{2}((\alpha,\beta)\setminus\{\alpha^{0},\beta^{0}\})$ by \eqref{function F}. Moreover, again by
\eqref{function a} and \eqref{function b}, $ A \in C^{3}((\alpha,\beta))$ and $ B \in C^{2}((\alpha,\beta))$ and,  by \eqref{eul 2},
$h^{\prime\prime}\in C^{2,1}((\alpha,\beta))\cap C^{3}((\alpha,\beta)\setminus\{\alpha^{0},\beta^{0}\})$, which proves \eqref{regularity h}. 

\textbf{Step 2:} We prove \eqref{equation h}. Fix $ \alpha<a<b<\beta$ and
let $\varphi\in C_{c}^{\infty}((\alpha,\beta))$ be such that
$\operatorname*{supp}\varphi\subset\lbrack a, b]$ and
\eqref{varphi zero average} holds. Define%
\begin{equation}
\overline{H}(x):=\int_{-\infty}^{x}(H(s)-H^{0}(s))\sqrt{1+((\check{h}^{0})^{\prime
}(s))^{2}}\,ds\,, \label{H bar}%
\end{equation}
By \eqref{J and J0}, \eqref{e-l1}, and \eqref{3.10 bis}, we have
\[
\gamma\int_{a}^{b} \frac{h^{\prime}\varphi^{\prime}}%
{J}dx+\nu_{0}\int_{a}^{b} \frac{h^{\prime\prime}%
\varphi^{\prime\prime}}{J^{5}}dx-\frac{5}{2}\nu_{0}\int_{a}^{b} %
\frac{h^{\prime}(h^{\prime\prime})^{2}\varphi^{\prime}}{J^{7}}dx+\int_{a}^{b}\overline{W}\varphi\,dx-\frac{1}{\tau}\int_{a}^{b}\overline{H}\varphi dx=0\,,
\]
 where $J$ is defined in \eqref{J and J0}. 
We integrate by parts once the integrals containing $\varphi^{\prime}$ and
twice the integral containing $\varphi^{\prime\prime}$ to obtain
\begin{align*}
-\gamma\int_{a}^{b} \Big(\frac{h^{\prime}}{J}\Big)^{\prime
}\varphi dx+\nu_{0}\int_{a}^{b} \Big(\frac{h^{\prime\prime
}}{J^{5}}\Big)^{\prime\prime}\varphi\,dx&+\frac{5}{2}\nu_{0}
\int_{a}^{b}\Big(\frac{h^{\prime}(h^{\prime\prime})^{2}}{J^{7}}\Big)^{\prime
}\varphi dx\\&+\int_{a}^{b} \overline{W}\varphi\,dx-\frac{1}{\tau
}\int_{a}^{b} \overline{H}\varphi dx=0\,,
\end{align*}
By the arbitrariness of $\varphi\in C_{c}^{\infty}(( a, b))$
satisfying \eqref{varphi zero average} we  obtain  
\begin{equation}
-\gamma\Big(\frac{h^{\prime}}{J}\Big)^{\prime}+\nu_{0}\Big(\frac
{h^{\prime\prime}}{J^{5}}\Big)^{\prime\prime}+\frac{5}{2}\nu_{0}\Big(\frac{h^{\prime
}(h^{\prime\prime})^{2}}{J^{7}}\Big)^{\prime}+\overline{W}-\frac{1}{\tau}\overline{H}= m   \label{equation fourth}%
\end{equation}
where $ m  $ is a Lagrange multiplier due to \eqref{varphi zero average}.
Using \eqref{H bar} and differentiating
\[
-\gamma\Big(\frac{h^{\prime}}{J}\Big)^{\prime\prime}+\nu_{0}\Big(\frac
{h^{\prime\prime}}{J^{5}}\Big)^{\prime\prime\prime}+\frac{5}{2}\nu_{0}\Big(\frac
{h^{\prime}(h^{\prime\prime})^{2}}{J^{7}}\Big)^{\prime\prime}+\overline{W}^{\prime
}-\frac{1}{\tau}(H-H^{0})J^{0}=0\,,
\]
 where $J_{0}$ is defined in \eqref{J and J0}.  Dividing by $J^{0}$ and differentiating once again  we obtain \eqref{equation h}.
\end{proof}

\begin{remark}
\label{remark lame}From the special form of $\mathbb{C}$, it follows that
\begin{align*}
\mu\Delta u+(\lambda+\mu)\nabla\operatorname{div}u  &  =0\quad\text{in }%
\Omega_{h}\,,\\
2\mu(Eu)\nu^{h}+\lambda(\operatorname{div}u)\nu^{h} &  =0\quad\text{on }\Gamma_{h}\,,
\end{align*}
where $\Gamma_{h}$ is the graph of $h$ in $(\alpha,\beta)$ and $\nu^{h}$ is the
outward  unit  normal.
\end{remark}

Given $(\alpha,\beta,h)\in\mathcal{A}_{s}$, $\alpha\leq a %
< b\leq\beta$, and $\eta>0$ we define%
\begin{equation}
\Omega_{h,\eta}^{ a,b}  :=\{(x,y)\in
\mathbb{R}^{2}:\, a <x< b\,,\,h(x)-\eta
<y<h(x)\}\,.\nonumber
\end{equation}

\begin{theorem}
\label{theorem regularity} Let $(\alpha,\beta,h)\in\mathcal{A}_{s}$, let $u\in\mathcal{A}_{e}(\alpha,\beta,h)$ be the minimizer of the functional $\mathcal{E}(\alpha,\beta,h,\cdot)$ defined in \eqref{elastic energy}, let $\delta>0$, $\eta>0$, and $M>1$.
Assume that there exist $\alpha < a<b< \beta$, 
with $ b-a>4\delta$, such that
\begin{align}
h(x)  &  \geq2\eta\,\quad\text{for all }x\in\lbrack a
, b]\,,\label{r01}\\
&  \int_{\alpha}^{\beta}|h^{\prime\prime}(x)|^{2}dx\leq M\,. \label{r02}%
\end{align}
Then  there exists a constant $C=C\left(
\delta,\eta,M\right)  >0$ (independent of $\alpha$, $\beta$, 
$a$, $b$, $h$, $h_{0}$, and $\tau$) such that%
\begin{equation}
\Vert u\Vert_{C^{1,1/2}(\overline{\Omega}_{h,\eta}^{ a%
+\delta,b-\delta})}\leq C\,. \label{eul 45}
\end{equation}
\end{theorem}

\begin{proof}
In what follows $C$ denotes a
positive constant whose value changes from  formula  to  formula  and which depends only
on $\delta$, $\eta$, $M$ and the fixed parameters $\lambda$, $\mu$,
$\gamma$, $\gamma_{0}$, $\sigma_{0}$, $A_{0}$, $e_{0}$, $L_{0}$, and $\nu_{0}$
of the problem. 
 Let $\Omega_{-h}:=\{(x,y)\in\mathbb{R}^{2}:\,\alpha<x<\beta\,,\,-h(x)<y<0\}$ and let $v$ be the function defined by $v(x,y):=u(x,y+h(x))$ for every $(x,y)\in \Omega_{-h}$.
Then $v\in H^1(\Omega_{-h};\mathbb{R}^{2})$ and by \eqref{estimate in H1} we have
\begin{equation}
\Vert v \Vert_{H^{1}(\Omega_{-h})}\le C\,.
\label{estimate v in H1}%
\end{equation}
 It can be shown that $v$ is a weak
solution to the  boundary value problem%
\begin{equation}
\left\{
\begin{array}
[c]{ll}%
\operatorname{div}(\mathbb{A}\nabla v)=0 & \text{in }\Omega_{-h},\\
(\mathbb{A}\nabla v)e_{2}=0 & \text{on }(\alpha,\beta)\times\{0\}\,,
\end{array}
\right.  \label{equation A}%
\end{equation}
where%
\begin{align}
(\mathbb{A}\xi)_{11}  &  =2\mu(\xi_{11}-\xi_{12}h^{\prime})+\lambda(\xi
_{11}-\xi_{12}h^{\prime}+\xi_{22})\,,\nonumber\\
(\mathbb{A}\xi)_{12}  &  =-2\mu(\xi_{11}-\xi_{12}h^{\prime})h^{\prime}%
+\lambda(\xi_{11}-\xi_{12}h^{\prime}+\xi_{22})h^{\prime}+\mu\left(  \xi
_{12}+\xi_{21}-\xi_{22}h^{\prime}\right)  \,,\label{coeff A}\\
(\mathbb{A}\xi)_{21}  &  =\mu\left(  \xi_{21}+\xi_{12}-\xi_{22}h^{\prime
}\right)  \,,\nonumber\\
(\mathbb{A}\xi)_{22}  &  =-\mu\left(  \xi_{21}+\xi_{12}-\xi_{22}h^{\prime
}\right)  h^{\prime}+2\mu\xi_{22}+\lambda(\xi_{11}+\xi_{22}-\xi_{12}h^{\prime
})\,.\nonumber
\end{align}

Define $\Omega_{-h}^{ a,b}:=\{(x,y)\in\mathbb{R}%
^{2}:\, a<x< b\,,\,-h(x)<y<0\}$. Since $h\in
C^{1,1/2}([\alpha,\beta])$ in view of \eqref{r02}, we have that%
\begin{equation}
\Vert\mathbb{A}\Vert_{C^{0,1/2}(\overline{\Omega}_{-h}^{ a,b
})}\leq C\,. \label{coefficients}%
\end{equation}
Moreover, recalling that $\mathbb{C}$ satisfies the Legendre--Hadamard
condition and the latter is preserved by changes of variables, we deduce that
$\mathbb{A}$ satisfies the Legendre--Hadamard condition. Hence, we are in a
position to apply Theorem 9.3 in \cite{agmon-douglis-nirenbergII} in each
rectangle of the form $(x-2\delta,x+2\delta)\times(-2\eta,0)$ with
$ a +2\delta\leq x\leq b-2\delta$. Using \eqref{estimate v in H1} we  obtain that $v\in
C^{1,1/2}([x-\delta,x+\delta]\times\lbrack-\eta,0];\mathbb{R}^{2})$ with
$\Vert v\Vert_{C^{1,1/2}([x-\delta,x+\delta]\times\lbrack-\eta,0])}\leq
C$. This implies that $\Vert v\Vert_{C^{1,1/2}([ a+\delta
, b-\delta]\times\lbrack-\eta,0])}\leq C$ and,  in turn,
\eqref{eul 45}. 
\end{proof}

 Let us define
\begin{equation}
B_{\tau}(H,{H^{0}},\alpha,{\alpha^{0}}, \beta,{\beta^{0}}):=1+\frac{1}{\tau}\int_{\mathbb{R}}%
|H-H^{0}|\,dx+\frac{1}{\tau}|\alpha-{\alpha^{0}}|
+\frac{1}{\tau}|\beta-{\beta^{0}}|\,, \label{B tau alpha beta}%
\end{equation}
where $H$ and $H^{0}$ are defined in \eqref{H and H0}. 

\begin{theorem}
\label{theorem C3}Under the assumptions of Theorem
\ref{theorem euler-lagrange}, suppose in addition that there exist
$0<\eta_{0}<1$, $0<\eta_{1}<1$,  and $M>1$  such that%
\begin{gather}
2\eta_{0}\leq h^{\prime}(\alpha)\,,\quad h^{\prime}(\beta)\leq-2\eta_{0}
\,,\,\label{h' alpha and beta}\\
h(x)\geq 2\eta_{1}\,\quad\text{for all }x\in\lbrack\alpha+\delta_{0}%
,\beta-\delta_{0}]\,,\label{h bounded from zero}\\
\int_{\alpha}^{\beta}|h^{\prime\prime}(x)|^{2}dx\leq M\,, \label{h'' L2 norm}%
\end{gather}
where
\begin{equation}
\delta_{0}:= \eta_{0}^{2}/(4M)<1/4\,. \label{delta0}%
\end{equation}
Then $W(Eu(\cdot,h(\cdot)))\in L^{1}((\alpha,\beta))$, $h\in W^{4,1}%
((\alpha,\beta))$, and there exists a constant $ c_{0} = c_{0} (\eta_{0},\eta_{1},M)>1$ 
(independent of $\alpha$, $\beta$, $h$, $h_{0}$, and $\tau$) such that%
\begin{align}
\int_{\alpha}^{\beta}W(Eu(x,h(x)))\,dx & \leq  c_{0} B_{\tau}(H,{H^{0}},\alpha,{\alpha^{0}}, \beta,{\beta^{0}})  \,,
\label{trace L1}%
\\
\Vert h^{\prime\prime}\Vert_{L^{\infty}((\alpha,\beta))}  &  \leq  c_{0} B_{\tau}(H,{H^{0}},\alpha,{\alpha^{0}}, \beta,{\beta^{0}}) \,,\label{h'' bound}\\
\Vert h^{\prime\prime\prime}\Vert_{L^{\infty}((\alpha,\beta))}  &  \leq
 c_{0} B_{\tau}(H,{H^{0}},\alpha,{\alpha^{0}},  \beta,{\beta^{0}})^{2} ,\label{h''' bound better}\\
\Vert h^{\prime\prime\prime}\Vert_{L^{1}((\alpha,\beta))}  &  \leq  c_{0} B_{\tau}(H,{H^{0}},\alpha,{\alpha^{0}}, \beta,{\beta^{0}})   \,,\label{h''' bound L1}\\
\Vert h^{(\operatorname*{iv})}\Vert_{L^{1}((\alpha,\beta))}  &  \leq  c_{0} B_{\tau}(H,{H^{0}},\alpha,{\alpha^{0}},\beta,{\beta^{0}})^{2}. \label{h iv bound better}%
\end{align}
\end{theorem}

 Unless otherwise indicated, in the proofs of the rest of the paper $C$ denotes a constant
depending only on the constants $\eta_{0}$, $\eta_{1}$, and $M$ of the previous theorem (and on the structural constants $e_{0}$, $\lambda$, $\mu$,
 and $L_{0}$, and possibly on the exponents considered in the statements). The value of $C$ can change from formula to formula.

In the proof of Theorem \ref{theorem C3} we use the following estimate.

\begin{lemma}\label{qwe1}
Let $\eta_{0}>0$, $M>0$, $\alpha<\beta$, and $h\in H^{2}((\alpha,\beta))\cap H_{0}^{1}((\alpha,\beta))$.
Assume that
\begin{gather}
2\eta_{0}\leq h^{\prime}(\alpha)\,,\quad h^{\prime}(\beta)\leq-2\eta_{0}
\,,\,\label{h' alpha and beta 11}\\
\int_{\alpha}^{\beta}|h^{\prime\prime}(x)|^{2}dx\leq M\,. \label{h'' L2 norm 11}%
\end{gather}
Then
\begin{equation}
\beta-\alpha\geq 16\eta_{0}^{2}/M\,.
\label{qwe2}
\end{equation}
\end{lemma}
\begin{proof}
By the fundamental theorem of
calculus and by \eqref{h' alpha and beta 11} for every $x\in(\alpha,\beta)$ we have
\begin{equation}
h^{\prime}(x)=h^{\prime}(\alpha)+\int_{\alpha}^{x}h^{\prime\prime}(s)\,ds
\geq 2\eta_{0} - \int_{\alpha}^{x}|h^{\prime\prime}(s)|\,ds\,.
\label{qwe3}
\end{equation}
We claim that there exists $x_{\alpha}\in(\alpha,\beta)$ such that
\begin{equation}
\int_{\alpha}^{x_{\alpha}}|h^{\prime\prime}(s)|\,ds=2\eta_{0}\quad\text{and}\quad \int_{\alpha}^{x}|h^{\prime\prime}(s)|\,ds<2\eta_{0}
\quad \text{for }x\in(\alpha,x_{\alpha})\,.
\label{qwe4}
\end{equation}
If not, by \eqref{qwe3} we would have $h^{\prime}(x)>0$ for every $x\in(\alpha,\beta)$, 
which contradicts the assumption $h\in H_{0}^{1}((\alpha,\beta))$ and proves the claim.
By \eqref{qwe3} and \eqref{qwe4} we have $h^{\prime}(x)>0$ for every $x\in(\alpha,x_{\alpha})$.

In the same way we prove that there exists $x_{\beta}\in(\alpha,\beta)$ such that
\begin{equation}
\int_{x_{\beta}}^{\beta}|h^{\prime\prime}(s)|\,ds=2\eta_{0}\quad\text{and}\quad h^{\prime}(x)<0
\quad \text{for }x\in(x_{\beta},\beta)\,.
\label{qwe5}
\end{equation}
Since $x_{\alpha}\leq x_{\beta}$, from \eqref{qwe4} and \eqref{qwe5} we deduce that
\[
4\eta_{0}\leq\int_{\alpha}^{\beta}|h^{\prime\prime}(s)|\,ds\,.
\]
By H\"older's inequaliy we get
\[
4\eta_{0}\leq M^{1/2}(\beta-\alpha)^{1/2}\,,
\]
which gives \eqref{qwe2}.
\end{proof}

\begin{proof}[Proof of Theorem \ref{theorem C3}]
\textbf{Step 1: A Variation of }$\boldsymbol{(\alpha,\beta,h,u)}$.\ Extend $h$
to  a  function $ h\in H^{2}((\alpha-1,\beta))$ by setting $ h(x):=h^{\prime}(\alpha)(x-\alpha)$ for $x\in(\alpha-1,\alpha]$. Using
H\"{o}lder's inequality for $x\in(\alpha,\beta)$ we have%
\begin{equation}
h^{\prime}(x)\geq h^{\prime}(\alpha)-(x-\alpha)^{1/2}\left(  \int_{\alpha}%
^{x}|h^{\prime\prime}(s)|^{2}ds\right)  ^{1/2}\geq2\eta_{0}-(x-\alpha
)^{1/2}M^{1/2}\geq\eta_{0}
\label{qyp1}
\end{equation}
provided $x-\alpha\leq\eta_{0}^{2}/M$.


Then,  using also the fact that $ h(x):=h^{\prime}(\alpha)(x-\alpha)$ for
$x\in(\alpha-1,\alpha]$,  we have%
\begin{equation}
\eta_{0}\leq h^{\prime}(x)\leq L_{0}\quad\text{for
every }x\in\lbrack\alpha-\delta_{0},\alpha+2\delta_{0}]\,. \label{h' bounded}%
\end{equation}
Since $h(\alpha)=0$  and $\alpha+4\delta_{0}<\beta$ by Lemma \ref{qwe1},  we obtain%
\begin{align}
\eta_{0}(x-\alpha)  &  \leq h(x)\quad\text{for every }x\in\lbrack\alpha
,\alpha+4\delta_{0}]\,,\label{h bounded}\\
| h(x)|  &  \leq L_{0}|x-\alpha|\quad\text{for
every }x\in\lbrack\alpha-\delta_{0},\alpha+2\delta_{0}]\,.
\label{h hat bounded}%
\end{align}

Take $\varphi_{0}\in C^{\infty}(\mathbb{R})$
with $\varphi_{0}(0)=1$, $\varphi_{0}(x)\geq 1/2$ for every $x\in[-\delta_{0}/2, \delta_{0}/2]$, $\int_{0}^{\delta_{0}}\varphi_{0}\,dx=0$, 
and $\operatorname*{supp}\varphi_{0}\subset(-\delta_{0},\delta_{0})$. Define
\begin{equation}
\varphi(x):=\varphi_{0}(x-\alpha)\,,\quad x\in\mathbb{R}\,.
\label{test function varphi}%
\end{equation}
Let $\varepsilon_{0}:=\min\{1,\frac{1}{2}\delta_{0}\eta_{0}/\Vert\varphi_{0}%
\Vert_{C^{1}}\}$. Then for every $\varepsilon\in\mathbb{R}$ with
$|\varepsilon|\leq\varepsilon_{0}$ we have%
\begin{align}
 h^{\prime}(x)+\varepsilon\varphi^{\prime}(x)  &  \geq\eta_{0}/2\,\quad\text{for all }x\in\lbrack\alpha-\delta_{0},\alpha+
\delta_{0}]\,,\label{100}\\
 h(\alpha-\delta_{0})+\varepsilon\varphi(\alpha-\delta_{0})  &
<0< h(\alpha+\delta_{0})+\varepsilon\varphi(\alpha+\delta_{0}%
)\,.\nonumber
\end{align}
This implies that there exists a unique $\alpha_{\varepsilon}$ such that
\begin{equation}
\alpha_{\varepsilon}\in(\alpha-\delta_{0},\alpha+\delta_{0})
\quad\hbox{and}\quad h(\alpha_{\varepsilon})+\varepsilon\varphi(\alpha_{\varepsilon})=0\,.
\label{102}%
\end{equation}
Moreover, by the Implicit Function Theorem the function $\varepsilon
\mapsto\alpha_{\varepsilon}$ is of class $C^{1}$ and so we can differentiate
the previous identity to get
\begin{equation}
\frac{d\alpha_{\varepsilon}}{d\varepsilon}=-\frac{\varphi(\alpha_{\varepsilon
})}{ h^{\prime}(\alpha_{\varepsilon})+\varepsilon\varphi^{\prime}%
(\alpha_{\varepsilon})}\,\,. \label{derivative endpoint l}%
\end{equation}
By \eqref{100},%
\[
\left\vert \frac{d\alpha_{\varepsilon}}{d\varepsilon}\right\vert \leq
\frac{2\Vert\varphi_{0}\Vert_{C^{0}}}{\eta_{0}}\quad\text{for all }\varepsilon
\in\lbrack-\varepsilon_{0},\varepsilon_{0}]\,,
\]
and, since ${\alpha^{0}}=\alpha$, this implies that%
\begin{equation}
|\alpha_{\varepsilon}-\alpha|\leq\frac{2\Vert\varphi_{0}\Vert_{C^{0}}}{\eta_{0}
}|\varepsilon|\,. \label{alpha epsilon lips}%
\end{equation}

Since 
$\varphi_{0}(x)\geq 1/2$ for $|x|\leq\delta_{0}/2$, we have

$\varphi(x)\geq 1/2$ for $|x-\alpha|\leq\delta_{0}/2$. This, together with
\eqref{alpha epsilon lips} implies that $\varphi(\alpha_{\varepsilon}%
)\geq\frac{1}{2}$ for $|\varepsilon|\leq\varepsilon_{1}$, where $\varepsilon
_{1}:=\min\{\delta_{0}\eta_{0}/(4\Vert\varphi_{0}\Vert_{C^{0}}),\varepsilon_{0}\}$. In
turn, by \eqref{100} and \eqref{derivative endpoint l},%
\begin{equation}
\frac{d\alpha_{\varepsilon}}{d\varepsilon}<0\quad\text{for }|\varepsilon
|\leq\varepsilon_{1}\,.\label{derivative negative}%
\end{equation}

Observe that \eqref{100} implies that $ h+\varepsilon\varphi>0$ in
$(\alpha_{\varepsilon},\alpha+\delta_{0})$. On the other hand $ h+\varepsilon\varphi=h\geq0$ on $[\alpha+\delta_{0},\beta]$.

In order to satisfy  the area constraint  \eqref{area constraint} we fix a function $\psi_{0}\in
C^{\infty}(\mathbb{R})$ such that $\operatorname*{supp}\psi_{0}\subset
(\delta_{0},2\delta_{0})$ 
  and $\int_{\delta_{0}}^{2\delta_{0}}\psi
_{0}\,dx=1$, and we define $\psi(x):=\psi_{0}(x-\alpha)$, $x\in\mathbb{R}$.
Consider the function $h_{\varepsilon}\in H^{2}((\alpha_{\varepsilon}%
,\beta))\cap H_{0}^{1}((\alpha_{\varepsilon},\beta))$ defined by%
\begin{equation}
h_{\varepsilon}(x):= h(x)+\varepsilon\varphi(x)+\omega_{\varepsilon}%
\psi(x)\,, \label{h epsilon}%
\end{equation}
where $\omega_{\varepsilon}\in\mathbb{R}$ is the unique constant such that%
\begin{equation}
\int_{\alpha_{\varepsilon}}^{\beta}h_{\varepsilon}\,dx=A_{0}\,.
\label{average h epsilon}%
\end{equation}
 Since  $\int_{\alpha}^{\beta}%
\varphi\,dx=0$, from \eqref{area constraint} and \eqref{average h epsilon} it follows that 
\[
\int_{\alpha_{\varepsilon}}^{\beta}( h(x)+\varepsilon\varphi
(x)+\omega_{\varepsilon}\psi(x))\,dx=A_{0}=\int_{\alpha}^{\beta}%
(h(x)+\varepsilon\varphi(x))\,dx\,.
\]
 Hence, using the fact that $\int_{\alpha_{\varepsilon}}^{\beta}%
\psi\,dx=1$, we get
\[
\int_{\alpha}^{\beta}( h(x)+\varepsilon\varphi(x))\,dx+\int%
_{\alpha_{\varepsilon}}^{\alpha}( h(x)+\varepsilon\varphi(x))\,dx+\omega
_{\varepsilon} =\int_{\alpha}^{\beta}(h(x)+\varepsilon\varphi(x))\,dx\,,
\]
and so
\[
\frac{\omega_{\varepsilon}}{\varepsilon}=-\frac{1}{\varepsilon}%
\int_{\alpha_{\varepsilon}}^{\alpha}h(x)\,dx-\int_{\alpha_{\varepsilon}%
}^{\alpha}\varphi(x)\,dx\,.
\]
To estimate the right-hand side we use \eqref{alpha epsilon lips} to get
\[
\left\vert \int_{\alpha_{\varepsilon}}^{\alpha}\varphi(x)\,dx\right\vert
\leq\frac{2\Vert\varphi_{0}\Vert_{C^{0}}^{2}}{\eta_{0}}|\varepsilon|\,,
\]
while, using also  \eqref{h hat bounded},  we obtain%
\[
\left\vert \int_{\alpha_{\varepsilon}}^{\alpha} h(x)\,dx\right\vert
\leq\frac{4\Vert\varphi_{0}\Vert_{C^{0}}^{2}}{\eta_{0}^{2}}\varepsilon^{2}%
(L_{0}+2^{1/2}\delta_{0}^{1/2}M^{1/2})\,.
\]
Combining these inequalities we have%
\begin{equation}
\left\vert \frac{\omega_{\varepsilon}}{\varepsilon}\right\vert \leq
 C|\varepsilon|\,. \label{omega eps go to zero}%
\end{equation}

Let $\varepsilon_{2}:=\min\{\frac{\eta_{0}}{2\Vert\varphi_{0}\Vert_{C^{1}}%
+2 C\Vert\psi_{0}\Vert_{C^{1}}},\varepsilon_{1}\}$. We claim that for all
$|\varepsilon|\leq\varepsilon_{2}$  we have %
\begin{equation}
h_{\varepsilon}^{\prime}\geq\eta_{0}/2\quad\text{in }(\alpha
_{\varepsilon},\alpha+2\delta_{0})\,. \label{h epsilon prime large}%
\end{equation}
Fix $|\varepsilon|\leq\varepsilon_{2}$. Then by \eqref{h' bounded},%
\[
h_{\varepsilon}^{\prime}   = h^{\prime}+\varepsilon\varphi^{\prime
}+\omega_{\varepsilon}\psi^{\prime}\geq\eta_{0}-|\varepsilon|\Vert\varphi_{0}%
\Vert_{C^{1}}- C\varepsilon^{2}\Vert\psi_{0}\Vert_{C^{1}}  \geq\eta_{0}-|\varepsilon|(\Vert\varphi_{0}\Vert_{C^{1}}+ C\Vert\psi
_{0}\Vert_{C^{1}})\geq\eta_{0}/2\,,
\]
which proves the claim.

 Since $\operatorname*{supp}\psi\subset (\alpha+\delta_{0},\alpha+2\delta_{0})$ and $\alpha_{\varepsilon}\in(\alpha
-\delta_{0},\alpha+\delta_{0})$, we have $\psi(\alpha_{\varepsilon})=0$. By \eqref{102} and \eqref{h epsilon}, this  gives  $h_{\varepsilon}(\alpha_{\varepsilon})=0$,  hence  \eqref{h epsilon prime large}
implies that $h_{\varepsilon}>0$ in $(\alpha_{\varepsilon},\alpha+2\delta_{0}%
)$. Moreover, since $\operatorname*{supp}\varphi\subset(\alpha-\delta
_{0},\alpha+\delta_{0})$ and $\operatorname*{supp}\psi\subset(\alpha
+\delta_{0},\alpha+2\delta_{0})$ we have that
\begin{equation}
h_{\varepsilon}=h\quad\text{on }[\alpha+2\delta_{0},\beta]\,.
\label{h epsilon =h}%
\end{equation}
We conclude that $h_{\varepsilon}\geq0$ in $(\alpha_{\varepsilon},\beta)$. In
turn, also by \eqref{average h epsilon}, $(\alpha_{\varepsilon},\beta
,h_{\varepsilon})\in\mathcal{A}_{s}$.

 Let $U$ be the interior of the set $\overline\Omega_{h}\cup([\alpha-\delta_{0},\beta]\times[-1,0])$ and let $\hat{u}\colon\overline U\to\mathbb{R}^{2}$ be the function defined by
\[
\hat{u}(x,y):=\left\{\setstretch{1.1}
\begin{array}
[c]{ll}%
u(x,y) & \text{if }(x,y)\in \overline{\Omega}_{h}\,,\\
(e_{0}x,0) & \text{if }(x,y)\in[\alpha-\delta_{0},\beta]\times[-1,0]\,.
\end{array}
\right.
\]
Since the two definitions match on $[\alpha,\beta]\times\{0\}$, we have $\hat{u}\in H^1(U;\mathbb{R}^{2})$.
Recalling that $U$ has Lipschitz boundary, we can extend $\hat{u}$ to a function defined on $\mathbb{R}^{2}$, still denoted by $\hat{u}$, such that $\hat{u}\in H^1(\mathbb{R}^{2};\mathbb{R}^{2})$.

Hence, if we let $u_{\varepsilon}$ be the
restriction of $\hat{u}$ to $\Omega_{h_{\varepsilon}}$,  we have $u_{\varepsilon}(x,0)=(e_{0}x,0)$ for a.e.\ $x\in(\alpha
_{\varepsilon},\beta)$, which gives  $(\alpha
_{\varepsilon},\beta,h_{\varepsilon},u_{\varepsilon}%
)\in\mathcal{A}$.

\noindent\textbf{Step 2: Evaluation of derivatives.} We claim that
\begin{align}
\liminf_{\varepsilon\rightarrow0}&\frac{\mathcal{S}(\alpha_{\varepsilon}%
,\beta,h_{\varepsilon})-\mathcal{S}(\alpha,\beta,h)}{\varepsilon}  \nonumber\\&
\geq\gamma\int_{\alpha}^{\beta}\frac{h^{\prime}\varphi^{\prime}}%
{\sqrt{1+(h^{\prime})^{2}}}dx+\gamma\sqrt{1+(h^{\prime}(\alpha))^{2}}\frac
{1}{h^{\prime}(\alpha)}-\gamma_{0}\frac{1}{h^{\prime}(\alpha)}
\nonumber\\
&  \quad+\nu_{0}\int_{\alpha}^{\beta}\frac{h^{\prime\prime}\varphi
^{\prime\prime}}{(1+(h^{\prime})^{2})^{5/2}}dx-\frac{5}{2}\nu_{0}\int_{\alpha}^{\beta
}\frac{h^{\prime}(h^{\prime\prime})^{2}\varphi^{\prime}}{(1+(h^{\prime}%
)^{2})^{7/2}}dx\,.
\label{liminf S}
\end{align}
Since $(\alpha_{\varepsilon},\beta,h_{\varepsilon})\in\mathcal{A}_{s}$  and $\varphi(\alpha)=1$, by
\eqref{derivative endpoint l}, \eqref{h epsilon}, and \eqref{omega eps go to zero}  we have
\begin{align}
\liminf_{\varepsilon\rightarrow0}\frac{\mathcal{S}(\alpha_{\varepsilon}%
,\beta,h_{\varepsilon})-\mathcal{S}(\alpha,\beta,h)}{\varepsilon}  &
=\gamma\int_{\alpha}^{\beta}\frac{h^{\prime}\varphi^{\prime}}{\sqrt
{1+(h^{\prime})^{2}}}dx+\gamma\sqrt{1+(h^{\prime}(\alpha))^{2}}\frac
{1}{h^{\prime}(\alpha)}%
\label{quotient S}\\
&  \quad-\gamma_{0}\frac{1}{h^{\prime}(\alpha)}+\frac{\nu_{0}}{2}\liminf_{\varepsilon\rightarrow0}I_{\varepsilon}\,,\nonumber
\end{align}
where 
\[
I_{\varepsilon}:=\frac{1}{\varepsilon}\left[  \int_{\alpha_{\varepsilon}%
}^{\beta}\frac{(h_{\varepsilon}^{\prime\prime})^{2}}{(1+(h_{\varepsilon
}^{\prime})^{2})^{5/2}}dx-\int_{\alpha}^{\beta}\frac{(h^{\prime\prime
})^{2}}{(1+(h^{\prime})^{2})^{5/2}}dx\right]  .
\]
Write%
\begin{align*}
I_{\varepsilon}  &  =\frac{1}{\varepsilon}\left[  \int_{\alpha_{\varepsilon}%
}^{\beta}\frac{(h_{\varepsilon}^{\prime\prime})^{2}}{(1+(h_{\varepsilon
}^{\prime})^{2})^{5/2}}dx-\int_{\alpha_{\varepsilon}}^{\beta}%
\frac{(h^{\prime\prime})^{2}}{(1+(h^{\prime})^{2})^{5/2}}dx\right] \\
&  \quad+\frac{1}{\varepsilon}\int_{\alpha_{\varepsilon}}^{\alpha}%
\frac{(h^{\prime\prime})^{2}}{(1+(h^{\prime})^{2})^{5/2}}%
dx=:I_{\varepsilon,1}+I_{\varepsilon,2}\,.
\end{align*}
By \eqref{h epsilon} and \eqref{omega eps go to zero},%
\begin{equation}
\lim_{\varepsilon\rightarrow0}I_{\varepsilon,1}=2\int_{\alpha}^{\beta}%
\frac{h^{\prime\prime}\varphi^{\prime\prime}}{(1+(h^{\prime})^{2})^{5/2}%
}dx-5\int_{\alpha}^{\beta}\frac{h^{\prime}(h^{\prime\prime})^{2}%
\varphi^{\prime}}{(1+(h^{\prime})^{2})^{7/2}}dx\,. \label{101}%
\end{equation}
By \eqref{derivative negative}, $\alpha_{\varepsilon}<\alpha$\ for
$\varepsilon>0$ and $\alpha_{\varepsilon}>\alpha$ for $\varepsilon<0$,
provided $|\varepsilon|\leq\varepsilon_{2}$. This implies that
\[
\liminf_{\varepsilon\rightarrow0}I_{\varepsilon,2}\geq0\,.
\]
 From this inequality and  from \eqref{quotient S} and \eqref{101} we conclude that
\eqref{liminf S} holds.
 Since $\operatorname*{Lip}h\leq L_{0}$, from \eqref{h'' L2 norm}, \eqref{delta0}, and \eqref{qwe2},
we deduce that
\begin{align}
\liminf_{\varepsilon\rightarrow0}&\frac{\mathcal{S}(\alpha_{\varepsilon}%
,\beta,h_{\varepsilon})-\mathcal{S}(\alpha,\beta,h)}{\varepsilon}
\nonumber\\&\geq-(\gamma L_{0}\delta_0
+\nu_{0}M^{1/2}\delta_{0}^{1/2}+(5/2)\nu_{0}L_{0}M)\Vert\varphi_{0}\Vert_{C^{2}}-\gamma/(2\eta_0)\,.
\label{liminf S simpler}
\end{align}

For simplicity of notation we abbreviate
\[
\mathcal{T}^{0}(\cdot,\cdot,\cdot):=\mathcal{T}_{\tau}(\cdot,\cdot
,\cdot;\alpha^{0},\beta^{0},h^{0})
\]
To evaluate $\left.  \frac{d}{d\varepsilon}\mathcal{T}^{0}(\alpha
_{\varepsilon},\beta,h_{\varepsilon})\right\vert _{\varepsilon=0}$ we define
\begin{equation}
H_{\varepsilon}(x):=\int_{\alpha_{\varepsilon}}^{x}\check{h}_{\varepsilon
}(\rho)\,d\rho\,,\quad\Phi(x):=\int_{\alpha}^{x}\check{\varphi}(\rho
)\,d\rho\,, \label{primitive}%
\end{equation}
where $\check{h}_{\varepsilon}$ and $\check{\varphi}$ are the extensions of
$h_{\varepsilon}$ and $\varphi$ by zero outside $(\alpha_{\varepsilon},\beta)$
and $(\alpha,\beta)$, respectively.  Since $\int_{0}^{\delta_{0}}\varphi_{0}\,dx=0$ and $\operatorname*{supp}\varphi_{0}\subset(-\delta_{0},\delta_{0})$, we have $\int_{\alpha}^{\beta}\varphi\,dx=0$, hence 
\begin{equation}
\Phi(x)=0\quad\text{for every }x\notin(\alpha,\beta)\,. \label{primitive zero}%
\end{equation}

Observe that if $\alpha<x<\beta$ then $\alpha_{\varepsilon}<x<\beta$ for all
$|\varepsilon|$ sufficiently small, and so by \eqref{derivative endpoint l},
\eqref{omega eps go to zero}, and the fact that $h(\alpha)=0$,
\begin{equation}
\left.  \frac{d}{d\varepsilon}H_{\varepsilon}(x)\right\vert _{\varepsilon
=0}=\left.  \frac{d}{d\varepsilon}\int_{\alpha_{\varepsilon}}^{x}(\hat
{h}+\varepsilon\varphi+\omega_{\varepsilon}\psi)\,d\rho\right\vert
_{\varepsilon=0}=\int_{\alpha}^{x}\varphi\,d\rho=\Phi(x)\,.
\label{derivative H epsilon}%
\end{equation}
On the other hand, if $x<\alpha$, then $x<\alpha_{\varepsilon}$ for all
$|\varepsilon|$ sufficiently small. Hence, $H_{\varepsilon}(x)=0$ and so
$\left.  \frac{d}{d\varepsilon}H_{\varepsilon}(x)\right\vert _{\varepsilon
=0}=0$. Moreover, $H_{\varepsilon}(x)=A_{0}$ for $x\geq\beta$ by
\eqref{average h epsilon}, and so again $\left.  \frac{d}{d\varepsilon
}H_{\varepsilon}(x)\right\vert _{\varepsilon=0}=0$. By
\eqref{primitive zero} it follows that
\eqref{derivative H epsilon} holds for all $x\in\mathbb{R}\setminus
\{\alpha\}$. 

By the regularity of $h$ it follows from the Lebesgue
dominated convergence theorem, \eqref{tau energy}, \eqref{derivative endpoint l}, and
\eqref{derivative H epsilon} that%
\begin{align}
\left.  \frac{d}{d\varepsilon}\mathcal{T}^{0}(\alpha_{\varepsilon}%
,\beta,h_{\varepsilon})\right\vert _{\varepsilon=0}  &  =\frac{1}{\tau}%
\int_{\mathbb{R}}(H-H^{0})\Phi\sqrt{1+((\check{h}^{0})^{\prime})^{2}}%
dx-\frac{\sigma_{0}}{\tau}(\alpha-\alpha^{0})\frac{1}{h^{\prime}(\alpha
)}\label{der T}\\
&  \geq -(1+L_{0})\Vert\varphi_{0}\Vert_{C^{0}}2\delta_{0}\frac{1}{\tau}%
\int_{\mathbb{R}}|H-H^{0}|\,dx-\frac{\sigma_{0}}{\eta_{0}}\frac{1}{\tau}%
|\alpha-{\alpha^{0}}|\,,\nonumber
\end{align}
where in the second inequality we used  \eqref{h' alpha and beta}  and the fact that
$\operatorname*{supp}\varphi_{0}\subset(-\delta_{0},\delta_{0})$. Integrating
by parts, and using \eqref{H bar},  \eqref{primitive},  and \eqref{primitive zero}, we obtain%
\begin{equation}
\int_{\mathbb{R}}(H-H^{0})\Phi\sqrt{1+((\check{h}^{0})^{\prime})^{2}}%
dx=-\int_{\alpha}^{\beta}\overline{H}\varphi dx\,.
\label{der T a}%
\end{equation}
\noindent\textbf{Step 3: Proof of \eqref{trace L1}}. In what follows $C$
denotes a positive constant whose value changes from  formula  to  formula  and which
depends only on $\eta_{0}$, $\eta_{1}$,  $M$,  and the fixed parameters $\lambda$,
$\mu$, $\gamma$, $\gamma_{0}$, $\sigma_{0}$, $A_{0}$, $e_{0}$, $L_{0}$, and
$\nu_{0}$ of the problem.

By Step 1, $(\alpha_{\varepsilon},\beta,h_{\varepsilon},u_{\varepsilon}%
)\in\mathcal{A}$ and because $(\alpha,\beta,h,u)\in\mathcal{A}$ is a minimizer
of the total energy functional $\mathcal{F}^{0}$, we have that
\[
\mathcal{F}^{0}(\alpha_{\varepsilon},\beta,h_{\varepsilon},u_{\varepsilon
})-\mathcal{F}^{0}(\alpha,\beta,h,u)\geq0\,.
\]
Then
\[
\limsup_{\varepsilon\rightarrow0^{-}}\frac{\mathcal{F}^{0}(\alpha
_{\varepsilon},\beta,h_{\varepsilon},u_{\varepsilon})-\mathcal{F}^{0}%
(\alpha,\beta,h,u)}{\varepsilon}\leq0
\]
and so, by \eqref{total energy}, \eqref{liminf S simpler},  and \eqref{der T}) we have%
\begin{equation}
\limsup_{\varepsilon\rightarrow0^{-}}\frac{\mathcal{E}(\alpha_{\varepsilon
},\beta,h_{\varepsilon},u_{\varepsilon})-\mathcal{E}(\alpha,\beta
,h,u)}{\varepsilon}\leq C B_{\tau}(H,{H^{0}},\alpha,{\alpha^{0}},\beta,{\beta^{0}})\,. \label{e3}%
\end{equation}
By \eqref{alpha epsilon lips}, \eqref{derivative negative}, and the fact that
$\varepsilon<0$ we have that $\alpha<\alpha_{\varepsilon}<\alpha+\delta_{0}/2$ for
all $-\varepsilon_{1}<\varepsilon<0$. Since
\[
\mathcal{E}(\alpha_{\varepsilon},\beta,h_{\varepsilon},u_{\varepsilon}%
)=\int_{\alpha_{\varepsilon}}^{\beta}\Bigl(\int_{0}^{h_{\varepsilon}%
(x)}W(E\hat{u}(x,y))\,dy\Bigr)dx\,,
\]
$\operatorname*{supp}\varphi\subset(\alpha-\delta_{0},\alpha+\delta_{0})$, and
$\operatorname*{supp}\psi\subset(\alpha+\delta_{0},\alpha+2\delta_{0})$, we
have that $h_{\varepsilon}=h$ in $(\alpha+2\delta_{0},\beta)$, and so  for $-\varepsilon_{1}<\varepsilon<0$  we can
write%
\begin{align}
\frac{\mathcal{E}(\alpha_{\varepsilon},\beta,h_{\varepsilon},u_{\varepsilon
})-\mathcal{E}(\alpha,\beta,h,u)}{\varepsilon}  &  =-\frac{1}{\varepsilon}%
\int_{\alpha_{\varepsilon}}^{\alpha+2\delta_{0}}\Bigl(\int_{h_{\varepsilon}%
(x)}^{h(x)}W(E\hat{u}(x,y))\,dy\Bigr)dx\nonumber\\
 \quad-\frac{1}{\varepsilon}\int_{\alpha}^{\alpha_{\varepsilon}}%
\Bigl(\int_{0}^{h(x)}W(Eu(x,y))\,dy\Bigr)dx &  \geq -\frac{1}{\varepsilon}%
\int_{\alpha_{\varepsilon}}^{\alpha+2\delta_{0}}\Bigl(\int_{h_{\varepsilon}%
(x)}^{h(x)}W(E\hat{u}(x,y))\,dy\Bigr)dx\,.\label{e6}
\end{align}
By \eqref{h bounded}   we have that
$h(x)\geq\eta_{0}\,\delta_{0}/4$ for all $x\in\lbrack\alpha+\delta_{0}/4,\alpha+4\delta_{0}]$. Hence, we can apply Theorem \ref{theorem regularity}
to obtain that $W(Eu(x,h(x)))\leq C$ for all $x\in\lbrack\alpha+\delta_{0}/2
,\alpha+2\delta_{0}]$. By \eqref{h epsilon}, this implies that%
\begin{equation}
\lim_{\varepsilon\rightarrow0^{-}}\frac{1}{\varepsilon}\int_{\alpha+\delta_{0}/2}
^{\alpha+2\delta_{0}}\Bigl(\int_{h_{\varepsilon}(x)}^{h(x)}W(E\hat
{u}(x,y))\,dy\Bigr)dx=-\int_{\alpha+\delta_{0}/2}^{\alpha+2\delta_{0}}%
 W(Eu(x,h(x)))\varphi(x)\,dx\leq C. \label{e5}%
\end{equation}
Together with \eqref{e3} and \eqref{e6} this implies that
\begin{equation}
\limsup_{\varepsilon\rightarrow0^{-}}\Big(-\frac{1}{\varepsilon}\int%
_{\alpha_{\varepsilon}}^{\alpha+\delta_{0}/2}\Bigl(\int_{h_{\varepsilon}(x)}%
^{h(x)}W(E\hat{u}(x,y))\,dy\Bigr)dx\Big)\leq C B_{\tau}(H,{H^{0}},\alpha,{\alpha^{0}},\beta,{\beta^{0}})\,. \label{e7}%
\end{equation}

Using the facts that $\varphi\geq 1/2$ in $(\alpha,\alpha+\delta_{0}/2)$ and $\operatorname*{supp}\psi\subset(\alpha+\delta_{0},\alpha+2\delta_{0})$ 
we obtain that $h_{\varepsilon}=h+\varepsilon\varphi\leq h$ in $(\alpha,\alpha+\delta_{0}/2)$ for every $-\varepsilon_{1}<\varepsilon<0$.
Hence,
\[
-\frac{1}{\varepsilon}\int_{h_{\varepsilon}(x)}^{h(x)}W(E\hat{u}%
(x,y))\,dy\geq0
\]
and
\[
-\frac{1}{\varepsilon}\chi_{(\alpha_{\varepsilon},\alpha+\delta_{0}/2)}%
(x)\int_{h_{\varepsilon}(x)}^{h(x)}W(E\hat{u}(x,y))\,dy\rightarrow
W(Eu(x,h(x)))\varphi(x)
\]
for every $x\in(\alpha,\alpha+\delta_{0}/2)$. By Fatou's lemma, the fact that $\varphi\geq 1/2$ in $(\alpha,\alpha+\delta_{0}/2)$,  and \eqref{e7},%
\begin{align*}
\frac12\int_{\alpha}^{\alpha+\delta_{0}/2} W(Eu(x,h(x)))\,dx 
&  \leq\liminf_{\varepsilon\rightarrow0^{-}}\Big(-\frac{1}{\varepsilon}%
\int_{\alpha_{\varepsilon}}^{\alpha+\delta_{0}/2}\Bigl(\int_{h_{\varepsilon}%
(x)}^{h(x)}W(E\hat{u}(x,y))\,dy\Bigr)dx\Big)\\
&  \leq C B_{\tau}(H,{H^{0}},\alpha,{\alpha^{0}},\beta,{\beta^{0}})\,.
\end{align*}

A similar argument gives the corresponding estimate over the interval
$[\beta-\delta_{0}/2,\beta]$. 
 To prove the estimate over the interval 
$[\alpha + \delta_{0}/2,\beta-\delta_{0}/2]$, we apply Theorem \ref{theorem regularity}
with $a=\alpha +\delta_{0}/4$, $b=\beta-\delta_{0}/4$, and $\delta=\delta_{0}/4$. We observe that \eqref{r01} is satisfied with $\eta=\min\{\eta_{0}\delta_{0}/8,\eta_{1}\}$
thanks to \eqref{h bounded from zero}, \eqref{h bounded}, and a similar estimate near $\beta$. By \eqref{W convex} and \eqref{eul 45} we obtain
\[
\int_{\alpha+\delta_{0}/2}^{\beta-\delta_{0}/2} W(Eu(x,h(x)))\,dx \leq C \leq C B_{\tau}(H,{H^{0}},\alpha,{\alpha^{0}},\beta,{\beta^{0}})\,,
\]
where we used the fact that $1\leq B_{\tau}(H,{H^{0}},\alpha,{\alpha^{0}},\beta,{\beta^{0}})$. 

Combining the inequalities on $[\alpha,\alpha + \delta_{0}/2]$, $[\alpha + \delta_{0}/2,\beta-\delta_{0}/2]$, and $[\beta-\delta_{0}/2,\beta]$ we obtain
 \eqref{trace L1}.

\noindent\textbf{Step 4: Regularity of }$\boldsymbol{h}$. Observe that by
\eqref{area constraint}  and \eqref{h bounded from zero}, we have that
$\beta-\alpha-2\delta_{0}\leq A_{0}/(2\eta_{1})$, and so%
\begin{equation}
\beta-\alpha\leq A_{0}/(2\eta_{1})+2\delta_{0}\,. \label{beta-alpha bound}%
\end{equation}

 Let us fix $x_0\in (\alpha+\delta_{0},\beta-\delta_{o})$. By 
\eqref{eul 2} we have
\begin{equation}
 A(x)h^{\prime\prime}(x)=\int_{x_{0}}^{x}( B (s)h^{\prime
}(s)-\,F_{ m  }(s))\,ds+ m_{1}  \label{eul 5}%
\end{equation}
for all $x\in(\alpha,\beta)$, where $ A$, $ B$, $\,F_{ m  }$ are defined in
\eqref{function a}, \eqref{function b}, and \eqref{function F}.  In particular we have
\begin{equation}
F_{m}(x):=F(x)-m (x-x_{0})\,,\quad\text{where}\quad F(x):=\int_{x_{0}}^{x} f(s)\,ds\,.
\label{eul 56789}
\end{equation}
By
\eqref{trace L1} we have that $\,F_{ m  }\in W^{1,1}((\alpha,\beta))$.
Since $ A\in C^{0,1/2}([\alpha,\beta])$, $ A\geq\nu_{0}/(1+L_{0}%
^{2})^{5/2}$, $ B\in L^{1}((\alpha,\beta))$, and $\,F_{ m  }\in
 C^{0}([\alpha,\beta])$, \eqref{eul 5} implies that $h^{\prime\prime}\in
 C^{0}([\alpha,\beta])$. In turn, this gives $ A\in C^{1}([\alpha,\beta])$ and $ B\in
 C^{0}([\alpha,\beta])$  by
\eqref{function a} and \eqref{function b}. Hence, the right-hand side of \eqref{eul 5} is
$C^{1}([\alpha,\beta])$, therefore $h^{\prime\prime}\in C^{1}([\alpha,\beta
])$. In turn, $ A\in C^{2}([\alpha,\beta])$ and $ B\in C^{1}([\alpha,\beta])$, and so by \eqref{eul 5},
$h^{\prime\prime}\in W^{2,1}((\alpha,\beta))$. By differentiating
\eqref{eul 5} we get%
\begin{equation}
 A(x)h^{\prime\prime\prime}(x)=- A^{\prime}(x)h^{\prime\prime}(x)+ B(x)h^{\prime
}(x)-\,F_{ m  }(x)\,. \label{ode 1}%
\end{equation}
By \eqref{function a}, \eqref{function b}, and \eqref{beta-alpha bound} we
have
\begin{equation}
\Vert  A\Vert_{H^{1}((\alpha,\beta))}+\Vert  B\Vert_{L^{1}((\alpha,\beta))}\leq
C\,. \label{a+b}%
\end{equation}

To estimate $\,F_{ m  }$, we first obtain bounds for the function $f$
defined in \eqref{function f}. In view of \eqref{trace L1}  and \eqref{beta-alpha bound},
\begin{equation}
\Vert f\Vert_{L^{1}((\alpha,\beta))}\leq C  B_{\tau}(H,{H^{0}},\alpha,{\alpha^{0}},\beta,{\beta^{0}})\,.\, \label{f bound}%
\end{equation}

Hence, by \eqref{eul 56789},
\begin{equation}
\Vert F\Vert_{L^{\infty}((\alpha,\beta))}\leq
C B_{\tau}(H,{H^{0}},\alpha,{\alpha^{0}},\beta,{\beta^{0}})\,. \label{r3}%
\end{equation}
Next we estimate the constant $ m  $ in \eqref{eul 56789}. Let $\zeta
_{0}\in C_{c}^{\infty}((-\delta_{0}/2,\delta_{0}/2))$ be such that $\int_{-\delta_{0}
/2}^{\delta_{0}/2}\zeta_{0}(x)\,dx=0$ and $\int_{-\delta_{0}/2}^{\delta_{0}/2}\zeta
_{0}(x)x^{2}dx=1$, and let $\zeta(x):=\zeta_{0}(x- x_{0}
)$. Since $\alpha+\delta_{0}<x_{0}<\beta-\delta_{0}$, we have that
$\zeta\in
C_{c}^{\infty}((\alpha+\delta_{0}/2,\beta-\delta_{0}/2))$. Multiplying
\eqref{eul 5} by $\zeta$ and integrating over $(\alpha+\delta_{0}
/2,\beta-\delta_{0}/2)$ we obtain%
\[
\int_{\alpha+\delta_{0}/2}^{\beta-\delta_{0}/2} A(x)h^{\prime\prime
}(x)\zeta(x)\,dx=\int_{\alpha+\delta_{0}/2}^{\beta-\delta_{0}/2}%
\zeta(x)\int_{x_{0}}^{x}( B(s)h^{\prime
}(s)-\,F(s))\,ds\,dx- \frac{m}{{ 2}}\,,
\]
where we used the facts that $\int_{\alpha+\delta_{0}/2}^{\beta
-\delta_{0}/2}\zeta(x)\,dx=0$ and $\int_{\alpha+\delta_{0}/2}^{\beta
-\delta_{0}/2}\zeta(x)(x- x_{0})^{2}dx=1$. From
\eqref{h'' L2 norm},  \eqref{a+b}, \eqref{r3},  and H\"older's inequality,
it follows that%
\begin{equation}
| m  |\leq C B_{\tau}(H,{H^{0}},\alpha,{\alpha^{0}},\beta,{\beta^{0}})\,. \label{r4}%
\end{equation}

Together with \eqref{beta-alpha bound}, \eqref{eul 56789}, and \eqref{r3}, 
this gives 
\begin{align}
&\Vert\,F_{ m  }\Vert_{L^{\infty}((\alpha,\beta))}\leq C B_{\tau}(H,{H^{0}},\alpha,{\alpha^{0}},\beta,{\beta^{0}})\nonumber\\
&\quad\text{and}\quad
\Vert\,F_{ m  }\Vert_{L^{1}((\alpha,\beta))}\leq C B_{\tau}(H,{H^{0}},\alpha,{\alpha^{0}},\beta,{\beta^{0}})\,.
\label{eul 57913}
\end{align}
Using the fact that $ A\geq \nu_{0}/(1+L_{0}%
^{2})^{5/2}$, \eqref{h'' L2 norm},  \eqref{ode 1}, \eqref{a+b}, \eqref{eul 57913},  and H\"older's inequality, we obtain%
\begin{equation}
\Vert h^{\prime\prime\prime}\Vert_{L^{1}((\alpha,\beta))}\leq CB_{\tau
}(H,{H^{0}},\alpha,{\alpha^{0}},\beta,{\beta^{0}})\,. \label{h''' bound}%
\end{equation}
This proves \eqref{h''' bound L1}.

For every $x\in (\alpha,\beta)$ we have
\[
|h^{\prime\prime}(x)|\le (\beta-\alpha)^{-1}\Vert h^{\prime\prime}\Vert_{L^{1}((\alpha,\beta))}+\Vert h^{\prime\prime\prime}\Vert_{L^{1}((\alpha,\beta))}\,.
\]
By \eqref{interval lower estimate}, \eqref{h'' L2 norm}, and \eqref{h''' bound}, this inequality yields
\begin{equation}
\Vert h^{\prime\prime}\Vert_{L^{\infty}((\alpha,\beta))}\leq CB_{\tau}(H,{H^{0}},\alpha,{\alpha^{0}},\beta,{\beta^{0}})\,, \label{h'' infty bound}%
\end{equation}
and proves \eqref{h'' bound}.

By \eqref{function a},
\eqref{function b}, \eqref{beta-alpha bound}, \eqref{h''' bound}, and \eqref{h'' infty bound}  we have
\begin{equation}
 \Vert  A^{\prime}\Vert_{L^{\infty}((\alpha,\beta))} + \Vert A^{\prime\prime}\Vert_{L^{1}((\alpha,\beta))}+
\Vert B\Vert_{L^{\infty}((\alpha,\beta
))}+\Vert B^{\prime}\Vert_{L^{1}((\alpha,\beta
))}\leq C B_{\tau}(H,{H^{0}},\alpha,{\alpha^{0}},\beta,{\beta^{0}})\,. \label{eul 46}%
\end{equation}
 Using again the fact that $ A\geq\nu_{0}/(1+L_{0}%
^{2})^{5/2}$ and \eqref{ode 1}, from \eqref{eul 57913}, \eqref{h'' infty bound}, and \eqref{eul 46} we get
\[
\Vert h^{\prime\prime\prime}\Vert_{L^{\infty}((\alpha,\beta))}\leq CB_{\tau
}(H,{H^{0}},\alpha,{\alpha^{0}},\beta,{\beta^{0}})^{2}\,, \label{h''' bound Linfty}%
\]
which proves \eqref{h''' bound better}. 

Differentiating \eqref{ode 1} gives%
\begin{equation*}
 A(x)h^{(\operatorname*{iv})}(x)   =-2 A^{\prime}(x)h^{\prime\prime\prime
}(x)+( B(x)- A^{\prime\prime}(x))h^{\prime\prime}(x)
  + B^{\prime}(x)h^{\prime}(x)-f(x)+ m  \,.
\end{equation*}
Using the fact that $ A\geq \nu_{0}/(1+L_{0}^{2})^{5/2}$, \eqref{r4}, \eqref{beta-alpha bound}, \eqref{f bound}, \eqref{h''' bound}%
, \eqref{h'' infty bound}, and  \eqref{eul 46} we obtain
\begin{equation}
\Vert h^{(\operatorname*{iv})}\Vert_{L^{1}((\alpha,\beta))}\leq C B_{\tau}(H,{H^{0}},\alpha,{\alpha^{0}},\beta,{\beta^{0}})^{2}\,, \label{h iv bound}%
\end{equation}
which proves \eqref{h iv bound better}  and concludes the proof of the theorem. 
\end{proof}

\begin{remark}
\label{remark derivative S}
 Since $h\in C^{2}([\alpha,\beta])$ in  view of Step 4 in the previous proof, the limit
inferior of $I_{\varepsilon}$ in \eqref{quotient S} is actually a limit and, by \eqref{derivative endpoint l} and 
\eqref{101},  it
is given by%
\[
\lim_{\varepsilon\rightarrow0}I_{\varepsilon}=2\int_{\alpha}^{\beta}%
\frac{h^{\prime\prime}\varphi^{\prime\prime}}{(1+(h^{\prime})^{2})^{5/2}%
}dx-5\int_{\alpha}^{\beta}\frac{h^{\prime}(h^{\prime\prime})^{2}%
\varphi^{\prime}}{(1+(h^{\prime})^{2})^{7/2}}dx+\frac{(h^{\prime\prime}%
(\alpha))^{2}}{(1+(h^{\prime}(\alpha))^{2})^{5/2}}\frac{1}{h^{\prime}(\alpha
)}\,.
\]
Hence, \eqref{quotient S} becomes%
\begin{align}
&  \left.  \frac{d}{d\varepsilon}\mathcal{S}(\alpha_{\varepsilon}%
,\beta,h_{\varepsilon})\right\vert _{\varepsilon=0}=\gamma\int_{\alpha}%
^{\beta}\frac{h^{\prime}\varphi^{\prime}}{\sqrt{1+(h^{\prime})^{2}}}%
dx+\gamma\sqrt{1+(h^{\prime}(\alpha))^{2}}\frac{1}{h^{\prime}(\alpha)}%
-\gamma_{0}\frac{1}{h^{\prime}(\alpha)}\nonumber\\
&  \quad+\nu_{0}\int_{\alpha}^{\beta}\frac{h^{\prime\prime}\varphi
^{\prime\prime}}{(1+(h^{\prime})^{2})^{5/2}}dx-\frac{5}{2}\nu_{0}\int_{\alpha}^{\beta
}\frac{h^{\prime}(h^{\prime\prime})^{2}\varphi^{\prime}}{(1+(h^{\prime}%
)^{2})^{7/2}}dx+\frac{\nu_{0}}{2}\frac{(h^{\prime\prime}(\alpha))^{2}}{(1+(h^{\prime
}(\alpha))^{2})^{5/2}}\frac{1}{h^{\prime}(\alpha)}\,. \label{der S}%
\end{align}
\end{remark}

\section{Flattening the Boundary}

In this section we transform the intersection of a neighborhood of
$(\alpha,0)$ with $\Omega_{h}$ into the triangle 
\begin{equation}
A_{r}^{m}:=\{(x,y)\in\mathbb{R}^{2}:
\,0<x<r\,,\,0<y<mx\}\,, \label{triangle}%
\end{equation}
for some $r>0$ and $m>0$.  We  fix $0<\eta_{0}<1$ and $M> 1$ and assume that
\begin{gather}
\alpha<\beta,\quad h\in W^{4,1}((\alpha,\beta))\subset C^{3}([\alpha
,\beta])\,,\label{hypotheses h}\\
h>0\text{ in }(\alpha,\beta)\,,\quad h(\alpha)=h(\beta)=0\,,\,\label{hypotheses h 2}\\
h^{\prime}(\alpha)\geq2\eta_{0}\,,\quad h^{\prime}(\beta)\leq-2\eta_{0}\,,\quad \operatorname*{Lip}h\leq L_{0}%
\,,\label{hypotheses h 3}\\
\int_{\alpha}^{\beta}|h^{\prime\prime}(x)|^{2} dx  \leq M\,. \label{hypotheses h 4}%
\end{gather}
For simplicity, in this section we assume that $\alpha=0$ and we write
$h_{0}^{\prime}:=h^{\prime}(0)$, $h_{0}^{\prime\prime}:=h^{\prime\prime}(0)$,
and $h_{0}^{\prime\prime\prime}:=h^{\prime\prime\prime}(0)$.

Given $r>0$ we set $I_{r}:=(0,r)$ and $\Omega_{h}^{0,r}:=\Omega_{h}\cap
(I_{r}\times\mathbb{R})$. Assume that
\begin{equation}
0<r\leq\delta_{0}=\frac{\eta_{0}^{2}}{4M} < \frac{1}{4}\,. \label{r interval1}%
\end{equation}
 By Lemma \ref{qwe1} we have $64r\leq \beta-\alpha$. 
Define for $x\in I_{ r}$,%
\begin{equation}
\sigma(x):=\frac{h_{0}^{\prime}x}{h(x)} \label{function sigma 1}%
\end{equation}
and the diffeomorphisms $\Phi:I_{ r}\times\mathbb{R}\rightarrow I_{ r%
}\times\mathbb{R}$ and $\Psi:I_{ r}\times\mathbb{R}\rightarrow I_{ r%
}\times\mathbb{R}$ by%
\begin{equation}
\Phi(x,y):=(x,y/\sigma(x))\,\quad\text{and}\quad\Psi(x,y):=(x,\sigma(x)y)\,.
\label{Phi 1}%
\end{equation}
 Throughout tis section we set $m=h_{0}^{\prime}$.  Observe that
\begin{equation}
\Psi(\Omega_{h}^{0,r})=A_{r}^{ m}\,, \label{R 2r star}%
\end{equation}
where $A_{r}^{ m}$ is the triangle introduced in \eqref{triangle}. By \eqref{hypotheses h} and a direct computation, we have
that $\sigma\in C^{2}(\overline{I}_{r})$ and that%
\begin{align}
\sigma^{\prime}(x)  &  =\frac{h_{0}^{\prime}h(x)- h_{0}^{\prime}%
xh^{\prime}(x)}{(h(x))^{2}}\,,\label{sigma' 1}\\
\sigma^{\prime\prime}(x)  &  =-\frac{h_{0}^{\prime}xh^{\prime\prime
}(x)h(x)+2h_{0}^{\prime}h(x)h^{\prime}(x)-2h_{0}^{\prime}x(h^{\prime}(x))^{2}%
}{(h(x))^{3}}\,, \label{sigma'' 1}%
\end{align}
if $x\in I_{r}$ and
\begin{equation}
\sigma(\alpha)=1\,,\quad\sigma^{\prime}(\alpha)=-\frac{h_{0}^{\prime\prime}%
}{2h_{0}^{\prime}}\,,\quad\sigma^{\prime\prime}(\alpha)=\frac{(h_{0}%
^{\prime\prime})^{2}}{2(h_{0}^{\prime})^{2}}-\frac{h_{0}^{\prime\prime\prime}%
}{3h_{0}^{\prime}}\,. \label{sigma at alpha}%
\end{equation}
In turn,
\begin{equation}
\Phi\in C^{2}(\overline{A_{r}^{ m}};\mathbb{R}^{2})\quad\text{and\quad}\Psi\in
C^{2}(\overline{\Omega_{h}^{0,r}};\mathbb{R}^{2})\,. \label{Phi C2}%
\end{equation}

\begin{lemma}
\label{lemma sigma alpha}Under the assumptions \eqref{hypotheses h}%
--\eqref{hypotheses h 4}, let $r$ be as in \eqref{r interval1}. Then there
exists a constant $C=C(\eta_{0},M)>0$, independent of $r$, such that%
\begin{align}
\Vert\sigma-1\Vert_{L^{\infty}(I_{r})}  &  \leq\frac{r^{1/2}}{\eta_{0}}\Vert
h^{\prime\prime}\Vert_{L^{2}(I_{r})}\leq Cr\Vert h^{\prime\prime}%
\Vert_{L^{\infty}(I_{r})}\,,\label{sigma -1 alpha}\\
\Vert\sigma^{\prime}\Vert_{L^{\infty}(I_{r})}  &  \leq C|h_{0}^{\prime\prime
}|+Cr\big(|h_{0}^{\prime\prime\prime}|+\Vert h^{(\operatorname*{iv})}\Vert
_{L^{1}(I_{r})}\big)\,,\label{sigma' norm alpha}\\
\Vert\sigma^{\prime\prime}\Vert_{L^{\infty}(I_{r})}  &  \leq C\big(|h_{0}%
^{\prime\prime}|^{2}+|h_{0}^{\prime\prime\prime}|+\Vert h^{(\operatorname*{iv}%
)}\Vert_{L^{1}(I_{r})}\big)+Cr^{2}\big(|h_{0}^{\prime\prime\prime}|^{2}+\Vert
h^{(\operatorname*{iv})}\Vert_{L^{1}(I_{r})}^{2}\big)\,.
\label{sigma'' norm alpha}%
\end{align}
Moreover,%
\begin{align}
\sup_{(x,y)\in\Omega_{h}^{0,r}}|y\sigma^{\prime}(x)|  &  \leq Cr|h_{0}%
^{\prime\prime}|+Cr^{2}\big(|h_{0}^{\prime\prime\prime}|+\Vert
h^{(\operatorname*{iv})}\Vert_{L^{1}(I_{r})}\big)\,,\label{sigma' y 2 alpha}\\
\sup_{(x,y)\in\Omega_{h}^{0,r}}|y\sigma^{\prime\prime}(x)|  &  \leq
Cr\big(|h_{0}^{\prime\prime}|^{2}+|h_{0}^{\prime\prime\prime}|+\Vert
h^{(\operatorname*{iv})}\Vert_{L^{1}(I_{r})}\big)+Cr^{3}\big(|h_{0}^{\prime
\prime\prime}|^{2}+\Vert h^{(\operatorname*{iv})}\Vert_{L^{1}(I_{r})}%
^{2}\big)\,,\label{sigma'' y 2 alpha}\\
\sup_{(x,y)\in A_{r}^{m}}|y\sigma^{\prime}(x)|  &  \leq Cr|h_{0}%
^{\prime\prime}|+Cr^{2}\big(|h_{0}^{\prime\prime\prime}|+\Vert
h^{(\operatorname*{iv})}\Vert_{L^{1}(I_{r})}\big)\,,\label{sigma' y alpha}\\
\sup_{(x,y)\in A_{r}^{m}}|y\sigma^{\prime\prime}(x)|  &  \leq
Cr\big(|h_{0}^{\prime\prime}|^{2}+|h_{0}^{\prime\prime\prime}|+\Vert
h^{(\operatorname*{iv})}\Vert_{L^{1}(I_{r})}\big)+Cr^{3}\big(|h_{0}^{\prime
\prime\prime}|^{2}+\Vert h^{(\operatorname*{iv})}\Vert_{L^{1}(I_{r})}^{2}\big)\,.
\label{sigma'' y alpha}%
\end{align}
\end{lemma}

\begin{proof}
Using Taylor's formula with integral remainder for $x\in I_{r}$ we can write%
\[
\sigma(x)-1=\frac{h_{0}^{\prime}x-h(x)}{h(x)}=-\frac{\int_{0}^{x}%
h^{\prime\prime}(s)(x-s)\,ds}{h(x)}\,.
\]
If $x\in I_{r}$ by H\"{o}lder's inequality, \eqref{hypotheses h 3},
\eqref{hypotheses h 4},  and \eqref{r interval1}  we have
\[
h^{\prime}(x)=h_{0}^{\prime}+\int_{0}^{x}h^{\prime\prime}(s)\,ds\geq
2\eta_{0}-M^{1/2}x^{1/2}\geq\eta_{0}\,.
\]
Hence,
\begin{equation}
h(x)=\int_{0}^{x}h^{\prime}(s)\,ds\geq\eta_{0} x\,. \label{h bound from belo}%
\end{equation}
In turn, by H\"{o}lder's inequality,%
\[
|\sigma(x)-1|\leq\frac{x^{3/2}}{\eta_{0} x}\Vert h^{\prime\prime}\Vert
_{L^{2}(I_{r})}\leq\frac{r^{1/2}}{\eta_{0}}\Vert h^{\prime\prime}\Vert
_{L^{2}(I_{r})}%
\]
for every $x\in I_{r}$, which gives the estimate \eqref{sigma -1 alpha}.

To prove \eqref{sigma' norm alpha}, in view of \eqref{sigma' 1}, we use
Taylor's formula with integral remainder to get
\begin{align*}
h_{0}^{\prime}h(x)  &  -h_{0}^{\prime}xh^{\prime}(x)=-\frac{1}{2}h_{0}%
^{\prime}h_{0}^{\prime\prime}x^{2}-\frac{1}{3}h_{0}^{\prime}h_{0}%
^{\prime\prime\prime}x^{3}\\
&  -\frac{1}{2}h_{0}^{\prime}x\int_{0}^{x}h^{(\operatorname*{iv})}%
(s)(x-s)^{2}ds+\frac{1}{6}h_{0}^{\prime}\int_{0}^{x}h^{(\operatorname*{iv}%
)}(s)(x-s)^{3}ds\,.
\end{align*}
Hence,%
\[
|h_{0}^{\prime}h(x)-h_{0}^{\prime}xh^{\prime}(x)|\leq\frac{1}{2}h_{0}^{\prime
}|h_{0}^{\prime\prime}|x^{2}+\frac{1}{3}h_{0}^{\prime}|h_{0}^{\prime
\prime\prime}|x^{3}+\frac{2}{3}h_{0}^{\prime}x^{3}\int_{0}^{x}%
|h^{(\operatorname*{iv})}(s)|\,ds\,.
\]
Using \eqref{h bound from belo}, if follows from \eqref{sigma' 1} that%
\[
|\sigma^{\prime}(x)|\leq\frac{L_{0}}{\eta_{0}^{2}}\left(  \frac{1}{2}%
|h_{0}^{\prime\prime}|+\frac{1}{3}r|h_{0}^{\prime\prime\prime}|+\frac{2}%
{3}r\Vert h^{(\operatorname*{iv})}\Vert_{L^{1}(I_{r})}\right)  \,,
\]
which proves \eqref{sigma' norm alpha}.

To prove \eqref{sigma'' norm alpha}, in view of \eqref{sigma'' 1}, we use
Taylor's formula with integral remainder and the inequality $2ab\leq
a^{2}+b^{2}$ to estimate
\begin{align*}
&  |h_{0}^{\prime}xh^{\prime\prime}(x)h(x)+2h_{0}^{\prime}h(x)h^{\prime
}(x)-2h_{0}^{\prime}x(h^{\prime}(x))^{2}|\\
&  \leq C\big[|h_{0}^{\prime\prime}|^{2}x^{3}+|h_{0}^{\prime\prime\prime
}|x^{3}+\Vert h^{(\operatorname*{iv})}\Vert_{L^{1}(I_{r})}x^{3}+|h_{0}%
^{\prime\prime\prime}|^{2}x^{5}+\Vert h^{(\operatorname*{iv})}\Vert
_{L^{1}(I_{r})}^{2}x^{5}\big]\,,
\end{align*}
where $C=C(L_{0})>0$. Using \eqref{sigma'' 1} and \eqref{h bound from belo},
for $x\in I_{r}$ we get%
\[
|\sigma^{\prime\prime}(x)|\leq\frac{C}{\eta_{0}^{3}}\big[|h_{0}^{\prime\prime
}|^{2}+|h_{0}^{\prime\prime\prime}|+\Vert h^{(\operatorname*{iv})}\Vert
_{L^{1}(I_{r})}+r^{2}(|h_{0}^{\prime\prime\prime}|^{2}+\Vert
h^{(\operatorname*{iv})}\Vert_{L^{1}(I_{r})}^{2})\big]\,,
\]
which  proves  \eqref{sigma'' norm alpha}.

To prove \eqref{sigma' y 2 alpha} and \eqref{sigma'' y 2 alpha}, in view of
\eqref{sigma' norm alpha} and \eqref{sigma'' norm alpha}, it is enough to
observe that if $(x,y)\in\Omega_{h}^{0,r}$, then $0<y<L_{0}r$.

Finally, if $(x,y)\in A_{r}^{m}$, then $0<y< mr\leq L_{0}r$, where in the last inequality we used the fact that $m=h_{0}^{\prime}\le L_{0}$  by {\eqref{hypotheses h 3}}) Together
with \eqref{sigma' norm alpha} and \eqref{sigma'' norm alpha}, this inequality
proves \eqref{sigma' y alpha} and \eqref{sigma'' y alpha}. This concludes the proof.
\end{proof}

\begin{remark}
\label{remark sigma close to 1 a} By \eqref{hypotheses h 4}, \eqref{r interval1}, and
\eqref{sigma -1 alpha} we have that
\[
\Vert\sigma-1\Vert_{L^{\infty}(I_{r})}\leq1/2\,.
\]
\end{remark}

\begin{lemma}
\label{lemma change of variables}Under  the assumptions  \eqref{hypotheses h}%
--\eqref{hypotheses h 4}, let $r$ be as in \eqref{r interval1}, and let
${ 1\le p<\infty}$. If $f\in L^{p}(\Omega_{h}^{0,r})$ and $w\in W^{2,p}(\Omega_{h}%
^{0,r})$, then $f\circ\Phi\in L^{p}(A_{r}^{ m})$ and $w\circ\Phi\in W^{2,p}(A_{r}^{ m}%
)$. Moreover the following estimates hold: %
\begin{align*}
\Vert f\circ\Phi\Vert_{L^{p}(A_{r}^{ m})}  &  \leq C_{ p}\Vert f\Vert_{L^{p}(\Omega
_{h}^{0,r})}\\
\Vert\nabla(w\circ\Phi)\Vert_{L^{p}(A_{r}^{ m})}  &  \leq C_{ p}\big(1+\sup_{(x,y)\in
 A_{r}^{m}}y|\sigma^{\prime}(x)|\big)\Vert\nabla w\Vert_{L^{p}(\Omega_{h}^{0,r})}\\
\Vert\nabla^{2}(w\circ\Phi)\Vert_{L^{p}(A_{r}^{ m})}  &  \leq C_{ p}\big(1+\sup
_{(x,y)\in A_{r}^{m}}y^{2}(\sigma^{\prime}(x))^{2}\big)\Vert\nabla
^{2}w\Vert_{L^{p}(\Omega_{h}^{0,r})}\\
+C_{ p}\big(1+\sup_{x\in I_{r}}|\sigma^{\prime}(x)|&+\sup_{(x,y)\in A_{r}^{m}}y(\sigma^{\prime}%
(x))^{2}+\sup_{(x,y)\in A_{r}^{m}}y|\sigma^{\prime\prime}%
(x)|\big)\Vert\nabla w\Vert_{L^{p}(\Omega_{h}^{0,r})}\,,
\end{align*}  where the constant $C_{p}$ depends only on $p$.
\end{lemma}

\begin{proof}
 In this proof $C$ is an absolute constant, independent of all other parameters, whose value can change from formula to formula.  Since $\det\nabla\Phi(x,y)=\frac{1}{\sigma(x)}$, by a change of variables and
Remark \ref{remark sigma close to 1 a} we have%
\begin{equation}
\int_{A_{r}^{ m}}|f\circ\Phi|^{p}dxdy  =\int_{\Omega_{h}^{0,r}}\sigma|f|^{p}%
dxdy \leq\frac{3}{2}\int_{\Omega_{h}^{0,r}}|f|^{p}dxdy\,.
\label{f change of variables a}
\end{equation}
By \eqref{Phi 1} and Remark \ref{remark sigma close to 1 a}, it follows by the
chain rule that
\begin{equation}
|\nabla(w\circ\Phi)(x,y)|\leq C(1+y|\sigma^{\prime}(x)|)|\nabla w(\Phi
(x,y))|\,. \label{chain rule}%
\end{equation}
Hence, by \eqref{f change of variables a}, we have
\begin{equation*}
\Vert\nabla   (w\circ\Phi)\Vert_{L^{p}(A_{r}^{ m})}
  \leq C\big(1+\sup_{(x,y)\in A_{r}^{ m}}y|\sigma^{\prime}(x)|\big)\Vert\nabla
w\Vert_{L^{p}(\Omega_{h}^{0,r})}\,.
\end{equation*}
Similarly, again by Remark \ref{remark sigma close to 1 a} and the chain rule%
\begin{align*}
|\nabla^{2}(w\circ\Phi)(x,y)|  &  \leq C(1+y^{2}(\sigma^{\prime}%
(x))^{2})|\nabla^{2}w(\Phi(x,y))|\\
&  \quad+C(1+|\sigma^{\prime}(x)|+ y(\sigma^{\prime}(x))^{2}+y|\sigma^{\prime\prime}(x)|)|\nabla
w(\Phi(x,y))|\,.
\end{align*}
  Hence, by \eqref{f change of variables a} we obtain
\begin{align*}
\Vert\nabla^{2}(w\circ\Phi)\Vert_{L^{p}(A_{r}^{ m})}  &  \leq C\big(1+\sup
_{(x,y)\in A_{r}^{m}} y^2(\sigma^{\prime}(x))^{2}\big)\Vert\nabla
^{2}w\Vert_{L^{p}(\Omega_{h}^{0,r})}\\
 +C\big(1+&\sup_{x\in I_{r}}|\sigma^{\prime}%
(x)|+\sup_{(x,y)\in A_{r}^{m}}y(\sigma^{\prime}%
(x))^{2}+\sup_{(x,y)\in A_{r}^{m}}y|\sigma^{\prime\prime}%
(x)|\big)\Vert\nabla w\Vert_{L^{p}(\Omega_{h}^{0,r})}\,,
\end{align*}
which concludes the proof.
\end{proof}

\begin{remark}
\label{remark change of variables}Similarly, one can show that if $f\in
L^{p}(A_{r}^{ m})$ and $w\in W^{2,p}(A_{r}^{ m})$, then $f\circ\Psi\in L^{p}(\Omega
_{h}^{0,r})$ and $w\circ\Psi\in W^{2,p}(\Omega_{h}^{0,r})$ and the following
estimates hold%
\begin{align*}
\Vert f\circ\Psi\Vert_{L^{p}(\Omega_{h}^{0,r})}  &  \leq C_{ p}\Vert f\Vert
_{L^{p}(A_{r}^{ m})}\\
\Vert\nabla(w\circ\Psi)\Vert_{L^{p}(\Omega_{h}^{0,r})}  &  \leq C_{ p}\big(1+\sup
_{(x,y)\in\Omega_{h}^{0,r}}y|\sigma^{\prime}(x)|\big)\Vert\nabla w\Vert_{L^{p}(A_{r}^{ m})}\\
\Vert\nabla^{2}(w\circ\Psi)\Vert_{L^{p}(\Omega_{h}^{0,r})}  &  \leq
C_{ p}\big(1+\sup_{(x,y)\in\Omega_{h}^{0,r}}y^{2}(\sigma^{\prime}(x))^{2}%
\big)\Vert\nabla^{2}w\Vert_{L^{p}(A_{r}^{ m})}\\
 +C_{ p}\big(1+&\sup_{x\in I_{r}}|\sigma^{\prime}(x)|+\sup_{(x,y)\in\Omega_{h}^{0,r}}y(\sigma^{\prime}%
(x))^{2}+\sup_{(x,y)\in\Omega_{h}^{0,r}}y|\sigma^{\prime\prime}%
(x)|\big)\Vert\nabla w\Vert_{L^{p}(A_{r}^{ m})}\,,
\end{align*}
 where the constant $C_{p}$ depends only on $p$.
\end{remark}

 Given $r>0$ as in \eqref{r interval1}, let  $\Gamma_{h}^{0,r}:=\Gamma_{h}\cap
(I_{r}\times\mathbb{R})$, where $\Gamma_{h}$ is the graph of $h$, 
and let $\Gamma:=\{(x,mx):\,0<x<r\}\subset\partial A_{r}^{m}$. 
For $x\in I_{r}$ let $\nu^{h}(x):={(-h^{\prime}(x),1)}/{\sqrt{1+(h^{\prime}(x))^{2}}}$
be the outer unit normal to $\Omega_{h}$ on $\Gamma_{h}$, let 
$\nu^{0}:={(-h_{0}^{\prime},1)}/{\sqrt{1+(h_{0}^{\prime})^{2}}}$
be the outer unit normal to $A_{r}^{m}$ on $\Gamma$, and let $\omega_{i}\colon I_{r}\to \mathbb{R}$ be the functions defined by
\begin{align}
\omega_{1}(x)  &  := \nu_{1}^{0}-\nu_{1}^{h}(x)=-\frac{h_{0}^{\prime}}{\sqrt{1+|h_{0}^{\prime}|^{2}}%
}+\frac{h^{\prime}(x)}{\sqrt{1+|h^{\prime}(x)|^{2}}}\,,\nonumber\\
\omega_{2}(x)  &  := \nu_{2}^{0}-\nu_{2}^{h}(x)=\frac{1}{\sqrt{1+|h_{0}^{\prime}|^{2}}}-\frac{1}%
{\sqrt{1+|h^{\prime}(x)|^{2}}}\,,\nonumber\\
\omega_{3}(x)  &  :=-\sigma^{\prime}(x)h(x)\nu_{1}^{h}(x)=\sigma^{\prime}(x)h(x)\frac{h^{\prime}(x)}{\sqrt
{1+|h^{\prime}(x)|^{2}}}\,,\nonumber\\
\omega_{4}(x)  &  :=-(\sigma(x)-1)\nu_{1}^{h}(x)=(\sigma(x)-1)\frac{h^{\prime}(x)}{\sqrt{1+|h^{\prime
}(x)|^{2}}}\,,\label{functions psi}\\
\omega_{5}(x)  &  :=\sigma^{\prime}(x)h(x)\nu_{2}^{h}(x)=\sigma^{\prime}(x)h(x)\frac{1}{\sqrt{1+|h^{\prime}%
(x)|^{2}}}\,,\nonumber\\
\omega_{6}(x)  &  :=(\sigma(x)-1)\nu_{2}^{h}(x)=(\sigma(x)-1)\frac{1}{\sqrt{1+|h^{\prime}(x)|^{2}}%
}\,.\nonumber
\end{align}

\begin{lemma}\label{lemma normals} Under the assumptions \eqref{hypotheses h}%
--\eqref{hypotheses h 4}, let $r$ be as in \eqref{r interval1}, { let $1\le p<\infty$},
let $u\in W_{\operatorname*{loc}}^{2,{ p}}(\Omega_{h}^{0,r};\mathbb{R}^{2})
$,
with $u\in W^{2,{p}}(\overline{\Omega}_{h}^{\rho,r};\mathbb{R}^{2})$
for every $0<\rho<r$, let $m=h_{0}^{\prime}$, let $v\colon A_{r}^{m}\to \mathbb{R}^{2}$
be defined by $v(x,y)=u(x,y/\sigma(x))$, and let 
$g\in W^{1,{ p}}(\Omega_{h}^{0,r};\mathbb{R}^{2})$.
Assume that
\begin{equation}\label{qwe31}
(\mathbb{C}Eu )\nu^{h}=2\mu(Eu )\nu^{h}+\lambda
(\operatorname{div}u )\nu^{h}= g \quad\text{on
}\Gamma_{h}^{0,r}\,.%
\end{equation}
Then 
\begin{equation}\label{qwe30}
(\mathbb{C}Ev)\nu=2\mu(Ev)\nu+\lambda(\operatorname{div}v)\nu=
g\circ\Phi+ \hat{g}^{v}+\check{g}^{v}\quad\text{on }\Gamma_{r}^{ m}\,,
\end{equation}
where { $\hat{g}^{v}=(\hat{g}_1^{v},\hat{g}_2^{v})$ and $\check{g}^{v}=(\check{g}_1^{v},\check{g}_2^{v})$ are } defined by 
\begin{align}
\hat{g}_{1}^{v}  &  :=(2\mu\omega_{1}+\lambda\omega_{1})\partial_{x}v_{1}+\mu\omega_{2}\partial_{x}v_{2}\,,\label{g1hat}
\\
\hat{g}_{2}^{v} &  :=
\lambda\omega_{2}  \partial_{x}v_{1} 
+\mu\omega_{1}%
\partial_{x}v_{2}\,.\label{g2hat}
\\
\check{g}_{1}^{v}  &  :=(\mu\omega_{2}+2\mu\omega_{3}+\lambda\omega_{3}-\mu\omega_{6}) \partial_{y}%
v_{1}+(\lambda\omega_{1}+\lambda\omega_{4}-\mu\omega_{5})\partial_{y}v_{2}
\,,\label{g1check}
\\
\check{g}_{2}^{v}  &  :=(\mu\omega_{1}+\mu\omega_{4}-\lambda\omega_{5})\partial_{y}v_{1}
+(2\mu\omega_{2}+\lambda\omega_{2}+\mu\omega_{3}-2\mu\omega_{6}-\lambda\omega_{6})\partial_{y}v_{2} \,.\label{g2check}
\end{align}
\end{lemma}

\begin{remark}\label{remqqq1}
If $u=0$ on $\partial\Omega_{h}^{0,r}\setminus\Omega_{h}^{0,r}$, then $\hat{g}^{v}=0$ on $(0,r)\times\{0\}$ and $\check{g}^{v}=0$ on $\{r\}\times(0,mr)$.
\end{remark}

\begin{proof}[Proof of Lemma \ref{lemma normals}]%
 Since
$u (x,y)=v(x,\sigma(x)y)$, 
the  partial derivatives  of $u$  are%
\[
\partial_{x}u (x,y)=\partial_{x}v(x,\sigma(x)y)+\partial_{y}%
v(x,\sigma(x)y)\sigma^{\prime}(x)y\quad\text{and}\quad\partial_{y}u (x,y)=\partial
_{y}v(x,\sigma(x)y)\sigma(x)\,.
\]
So at $(x,h(x))$ 
 the  first component  of $2\mu(Eu )\nu^{h}+\lambda(\operatorname{div}u )\nu^{h}$
is 
\[
2\mu\left(  \partial_{x}v_{1}+\partial_{y}v_{1}\sigma^{\prime}y\right)
\nu_{1}^{h}+\mu(\partial_{x}v_{2}+\partial_{y}v_{2}\sigma^{\prime}%
y+\partial_{y}v_{1}\sigma)\nu_{2}^{h}+\lambda\left(  \partial_{x}%
v_{1}+\partial_{y}v_{1}\sigma^{\prime}y+\partial_{y}v_{2}\sigma\right)
\nu_{1}^{h}\,,%
\]
 while the  second component  is 
\[
\mu\left(  \partial_{x}v_{2}+\partial_{y}v_{2}\sigma^{\prime}y+\partial
_{y}v_{1}\sigma\right)  
\nu_{1}^{h}+2\mu\partial_{y}v_{2}\sigma\nu_{2}^{h}+\lambda\left(  
\partial_{x}v_{1}+\partial_{y}v_{1}\sigma^{\prime}y+\partial_{y}v_{2}\sigma\right)  \nu_{2}^{h}%
\]
where $v_{1}$ and $v_{2}$ are computed at $(x,\sigma(x)h(x))=(x, h_{0}^{\prime}x)$ and $\sigma$ and
$\sigma^{\prime}$ and $\nu^{h}$ at $x$.  By adding and subtracting some 
terms containing $\nu_{1}^{0}$ and $\nu_{2}^{0}$, the  first component can be written as%
\begin{align*}
& 2\mu\partial_{x}v_{1} \nu_{1}^{0}+\mu(\partial_{x}v_{2}+\partial_{y}v_{1}%
) \nu_{2}^{0}+\lambda\left(  \partial_{x}v_{1}+\partial_{y}v_{2}\right)   \nu_{1}^{0}\\
&  -2\mu\partial_{x}v_{1}( \nu_{1}^{0}- \nu_{1}^{h})-\mu(\partial_{x}v_{2}%
+\partial_{y}v_{1})( \nu_{2}^{0}-\nu_{2}^{h})-\lambda\left(  \partial_{x}%
v_{1}+\partial_{y}v_{2}\right)  ( \nu_{1}^{0}-\nu_{1}^{h})\\
&  +2\mu\partial_{y}v_{1}\sigma^{\prime}y\nu_{1}^{h}+\mu(\partial_{y}%
v_{2}\sigma^{\prime}y+\partial_{y}v_{1}(\sigma-1))\nu_{2}^{h}+\lambda\left(
\partial_{y}v_{1}\sigma^{\prime}y+\partial_{y}v_{2}(\sigma-1)\right)  \nu
_{1}^{h}%
\end{align*}
and the second component as%
\begin{align*}
&  \mu\left(  \partial_{x}v_{2}+\partial_{y}v_{1}\right)   \nu_{1}^{0}%
+2\mu\partial_{y}v_{2} \nu_{2}^{0}+\lambda\left(  \partial_{x}v_{1}+\partial
_{y}v_{2}\right)   \nu_{2}^{0}\\
&  -\mu\left(  \partial_{x}v_{2}+\partial_{y}v_{1}\right)  ( \nu_{1}^{0}-\nu
_{1}^{h})-2\mu\partial_{y}v_{2}( \nu_{2}^{0}-\nu_{2}^{h})-\lambda\left(
\partial_{x}v_{1}+\partial_{y}v_{2}\right)  ( \nu_{2}^{0}-\nu_{2}^{h})\\
&  +\mu\left(  \partial_{y}v_{2}\sigma^{\prime}y+\partial_{y}v_{1}%
(\sigma-1)\right)  \nu_{1}^{h}+2\mu\partial_{y}v_{2}(\sigma-1)\nu_{2}^{h}
+\lambda\left(  \partial_{y}v_{1}\sigma^{\prime}y+\partial_{y}v_{2}%
(\sigma-1)\right)  \nu_{2}^{h}%
\end{align*}
Hence, using \eqref{qwe31} 
we obtain \eqref{qwe30}
with 
{ 
\begin{align*}
\hat{g}_{1}^{v}  &  :=2\mu\partial_{x}v_{1}( \nu_{1}^{0}-\nu_{1}^{h})+\mu\partial
_{x}v_{2}( \nu_{2}^{0}-\nu_{2}^{h})+ \lambda \partial
_{x}v_{1} ( \nu_{1}^{0}-\nu_{1}^{h})\\
\hat{g}_{2}^{v}  &  :=\mu\partial_{x}v_{2}
( \nu_{1}^{0}-\nu_{1}^{h})+ \lambda
 \partial_{x}v_{1} ( \nu_{2}^{0}-\nu_{2}^{h})
\end{align*}
and 
\begin{align*}
\check{g}_{1}^{v}  &  :=\mu\partial_{y}v_{1}( \nu_{2}^{0}-\nu_{2}^{h})+ \lambda\partial_{y}v_{2}  ( \nu_{1}^{0}-\nu_{1}^{h})- 2\mu\partial_{y}v_{1}\sigma^{\prime}h  \nu_{1}^{h}\\
&  \quad
- \mu(\partial_{y}v_{2}\sigma^{\prime}h+\partial_{y}v_{1}(\sigma-1))\nu_{2}^{h}
- \lambda\left(  \partial_{y}v_{1}\sigma^{\prime}h+\partial_{y}v_{2}%
(\sigma-1)\right)  \nu_{1}^{h}\,,
\\
\check{g}_{2}^{v}  &  :=\mu\partial_{y}v_{1}
( \nu_{1}^{0}-\nu_{1}^{h})+ 2\mu\partial_{y}v_{2}( \nu_{2}^{0}-\nu_{2}^{h})+ \lambda
\partial_{y}v_{2}  ( \nu_{2}^{0}-\nu_{2}^{h})\\
&  \quad- \mu\left(  \partial_{y}v_{2}\sigma^{\prime}h+\partial_{y}v_{1}%
(\sigma-1)\right)  \nu_{1}^{h}- 2\mu\partial_{y}v_{2}(\sigma-1)\nu_{2}^{h}
- \lambda\left(  \partial_{y}v_{1}\sigma^{\prime}h+\partial_{y}v_{2}%
(\sigma-1)\right)  \nu_{2}^{h} \,.
\end{align*}
}
 Using \eqref{functions psi} we obtain \eqref{g1hat}, \eqref{g2hat}, \eqref{g1check}, and \eqref{g2check}.
\end{proof}

For technical reasons we need a precise estimate
of the $L^{\infty}$  norms of  the functions $\omega_{i}$ defined in
\eqref{functions psi} and of their derivatives.

\begin{lemma}
\label{lemma fractional a}Under  the assumptions  \eqref{hypotheses h}--\eqref{hypotheses h 4},
let $r$ be as in \eqref{r interval1}, and let $\omega_{i}$, $i=1,\ldots,6$ be
defined as in \eqref{functions psi}. Then there exists a constant
$C=C(\eta_{0},M)>0$ such that
\begin{align}
\Vert\omega_{i}\Vert_{ L^{\infty}(I_{r})}  &  \leq Cr\Vert h^{\prime\prime}%
\Vert_{L^{\infty}(I_{r})}+Cr^{2}(\Vert h^{\prime\prime\prime}\Vert_{L^{\infty
}(I_{r})}+\Vert h^{(\operatorname*{iv})}\Vert_{L^{1}(I_{r})})
\label{omega i estimate a}%
\\
\Vert\omega_{i}^{\prime}\Vert_{ L^{\infty}(I_{r})}  &  \leq C\Vert h^{\prime\prime
}\Vert_{L^{\infty}(I_{r})}+Cr\big(\Vert h^{\prime\prime}\Vert_{L^{\infty}(I_{r}%
)}^{2}+\Vert h^{\prime\prime\prime}\Vert_{L^{\infty}(I_{r})}%
+\Vert h^{(\operatorname*{iv})}\Vert_{L^{1}(I_{r})}\big)
\label{omega' i estimate a}\\
&  \quad+Cr^{3}\big(\Vert
h^{\prime\prime\prime}\Vert_{L^{\infty}(I_{r})}^{2}+\Vert
h^{(\operatorname*{iv})}\Vert_{L^{1}(I_{r})}^{2}\big)\nonumber
\end{align}
for $i=1,\ldots,6$.
\end{lemma}

\begin{proof}
Define%
\begin{equation}
f_{1}(t)=\frac{t}{\sqrt{1+t^{2}}}\,\quad\text{and}\quad f_{2}(t)=\frac
{1}{\sqrt{1+t^{2}}}\quad\text{for }t\in\mathbb{R}\,. \label{2999}%
\end{equation}
Observe that $\omega_{1}(x)=f_{1}(h^{\prime}(x))-f_{1}(h_{0}^{\prime})$ and
$\omega_{2}(x)=f_{2}(h_{0}^{\prime})-f_{2}(h^{\prime}(x))$. Since%
\begin{equation}
f_{1}^{\prime}(t)=\frac{1}{(1+t^{2})^{3/2}}\quad\text{and}\,\quad
f_{2}^{\prime}(t)=-\frac{t}{(1+t^{2})^{3/2}}\,, \label{3000}%
\end{equation}
the functions $f_{1}$ and $f_{2}$ are Lipschitz continuous with Lipschitz
constant one. Hence,
\[
|\omega_{i}(x)|\leq|h^{\prime}(x)-h_{0}^{\prime}|\quad\text{for }i=1,2\,,
\]
and the estimate \eqref{omega i estimate a} follows by the mean value theorem.
On the other hand,%
\[
|\omega_{1}^{\prime}(x)|=|f_{1}^{\prime}(h^{\prime}(x))h^{\prime\prime
}(x)|\leq|h^{\prime\prime}(x)|\quad\text{and}\quad|\omega_{2}^{\prime
}(x)|=|f_{2}^{\prime}(h^{\prime}(x))h^{\prime\prime}(x)|\leq|h^{\prime\prime
}(x)|\,,
\]
which gives \eqref{omega' i estimate a} for $i=1$ and $2$.

 By \eqref{hypotheses h 2} and \eqref{hypotheses h 3} we have  $h(x)\leq L_{0}x$ for $x\in (\alpha,\beta)$,  hence  $|\omega
_{3}(x)|\leq L_{0}r|\sigma^{\prime}(x)|$  for $x\in I_{r}$, and \eqref{omega i estimate a} is a
consequence of \eqref{sigma' norm alpha}, while
\begin{align*}
|\omega_{3}^{\prime}(x)|  &  \leq|\sigma^{\prime\prime}(x)h(x)f_{1}(h^{\prime
}(x))|+|\sigma^{\prime}(x)h^{\prime}(x)f_{1}(h^{\prime}(x))|+|\sigma^{\prime
}(x)h(x)f_{1}^{\prime}(h^{\prime}(x))h^{\prime\prime}(x)|\\
&  \leq L_{0}r|\sigma^{\prime\prime}(x)|+L_{0}|\sigma^{\prime}(x)|+L_{0}%
r|\sigma^{\prime}(x)h^{\prime\prime}(x)|\,.
\end{align*}
Using \eqref{sigma' norm alpha}, \eqref{sigma'' norm alpha}, and the
inequality
\begin{align*}
&  2r^{2}\Vert h^{\prime\prime}\Vert_{L^{\infty}(I_{r})}(\Vert h^{\prime
\prime\prime}\Vert_{L^{\infty}(I_{r})}+\Vert h^{(\operatorname*{iv})}%
\Vert_{L^{1}(I_{r})})\\
&  \leq2r\Vert h^{\prime\prime}\Vert_{L^{\infty}(I_{r})}^{2}+r^{3}(\Vert
h^{\prime\prime\prime}\Vert_{L^{\infty}(I_{r})}^{2}+\Vert
h^{(\operatorname*{iv})}\Vert_{L^{1}(I_{r})}^{2})\,,
\end{align*}
we obtain \eqref{omega' i estimate a}. The estimates for $\omega_{5}$ are similar.

We have $|\omega_{4}(x)|\leq|(\sigma(x)-1)f_{1}(h^{\prime}(x))|\leq
|\sigma(x)-1|$, and so \eqref{omega i estimate a} follows from
\eqref{sigma -1 alpha}. On the other hand,
\begin{align*}
|\omega_{4}^{\prime}(x)|  &  \leq|\sigma^{\prime}(x)f_{1}(h^{\prime
}(x))|+|(\sigma(x)-1)f_{1}^{\prime}(h^{\prime}(x))h^{\prime\prime}(x)|\\
&  \leq|\sigma^{\prime}(x)|+|(\sigma(x)-1)h^{\prime\prime}(x)|\,,
\end{align*}
so that \eqref{omega' i estimate a} is a consequence of \eqref{sigma -1 alpha}
and \eqref{sigma' norm alpha}. A similar estimate holds for $\omega_{6}$.
\end{proof}

\section{Regularity:  Preliminary  Results}
 In this section  we study the regularity of solutions to the Lam\'{e} system in the
triangle $A_{r}^{ m}$  introduced in \eqref{triangle}. The exponent in the
regularity theorem will depend on the complex solutions $z$ of the equation%
\begin{equation}
\sin^{2}(z\omega)=K_{1}-Kz^{2}\sin^{2}\omega\,, \label{equation angle}%
\end{equation}
where $\omega=\arctan m$ is the angle of the triangle $ A_{r}^{m}$ at the vertex
$(0,0)$, and%
\begin{equation}
K:=\frac{\lambda+\mu}{\lambda+3\mu}<1<K_{1}:=\frac{(\lambda+2\mu)^{2}%
}{(\lambda+\mu)(\lambda+3\mu)}\,, \label{constants K}%
\end{equation}
where $\lambda$ and $\mu$ are the Lam\'{e} coefficients. In particular, we
will use the results of the following lemma.

\begin{lemma}
\label{lemma no real}There exists a constant $\xi_{0}=\xi_{0}(\lambda,\mu
)\in(1/2,1)$ such  that for every $0<\omega\le\pi/2$  the equation \eqref{equation angle} has no complex solutions
$z$ with $\operatorname*{Re}z\in(0,\xi_{0})$.
\end{lemma}

\begin{proof}
We follow the proof of \cite[Theorem 2.2]{nicaise1992lame}. If
$\operatorname*{Im}z\neq0$, then using the fact that $0<\omega\leq\frac{\pi
}{2}$, the argument in the first case of that proof shows that there are no
solutions $z$ with $\operatorname*{Re}z\in(0,1]$. If $\operatorname*{Im}z=0$,
then \eqref{equation angle} reduces to
\begin{equation}
\sin^{2}(\omega\operatorname*{Re}z)=K_{1}-K(\operatorname*{Re}z)^{2}\sin
^{2}\omega\,.
\label{impossible}
\end{equation}
By \eqref{constants K} there exists $\varepsilon_{0}>0$ such that
$K<1<K_{1}-\varepsilon_{0}$. By a trigonometric computation we find that
 $\sin^{2}(\omega/2)\leq 1-(\sin^{2}\omega)/4$, hence  
\[
\sin^{2}(\omega/2)<K_{1}-\varepsilon_{0}-K(\sin^{2}\omega)/4\,.
\]
 Since for every $0<\omega\leq\frac{\pi}{2}$ the function $\xi\mapsto\sin^{2}(\omega\xi)+K\xi^{2}\sin^{2}\omega$ is Lipschitz continuous on $[0,1]$ with Lipschitz constant $\pi+2$,  there exists a constant
$\xi_{0}\in(1/2,1)$, depending only on $\varepsilon_{0}$, 
such that
\[
\sin^{2}(\omega\xi_{0})<K_{1}-K\xi_{0}^{2}\sin^{2}\omega
\]
 for every $0<\omega\leq\frac{\pi}{2}$. 
Since $\xi\mapsto\sin^{2}(\omega\xi)$ is increasing and $\xi\mapsto K_{1}%
-K\xi^{2}\sin^{2}\omega$ is decreasing on $[0,\xi_{0}]$, it follows that
\[
\sin^{2}(\omega\operatorname*{Re}z)<K_{1}-K(\operatorname*{Re}z)^{2}\sin
^{2}\omega
\]
for all $z$ with $\operatorname*{Re}z\in\lbrack0,\xi_{0}]$.  This shows that
\eqref{impossible} is impossible and concludes the proof. 
\end{proof}

Let $p_{0}$ be such that $2-\frac{2}{p_{0}}=\xi_{0}$. Then
\begin{equation}
\frac{4}{3}<p_{0}=\frac{2}{2-\xi_{0}}<2. \label{p0}%
\end{equation}
 We recall that for every $0<m\le L_{0}$ and $r>0$ the triangle $A_{r}^{m}$ is defined in \eqref{triangle}.
 Since the conjugate exponent of $p_{0}$ satisfies $p_{0}^{\prime}<4$, by the
Sobolev embedding theorem we have $H_{0}^{1}(A_{r}^{ m};
\mathbb{R}^2)\hookrightarrow
L^{p_{0}^{\prime}}(A_{r}^{ m};\mathbb{R}^2)$. Hence, by duality $L^{p_{0}}(A_{r}^{ m};\mathbb{R}^2)\hookrightarrow
H^{-1}(A_{r}^{ m};\mathbb{R}^2)$. Similarly, 
given $g\in W^{1,p_{0}}(A_{r}^{ m};\mathbb{R}^2)$, its trace on
\begin{equation}
\Gamma_{r}^{m}=\{(x,mx):\,0<x<r\} \label{Gamma r}\,,%
\end{equation}
still denoted by $g$, belongs to $W^{1-1/p_{0},p_{0}}(\Gamma_{r}^{m};\mathbb{R}^2)$
and,  by the embedding theorem for fractional Sobolev
spaces, we have $g\in L^{2}(\Gamma_{r}^{ m} ;\mathbb{R}^2)
\hookrightarrow 
H^{-1/2}(\Gamma_{r}^{ m}
;\mathbb{R}^2)$. Therefore, given $f\in
L^{p_{0}}(A_{r}^{ m};\mathbb{R}^{2})$ and $g\in  W^{1,p_{0}}(A
_{r}^{ m};\mathbb{R}^{2})$ there exists a unique weak solution $w\in H^{1}%
(A_{r}^{ m};\mathbb{R}^{2})$ to the problem%
\begin{equation}
\left\{
 {\setstretch{1.1}
\begin{array}
[c]{ll}%
-\operatorname{div}\mathbb{C}Ew
=f & \text{in }A_{r}^{ m}\,,
\\
(\mathbb{C}Ew)\nu^{ m}=g & \text{on }\Gamma_{r}^{ m}\,,\\
w=0 & \text{on }\partial A_{r}^{ m}\setminus\Gamma_{r}^{ m}\,, 
\end{array}
}
\right.
\label{eul13}%
\end{equation}
 where $\nu^{m}:=(-m,1)/{\sqrt{1+m^{2}}}$ is the outer unit normal to $A_{r}^{m}$ on $\Gamma_{r}^{m}$.

In the next theorem we will use \cite[Theorems 2.1 and
2.2]{nicaise1992lame} (see also \cite[Theorem I]%
{grisvard1989singularites}).

\begin{theorem}
\label{theorem grisvard 1}Let $p_{0}$ be as in \eqref{p0}, let $r>0$,  let $0<\eta_{0}\leq m\leq L_{0}$, let  $A_{r}^{m}$ and $\Gamma_{r}^{m}$ be defined by \eqref{triangle} and 
\eqref{Gamma r}, 
let
$f\in L^{p_{0}}(A_{r}^{ m};\mathbb{R}^{2})$, let $g\in  W^{1,p_{0}}%
( A_{r}^{m};\mathbb{R}^{2})$,  with $g=0$ on 
one of the sides of the triangle $A_{r}^{m}$ different from $\Gamma_{r}^{m}$,  and let
$w\in H^{1}(A_{r}^{ m};\mathbb{R}^{2})$ be the unique weak solution to the problem
\eqref{eul13}. Then $w$ belongs to $W^{2,p_{0}}(A_{r}^{ m}%
;\mathbb{R}^{2})$ and
\begin{equation}\Vert\nabla^{2}w\Vert_{L^{p_{0}}(A_{r}^{ m})}   \leq\kappa\big(\Vert
f\Vert_{L^{p_{0}}(A_{r}^{ m})}+ \Vert\nabla g\Vert_{L^{p_{0}}(A_{r}^{m})}\big)\,,
\label{g14}%
\end{equation}
for  a constant $\kappa>0$  depending on $\lambda$, $\mu
$,  $\eta_{0}$, and $L_{0}$, but independent of $r$, $m$, $f$,  and $g$.
\end{theorem}

\begin{proof}  By a rescaling argument we see that it is enough to prove the result for $r=1$.

\textbf{Step 1:} Assume $g=0$. By \cite[Theorem 3.1]{rossle2000corner}, Lemma
\ref{lemma no real}, and \eqref{p0}, we have that $w\in W^{2,p_{0}}%
(A_{1}^{ m};\mathbb{R}^{2})$. Let $X^{ m}$ be the space of functions $z\in W^{2,p_{0}}(A_{1}^{ m};\mathbb{R}^{2})$  such that $z=0$
on $\partial A_{1}^{ m}\setminus\Gamma_{1}^{ m}$ and $(\mathbb{C}Ew)\nu=0$ on
$\Gamma_{1}^{ m}$ endowed with the norm
 of $W^{2,p_{0}}%
(A_{1}^{ m};\mathbb{R}^2)$. 
Consider the continuous linear operator $\mathcal{L}:X^{ m}\rightarrow L^{p_{0}%
}(A_{1}^{ m};\mathbb{R}^{2})$ defined by $\mathcal{L}(z):=-\operatorname{div}%
\mathbb{C}Ez$. By what we just proved (see \eqref{eul13}), $\mathcal{L}$ is invertible, and so by
the Closed Graph Theorem we obtain that there exists a constant $C>0$,
depending on $\lambda$, $\mu$, and $m$, such that%
\begin{equation}
\Vert w\Vert_{W^{2,p_{0}}(A_{1}^{ m})}\leq C\Vert
f\Vert_{L^{p_{0}}(A_{1}^{ m})} \label{g11}%
\end{equation}
for every solution $w$ to \eqref{eul13}  with $g=0$.

\textbf{Step 2: } In the general case $g\in  W^{1,p_{0}}(A_{1}^{m};\mathbb{R}^{2})$, 
 with $g=0$ on 
one of the sides of the triangle $A_{1}^{m}$ different from $\Gamma_{1}^{m}$, 
recalling that $1<p_{0}<2$, we can reason as in
\cite[Lemma 3.12]{fonseca2007equilibrium} and using
\cite[Theorem 1.5.2.8]{grisvard2011elliptic} we can find $w^{1}\in W^{2,p_{0}}(A_{1}^{ m}%
;\mathbb{R}^{2})$, with  $(\mathbb{C}Ew^{1})\nu^{ m}=g$ on $\Gamma_{1}^{ m}$ and
$w^{1}=0$ on $\partial A_{1}^{ m}$,  such that %
\begin{equation}
\Vert w^{1}\Vert_{W^{2,p_{0}}(A_{1}^{ m})}\leq C_{1}\Vert g\Vert_{W^{1-1/p_{0},p_{0}%
}(\Gamma_{1}^{ m})}  \leq C_{2}\Vert \nabla g\Vert_{L^{p_{0}%
}(A_{1}^{ m})} \label{g12}%
\end{equation}
for some  constants $C_{1},\,C_{2}>0$  depending on $\lambda$, $\mu$, and $m$, where in 
the last inequality we used the trace estimate and Poincar\'e's inequality. Then the
function $v:=w-w^{1}$ is a weak solution to \eqref{eul13}  in $A_{1}^{m}$  with
$f$ replaced by $f+\operatorname{div}\mathbb{C}Ew^{1}$ and $g$ replaced by
zero. Hence, by the previous step $v\in W^{2,p_{0}}(A_{1}^{ m};\mathbb{R}^{2})$. Moreover, using
\eqref{g11} for $v$ and \eqref{g12} for $w_{1}$ we obtain a constant
$\kappa_{m}>0$, depending on
$\lambda$, $\mu$, and $m$, but independent of $f$ and $g$, such that
\begin{equation}
\Vert\nabla^{2}w\Vert_{L^{p_{0}}(A_{1}^{m})}    \leq\kappa_{m}\big(\Vert
f\Vert_{L^{p_{0}}(A_{1}^{m})}+ \Vert\nabla g\Vert_{L^{p_{0}}(A_{1}^{m})}\big)\,.
\label{g14bis}%
\end{equation}

\textbf{Step 3: }Let $m_{0}\in[\eta_{0},L_{0}]$. We want to prove that there exists $\varepsilon_{0}>0$ such that, if $m\in[\eta_{0},L_{0}]$ and $|m-m_{0}|<\varepsilon_{0}$, then $w$ satisfies the estimate
\begin{equation}
\Vert\nabla^{2}w\Vert_{L^{p_{0}}(A_{1}^{m})}  \leq2\kappa_{m_{0}}\big(\Vert
f\Vert_{L^{p_{0}}(A_{1}^{m})}+\Vert\nabla g\Vert_{L^{p_{0}}(A_{1}^{m})}\big)\,,
\label{g14ter}%
\end{equation}
with the constant $\kappa_{m_{0}}$ corresponding to $m_{0}$.

Since $w\in W^{2,p_{0}}(A_{1}^{m};\mathbb{R}^{2})$ satisfies
\eqref{eul13}, a direct computation (using Remark \ref{remark lame}) shows that the function $z\in W^{2,p_{0}}(A_{1}^{m_{0}};\mathbb{R}^{2})$ defined by $z(x,y):=w(x,\frac{m}{m_{0}}y)$ satisfies
\begin{gather}
-\operatorname{div}\mathbb{C}Ez
=\tilde{f}+f^{z}\quad\text{in }A_{1}^{m_{0}}\,,
\label{eul11bis}\\
(\mathbb{C}Ez)\nu^{ m_{0}}=\tilde{g}+\hat{g}^{z}+\check{g}^{z}\quad\text{on }\Gamma_{1}^{m_{0}}\,,\label{eul12bis}\\
z=0\quad\text{on }\partial A_{1}^{ m_{0}}\setminus\Gamma_{1}^{m_{0}}\,, \label{eul13bis}%
\end{gather}
with $\tilde{f}, f^{z}\in L^{p_{0}}(A_{1}^{m_{0}};\mathbb{R}^{2})$ and
 $\tilde{g}, \hat{g}^{z}, \check{g}^{z}\in W^{1,p_{0}}(A_{1}^{m_{0}};\mathbb{R}^{2})$
 defined by 
 \begin{align}
 &\tilde{f}(x,y):=f(x,\tfrac{m}{m_{0}}y)
 \quad\text{and} \quad\tilde{g}(x,y):=g(x,\tfrac{m}{m_{0}}y)\,,\label{wry1}
 \\
 &f_{1}^{z}:=\mu\big((\tfrac{m_{0}}{m})^{2}-1\big)\partial_{yy}^{2}z_{1}+
 (\lambda+\mu)
 (\tfrac{m_{0}}{m}-1)\partial_{xy}^{2}z_{2}\,,
 \label{wry2}
 \\
 &f_{2}^{z}:=(\lambda+2\mu)\big((\tfrac{m_{0}}{m})^{2}-1\big)\partial_{yy}^{2}z_{2}+
 (\lambda+\mu)
 (\tfrac{m_{0}}{m}-1)\partial_{xy}^{2}z_{1}\,,
 \label{wry3}
 \\
 &\hat{g}_{1}^{z}:=(2\mu+\lambda)\big(\tfrac{m}{\sqrt{1+m^{2}}}-
 \tfrac{m_{0}}{\sqrt{1+m_{0}^{2}}}\big)\partial_{x}z_{1}
 -\mu\big(\tfrac{1}{\sqrt{1+m^{2}}}-
 \tfrac{1}{\sqrt{1+m_{0}^{2}}}\big)\partial_{x}z_{2}
\,,\label{wry4}
 \\
 &\check{g}_{1}^{z}:=
 -\mu\big(\tfrac{1}{\sqrt{1+m^{2}}}-
 \tfrac{1}{\sqrt{1+m_{0}^{2}}}\big)\partial_{y}z_{1}
 +\lambda\big(\tfrac{m}{\sqrt{1+m^{2}}}-
 \tfrac{m_{0}}{\sqrt{1+m_{0}^{2}}}\big)
 \partial_{y}z_{2}
 \nonumber
 \\
 &\qquad - \tfrac{\mu}{\sqrt{1+m^{2}}}(\tfrac{m_{0}}{m}-1)\partial_{y}z_{1}+\tfrac{\lambda m}{\sqrt{1+m^{2}}}(\tfrac{m_{0}}{m}-1)\partial_{y}z_{2}\,,
 \label{wry5}
 \\
&\hat{g}_{2}^{z}:=\mu\big(\tfrac{m}{\sqrt{1+m^{2}}}-
 \tfrac{m_{0}}{\sqrt{1+m_{0}^{2}}}\big)
 \partial_{x}z_{2}
 + \lambda\big(\tfrac{ m }{\sqrt{1+m^{2}}}-
 \tfrac{ m_{0} }{\sqrt{1+m_{0}^{2}}}\big)
 \partial_{x}z_{1}\,,\label{wry6}
 \\
&\check{g}_{2}^{z}:=\mu\big(\tfrac{m}{\sqrt{1+m^{2}}}-
 \tfrac{m_{0}}{\sqrt{1+m_{0}^{2}}}\big)
 \partial_{y}z_{1}
 -2\mu\big(\tfrac{1}{\sqrt{1+m^{2}}}-
 \tfrac{1}{\sqrt{1+m_{0}^{2}}}\big)\partial_{y}z_{2}-\tfrac{2\mu}{\sqrt{1+m^{2}}}(\tfrac{m_{0}}{m}-1)\partial_yz_2
 \nonumber
 \\
  &\qquad-\lambda\big(\tfrac{1}{\sqrt{1+m^{2}}}-
 \tfrac{1}{\sqrt{1+m_{0}^{2}}}\big)
\partial_{y}z_{2}
 + \tfrac{\mu m}{\sqrt{1+m^{2}}}(\tfrac{m_{0}}{m}-1)\partial_{y}z_{1}
  -  \tfrac{\lambda}{\sqrt{1+m^{2}}}(\tfrac{m_{0}}{m}-1)\partial_{y}z_{2}\,.\label{wry7}
 \end{align}
 Since $z\in W^{2,p_{0}}(A_{1}^{m_{0}};\mathbb{R}^{2})$ and $z=0$ on $\partial A_{1}^{ m_{0}}\setminus\Gamma_{1}^{m_{0}}$, we have that $\partial_{x}z=0$ a.e.\ on $[0,1]\times\{0\}$ and 
 $\partial_{y}z=0$ a.e.\ on $\{1\}\times[0,m_{0}]$, hence
 $\hat{g}^{z}=0$ a.e.\ on $[0,1]\times\{0\}$ and 
 $\check{g}^{z}=0$ a.e.\ on $\{1\}\times[0,m_{0}]$. Moreover $\tilde{g}$ vanishes on one of the sides of $A_{1}^{m_{0}}$ different from $\Gamma_{1}^{m_{0}}$. We may assume that $\tilde{g}=0$ a.e.\ on $[0,1]\times\{0\}$, the other case being analogous.

 Let $\hat{z}$ and $\check{z}$ be the unique solutions of the problems
 \[
 \left\{
 {\setstretch{1.1}
\begin{array}
[c]{ll}%
-\operatorname{div}\mathbb{C}E\hat{z}
=\tilde{f}+f^{z}&\text{in }A_{1}^{m_{0}}\,,
\\
(\mathbb{C}E\hat{z})\nu^{ m_{0}}=\tilde{g}+\hat{g}^{z}&\text{on }\Gamma_{1}^{m_{0}}\,,\\
\hat{z}=0 &\text{on }\partial A_{1}^{ m_{0}}\setminus\Gamma_{1}^{m_{0}}\,, 
\end{array}
}
\right.
\quad\text{ and }\quad
\left\{
 {\setstretch{1.1}
\begin{array}
[c]{ll}%
-\operatorname{div}\mathbb{C}E\check{z}
=0 & \text{in }A_{1}^{m_{0}}\,,
\\
(\mathbb{C}E\check{z})\nu^{ m_{0}}=\check{g}^{z} & \text{on }\Gamma_{1}^{m_{0}}\,,
\\
\check{z}=0 & \text{on }\partial A_{1}^{ m_{0}}\setminus\Gamma_{1}^{m_{0}}\,.
\end{array}
}
\right.%
\]
By linearity we have $z=\hat{z}+\check{z}$. Since $\tilde{g}+\hat{g}^{z}$ and $\check{g}^{z}$ vanish on one of the sides of $A_{1}^{m_{0}}$ different from $\Gamma_{1}^{m_{0}}$, by \eqref{g14bis} we have
\begin{align*}
\Vert\nabla^{2}\hat{z}\Vert_{L^{p_{0}}(A_{1}^{m_{0}})}  &  \leq\kappa_{m_{0}}\big(\Vert
\tilde{f}+f^{z}\Vert_{L^{p_{0}}(A_{1}^{m_{0}})}+\Vert\nabla \tilde{g}+\nabla\hat{g}^{z}\Vert_{L^{p_{0}}(A_{1}^{m_{0}})}\big)\,,
\\
\Vert\nabla^{2}\check{z}\Vert_{L^{p_{0}}(A_{1}^{m_{0}})}  &  \leq\kappa_{m_{0}}\Vert\nabla \check{g}^{z}\Vert_{L^{p_{0}}(A_{1}^{m_{0}})}\,.
\end{align*}
hence
\begin{equation}
\Vert\nabla^{2}z\Vert_{L^{p_{0}}(A_{1}^{m_{0}})}    \leq\kappa_{m_{0}}\big(\Vert
\tilde{f}+f^{z}\Vert_{L^{p_{0}}(A_{1}^{m_{0}})}+\Vert\nabla \tilde{g}+\hat{g}^{z}\Vert_{L^{p_{0}}(A_{1}^{m_{0}})}+\Vert\nabla \check{g}^{z}\Vert_{L^{p_{0}}(A_{1}^{m_{0}})}\big)\,.
\label{wry11}
\end{equation}

Let us fix $\omega>0$. Since $z\in W^{2,p_{0}}( A_1^{m_0} ;\mathbb{R}^2)$, by \eqref{wry1}, \eqref{wry2}, \eqref{wry3}, \eqref{wry4}, \eqref{wry5}, \eqref{wry6}, and \eqref{wry7} there exists $\varepsilon_{\omega}>0$ such that, if $m\in[\eta_{0},L_{0}]$ and $|m-m_{0}|<\varepsilon_{\omega}$, then
\begin{gather}
\Vert \nabla^{2}w\Vert
_{L^{p_{0}}(A_{1}^{m})}
\leq (1+\omega)\Vert \nabla z^{2} \Vert_{L^{p_{0}}(A_{1}^{m_{0}})}
\,,\label{qrt1}
\\
\Vert \tilde{f}\Vert
_{L^{p_{0}}(A_{1}^{m_{0}})}+\Vert \nabla\tilde{g}\Vert
_{L^{p_{0}}(A_{1}^{m_{0}})}
\leq (1+\omega)\big(\Vert f\Vert_{L^{p_{0}}(A_{1}^{m})}+\Vert \nabla g\Vert_{L^{p_{0}}(A_{1}^{m})}\big)
\,,\label{qrt2}
\\
\Vert f^{z}\Vert
_{L^{p_{0}}(A_{1}^{m_{0}})}+\Vert \nabla\hat{g}^{z}\Vert_{L^{p_{0}}(A_{1}^{m_{0}})}+
\Vert \nabla\check{g}^{z}\Vert_{L^{p_{0}}(A_{1}^{m_{0}})}
\leq\omega
\Vert \nabla^{2} z \Vert
_{L^{p_{0}}(A_{1}^{m_{0}})}
\,.\label{qrt9}
\end{gather}

From \eqref{wry11}, \eqref{qrt2}, and \eqref{qrt9} we obtain
\begin{equation}
\Vert\nabla^{2}z\Vert_{L^{p_{0}}(A_{1}^{m_{0}})}   \leq(1+\omega)\kappa_{m_{0}}\big(\Vert
f\Vert_{L^{p_{0}}(A_{1}^{m})}+\Vert\nabla g\Vert_{L^{p_{0}}(A_{1}^{m})}\big)
+\omega \kappa_{m_{0}} \Vert \nabla^{2}z\Vert_{L^{p_{0}}(A_{1}^{m_{0}})}
\,.
\label{wry111}
\end{equation}
We now choose $\omega<1/3$ such that $\omega \kappa_{m_{0}}<1/9$. If $m\in[\eta_{0},L_{0}]$ and $|m-m_{0}|<\varepsilon_{\omega}$, then \eqref{wry111} gives
\[
\tfrac{8}{9}\Vert\nabla^{2}z\Vert_{L^{p_{0}}(A_{1}^{m_{0}})}   \leq\tfrac{4}{3}\kappa_{m_{0}}\big(\Vert
f\Vert_{L^{p_{0}}(A_{1}^{m})}+\Vert\nabla g\Vert_{L^{p_{0}}(A_{1}^{m})}\big)\,.
\]
Using \eqref{qrt1} we obtain
\[
\tfrac{2}{3}\Vert\nabla^{2}w\Vert_{L^{p_{0}}(A_{1}^{m})}   \leq\tfrac{4}{3}\kappa_{m_{0}}\big(\Vert
f\Vert_{L^{p_{0}}(A_{1}^{m})}+\Vert\nabla g\Vert_{L^{p_{0}}(A_{1}^{m})}\big)\,,
\]
which implies \eqref{g14ter}.

\textbf{Step 4: }By compactness we can cover $[\eta_{0},L_{0}]$ by a finite number of open intervals $(m_{i}-\varepsilon_{i},m_{i}+\varepsilon_{i})$, with $m_{i}\in[\eta_{0},L_{0}]$, such that \eqref{g14ter} holds with $m_{0}$ and $\varepsilon_{0}$ replaced by $m_{i}$ and $\varepsilon_{i}$. Then \eqref{g14} holds with $\kappa:=2\max_{i}\kappa_{m_{i}}$.
\end{proof}

\section{Regularity at Corners}

In this section we obtain  a  precise estimate on the $W^{2,p_{0}}$ norm of the
equilibrium solution $u$ in a neighborhood of the corners.

Next we use a fixed point theorem to prove that $u$
belongs to $W^{2,p_{0}}$ near $(\alpha,0)$.

\begin{theorem}
\label{theorem regularity at corner}Under the assumptions of Theorem
\ref{theorem C3}, we have that $u\in W^{2,p_{0}}(\Omega_{h})$.
\end{theorem}

To prove this theorem we  need  some auxiliary results. We begin with
Poincar\'{e}'s inequality. 

\begin{lemma}
[Poincar\'{e}'s inequality]\label{lemma poincare corner}Let $0<r<\beta-\alpha$
and let $p\geq1$. Then%
\begin{equation}
\Vert  v \Vert_{L^{p}(\Omega_{h}^{\alpha,\alpha+r})}\leq L_{0}r\Vert\nabla
 v\Vert_{L^{p}(\Omega_{h}^{\alpha,\alpha+r})}%
\label{qwe41}
\end{equation}
for every $ v\in W^{1,p}(\Omega_{h}^{\alpha,\alpha+r})$ such that
$ v(x,0)=0$ for $x\in (\alpha,\alpha+r)$ (in the sense of traces).
\end{lemma}

\begin{proof}
By density we can assume that $ v\in C^{\infty}( %
\mathbb{R}^{2})$. For $(x,y)\in\Omega_{h}^{\alpha,\alpha+r}$, by the
fundamental theorem of calculus
 and  H\"{o}lder's inequality  we have 
\[
| v(x,y)|
\leq(L_{0}r)^{1/p^{\prime}}\Big(  \int_{0}^{h(x)}%
|\partial_{y} v(x,t)|^{p}dt\Big)^{1/p},
\]
where we used the fact that $h(x)\leq L_{0}r$. Raising both sides to power $p$
and integrating over $\Omega_{h}^{\alpha,\alpha+r}$ gives%
\[
\int_{\Omega_{h}^{\alpha,\alpha+r}}| v(x,y)|^{p}dxdy   \leq
(L_{0}r)^{1+p/p^{\prime}}\int_{\Omega_{h}^{\alpha,\alpha+r}}|\partial
_{y} v(x,y)|^{p}dxdy\,,
\]
which  yields \eqref{qwe41} and  concludes the proof.
\end{proof}

\begin{remark}
\label{remark poincare corner} Let
$r>0$, let $0<\eta_{0}\leq m\leq L_{0}$, and let  $A_{r}^{ m}$ be the triangle defined in \eqref{triangle}.
With a proof similar to the one of Lemma \ref{lemma poincare corner}, one can show that%
\[
\Vert  v \Vert_{L^{p}(A_{r}^{ m})}\leq \max\{L_{0},1/\eta_{0}\}
r\Vert\nabla  v\Vert_{L^{p}(A_{r}^{ m})}%
\]
for every $ v\in W^{1,p}(A_{r}^{ m})$ such that $v(x,0)=0$ for $x\in(0,r)$ or  
$v(r,y)=0$ for $y\in(0,mr)$  (in the
sense of traces).
\end{remark}

We turn to the proof of Theorem \ref{theorem regularity at corner}.

\begin{proof}
[Proof of Theorem \ref{theorem regularity at corner}]In view of Theorem
\ref{theorem euler-lagrange}, it is enough to prove that there exists $r>0$
sufficiently small such that%
\[
u\in W^{2, p_{0}}(\Omega_{h}^{\alpha,\alpha+r}\cup\Omega_{h}^{\beta-r,\beta})\,,
\]
where $\Omega_{h}^{c,d}$ is defined in \eqref{Omega alpha beta}. We will only
show that $u\in W^{2, p_{0}}(\Omega_{h}^{\alpha,\alpha+r})$, since the other
endpoint can be treated in a similar way. By a translation, without loss of
generality, we may assume that $\alpha=0$. We modify $u$ far from $(0,0)$ to
obtain a new function $\tilde{u}$
vanishing away from
$(0,0)$. We then use the transformation \eqref{Phi 1} to map $\Omega_{h}%
^{0,r}$ into the triangle $ A_{r}^{m}$ defined in \eqref{triangle} with
$m=h^{\prime}(0)$. A fixed point argument will  allow us to show that the
resulting system has a solution in $W^{2,p_{0}}( A_{r}^{m})$, and due to uniqueness, 
this solution is $\tilde{u}\circ\Phi$ itself. To apply the Banach fixed point
theorem, we will use Lemmas \ref{lemma sigma alpha} and
\ref{lemma fractional a}.

\noindent\textbf{Step 1: Localization. } Let
\begin{equation}
w^{0}(x,y):=(e_{0}x,0)\,,\quad(x,y)\in\mathbb{R}^{2}.\label{function w0 a}%
\end{equation}
For every $r>0$ let $\varphi_{r}\in C^{\infty
}(\mathbb{R}^{2})$ be a function such that $0\leq\varphi_{r}\leq1$, $\varphi_{r}(x,y)=1$ for $x\leq 5r/8$, $\varphi_{r}(x,y)=0$ for $x\geq 7r/8$, $\Vert
\nabla\varphi_{r}\Vert_{L^{\infty}(\mathbb{R}^{2})}\leq C/r$ and $\Vert
\nabla^{2}\varphi_{r}\Vert_{L^{\infty}(\mathbb{R}^{2})}\leq C/r^{2}$, where
$C>0$ is a constant independent of $r$. 
In $\Omega_{h}$ we define 
\begin{equation}
\tilde{u}:=(u-w^{0})\varphi_{r}\,. \label{u tilde a}%
\end{equation}
 If $0<r<\beta$ we have $\tilde{u}(x,0)=0$ for $0\leq x\le r$ and $\tilde{u}(r,y)=0$ for $0\leq y\leq h(r)$. In other words $\tilde{u}=0$ on $\partial\Omega_{h}^{0,r}\setminus\Gamma_{h}^{0,r}$, where $\Gamma_{h}^{0,r}=\{(x,h(x)): 0< x< r\}$. Moreover,  we have
\[
\nabla\tilde{u}=\varphi_{r}\nabla(u-w^{0})+(u-w^{0})\otimes\nabla\varphi
_{r}\,.
\]
By Poincar\'{e}'s inequality (see Lemma \ref{lemma poincare corner}) we obtain
\begin{equation}
\Vert\nabla\tilde{u}\Vert_{L^{2}(\Omega_{h}^{0, r})}\leq C\Vert\nabla
(u-w^{0})\Vert_{L^{2}(\Omega_{h}^{0, r})}\,. \label{gradient u tilde a}%
\end{equation}

By direct computation, it follows from Remark \ref{remark lame} that
\begin{equation}
\left\{
{\setstretch{1.3}
\begin{array}
[c]{ll}%
-\operatorname{div}
\mathbb{C}E\tilde{u}=-\mu\Delta\tilde{u}-(\lambda
+\mu)\nabla\operatorname{div}\tilde{u}= f & \text{in
}\Omega_{h}^{0,r}\,,\\
(\mathbb{C}E\tilde{u})\nu^{h}=2\mu(E\tilde{u})\nu^{h}+\lambda(\operatorname{div}%
\tilde{u})\nu^{h}= g  & \text{on }\Gamma_{h}^{0,r}\,,\\
\tilde{u}=0 & \text{on }\partial\Omega_{h}^{0,r}\setminus\Gamma_{h}^{0,r} \,,
\end{array}
}
\right.  \label{qwe50}
\end{equation}
where $\nu^{h}(x):=(-h^{\prime}(x),1)/\sqrt{1+(h^{\prime}(x))^{2}}$
denotes the outer unit normal to $\Omega_{h}$ on $\Gamma_{h}$,  and%
\begin{align}
 f &  :=-\mu\big( 2 (\nabla u-\nabla w^{0})\nabla\varphi_{r}
+(u-w^{0})\Delta\varphi_{r}\big)\label{f tilde v}\\
&  \quad-(\lambda+\mu)\big((\operatorname{div}u-\operatorname{div}w^{0}%
)\nabla\varphi_{r}+(\nabla u-\nabla w^{0})^{ T}
\nabla\varphi_{r}+\nabla
^{2}\varphi_{r}(u-w^{0})\big)\,,\nonumber\\
 g  &  :=\mu\big((u-w^{0})\otimes\nabla\varphi_{r}%
+\nabla\varphi_{r}\otimes(u-w^{0})\big)\nu^{h}+\lambda\operatorname*{trace}((u-w^{0})\otimes\nabla\varphi_{r}%
)\nu^{h}\,.\label{g tilde a}
\end{align}
 Note that $f\in L^{2}(\Omega_{h}^{0,r};\mathbb{R}^{2})$.
Since $h\in C^{3}([0,\beta])$ by Theorem \ref{theorem C3}, we can consider $\nu^{h}$ as a $C^{2}$ function in the strip
$[0,\beta]\times\mathbb{R}$. This shows that $g\in H^{1}(\Omega_{h}^{0,r};\mathbb{R}^{2})$. Since $u(x,0)-w^{0}(x,0)=0$ for a.e. $x\in[0,r]$ and $\varphi_{r}(r,y)=0$ for every $y\in\mathbb{R}$, we conclude that the trace of $g$ vanishes on $\partial\Omega_{h}^{0,r}\setminus\Gamma_{h}^{0,r}$. We also observe that for every $f\in L^{2}(\Omega_{h}^{0,r};\mathbb{R}^{2})$ and every $g\in H^{1}(\Omega_{h}^{0,r};\mathbb{R}^{2})$ problem 
\eqref{qwe50} has a unique weak solution in $H^1(\Omega_{h}^{0,r};\mathbb{R}^{2})$.

\noindent\textbf{Step 2: Straightening the boundary. }Let  $r$ be as in \eqref{r interval1}
and let 
\begin{equation}
v:=\tilde{u}\circ\Phi\in H^{1}( A_{r}^{m};\mathbb{R}^{2})\,, \label{function v}%
\end{equation}
where $ A_{r}^{m}$ and $\Phi$ are defined in \eqref{triangle} and \eqref{Phi 1}, with $m=h_{0}^{\prime}$. 
 Recalling  Remark \ref{remark lame} and the fact that
$\tilde{u}(x,y)=v(x,\sigma(x)y)$, it
follows by direct computation  and by Lemma \ref{lemma normals}  that $v$ satisfies the boundary value
problem
\begin{equation}
\left\{
{\setstretch{1.1}
\begin{array}
[c]{ll}%
-\operatorname{div}\mathbb{C}Ev=-\mu\Delta v-(\lambda+\mu)\nabla
\operatorname{div}v= f\circ\Phi+f^{v} &
\text{in } A_{r}^{m}\,,\\
(\mathbb{C}Ev)\nu^{0}=2\mu(Ev)\nu^{0}+\lambda(\operatorname{div}v)
\nu^{0}= g\circ\Phi+\hat{g}^{v}+\check{g}^{v} & \text{on }\Gamma_{r}^{m}\,,\\
v=0 & \text{on }\partial A_{r}^{m}\setminus\Gamma_{r}^{m}\,,
\end{array}
}
\right.  \label{eul 8}%
\end{equation}
where $\nu^{0}:=(-h_0^{\prime},1)/\sqrt{1+(h_0^{\prime}%
)^{2}}$ is the outer unit normal to  $A_{r}^{m}$ on $\Gamma_{r}^{m}:=\{(x,h_0^{\prime}x):{0<x<r}\}$, $f^{v}%
=(f_{1}^{v},f_{2}^{v})\in H^{-1}( A_{r}^{m};\mathbb{R}^{2})$ is defined by%
\begin{align}
f_{1}^{v}  &  :=\mu\bigg[\left(  \sigma^{2}-1+y^{2}\frac{(\sigma^{\prime}%
)^{2}}{\sigma^{2}}\right)  \partial_{yy}^{2}v_{1}+2y\frac{\sigma^{\prime}%
}{\sigma}\partial_{xy}^{2}v_{1}\bigg]\nonumber\\
&  \quad+(\lambda+\mu)\bigg[y^{2}\frac{(\sigma^{\prime})^{2}}{\sigma^{2}%
}\partial_{yy}^{2}v_{1}+2y\frac{\sigma^{\prime}}{\sigma}\partial_{xy}^{2}%
v_{1}+(\sigma-1)\partial_{xy}^{2}v_{2}\,+y\sigma^{\prime}\partial_{yy}%
^{2}v_{2}\,\bigg]\,,\label{f1}\\
&  \quad+\mu y\frac{\sigma^{\prime\prime}}{\sigma}\partial
_{y}v_{1}+(\lambda+\mu)\bigg[y\frac{\sigma^{\prime\prime}}{\sigma}\partial
_{y}v_{1}+\sigma^{\prime}\partial_{y}v_{2}\bigg]\,,\nonumber\\
f_{2}^{v}  &  :=\mu\bigg[\left(  \sigma^{2}-1+y^{2}\frac{(\sigma^{\prime}%
)^{2}}{\sigma^{2}}\right)  \partial_{yy}^{2}v_{2}+2y\frac{\sigma^{\prime}%
}{\sigma}\partial_{xy}^{2}v_{2}\bigg]\nonumber\\
&  \quad+(\lambda+\mu)[(\sigma-1)\partial_{xy}^{2}v_{1}+\sigma^{\prime
}y\partial_{yy}^{2}v_{1}+(\sigma^{2}-1)\partial_{yy}^{2}v_{2}] \label{f2}%
\\
&  \quad+\mu y\frac{\sigma^{\prime\prime}}{\sigma}\partial
_{y}v_{2}+(\lambda+\mu)\sigma^{\prime}\partial_{y}v_{1}\,,\nonumber
\end{align}
 while $\hat{g}^{v}=(\hat{g}_{1}^{v},\hat{g}_{2}^{v})\in W^{1,p_{0}}(A_{r}^{m};\mathbb{R}^{2})$ 
and 
$\check{g}^{v}=(\check{g}_{1}^{v},\check{g}_{2}^{v})\in W^{1,p_{0}}(A_{r}^{m};\mathbb{R}^{2})$ are
defined by \eqref{g1hat}--\eqref{g2check}.

\noindent\textbf{Step 3: Fixed Point Argument. } To prove the $W^{2,p_{0}}$
regularity of $v$ we use a fixed point argument in  the space
\[
X_{r}^{ m}:=\{z\in W^{2,p_{0}}(A_{r}^{ m};\mathbb{R}^{2}):\,z=0\text{ on }\partial A_{r}^{ m}\setminus \Gamma_{r}^{ m}\}\,,
\]
endowed with the norm 
\begin{equation}
\Vert z\Vert_{X_{r}^{ m}}:=\Vert\nabla^{2}%
z\Vert_{L^{p_{0}}(A_{r}^{ m})}\,. \label{norm Z}%
\end{equation}

Let us first prove that $\Vert\cdot\Vert_{X_{r}^{m}}$ is a norm equivalent to $\Vert\cdot\Vert_{W^{2,p_{0}}(A_{r}^{m},\mathbb{R}^{2})}$. By Poincar\'e's inequality (see Remark \ref{remark poincare corner}) there exists a constant $C>0$, depending only on $\eta_{0}$ and $L_{0}$, such that
\begin{equation}
\Vert\varphi\Vert_{L^{p_{0}}(A_{r}^{m})}\leq Cr\Vert\nabla\varphi\Vert_{L^{p_{0}}(A_{r}^{m})}
\label{qty32}
\end{equation}
for every $\varphi\in W^{1,p_{0}}(A_{r}^{m},\mathbb{R}^{2})$ such that $\varphi=0$ on one of the sides of the triangle $A_{r}^{m}$ different from $\Gamma_{r}^{m}$. 
If $z\in X_{r}^{m}$, then $z$, $\partial_{x}z$, and $\partial_{y}z$ satisfy this property, hence
\begin{align*}
\Vert z \Vert_{L^{p_{0}}(A_{r}^{m})}&\leq Cr\Vert\nabla z\Vert_{L^{p_{0}}(A_{r}^{m})}\,,
\\
\Vert \partial_{x}z \Vert_{L^{p_{0}}(A_{r}^{m})}&\leq Cr\Vert\nabla \partial_{x}z\Vert_{L^{p_{0}}(A_{r}^{m})}\,,
\\
\Vert \partial_{y}z \Vert_{L^{p_{0}}(A_{r}^{m})}&\leq Cr\Vert\nabla \partial_{y}z\Vert_{L^{p_{0}}(A_{r}^{m})}\,.
\end{align*}
These inequalities imply that there exists a constant $K_{r}>0$ such that
\[
\Vert z \Vert_{W^{2,p_{0}}(A_{r}^{m})}\leq K_{r}\Vert\nabla^{2} z\Vert_{L^{p_{0}}(A_{r}^{m})}
\]
for every $z\in X_{r}^{m}$, concluding the proof of the equivalence between the norms $\Vert\cdot\Vert_{X_{r}^{m}}$ and $\Vert\cdot\Vert_{W^{2,p_{0}}(A_{r}^{m},\mathbb{R}^{2})}$. This shows, in particular, that $X_{r}^{m}$  is a Banach space.

Let $r$ be as in \eqref{r interval1}.
 Given $z\in W^{2,p_{0}%
}(A_{r}^{ m};\mathbb{R}^{2})$, let $w:=F_{r}^{ m}(z)$ be the solution to the boundary value
problem
\begin{equation}
\left\{
{\setstretch{1.1}
\begin{array}
[c]{ll}%
-\operatorname{div}(\mathbb{C}Ew)= f\circ\Phi+ f^{z}
 & \text{in }A_{r}^{ m}\,,\\
(\mathbb{C}Ew)\nu^{ 0}= g\circ\Phi + \hat{g}^{z} +\check{g}^{z} & \text{on
}\Gamma_{r}^{ m}\,,\\
w=0 & \text{on }\partial A_{r}^{m}\setminus \Gamma_{r}^{m}\,,
\end{array}
}
\right.  \label{eul24}%
\end{equation}
let $ w_{1}$ be the solution to the boundary value problem
\[
\left\{
{\setstretch{1.1}
\begin{array}
[c]{ll}%
-\operatorname{div}(\mathbb{C}E w_{1})= f\circ\Phi & \text{in }A_{r}^{ m}\,,\\
(\mathbb{C}E w_{1})\nu^{ 0}= g\circ\Phi & \text{on }\Gamma_{r}^{ m}\,,\\
 w_{1}=0 & \text{on }\partial A_{r}^{m}\setminus \Gamma_{r}^{m}\,,
\end{array}
}
\right.
\]
and let $ w_{2}:=G_{r}^{ m}(z)$ be the solution to the boundary value problem
\[
\left\{
{\setstretch{1.1}
\begin{array}
[c]{ll}%
-\operatorname{div}(\mathbb{C}E w_{2})= f^{z} & \text{in }A_{r}^{ m}\,,\\
(\mathbb{C}E w_{2})\nu^{ 0}= \hat{g}^{z}+\check{g}^{z} & \text{on }\Gamma_{r}^{ m}\,,\\
 w_{2}=0 & \text{on }\partial A_{r}^{m}\setminus \Gamma_{r}^{m}\,.
\end{array}
}
\right.
\]
By linearity, we have $F_{r}^{ m}(z)= w_{1}+ G_{r}^{ m}(z)$.

We claim that  for $r>0$ sufficiently small the linear map  $G_{r}^{ m}:X_{r}^{ m}\rightarrow X_{r}^{ m}$ is a contraction. By linearity it suffices to show that%
\begin{equation}
\Vert G_{r}^{ m}(z)\Vert_{X_{r}^{ m}}\leq1/2\quad\text{for all }z\in X_{r}^{ m}\text{ with
}\Vert z\Vert_{X_{r}^{ m}}\leq1\,. \label{L2}%
\end{equation}
Fix $z\in X_{r}^{ m}$ with $\Vert z\Vert_{X_{r}^{ m}}\leq1$. By  linearity, using  Theorem
\ref{theorem grisvard 1} we  obtain  that $ w_{2}=G_{r}^{ m}(z)\in W^{2,p_{0}%
}(A_{r}^{ m};\mathbb{R}^{2})$ with 
\begin{equation}
\Vert\nabla^{2} w_{2}\Vert_{L^{p_{0}}(A_{r}^{ m})}  \leq\kappa\big(\Vert f%
^{z}\Vert_{L^{p_{0}}(A_{r}^{ m})}+
 \Vert\nabla \hat{g}^{z}\Vert_{L^{p_{0}}(A_{r}^{m})} + \Vert\nabla \check{g}^{z}\Vert_{L^{p_{0}}(A_{r}^{m})}\big)\,, \label{L4}%
\end{equation}
 where $\kappa$ depends only on $\lambda$, $\mu$, $\eta_{0}$, and $L_{0}$; in particular it does not depend on $r$, $h$, $f$, $g$, and $z$. 
Therefore, \eqref{L2} follows provided we can prove that%
\begin{equation}
\kappa\big(\Vert f^{z}\Vert_{L^{p_{0}}(A_{r}^{ m})}+
 \Vert\nabla \hat{g}^{z}\Vert_{L^{p_{0}}(A_{r}^{m})} + \Vert\nabla \check{g}^{z}\Vert_{L^{p_{0}}(A_{r}^{m})}\big) \leq1/2\,. \label{L6}%
\end{equation}

\noindent\textbf{Step 4: Proof of \eqref{L6}.} In this step we prove \eqref{L6}.  Below  $ C_{h} $ denotes a constant 
independent of $r\in (0,\eta_{0}^{2}/(4M))$ and $m\in[2\eta_{0},L_{0}]$, but depending on $h$, whose value can change from formula to formula, while $C_{\lambda,\mu}$ denotes a constant depending only on the Lam\'e coefficients. 
 We begin with $\Vert f^{z}\Vert_{L^{p_{0}}%
(A_{r}^{ m})}$. For $0<r<\eta_{0}^{2}/(4M)$, by Lemma \ref{lemma sigma alpha}, Remark \ref{remark sigma close to 1 a},
\eqref{f1},  and \eqref{f2},
in $A_{r}^{ m}$ we have the pointwise estimate%
\begin{align}
&| f^{z}|\leq  C_{\lambda,\mu}\Big(\Vert\sigma-1\Vert_{L^{\infty}(I_{ r})}%
+\sup_{(x,y)\in A_{r}^{m}}|y\sigma^{\prime}(x)|+\sup_{(x,y)\in A_{r}^{m}}|y\sigma^{\prime}%
(x)|^{2}\Big)|\nabla^{2}z|\nonumber
\\
&\qquad  +C_{\lambda,\mu}\Big(\Vert\sigma^{\prime}\Vert_{L^{\infty}(I_{r})} + 
\sup_{(x,y)\in  A_{r}^{m}}|y\sigma^{\prime\prime}(x)|\Big)|\nabla z|\leq  C_{h}  r |\nabla^{2}z| + C_{h} |\nabla z|\,, 
\label{f v hat estimate}%
\end{align}
where  in the last inequality we used \eqref{sigma -1 alpha}, \eqref{sigma' y 2 alpha}, and \eqref{sigma'' y alpha}. In turn, using Poincar\'e's inequality given in Remark \ref{remark poincare corner} together with the inequality $\Vert z\Vert_{X_{r}^{m}}\leq1$, with norm defined by \eqref{norm Z}, we obtain 
\begin{equation}
\Vert f^{z}\Vert_{L^{p_{0}}(A_{r}^{ m})}  \leq C_{h}r\Vert\nabla^{2}z\Vert_{L^{p_{0}}(A_{r}^{ m})}+ C\Vert\nabla z\Vert_{L^{p_{0}}(A_{r}^{ m})}
\leq  C_{h}r \Vert\nabla^{2}z\Vert_{L^{p_{0}}(A_{r}^{m})}\leq C_{h}r \,.
\label{f hat estimate a}%
\end{equation}

We now estimate $\Vert\nabla \hat{g}^{z}\Vert_{L^{p_0}(A_{r}^{m})}$. We observe that by
\eqref{g1hat} and \eqref{g2hat} we can write each component of $\nabla \hat{g}^{z}$ as linear
combinations of products of the functions
$\omega_{i}$ and second order partial
derivatives of the component of $z$, as well as products of the derivatives
$\omega_{i}^{\prime}$ and first order partial
derivatives of the components of $z$. Therefore, we obtain
that
\begin{align}
\Vert \nabla \hat{g}^{z}\Vert_{L^{p_{0}}(A_{r}^{m})}
&\leq C_{\lambda,\mu}\Big(\sum_{i=1}^{6}
\Vert\omega_{i}\Vert_{L^{\infty}(I_{r})}%
\Vert\nabla^{2} z\Vert_{L^{p_{0}}(A_{r}^{m})} +
\sum_{i=1}^{6}\Vert\omega_{i}^{\prime}\Vert_{L^{\infty}(I_{r})}%
\Vert\nabla z\Vert_{L^{p_{0}}(A_{r}^{m})}\Big)
 \nonumber\\&\leq C_{h} r
\,,
\label{qwe64}
\end{align}
where in the last inequality we used the fact that $\Vert\nabla^{2} z\Vert_{L^{p_{0}}(A_{r}^{m})}=\Vert
z\Vert_{X_{r}^{m}}\leq1$ and the estimates \eqref{omega i estimate a} and \eqref{omega' i estimate a}, together with Poincar\'e's inequality for $\partial_{x}z$ and $\partial_{y}z$, which vanish on one of the sides of $A_{r}^{m}$ different from $\Gamma_{r}^{m}$ (see Remark \ref{remark poincare corner}). Similarly, we prove that
\begin{equation}
\Vert \nabla \check{g}^{z}\Vert_{L^{p_{0}}(A_{r}^{m})}
\leq C_{h} r
\,.
\label{qwe65}
\end{equation}
From \eqref{f hat estimate a}, \eqref{qwe64}, and \eqref{qwe65}
it follows that
\begin{equation}
\kappa\big(\Vert f^{z}\Vert_{L^{p_{0}}( A_{r}^{m})}+
 \Vert\nabla \hat{g}^{z}\Vert_{L^{p_{0}}( A_{r}^{m})}
 +\Vert\nabla \check{g}^{z}\Vert_{L^{p_{0}}( A_{r}^{m})}\big) 
  \leq C_{h} r\,. \label{L6*}%
\end{equation}
 Therefore \eqref{L6} is satisfied if $r\in (0,\eta_{0}^{2}/{ (4M))}$ is sufficiently small. This shows that $G_{r}^{m}$ is a contraction in the Banach space $X_{r}^{m}$. Consequently $F_{r}^{m}$ is a contraction.

\noindent\textbf{Step 5: Conclusion. }By the Banach fixed point theorem
applied to $ F_{r}^{m}$, there exists $ z_{0} \in X_{r}^{ m}$ such that $z_{0}= F_{r}^{m}(z_{0})$.  By \eqref{eul24} the function $z_{0}$ solves \eqref{eul 8}. By \eqref{Phi C2} the function $u_{0}:=z_{0}\circ\Psi$ belongs to $W^{2,p_{0}}(\Omega_{h}^{0,r};\mathbb{R}^{2})$ and, by direct computation, it solves \eqref{qwe50}. Since $W^{2,p_{0}}(\Omega_{h}^{0,r};\mathbb{R}^{2})\subset H^{1}(\Omega_{h}^{0,r};\mathbb{R}^{2})$ and problem \eqref{qwe50} has a unique weak solution in $H^{1}(\Omega_{h}^{0,r};\mathbb{R}^{2})$, we conclude that $u_{0}=\tilde{u}$, hence $\tilde{u}\in W^{2,p_{0}}(\Omega_{h}^{0,r};\mathbb{R}^{2})$. Recalling that $\tilde{u}=u-w^{0}$ in $\Omega_{h}^{0,r/2}$, we obtain that 
$u\in W^{2,p_{0}}(\Omega_{h}^{0,r/2};\mathbb{R}^{2})$,  which concludes the proof.
\end{proof}

\begin{theorem}
\label{theorem endpoints}Under the hypotheses of Theorem \ref{theorem C3}, we
have%
\begin{equation}
h^{\prime\prime}(\alpha)=h^{\prime\prime}(\beta)=0\,,
\label{h'' zero at endpoints}%
\end{equation}
and
\begin{align}
\sigma_{0}\frac{\alpha-\alpha^{0}}{\tau}  &  =\frac{\gamma}{J(\alpha)}%
-\gamma_{0}+\nu_{0}\frac{h^{\prime}(\alpha)}{J(\alpha)^{2}}\Big(\frac
{h^{\prime\prime}}{J^{3}}\Big)^{\prime}(\alpha)\,,\label{equation x l1}\\
\sigma_{0}\frac{\beta-\beta^{0}}{\tau}  &  =-\frac{\gamma}{J(\beta)}%
+\gamma_{0}-\nu_{0}\frac{h^{\prime}(\beta)}{J(\beta)^{2}}\Big(\frac
{h^{\prime\prime}}{J^{3}}\Big)^{\prime}(\beta)\,, \label{equation x r1}%
\end{align}
 where $J$ is defined in \eqref{J and J0}. 
\end{theorem}

 We begin with a preliminary lemma.

\begin{lemma}
\label{lemma extension}
Under the hypotheses of Theorem \ref{theorem C3}, let $1\leq p<\infty$ and let $v\in W^{2,p}(\Omega_{h};\mathbb{R}^2)$ be such that $v(x,0)=0$ for a.e.\ $x\in(\alpha,\beta)$. Then there exists
$\hat{v}\in W^{2,p}\mathbb{R}^2;\mathbb{R}^2)$
such that $\hat{v}=v$ in $\Omega_{h}$ and 
$\hat{v}(x,0)=0$ for a.e.\ $x\in \mathbb{R}$.
\end{lemma}

\begin{proof}
Since $\Omega_{h}$ is a
domain  with Lipschitz boundary, we can extend $v$ to a function $\hat{v}\in
W^{2,p_{0}}((\alpha,\beta)\times\mathbb{R};\mathbb{R}^{2})$ (see \cite[Theorem 13.17]%
{leoni-book2}). Let $\tilde{\alpha}:=\alpha-\frac{\beta-\alpha}{2}$ and $\tilde{\beta}:=\beta+\frac{\beta-\alpha}{2}$. For $x\in[\tilde{\alpha},\alpha]$ and $y\in\mathbb{R}$ we set
\[\hat{v}(x,y):=3\hat{v}(2\alpha-x,y)-2\hat{v}(3\alpha-2x,y)\,.\]
Since $\hat{v}$, $\partial_{x}\hat{v}$, and $\partial_{y}\hat{v}$ have the same trace on both sides of the line $x=\alpha$, the function $\hat{v}$ belongs to $W^{2,p_{0}}((\tilde{\alpha},\beta)\times\mathbb{R};\mathbb{R}^{2})$. Similarly, for $x\in[\beta,\tilde{\beta}]$ and $y\in\mathbb{R}$, we set
\[\hat{v}(x,y):=3\hat{v}(2\beta-x,y)-2\hat{v}(3\beta-2x,y)\]
and we obtain that $\hat{v}$ belongs to $W^{2,p_{0}}((\tilde{\alpha},\tilde{\beta})\times\mathbb{R};\mathbb{R}^{2})$. By construction, we have 
$\hat{v}(x,0)=0$ for a.e.\ $x\in(\tilde{\alpha},\tilde{\beta})$. Using a suitable cut-off function, we can modify $\hat{v}$ near the lines $x=\tilde{\alpha}$ and $x=\tilde{\beta}$ so that the modified function vanishes near these lines. The conclusion can be obtained by setting $\hat{v}=0$ outside the strip $(\tilde{\alpha},\tilde{\beta})\times\mathbb{R}$.
\end{proof}

\begin{proof}[ Proof of Theorem \ref{theorem endpoints}]
Since $u\in W^{2,p_{0}}(\Omega_{h};\mathbb{R}^{2})$ by Theorem
\ref{theorem regularity at corner}, by Lemma \ref{lemma extension}  we can extend $u$ to a function $\hat{u}\in
 W_{\operatorname*{loc}}^{2,p_{0}}(\mathbb{R}^{2};\mathbb{R}^{2})$  such that $u(x,0)=(e_{0}x,0)$ for a.e.\ $x\in\mathbb{R}$. We take $\varphi$, $\alpha_{\varepsilon}$, and $h_{\varepsilon
}$ as in \eqref{test function varphi}, \eqref{102}, and \eqref{h epsilon}, respectively, and we define $u_{\varepsilon}$ as
restriction of $\hat{u}$ to $\Omega_{h_{\varepsilon}}$. Due to our choice of the extension, we have $u_{\varepsilon}(x,0)=(e_{0}x,0)$ for a.e.\ $x\in(\alpha
_{\varepsilon},\beta)$.

\noindent\textbf{Step 1:} We claim that
\begin{equation}
\left.  \frac{d}{d\varepsilon}\mathcal{E}(\alpha_{\varepsilon},\beta
,h_{\varepsilon},u_{\varepsilon})\right\vert _{\varepsilon=0}=\int_{\alpha
}^{\beta}W(Eu(x,h(x)))\varphi(x)\,dx\,. \label{der E1}%
\end{equation}
By \eqref{e6}  there exists $\varepsilon_{1}>0$ such that 
\begin{align*}
\frac{\mathcal{E}(\alpha_{\varepsilon},\beta,h_{\varepsilon},u_{\varepsilon
})-\mathcal{E}(\alpha,\beta,h,u)}{\varepsilon}  &  =-\frac{1}{\varepsilon}%
\int_{\alpha_{\varepsilon}}^{\alpha+2\delta_{0}}\Bigl(\int_{h_{\varepsilon}%
(x)}^{h(x)}W(E\hat{u}(x,y))\,dy\Bigr)dx\nonumber\\
 &\quad-\frac{1}{\varepsilon}\int_{\alpha}^{\alpha_{\varepsilon}}%
\Bigl(\int_{0}^{h(x)}W(Eu(x,y))\,dy\Bigr)dx
=:I_{\varepsilon}+II_{\varepsilon
}\,
\end{align*}
for all $-\varepsilon_{1}<\varepsilon<0$. Since $\hat{u}\in W^{2,p_{0}%
}(\mathbb{R}^{2};\mathbb{R}^{2})$, with $p_{0}>\frac{4}{3}$, by the
Sobolev--Gagliardo--Nirenberg embedding theorem, we have that $\nabla\hat
{u}\in L^{p_{0}^{\ast}}(\mathbb{R}^{2};\mathbb{R}^{2\times2})$, where
$p_{0}^{\ast}>4$. Hence, by H\"{o}lder's inequality and \eqref{W convex},%
\begin{align*}
\int_{\alpha}^{\alpha_{\varepsilon}}\int_{0}^{h(x)}W(Eu(x,y))\,dydx  &  \leq
C_{W}\int_{\alpha}^{\alpha_{\varepsilon}}\int_{0}^{h(x)}|\nabla u(x,y)|^{2}%
dydx\\
&  \leq C_{W}\Big(\int_{\alpha}^{\alpha_{\varepsilon}}\int_{0}^{h(x)}|\nabla
u(x,y)|^{4}dydx\Big)^{1/2}\Big(\int_{\alpha}^{\alpha_{\varepsilon}%
}h(x)\,dx\Big)^{1/2}\\
&  \leq C\Big(\int_{\alpha}^{\alpha_{\varepsilon}}\int_{0}^{h(x)}|\nabla
u(x,y)|^{4}dydx\Big)^{1/2}|\varepsilon|
\end{align*}
where we used the fact that $|h(x)|\le L_0|\alpha_\varepsilon-\alpha|$ and
\eqref{alpha epsilon lips}. This shows that $II_{\varepsilon}\rightarrow0$ as
$\varepsilon\rightarrow0^{-}$.

Since the function $\zeta:=E\hat u$ belongs to $W^{1,p_{0}}_{\operatorname*{loc}}(\mathbb{R}^{2};\mathbb{R}^{2})$, we have that
$W\circ\zeta\in W_{\operatorname*{loc}}^{1,1}(\mathbb{R}^{2})$. Indeed, using
the fact that $W(\xi)=\mu|\xi|^{2}+(\lambda+\mu)(\operatorname*{tr}\xi)^{2}$,
for every bounded Lebesgue measurable set $E\subset\mathbb{R}^{2}$, by
H\"{o}lder's inequality, we have
\begin{equation}
\int_{E}|\nabla(W\circ\zeta)|\,dxdy\leq C\int_{E}|\zeta||\nabla\zeta
|\,dxdy\leq C\Vert\zeta\Vert_{L^{p_{0}^{\prime}}(E)}\Vert\nabla\zeta
\Vert_{L^{p_{0}}(E)}.\label{77777}%
\end{equation}
Since $p_{0}>\frac{4}{3}$, it follows that $p_{0}^{\prime}=\frac{p_{0}}%
{p_{0}-1}<p_{0}^{\ast}=\frac{2p_{0}}{2-p_{0}}$, and thus the right-hand side
of \eqref{77777} is finite. 

By considering the representative of $\zeta$ that is locally absolutely
continuous on a.e. line parallel to the axes, we have that
\[
W(\zeta(x,y))=W(\zeta(x,h(x)))-\int_{y}^{h(x)}\partial_{y}(W\circ
\zeta)(x,s)\,ds
\]
for a.e. $x\in(\alpha_{\varepsilon},\alpha+2\delta_{0})$. Hence,%
\begin{align*}
 I_{\varepsilon}=-\frac{1}{\varepsilon}\int_{\alpha_{\varepsilon}}^{\alpha+2\delta_{0}}%
\int_{h_{\varepsilon}(x)}^{h(x)}&W(\zeta(x,y))\,dydx  =-\frac{1}{\varepsilon
}\int_{\alpha_{\varepsilon}}^{\alpha+2\delta_{0}}\int_{h_{\varepsilon}%
(x)}^{h(x)}W(\zeta(x,h(x)))\,dydx\\
& +\frac{1}{\varepsilon}\int_{\alpha_{\varepsilon}}^{\alpha+2\delta_{0}%
}\int_{h_{\varepsilon}(x)}^{h(x)}\int_{y}^{h(x)}\partial_{y}(W\circ
\zeta)(x,s)\,dsdydx=:A_{\varepsilon}+B_{\varepsilon}.
\end{align*}
By \eqref{h epsilon} and \eqref{omega eps go to zero}, we have%
\begin{equation}
A_{\varepsilon}=\int_{\alpha_{\varepsilon}}^{\alpha+2\delta_{0}}%
W(\zeta(x,h(x)))\left(  \varphi(x)+\frac{\omega_{\varepsilon}}{\varepsilon
}\psi(x)\right)  \,dx\rightarrow\int_{\alpha}^{\alpha+2\delta_{0}}%
W(\zeta(x,h(x)))\varphi(x)\,dx\,.\label{cc1}
\end{equation}
On the other hand, by Fubini's theorem%
\[
B_{\varepsilon}=\frac{1}{\varepsilon}\int_{\alpha_{\varepsilon}}%
^{\alpha+2\delta_{0}}\int_{h_{\varepsilon}(x)}^{h(x)}(s-h_{\varepsilon
}(x))\partial_{y}(W\circ\zeta)(x,s)\,dsdx\,,
\]
and so,%
\begin{align}
&|B_{\varepsilon}|   \leq C\left\Vert \varphi+\frac{\omega_{\varepsilon}%
}{\varepsilon}\psi\right\Vert _{\infty}\int_{\alpha_{\varepsilon}}%
^{\alpha+2\delta_{0}}\int_{h_{\varepsilon}(x)}^{h(x)}|\partial_{y}(W\circ
\zeta)(x,s)|\,dsdx\label{cc2}\\
& \leq C\left\Vert \varphi+\frac{\omega_{\varepsilon}}{\varepsilon}%
\psi\right\Vert _{\infty}\Vert\zeta\Vert_{L^{p_{0}^{\prime}}((\alpha
,\beta)\times(0,L_{0}(\beta-\alpha)))}\left(  \int_{\alpha_{\varepsilon}%
}^{\alpha+2\delta_{0}}\int_{h_{\varepsilon}(x)}^{h(x)}|\partial_{y}%
\zeta(x,s)|^{p_{0}}dsdx\right)  ^{1/p_{0}}\rightarrow0\nonumber
\end{align}
as $\varepsilon\rightarrow0^{-}$, where in the last inequality we used
\eqref{77777}.

It follows from \eqref{cc1} and \eqref{cc2} that
$$I_\varepsilon\to \int_{\alpha
}^{\beta}W(Eu(x,h(x)))\varphi(x)\,dx\,,$$ 
which proves  \eqref{der E1} when $\varepsilon\to 0^-$.

On the other hand, if $0<\varepsilon<\varepsilon_{1}$, then by \eqref{102} and 
\eqref{derivative negative}  we have  $\alpha-\delta_{0}<\alpha_{\varepsilon}<\alpha$, and since
$\operatorname*{supp}\psi\subset(\alpha+
\delta_{0},\alpha+2\delta_{0})$
we write%
\begin{align*}
\frac{\mathcal{E}(\alpha_{\varepsilon},\beta,h_{\varepsilon},u_{\varepsilon
})-\mathcal{E}(\alpha,\beta,h,u)}{\varepsilon}  &  =\frac{1}{\varepsilon}%
\int_{\alpha}^{\alpha+2\delta_{0}}\Bigl(\int_{h(x)}^{h_{\varepsilon}(x)}%
W(E\hat{u}(x,y))\,dy\Bigr)dx\\
&  \quad+\frac{1}{\varepsilon}\int_{\alpha_{\varepsilon}}^{\alpha}%
\Bigl(\int_{0}^{ h(x)+\varepsilon\varphi(x)}W(E\hat{u}%
(x,y))\,dy\Bigr)dx\\&=:III_{\varepsilon}+IV_{\varepsilon}\,.
\end{align*}
The proof that $IV_{\varepsilon}\rightarrow0$ as $\varepsilon\rightarrow0^{+}$
can be done as in the proof of $II_{\varepsilon}$. The term $III_{\varepsilon
}$ can be treated similarly to $I_{\varepsilon}$.

\noindent\textbf{Step 2:} By  \eqref{der T}, \eqref{der S},
and \eqref{der E1}, we have that the derivative $\frac{d}{d\varepsilon
}\mathcal{F}^{0}(\alpha_{\varepsilon},\beta,h_{\varepsilon},u_{\varepsilon})$
exists at $\varepsilon=0$, and by minimality,
\begin{equation}
\left.  \frac{d}{d\varepsilon}\mathcal{F}^{0}(\alpha_{\varepsilon}%
,\beta,h_{\varepsilon},u_{\varepsilon})\right\vert _{\varepsilon=0}=0\,.
\label{derivative Fi zero1}%
\end{equation}
Hence,  using also \eqref{der T a},  we  get %
\begin{align*}
\gamma\int_{\alpha}^{\beta}  &  \frac{h^{\prime}\varphi^{\prime}}{J}%
dx+\gamma\frac{J(\alpha)}{h^{\prime}(\alpha)}-\gamma_{0}\frac{1}{h^{\prime
}(\alpha)}+\nu_{0}\int_{\alpha}^{\beta}\frac{h^{\prime\prime}\varphi
^{\prime\prime}}{J^{5}}dx-\frac{5}{2}\nu_{0}\int_{\alpha}^{\beta}\frac{h^{\prime
}(h^{\prime\prime})^{2}\varphi^{\prime}}{J^{7}}dx\\
&  \,+\frac{\nu_{0}}{2}\frac{(h^{\prime\prime}(\alpha))^{2}}{J(\alpha)^{5}}\frac
{1}{h^{\prime}(\alpha)}+\int_{\alpha}^{\beta}\overline{W}\varphi\,dx-\frac{1}{\tau
}\int_{\alpha}^{\beta}\overline{H}\varphi dx\,-\sigma_{0}\frac{\alpha-\alpha^{0}%
}{\tau}\frac{1}{h^{\prime}(\alpha)}=0\,,
\end{align*}
 where $\overline{W}$, $\overline{H}$, and $J$ are defined in \eqref{H bar} and \eqref{J and J0}. 
We integrate by parts once the integrals containing $\varphi^{\prime}$ and
twice the integral containing $\varphi^{\prime\prime}$ to obtain
\begin{align*}
-\gamma\int_{\alpha}^{\beta}  &  \Big(\frac{h^{\prime}}{J}\Big)^{\prime
}\varphi dx+\nu_{0}\int_{\alpha}^{\beta}\Big(\frac{h^{\prime\prime}}{J^{5}%
}\Big)^{\prime\prime}\varphi\,dx+\frac{5}{2}\nu_{0}\int_{\alpha}^{\beta}\Big(\frac
{h^{\prime}(h^{\prime\prime})^{2}}{J^{7}}\Big)^{\prime}\varphi dx\\
&  +\int_{\alpha}^{\beta}\overline{W}\varphi\,dx-\frac{1}{\tau}\int_{\alpha}%
^{\beta}\overline{H}\varphi dx-\gamma\frac{h^{\prime}(\alpha)}{J(\alpha)}%
+\gamma\frac{J(\alpha)}{h^{\prime}(\alpha)}-\gamma_{0}\frac{1}{h^{\prime
}(\alpha)}
 -\nu_{0}\frac{h^{\prime\prime}(\alpha)}{J(\alpha)^5}\varphi^{\prime}%
(\alpha)
  \\
&  +\nu_{0}\Big(\frac{h^{\prime\prime}}{J^{5}}\Big)^{\prime}(\alpha)+\frac{5}{2}\nu_{0}\frac{h^{\prime}(\alpha)(h^{\prime\prime}(\alpha))^{2}}%
{J(\alpha)^{7}}+\frac{\nu_{0}}{2}\frac{(h^{\prime\prime}(\alpha))^{2}}{J(\alpha)^{5}%
}\frac{1}{h^{\prime}(\alpha)}-\sigma_{0}\frac{\alpha-\alpha^{0}}{\tau}%
\frac{1}{h^{\prime}(\alpha)}\,,
\end{align*}
where we used the facts that $\operatorname*{supp}\varphi\subset(\alpha
-\delta_{0},\alpha+\delta_{0})$  and  $\varphi(\alpha)=1$.  By \eqref{equation fourth}, and the
fact that $ m  \int_{\alpha}^{\beta}\varphi\,dx=0$, we have
\begin{align*}
&  -\gamma\frac{h^{\prime}(\alpha)}{J(\alpha)}+\gamma\frac{J(\alpha
)}{h^{\prime}(\alpha)}-\gamma_{0}\frac{1}{h^{\prime}(\alpha)}
-\nu_{0}\frac{h^{\prime\prime}(\alpha)}{J(\alpha)^5}\varphi^{\prime}%
(\alpha)
+\nu_{0}\Big(\frac{h^{\prime\prime}}{J^{5}}\Big)^{\prime}(\alpha)\\
&  +\frac{5}{2}\nu_{0}\frac{h^{\prime}(\alpha)(h^{\prime\prime}(\alpha))^{2}}%
{J(\alpha)^{7}}+\frac{\nu_{0}}{2}\frac{(h^{\prime\prime}(\alpha))^{2}}{J(\alpha)^{5}%
}\frac{1}{h^{\prime}(\alpha)}-\sigma_{0}\frac{\alpha-\alpha^{0}}{\tau}%
\frac{1}{h^{\prime}(\alpha)}=0\,.
\end{align*}
Since $\varphi_{0}^{\prime}(0)$ can be chosen arbitrarily, by
\eqref{test function varphi} so can $\varphi^{\prime}(\alpha)$. Hence, by
dividing the previous equation by $\varphi^{\prime}(\alpha)>0$ and letting
$\varphi^{\prime}(\alpha)\rightarrow\infty$ we get%
\begin{equation}
h^{\prime\prime}(\alpha)=0. \label{second derivatives zero1}%
\end{equation}
Using \eqref{second derivatives zero1} and multiplying the previous equation
by $h^{\prime}(\alpha)$ we obtain
\begin{align*}
\sigma_{0}\frac{\alpha-\alpha^{0}}{\tau}  &  =\gamma\frac{(J(\alpha
))^{2}-(h^{\prime}(\alpha))^{2}}{J(\alpha)}-\gamma_{0}+\nu_{0} h^{\prime
}(\alpha)\Big(\frac{h^{\prime\prime}}{J^{5}}\Big)^{\prime}(\alpha)\\
&  =\frac{\gamma}{J(\alpha)}-\gamma_{0}+\nu_{0}\frac{h^{\prime}(\alpha
)}{J(\alpha)^{2}}\Big(\frac{h^{\prime\prime}}{J^{3}}\Big)^{\prime}(\alpha)\,.
\end{align*}
In a similar way we can prove that $h^{\prime\prime}(\beta)=0$ and that%
\[
\sigma_{0}\frac{\beta-\beta^{0}}{\tau}=-\frac{\gamma}{J(\beta)}+\gamma
_{0}-\nu_{0} h^{\prime}(\beta)\Big(\frac{h^{\prime\prime}}{J^{5}}%
\Big)^{\prime}(\beta)=-\frac{\gamma}{J(\beta)}+\gamma_{0}-\nu_{0}
\frac{h^{\prime}(\beta)}{J(\beta)^{2}}\Big(\frac{h^{\prime\prime}}{J^{3}%
}\Big)^{\prime}(\beta)\,.
\]
We omit the details. This concludes the proof.
\end{proof}

We now estimate the $L^{p_{0}}$ norm of $\nabla^{2}u$ near $\alpha$.  To simplify the exposition we assume that $\alpha=0$. Given $r>0$, we set $I_{r}:=(0,r)$, $\Omega_{h}^{0,r}:=\Omega_{h}\cap
(I_{r}\times\mathbb{R})$, and $\Gamma_{h}^{0,r}:=\Gamma_{h}\cap
(I_{r}\times\mathbb{R})=\{(x,h(x)):0< x< r\}$.

\begin{theorem}
\label{theorem apriori estimates}Under the assumptions of Theorem
\ref{theorem C3}, let $0<r<\delta_{0}$. Then there exist two constants
$0< c_{1}=c_{1}(\eta_{0},\eta_{1},M)<1$ and $c_{2}=c_{2}(\eta_{0},\eta_{1},M)>0$,
independent of $r$, such that if%
\begin{equation}
r^{2}(\Vert h^{\prime\prime\prime}\Vert_{L^{\infty}( I_{r})}+\Vert
h^{(\operatorname*{iv})}\Vert_{L^{1}( I_{r})})\leq c_{1}\,,
\label{r hypothesis}%
\end{equation}
then
\begin{align}
\Vert\nabla^{2}u\Vert_{L^{p_{0}}(\Omega_{h}^{0,r/2})}  &  \leq
c_{2}+\frac{c_{2}}{r}\Vert\nabla u\Vert_{L^{p_{0}}(\Omega_{h}^{0,r})}\,,\label{second derivative estimate alpha}\\
\Vert\nabla u\Vert_{L^{p_{0}/(2-p_{0})}(\Gamma_{h}^{0,r/2})}  &
\leq c_{2}+\frac{c_{2}}{r}\Vert\nabla u\Vert_{L^{p_{0}}(\Omega_{h}^{0,r})}\,. 
\label{trace estimate alpha}%
\end{align}
\end{theorem}

We begin with a preliminary lemma.

\begin{lemma}
\label{lemma extension1}Let $ 0<\eta_{0}\leq m \leq L_{0}$, $1\le p<\infty$,
 $\mathbb{R}_{+}^{2}:=\mathbb{R}\times(0,+\infty)$, 
\begin{equation}\label{try01}
A_{\infty}^{ m}:=\{(x,y)\in\mathbb{R}^{2}:\,x>0\,,\,0<y< mx \}
\end{equation}
and let $v\in W^{2,p}(A_{\infty}^{ m})$ be such that $v(x,0)=0$ for $x>0$. Then the
function%
\[
\hat{v}(x,y):=\left\{
{\setstretch{1.3}
\begin{array}
[c]{ll}%
3v(2 y/m -x,y)-2v(3 y/m -2x,y) & \text{if }%
(x,y)\in\mathbb{R}_{+}^{2}\setminus A_{\infty}^{ m}\,,\\
v(x,y) & \text{if }(x,y)\in A_{\infty}^{ m}\,,
\end{array}
}
\right.
\]
belongs to $W^{2,p}(\mathbb{R}_{+}^{2})$ with $\hat{v}(x,0)=0$ for all
$x\in\mathbb{R}$. Moreover,%
\begin{equation}
\Vert\nabla^{2}\hat{v}\Vert_{L^{ p}(\mathbb{R}_{+}^{2})}\leq C\Vert
\nabla^{2}v\Vert_{L^{ p}(A_{\infty}^{ m})}%
\label{qup3}
\end{equation}
for some constant $C>0$ depending on  $\eta_{0}$, $L_{0}$,  and $p$.
\end{lemma}

\begin{proof}  Since $\hat{v}$, $\partial_{x}\hat{v}$, and $\partial_{y}\hat{v}$ have the same trace on both sides of the line $y=mx$, the function $\hat{v}$ belongs to $W^{2,p_{0}}(\mathbb{R}_{+}^{2};\mathbb{R}^{2})$. 
By direct computation we check  that 
$\hat{v}(x,0)=0$ for  a.e.\ $x<0$.
 The estimate \eqref{qup3} can be obtained by computing the second derivatives of $\hat{v}$. 
\end{proof}

\medskip

\begin{proof}[ Proof of Theorem \ref{theorem apriori estimates}]
\textbf{Step 1:} We proceed as in the proof of Theorem
\ref{theorem regularity at corner}. We recall that we are taking $\alpha=0$.
Let $m=h^{\prime}(0)$, let $r$ be as in \eqref{r interval1}, let $A_{r}^{m}$
be the triangle  defined in \eqref{triangle}, and let  $v$ be defined as in
\eqref{function v}. In view of Theorem \ref{theorem regularity at corner}, we
have that $v\in W^{2,p_{0}}(A_{r}^{ m};\mathbb{R}^{2})$, and so, by \eqref{eul 8} and
Theorem \ref{theorem grisvard 1}, we have that
\begin{align*}
\Vert\nabla^{2}v\Vert_{L^{p_{0}}(A_{r}^{ m})}&\leq\kappa\big(\Vert f\circ\Phi\Vert_{L^{p_{0}}(A_{r}^{m})}+
\Vert f^{v}\Vert_{L^{p_{0}}(A_{r}^{m})}
+\Vert \nabla(g\circ\Phi)\Vert_{L^{p_{0}}(A_{r}^{m})}\\&\quad+\Vert \hat{g}^{v}\Vert_{L^{p_{0}}(A_{r}^{m})}+\Vert \check{g}^{v}\Vert_{L^{p_{0}}(A_{r}^{m})}\big)\,.
\end{align*}
 Since $v=0$ near the line $x=r$ by \eqref{u tilde a}, we can extend it to $A_{\infty}^{m}$ by setting $v=0$ in $A_{\infty}^{m}\setminus A_{r}^{m}$, and the extended function belongs to $W^{2,p_{0}}(A_{\infty}^{m};\mathbb{R}^{2})$ and satisfies the estimate
\begin{align}
\Vert\nabla^{2}v\Vert_{L^{p_{0}}(A_{\infty}^{m})}
&\leq\kappa\big(\Vert f\circ\Phi\Vert_{L^{p_{0}}(A_{r}^{m})}+
\Vert f^{v}\Vert_{L^{p_{0}}(A_{r}^{m})}
\nonumber
\\&\quad+\Vert \nabla(g\circ\Phi)\Vert_{L^{p_{0}}(A_{r}^{m})}+\Vert \hat{g}^{v}\Vert_{L^{p_{0}}(A_{r}^{m})}+\Vert \check{g}^{v}\Vert_{L^{p_{0}}(A_{r}^{m})}\big)\,.
\label{hessian v alpha}%
\end{align}

We underline that { $\kappa$} is
independent of $r\in (0,\delta_{0})$, $m\in[2\eta_{0},L_{0}]$, and $h$ satisfying \eqref{hypotheses h}--\eqref{hypotheses h 4}. 
Since $h^{\prime\prime}(\alpha)=0$ by Theorem \ref{theorem endpoints}, it
follows by the mean value theorem that%
\begin{equation}
\Vert h^{\prime\prime}\Vert_{L^{\infty}(I_{ r})}\leq r\Vert h^{\prime
\prime\prime}\Vert_{L^{\infty}(I_{ r})}\,. 
\label{h'' bound a}%
\end{equation}
Hence, also by \eqref{f v hat estimate}, \eqref{sigma -1 alpha},  \eqref{sigma' norm alpha}, 
\eqref{sigma' y alpha},  and  \eqref{sigma'' y alpha}  we have
\begin{align*}
&  \Vert f^{v}\Vert_{L^{p_{0}}(A_{r}^{ m})}\leq C  r^{2}(\Vert
h^{\prime\prime\prime}\Vert_{L^{\infty}(I_{ r})}+\Vert h^{(\operatorname*{iv}%
)}\Vert_{L^{1}(I_{ r})})\Vert\nabla^{2}v\Vert_{L^{p_{0}}(A_{r}^{ m})}\\
&  \quad+ C  r^{4}(\Vert h^{\prime\prime\prime}\Vert_{L^{\infty}%
(I_{ r})}^{2}+\Vert h^{(\operatorname*{iv})}\Vert_{L^{1}(I_{ r})}^{2}%
)\Vert\nabla^{2}v\Vert_{L^{p_{0}}(A_{r}^{ m})}
\\
& \quad +  C  r(\Vert h^{\prime
\prime\prime}\Vert_{L^{\infty}(I_{ r})}+\Vert h^{(\operatorname*{iv})}%
\Vert_{L^{1}(I_{ r})})\Vert\nabla v\Vert_{L^{p_{0}}(A_{r}^{ m})}%
\\
&  \quad+ C  r^{3}(\Vert h^{\prime\prime\prime}\Vert_{L^{\infty}%
(I_{ r})}^{2}+\Vert h^{(\operatorname*{iv})}\Vert_{L^{1}(I_{ r})}^{2}%
)\Vert\nabla v\Vert_{L^{p_{0}}(A_{r}^{ m})}\,.
\end{align*}
Since { $\partial_{x}v$ and $\partial_{y}v$} vanish on one of the sides of $A_{r}^{m}$ different from $\Gamma_{r}^{m}$,
by Poincar\'e's inequality (see Remark \ref{remark poincare corner}) we obtain 
\begin{align}
&  \Vert f^{v}\Vert_{L^{p_{0}}(A_{r}^{ m})}\leq C  r^{2}(\Vert
h^{\prime\prime\prime}\Vert_{L^{\infty}(I_{ r})}+\Vert h^{(\operatorname*{iv}%
)}\Vert_{L^{1}(I_{ r})})\Vert\nabla^{2}v\Vert_{L^{p_{0}}(A_{r}^{ m})}\nonumber\\
&  \quad+ C  r^{4}(\Vert h^{\prime\prime\prime}\Vert_{L^{\infty}%
(I_{ r})}^{2}+\Vert h^{(\operatorname*{iv})}\Vert_{L^{1}(I_{ r})}^{2}%
)\Vert\nabla^{2}v\Vert_{L^{p_{0}}(A_{r}^{ m})}\,.
\label{tyu1}
\end{align}

By \eqref{qwe64}, \eqref{omega i estimate a}, and \eqref{omega' i estimate a} we have
\begin{align}
\Vert\nabla\hat{g}^{v}\Vert_{L^{p_{0}}(A_{r}^{m})} & \leq   C r^{2}\big(\Vert h^{\prime\prime\prime}\Vert_{L^{\infty}%
(I_{r})}+\Vert h^{(\operatorname*{iv})}\Vert_{L^{1}(I_{r})}\big)\Vert\nabla
^{2}v\Vert_{L^{p_{0}}(A_{r}^{m})}\nonumber\\
&\quad +  C r(\Vert
h^{\prime\prime\prime}\Vert_{L^{\infty}(I_{r})}+\Vert h^{(\operatorname*{iv}%
)}\Vert_{L^{1}(I_{r})})\Vert\nabla v\Vert_{L^{p_{0}}(A_{r}^{m})}\nonumber\\
&  \quad+ C r^{3}(\Vert h^{\prime\prime\prime}\Vert_{L^{\infty}%
(I_{r})}^{2}+\Vert h^{(\operatorname*{iv})}\Vert_{L^{1}(I_{r})}^{2}%
)\Vert\nabla v\Vert_{L^{p_{0}}(A_{r}^{m})}\,.\nonumber
\end{align}
Using again Poincar\'e's inequality (see Remark \ref{remark poincare corner}) we obtain
\begin{align}
\Vert\nabla\hat{g}^{v}\Vert_{L^{p_{0}}(A_{r}^{m})} & \leq   C r^{2}\big(\Vert h^{\prime\prime\prime}\Vert_{L^{\infty}%
(I_{r})}+\Vert h^{(\operatorname*{iv})}\Vert_{L^{1}(I_{r})}\big)\Vert\nabla
^{2}v\Vert_{L^{p_{0}}(A_{r}^{m})}\nonumber\\
&  \quad+ C r^{4}(\Vert h^{\prime\prime\prime}\Vert_{L^{\infty}%
(I_{r})}^{2}+\Vert h^{(\operatorname*{iv})}\Vert_{L^{1}(I_{r})}^{2}%
)\Vert\nabla^{2} v\Vert_{L^{p_{0}}(A_{r}^{m})}
\label{tyu2}
\end{align}
In the same way we prove that
\begin{align}
\Vert\nabla\check{g}^{v}\Vert_{L^{p_{0}}(A_{r}^{m})} & \leq   C r^{2}\big(\Vert h^{\prime\prime\prime}\Vert_{L^{\infty}%
(I_{r})}+\Vert h^{(\operatorname*{iv})}\Vert_{L^{1}(I_{r})}\big)\Vert\nabla
^{2}v\Vert_{L^{p_{0}}(A_{r}^{m})}\nonumber\\
&  \quad+ C r^{4}(\Vert h^{\prime\prime\prime}\Vert_{L^{\infty}%
(I_{r})}^{2}+\Vert h^{(\operatorname*{iv})}\Vert_{L^{1}(I_{r})}^{2}%
)\Vert\nabla^{2} v\Vert_{L^{p_{0}}(A_{r}^{m})}
\label{tyu3}
\end{align}

 By 
\eqref{f tilde v}, it follows from Lemma \ref{lemma change of variables}  that %
\begin{align}
\Vert f\circ\Phi\Vert_{L^{p_{0}}(A_{r}^{ m})}  &  \leq 
 C \Vert f\Vert_{L^{p_{0}}(\Omega_{h}^{0, r})}
\nonumber\\
&\leq
  C  \Vert\nabla^{2}\varphi_{r}\Vert_{C^{0}(\mathbb{R}^{2})}\Vert
u-w^{0}\Vert_{L^{p_{0}}(\Omega_{h}^{0, r})}\nonumber
\\&\quad+  C  \Vert\nabla\varphi_{r}\Vert_{C^{0}(\mathbb{R}^{2})}%
\Vert\nabla u-\nabla w^{0}\Vert_{L^{p_{0}}(\Omega_{h}^{0, r})}%
\label{f u tilde estimate}\\
&  \leq\frac{  C  }{r}\Vert\nabla u-\nabla w^{0}\Vert_{L^{p_{0}}%
(\Omega_{h}^{0, r})}\,,\nonumber
\end{align}
where we used the Poincar\'{e} inequality (see Lemma
\ref{lemma poincare corner}) and the inequalities $\Vert\nabla\varphi_{r}%
\Vert_{C^{0}(\mathbb{R}^{2})}\leq C/r$, and $\Vert\nabla^{2}\varphi_{r}%
\Vert_{C^{0}(\mathbb{R}^{2})}\leq C/r^{2}$.

By \eqref{g tilde a}, it follows from  Lemmas \ref{lemma sigma alpha} and  \ref{lemma change of variables} that
\begin{align}
\Vert\nabla(g\circ\Phi)\Vert_{L^{p_{0}}(A_{r}^{m})}  &\leq  C \big(1+\sup_{(x,y)\in
A_{r}^{m}}y|\sigma^{\prime}(x)|\big) \Vert\nabla g\Vert_{L^{p_{0}}(\Omega_{h}^{0,r})}
\nonumber\\
&\leq
 C \big(1+\sup_{(x,y)\in
A_{r}^{m}}y|\sigma^{\prime}(x)|\big)
\bigl[ \Vert\nabla^{2}\varphi_{r}\Vert_{C^{0}(\mathbb{R}^{2})}\Vert
u-w^{0}\Vert_{L^{p_{0}}(\Omega_{h}^{0,r})}\nonumber\\&\quad+ \Vert\nabla\varphi_{r}\Vert_{C^{0}(\mathbb{R}^{2})}%
\Vert\nabla u-\nabla w^{0}\Vert_{L^{p_{0}}(\Omega_{h}^{0,r})}\nonumber\\
&\quad + \Vert\nabla\varphi_{r}\Vert_{C^{0}(\mathbb{R}^{2})}%
\Vert
u-w^{0}\Vert_{L^{p_{0}}(\Omega_{h}^{0,r})}
\Vert h^{\prime\prime}\Vert_{L^{\infty}(I_{r})}\bigr]
\label{rty5}\\
&\leq \frac{C}{r}\big[1+ r^2(\Vert
h^{\prime\prime\prime}\Vert_{L^{\infty}(I_{ r})}+\Vert h^{(\operatorname*{iv})}\Vert_{L^{1}(I_{ r})}))
+r^4(\Vert
h^{\prime\prime\prime}\Vert_{L^{\infty}(I_{ r})}^2\nonumber\\&\quad+\Vert h^{(\operatorname*{iv})}\Vert_{L^{1}(I_{ r})})^2\big]
\Vert\nabla u-\nabla w^{0}\Vert_{L^{p_{0}}%
(\Omega_{h}^{0, r})}\,,\nonumber
\end{align}
where we used the Poincar\'{e} inequality (see Lemma
\ref{lemma poincare corner}), the inequalities $\Vert\nabla\varphi_{r}%
\Vert_{C^{0}(\mathbb{R}^{2})}\leq C/r$ and $\Vert\nabla^{2}\varphi_{r}%
\Vert_{C^{0}(\mathbb{R}^{2})}\leq C/r^{2}$, and \eqref{h'' bound a}.

Combining inequalities \eqref{hessian v alpha},  and 
\eqref{tyu1}--\eqref{rty5}, we finally
obtain%
\begin{align}
\Vert&\nabla^{2}v\Vert_{L^{p_{0}}(A_{r}^{ m})}    \leq C  r^{2}(\Vert
h^{\prime\prime\prime}\Vert_{L^{\infty}(I_{ r})}+\Vert h^{(\operatorname*{iv}%
)}\Vert_{L^{1}(I_{ r})})\Vert\nabla^{2}v\Vert_{L^{p_{0}}(A_{r}^{ m})}\nonumber\\
& + C  r^{4}(\Vert h^{\prime\prime\prime}\Vert_{L^{\infty}(I_{ r}%
)}^{2}+\Vert h^{(\operatorname*{iv})}\Vert_{L^{1}(I_{ r})}^{2})\Vert\nabla
^{2}v\Vert_{L^{p_{0}}(A_{r}^{ m})}\label{rty10}\\
& +\frac{C}{r}\big[1+ r^2(\Vert
h^{\prime\prime\prime}\Vert_{L^{\infty}(I_{ r})}+\Vert h^{(\operatorname*{iv})}\Vert_{L^{1}(I_{ r})}))\nonumber\\&
+r^4(\Vert
h^{\prime\prime\prime}\Vert_{L^{\infty}(I_{ r})}^2+\Vert h^{(\operatorname*{iv})}\Vert_{L^{1}(I_{ r})})^2\big]
\Vert\nabla u-\nabla w^{0}\Vert_{L^{p_{0}}(\Omega
_{h}^{0,r})}\,.\nonumber
\end{align}

Let us fix $c_1=c_1(\eta_{0},\eta_{1},M)>0$ such that $c_1<1$ and $Cc_1<1/4$, where $C=C(\eta_{0},\eta_{1},M)$ is the constant in \eqref{rty10}.
Suppose that \eqref{r hypothesis} holds. Then $r^{4}\Vert
h^{\prime\prime\prime}\Vert_{L^{\infty}(I_{r})}^{2}\leq r^{2}\Vert
h^{\prime\prime\prime}\Vert_{L^{\infty}(I_{r})}$ and 
$r^{4}\Vert
h^{(\operatorname*{iv})}\Vert_{L^{1}(I_{r})}^{2}\leq r^{2}\Vert
h^{(\operatorname*{iv})}\Vert_{L^{1}(I_{r})}$, hence \eqref{rty10} and the inequality $|\Omega_{h}^{0,r}|\leq L_{0}r^{2}/2$ give

\[
\frac{1}{2}\Vert\nabla^{2}v\Vert_{L^{p_{0}}(A_{r}^{ m})}\leq
\frac{ C}{r}\Vert\nabla u-\nabla w^{0}\Vert_{L^{p_{0}}(\Omega
_{h}^{0, r})}\leq 
\frac{C}{r}\Vert
\nabla u\Vert_{L^{p_{0}}(\Omega_{h}^{ 0,r})}
+
\frac{C}{r}r^{2/p_{0}}
\,,
\]
 with new constants $C$ depending only on $\eta_{0}$, $\eta_{1}$, and $M$.
Since $p_{0}<2$ by \eqref{p0}, we obtain
\begin{equation}\label{try2}
\Vert\nabla^{2}v\Vert_{L^{p_{0}}(A_{r}^{ m})}\leq
\frac{C}{r}\Vert
\nabla u\Vert_{L^{p_{0}}(\Omega_{h}^{ 0,r})}+C
\end{equation}

In turn, by Remark \ref{remark change of variables},  \eqref{sigma' norm alpha}, \eqref{sigma' y 2 alpha}, \eqref{sigma'' y alpha},   \eqref{function v},  \eqref{r hypothesis}, \eqref{h'' bound a},  and  Remark \ref{remark poincare corner} applied to $\partial_xv$ and  $\partial_yv$,  the previous inequality  gives 
\[
\Vert\nabla^{2}\tilde{u}\Vert_{L^{p}(\Omega_{h}^{0,r})}   \leq C\Vert
\nabla^{2}v\Vert_{L^{p}(A_{r}^{ m})}  \leq \frac{C}{r}\Vert
\nabla u\Vert_{L^{p_{0}}(\Omega_{h}^{ 0,r})}+C\,.
\]
Since $\tilde{u}=u-w^{0}$ in $\Omega_{h}^{0, r/2}$ and $w^{0}$ is linear, 
inequality \eqref{second derivative estimate alpha}  follows.

\textbf{Step 2: }It follows from \eqref{u tilde a} and \eqref{function v} that
the function $v \in W^{2,p_{0}}(A_{r}^{m};\mathbb R^{2})$ is zero outside $A_{7r/8}^{ m}$. Hence, by extending $v$ to be
zero, we can assume that $v\in W^{2,p_{0}}(A_{\infty}^{ m};\mathbb{R}^{2})$, where  $A_{\infty}^{m}$ is defined in \eqref{try01} with $m=h^{\prime}(0)$. By Lemma \ref{lemma extension1} we have that the
function%
\[
\hat{v}(x,y):=\left\{
{\setstretch{1.3}
\begin{array}
[c]{ll}%
3v(2 y/m -x,y)-2v(3 y/m -2x,y) & \text{if }%
(x,y)\in\mathbb{R}_{+}^{2}\setminus A_{\infty}^{ m}\,,\\
v(x,y) & \text{if }(x,y)\in A_{\infty}^{ m}\,,
\end{array}
}
\right.
\]
belongs to $W^{2,p_0}(\mathbb{R}_{+}^{2})$ with $\hat{v}(x,0)=0$ for all
$x\in\mathbb{R}$. Moreover,%
\[
\Vert\nabla^{2}\hat{v}\Vert_{L^{p_{0}}(\mathbb{R}_{+}^{2})}\leq C\Vert
\nabla^{2}v\Vert_{L^{p_{0}}(A_{\infty}^{ m})}\,.%
\]
Next we extend
$\hat{v}$ across the $x$-axis by setting%
\[
\bar{v}(x,y):=\left\{
{\setstretch{1.3}
\begin{array}
[c]{ll}%
3\hat{v}(x,-y)-2\hat{v}(x,-2y) & \text{if }(x,y)\in\mathbb{R}^{2}%
\setminus\mathbb{R}_{+}^{2},\\
\hat{v}(x,y) & \text{if }(x,y)\in\mathbb{R}_{+}^{2}\,.
\end{array}
}
\right.
\]
Then $\bar{v}\in W^{2,p_0}(\mathbb{R}^{2})$ with%
\[
\Vert\nabla^{2}\bar{v}\Vert_{L^{p_{0}}(\mathbb{R}^{2})}\leq C\Vert\nabla
^{2}\hat{v}\Vert_{L^{p_{0}}(\mathbb{R}_{+}^{2})}\leq C\Vert\nabla^{2}%
v\Vert_{L^{p_{0}}(A_{r}^{ m})}.
\]
It follows from \cite[Theorem 18.24]{leoni-book2}  that the trace of
$\nabla\bar{v}$ along the  half- line $ \Gamma_{\infty}^{m}:=\{(x,mx): x\ge 0\}$ belongs to $L^{p_{0}%
/(2-p_{0})}(\Gamma_{\infty}^{ m};\mathbb{R}^{2\times2})$ with%
\begin{equation}
\Vert\nabla\bar{v}\Vert_{L^{p_{0}/(2-p_{0})}(\Gamma_{\infty}^{ m})}\leq
C\Vert\nabla^{2}\bar{v}\Vert_{L^{p_{0}}(\mathbb{R}^{2})}\leq C\Vert\nabla
^{2}v\Vert_{L^{p_{0}}(A_{r}^{ m})}\,. \label{trace inequality v bar}%
\end{equation}
Since $\varphi_{r}=1$ on $A_{r/2}^{ m}$, it follows from \eqref{u tilde a} and
\eqref{function v} that $\bar{v}=(u-w^{0})\circ\Phi$ in $A_{r{/2}}^{ m}$. Hence,
$u-w^{0}= \bar{v}\circ\Psi$  in  $\Omega_{h}^{0, r/2}$. By \eqref{Phi 1} and the chain
rule,
\[
|\nabla(u-w^{0})(x,y)|\leq C(1+|\sigma^{\prime}(x)y|)|\nabla\bar{v}%
(x,\sigma(x)y)|\,.
\]
Taking $y=h(x)$ and using \eqref{function sigma 1} gives
\[
|\nabla(u-w^{0})(x,h(x))|\leq C\sup_{(x,y)\in\Omega_{h}^{0, r/2}}(1+|\sigma^{\prime
}(x)y|)|\nabla\bar{v}(x, m x)|
\]
 for a.e.\ $0<x<r{/2}$. 
Hence, by \eqref{trace inequality v bar},%
\begin{align*}
\Vert\nabla(u-w^{0})\Vert_{L^{p_{0}/(2-p_{0})}(\Gamma_{h}^{0, r/2})}  &  \leq
C\sup_{(x,y)\in\Omega_{h}^{0, r/2}}(1+|\sigma^{\prime}(x)y|)\Vert\nabla\bar{v}%
\Vert_{L^{p_{0}/(2-p_{0})}(\Gamma_{\infty}^{ m})}\\
&  \leq C\sup_{(x,y)\in\Omega_{h}^{0,r}}(1+|\sigma^{\prime}(x)y|)\Vert\nabla^{2}%
v\Vert_{L^{p_{0}}(A_{r}^{ m})}\,.
\end{align*}
Using \eqref{sigma' y 2 alpha}, \eqref{h'' zero at endpoints}, \eqref{r hypothesis}, and 
\eqref{h'' bound a},%
\[
\Vert\nabla(u-w^{0})\Vert_{L^{p_{0}/(2-p_{0})}(\Gamma_{h}^{0, r/2})}\leq
C\Vert\nabla^{2}v\Vert_{L^{p_{0}}(A_{r}^{ m})}\leq C+\frac{C}{r}\Vert\nabla
u\Vert_{L^{p_{0}}(\Omega_{h}^{0, r})}\,,
\]
where the last inequality follows from \eqref{try2}. Since { by (\ref{function w0 a}),}
\[
\Vert\nabla w^{0}\Vert_{L^{p_{0}/(2-p_{0})}(\Gamma_{h}^{0, r/2})}\leq |e_{0}|{ ((1+L_{0}^2)^{1/2}r/8)^{(2-p_{0})/p_{0}}}\leq C\,,
\]
inequality \eqref{trace estimate alpha} follows.
\end{proof}

\begin{corollary}
\label{corollary Btau alpha}Under the hypotheses of Theorem \ref{theorem C3},
let $1\leq p\leq p_{0}/(2-p_{0})$ and let
\begin{equation}
r_{1}:=\min\Big\{  \frac{c_{1}^{1/2}}{2c_{0}}\frac{1}{B_{\tau
}(H,{H^{0}},\alpha,{\alpha^{0}},\beta,{\beta^{0}})},\frac{\delta_{0}}{2},\frac{\eta_{1}}{2}\Big\}  \,,
\label{r one over B tau}%
\end{equation}
where $B_{\tau}(H,{H^{0}},\alpha,{\alpha^{0}},\beta,{\beta^{0}})$ is defined in
\eqref{B tau alpha beta}, $c_{0}>1$ is the constant in the statement of Theorem
\ref{theorem C3}, and $0< c_{1}<1$ is the constant in \eqref{r hypothesis}. Then
there exist two constants $c_{3}=c_{3}(\eta_{0},\eta_{1},M)>0$ and $c_{4}%
=c_{4}(\eta_{0},\eta_{1},M,p)>0$ such that
\begin{align}
\Vert\nabla^{2}u\Vert_{L^{p_{0}}(\Omega_{h}^{\alpha,\alpha+r_{1}})}  &  \leq
c_{3}B_{\tau}(H,{H^{0}},\alpha,{\alpha^{0}},\beta,{\beta^{0}})^{2-2/p_{0}%
}\,,\label{second derivative estimate alpha tau}\\
\Vert\nabla u\Vert_{L^{p}(\Gamma_{h}^{\alpha,\alpha+r_{1}})}  &  \leq
c_{4}B_{\tau}(H,{H^{0}},\alpha,{\alpha^{0}},\beta,{\beta^{0}})^{1-1/p}\,.
\label{trace estimate alpha tau}%
\end{align}
\end{corollary}

\begin{proof}
By H\"{o}lder's inequality with exponent $q=2/p_{0}$, and Lemma
\ref{lemma korn}, for every $0<r\leq \beta-\alpha$,%
\begin{equation}
\Vert\nabla u\Vert_{L^{p_{0}}(\Omega_{h}^{\alpha,\alpha+r})}\leq
C(r^{2})^{1/(q^{\prime}p_{0})}\Vert\nabla u\Vert_{L^{2}(\Omega_{h}%
^{\alpha,\beta})}\leq Cr^{2/p_{0}-1}\,, \label{holder p0 2}%
\end{equation}
where we used the fact that $q^{\prime}=2/(2-p_{0})$. 

By \eqref{h''' bound better}, \eqref{h iv bound better}, and
\eqref{r one over B tau}, recalling that $2c_{0}>1$, we have 
\[
r_{1}^{2}(\Vert h^{\prime\prime\prime}\Vert_{L^{\infty}((\alpha,\beta))}+\Vert
h^{(\operatorname*{iv})}\Vert_{L^{1}((\alpha,\beta))})\leq2c_{0}r_{1}%
^{2}B_{\tau}(H,{H^{0}},\alpha,{\alpha^{0}},\beta,{\beta^{0}})^{2}\leq c_{1}\,.
\]
Hence, we can apply Theorem \ref{theorem apriori estimates} to obtain
\eqref{second derivative estimate alpha} and \eqref{trace estimate alpha}. In
turn, by \eqref{holder p0 2},%
\begin{align*}
\Vert\nabla^{2}u\Vert_{L^{p_{0}}(\Omega_{h}^{\alpha,\alpha+r_{1}})}  &  \leq
c_{2}+\frac{c_{2}}{r_{1}}\Vert\nabla u\Vert_{L^{p_{0}}(\Omega_{h}%
^{\alpha,\alpha+2r_{1}})}\\
&  \leq c_{2}+c_{2} C\frac{1}{r_{1}^{2-2/p_{0}}}\leq c_{3}B_{\tau}%
(H,{H^{0}},\alpha,{\alpha^{0}},\beta,{\beta^{0}})^{2-2/p_{0}},
\end{align*}
 where in the last inequality we used \eqref{r one over B tau} and the inequality $1\leq B_{\tau}%
(H,{H^{0}},\alpha,{\alpha^{0}},\beta,{\beta^{0}})$. 

If $p<p_{0}/(2-p_{0})$, we use H\"{o}lder's inequality with exponent
$p_{1}=p_{0}/[p(2-p_{0})]$, \eqref{trace estimate alpha}, \eqref{holder p0 2},
and the fact that $r_{1}\leq1$ to get%
\begin{align*}
\Vert\nabla u\Vert_{L^{p}(\Gamma_{h}^{\alpha,\alpha+r_{1}})}  &  \leq
Cr_{1}^{1/(pp_{1}^{\prime})}\Vert\nabla u\Vert_{L^{p_{0}/(2-p_{0})}(\Gamma
_{h}^{\alpha,\alpha+r_{1}})}\\&=Cr_{1}^{1+1/p-2/p_{0}}\Vert\nabla u\Vert
_{L^{p_{0}/(2-p_{0})}(\Gamma_{h}^{\alpha,\alpha+r_{1}})}\\
&  \leq C+Cr_{1}^{1/p-2/p_{0}}\Vert\nabla u\Vert_{L^{p_{0}}(\Omega_{h}%
^{\alpha,\alpha+2r_{1}})}\leq C+Cr_{1}^{1/p-1}\\
&  \leq CB_{\tau}(H,{H^{0}},\alpha,{\alpha^{0}},\beta,{\beta^{0}})^{1-1/p},
\end{align*}
 where in the last inequality we used \eqref{r one over B tau} and the inequality $1\leq B_{\tau}%
(H,{H^{0}},\alpha,{\alpha^{0}},\beta,{\beta^{0}})$. 
The same inequality holds if $ p=p_{0}/(2-p_{0})$.
\end{proof}

\section{Global Regularity}

In this section we obtain global estimates for $\nabla^{2}u$ in the entire
domain $\Omega_{h}$.
 Given $r>0$ and $x_{0}\in\mathbb{R}$, we set $I_{r}(x_{0}):=(x_{0}-r,x_{0}+r)$.  Under the assumptions of  Theorem
\ref{theorem C3}, let us fix $r$ and $x_{0}$ such that  
\begin{equation}
0<2 r\leq \min\{\delta_{0},\eta_{1}\}
\quad\text{and}\quad
\alpha+8 r/\eta_{0}\leq x_{0}\leq\beta-8 r/\eta_{0}\,.\label{x0 interval}%
\end{equation}
If $x\in I_{4 r}(x_{0})$,  the inequality $\eta_{0}<1$ implies that 
\begin{equation}
\alpha+4r/\eta_{0}<\alpha+8 r/\eta_{0}-4 r< x<\beta-8 r/\eta_{0}+ 4 r< \beta-4r/\eta_{0}\,.\label{x0 interval*}%
\end{equation}
 Using  the fact that  $2r\leq \delta_{0}<\eta_{0}$, by \eqref{h bounded}  we have $h(x)\geq\eta_{0}(x-\alpha
) { \ge} 4 r$ for $\alpha+4 r/\eta_{0}\leq x\leq\alpha+\delta_{0}$;
in the same way we prove that 
$h(x){ \ge}4r$ for $\beta-\delta_{0}\leq x\leq\beta-4 r/\eta_{0}$. Using \eqref{h bounded from zero} we have also 
$h(x)\geq2\eta_{1}{ \ge}4 r$ for
$\alpha+\delta_{0}\leq x\leq\beta-\delta_{0}$.
Hence, by \eqref{x0 interval*}, we
have
\begin{equation}
h(x)\geq 4 r\label{h geq  8Lr}%
\end{equation}
for every $x\in I_{4 r}(x_{0})$ and $r$ and $x_{0}$ as in \eqref{x0 interval}. 

We recall that
\begin{align*}
\Omega_{-h}&:=\{(x,y)\in\mathbb{R}^{2}:\,\alpha<x<\beta\,,\,-h(x)<y<0\}\,,\\ 
\Omega_{h,r}^{ a,b}  & :=\{(x,y)\in
\mathbb{R}^{2}:\, a<x< b\,,\,h(x)- r
<y<h(x)\}\,,\\
\Gamma_{h}^{a,b}  & :=\{(x,h(x)):\,a<x<b\}\,.
\end{align*}

\begin{theorem}
\label{theorem regularity away}Under the assumptions of Theorem
\ref{theorem C3}, let $r$ and $x_{0}$ be as in
\eqref{x0 interval}. Then for every $2\leq q<\infty$ there exist constants $c_{5}=c_{5}(\eta_{0},\eta_{1},M)>0$, and
$c_{6}=c_{6}(\eta_{0},\eta_{1},M,q)>0$, independent of  $h$,  $r$ and $x_{0}$, such that
if%
\begin{equation}
\Vert h^{\prime\prime}\Vert_{ L^{\infty}(I_{4r}(x_{0}))}\leq\frac{1%
}{r}\,, \label{smallness of r 2}%
\end{equation}
then $u\in H^{2}(\Omega_{h,2r}^{ x_{0}-2r, x_{0}+2r})$ and we have%
\begin{align}
\Vert\nabla^{2}u\Vert_{L^{2}(\Omega_{h,2r}^{x_{0}-2 r,x_{0}+2r})}  &  \leq
\frac{c_{5}}{r}\Vert\nabla u\Vert_{L^{2}(\Omega_{h,4r}^{x_{0}-4 r,x_{0}+4r}%
)}\,,\label{second derivative estimate}\\
\Vert\nabla u\Vert_{L^{q}(\Gamma_{h}^{x_{0}-2 r,x_{0}+2 r})}  &  \leq\frac{c_{6}%
}{r^{1-1/q}}\Vert\nabla u\Vert_{L^{2}(\Omega_{h,4r}^{x_{0}-4 r,x_{0}+4r})}\,.
\label{trace estimate away}%
\end{align}
\end{theorem}

\begin{proof}
Without loss of generality we assume that $x_{0}=0$. 
 For every $\rho>0$  define $ R_{\rho} %
:=(- \rho , \rho )\times(- \rho ,0)$  and $J_{ \rho }:=(- \rho , \rho )\times\{0\}$.

\textbf{Step 1:} Let 
\begin{equation}
 v(x,y):=u(x,y+h(x))\,. 
\label{function v 1}%
\end{equation}
In view of \eqref{h geq  8Lr}  we have $R_{4 r}%
\subset\Omega_{-h}$. Hence, \eqref{equation A} gives
\[
\left\{
{\setstretch{1.3}
\begin{array}
[c]{ll}%
\operatorname{div}(\mathbb{A}\nabla v)=0 & \text{in } R_{4r} ,\\
(\mathbb{A}\nabla v)e_{2}=0 & \text{on }J_{4 r}\,.
\end{array}
}
\right.
\]
Define $v_{r}(x,y):=v(rx,ry)$, $(x,y)\in { R_{ 4}}$. Then $v_{r}$ satisfies the
boundary value problem
\[
\left\{
{\setstretch{1.3}
\begin{array}
[c]{ll}%
\operatorname{div}(\mathbb{A}_{r}\nabla v_{r})=0 & \text{in } R_{4} ,\\
(\mathbb{A}_{r}\nabla v_{r})e_{2}=0 & \text{on }J_{4}\,,
\end{array}
}
\right.
\]
where $\mathbb{A}_{r}(x):=\mathbb{A}(rx)$. Using using \eqref{coeff A} and
\eqref{smallness of r 2} and the fact that $\operatorname*{Lip}h\leq L_{0}$,
we have%
\[
\Vert\mathbb{A}_{r}\Vert_{C^{1}({ \overline{R}_{2}})}\leq  C\,.%
\]
Standard elliptic
regularity { (\cite[Proof of Theorem 20.4]{wloka-book})} gives%
\[
\Vert\nabla^{2}v_{r}\Vert_{L^{2}( R_{2})}\leq  C\Vert\nabla v_{r}%
\Vert_{L^{2}( R_{4} )}\,.%
\]
In turn, we obtain%
\begin{equation}
\Vert\nabla^{2}v\Vert_{L^{2}( R_{2r} )}\leq\frac{ C}{r}\Vert\nabla
v\Vert_{L^{2}( R_{4r} )}\,. \label{50000}%
\end{equation}
Since $u(x,y)=v(x,y-h(x))$  and 
\begin{align}
\partial_{x}u  &  =\partial_{x}v-\partial_{y}vh^{\prime}\,,\quad\partial
_{y}u=\partial_{y}v\label{50001}\\
\partial_{xx}^{2}u  &  =\partial_{x}^{2}v-2\partial_{xy}^{2}vh^{\prime
}+\partial_{yy}^{2}v(h^{\prime})^{2}-\partial_{y}vh^{\prime\prime}%
\,,\quad\partial_{xy}^{2}u=\partial_{xy}v-\partial_{yy}^{2}vh^{\prime
}\,,\nonumber\\
\partial_{yy}^{2}u  &  =\partial_{yy}^{2}v\,,\nonumber
\end{align}
we have%
\begin{align*}
\Vert\nabla^{2}u\Vert_{L^{2}(\Omega_{h,2r}^{-2r,2r})}  &  \leq  C\big(\Vert
\nabla^{2}v\Vert_{L^{2}( R_{2r} )}+\Vert h^{\prime\prime}\Vert_{ L^{\infty}%
(I_{4r})}\Vert\nabla v\Vert_{L^{2}( R_{2r} )}\big)\\
&  \leq\frac{ C}{r}\Vert\nabla v\Vert_{L^{2}( R_{4r} )}\leq\frac{ C}%
{r}\Vert\nabla u\Vert_{L^{2}(\Omega_{h,4r}^{- 4r,4r})}\,,
\end{align*}
 which proves \eqref{second derivative estimate}. 

\textbf{Step 2: }In this step we prove \eqref{trace estimate away}. We begin
by proving%
\begin{align}
\Vert\nabla v\Vert_{L^{2}(J_{2 r})}  &  \leq\frac{C}{r^{1/2}}\Vert\nabla
v\Vert_{L^{2}( R_{4r} )}\,,\label{L2 trace}\\
|\nabla v|_{H^{1/2}(J_{2 r})}  &  \leq\frac{C}{r}\Vert\nabla v\Vert
_{L^{2}( R_{4r} )}\,. \label{H1/2 seminorm}%
\end{align}
Define $z=\nabla v$. Then by standard trace theory \cite[Theorem 2.5.3]{necas-book} and a rescaling argument,
we have
\begin{align*}
\Vert z\Vert_{L^{2}(J_{2 r})}^{2}  &  \leq\frac{C}{r}\Vert z\Vert_{L^{2}( R_{2r}%
)}^{2}+Cr\Vert\nabla z\Vert_{L^{2}( R_{2r} )}^{2}\,,\\
|z|_{H^{1/2}(J_{2 r})}^{2}  &  \leq C\Vert\nabla z\Vert_{L^{2}( R_{2r})}^{2}\,.
\end{align*}
Combining the previous  inequalities  with \eqref{50000} gives \eqref{L2 trace}
and \eqref{H1/2 seminorm}.

Next, fix $2<q<\infty$. We claim that%
\begin{equation}
\Vert\nabla v\Vert_{L^{q}(J_{2r})}\leq\frac{C}{r^{1-1/q}}\Vert\nabla
v\Vert_{L^{2}( R_{4r} )}\,. \label{Lq trace}%
\end{equation}
Let $s=\frac{q-2}{2q}<\frac{1}{2}$. Then $s$ is subcritical so, by
\cite[Corollary 2.3]{leoni-book-fractional},%
\[
\Vert\nabla v\Vert_{L^{q}(J_{2r})}=\Vert\nabla v\Vert_{L^{2_{s}^{\ast}}(J_{2r}%
)}\leq\frac{C}{r^{1-1/q}}\Vert\nabla v\Vert_{L^{1}(J_{2r})}+C|\nabla
v|_{H^{s}(J_{2r})}\,,
\]
where $2_{s}^{\ast}=\frac{2}{1-2s}=q$ is Sobolev critical exponent. On the
other hand, by \cite[Lemma 2.6]{leoni-book-fractional},%
\[
|\nabla v|_{H^{s}(J_{2r})}\leq Cr^{1/2-s}|\nabla v|_{H^{1/2}(J_{2r})}%
=Cr^{1/q}|\nabla v|_{H^{1/2}(J_{2r})}\,.
\]
Combining the last two inequalities and using H\"{o}lder's inequality and
\eqref{L2 trace} and \eqref{H1/2 seminorm}, we deduce%
\begin{align*}
\Vert\nabla v\Vert_{L^{q}(J_{2r})}  &  \leq\frac{C}{r^{1/2-1/q}}\Vert\nabla
v\Vert_{L^{2}(J_{2r})}+Cr^{1/q}|\nabla v|_{H^{1/2}(J_{2r})}\\
&  \leq\frac{C}{r^{1-1/q}}\Vert\nabla v\Vert_{L^{2}( R_{4r} )}\,.
\end{align*}
By changing variables using \eqref{50001}, we obtain
\eqref{trace estimate away}.
\end{proof}

\begin{corollary}
\label{corollary regularity away}Under the assumptions of Theorem
\ref{theorem C3},  let $r$ be such that 
\begin{equation}
0<2 r\leq \min\{\delta_{0},\eta_{1}\}\,, \label{r interval 5}%
\end{equation}
and let
\begin{equation}
\alpha_{r} :=\alpha+9 r/\eta_{0}\quad\text{and}\quad\beta_{r} %
:=\beta-9 r/\eta_{0}\,. \label{alpha tilde beta tilde}%
\end{equation}
Then there exists  a  constant  
$c_{7}=c_{7}(\eta_{0},\eta_{1},M)>0$ such that, if%
\begin{equation}
\Vert h^{\prime\prime}\Vert_{ L^{\infty}(I_{4 r}(x_{0}))}\leq\frac{1}{r}
\label{smallness of r 4}%
\end{equation}
for every $x_{0}\in[\alpha_{r},\beta_{r}]$, then $u\in H^{2}(\Omega
_{h}^{\alpha_{r} ,\beta_{r} })$ and%
\begin{equation}
\Vert\nabla^{2}u\Vert_{L^{2}(\Omega_{h}^{\alpha_{r} ,\beta_{r} })}%
\leq\frac{c_{7}}{r}\Vert\nabla u\Vert_{L^{2}(\Omega_{h})\,.}
\label{second derivative estimate 13}%
\end{equation}

\end{corollary}

\begin{proof}
Assume \eqref{smallness of r 4} and let
\begin{align*}
&  \Omega_{1}:=\{(x,y):\alpha_{r} <x<\beta_{r} ,\ 0<y<2 r\}\,,\\
&  \Omega_{2}:=\{(x,y):\alpha_{r} <x<\beta_{r} ,\  r<y<h(x)- r\}\,,\\
&  \Omega_{3}:=\{(x,y):\alpha_{r} <x<\beta_{r} ,\ h(x)- 2 r<y<h(x)\}\,.
\end{align*}
Since $\Omega_{h}^{\alpha_{r} ,\beta_{r} }=\Omega_{1}\cup\Omega_{2}%
\cup\Omega_{3}$, it is enough to prove that $u\in H^{2}(\Omega_{i})$ and that
\begin{equation}
\Vert\nabla^{2}u\Vert_{L^{2}(\Omega_{i})}\leq\frac{ C}{r}\Vert\nabla
u\Vert_{L^{2}(\Omega_{h})}\,. \label{second derivative estimate 3}%
\end{equation}

For every $(x_{0},y_{0})$ in $\mathbb{R}^{2}$ and every $\rho>0$ let $Q_{\rho
}(x_{0},y_{0})$ be the cube with center $(x_{0},y_{0})$ and sides with length
$2\rho$  parallel to the coordinate axes. We set $Q_{\rho}^{+}(x_{0}%
,y_{0}):=\{(x,y)\in Q_{\rho}(x_{0},y_{0}):y>0\}$. We take for granted that every solution of the
Lam\'{e} system in a cube of the form $Q_{2\rho}(x_{0},y_{0})$ belongs to
$H^{2}(Q_{\rho}(x_{0},y_{0}))$ and satisfies the inequality
\begin{equation}
\int_{Q_{\rho}(x_{0},y_{0})}|\nabla^{2}u|^{2}\,dxdy\leq\frac{C}{\rho^{2}}%
\int_{Q_{2\rho}(x_{0},y_{0})}|\nabla u|^{2}\,dxdy\,,
\label{second derivative estimate 4}%
\end{equation}
(see \cite[Theorem 20.1]{wloka-book}).
Moreover, every
solution in a rectangle of the form $Q_{2\rho}^{+}(x_{0},0)$ with homogeneous
Dirichlet boundary condition on $I_{2\rho}(x_{0})\times\{0\}$ belongs to
$H^{2}(Q_{\rho}^{+}(x_{0},0))$ and satisfies the inequality
\begin{equation}
\int_{Q_{\rho}^{+}(x_{0},0)}|\nabla^{2}u|^{2}\,dxdy\leq\frac{ C}{\rho^{2}%
}\int_{Q_{2\rho}^{+}(x_{0},0)}|\nabla u|^{2}\,dxdy\,.
\label{second derivative estimate 5}%
\end{equation}
In both cases the dependence of the estimate on $\rho$ follows from a standard
dimensional argument.

To prove the estimate for $\Omega_{1}$,  for every integer $i$ we define
$
x_{i}:=2ir$
 and we consider the set $Z_{1}$ of integers $i$ such that $\alpha_{r}<x_{i}<\beta_{r}$. Since $x_{i}$ 
satisfies \eqref{x0 interval}, by \eqref{h geq 8Lr}  we have 
$h(x)\geq4 r$ for every  $i\in Z_{1}$ and  $x\in I_{4 r}(x_{ i})$. 
It follows
that $Q_{4r}^{+}(x_{ i},0)\subset\Omega_{h}$. 
Therefore $u\in H^{2}(Q_{2r}%
^{+}(x_{i},0))$ and by \eqref{second derivative estimate 5},
\[
\int_{Q_{2r}^{+}(x_{i},0)}|\nabla^{2}u|^{2}
\,dxdy\leq\frac{ C}{4r^{2}}%
\int_{Q_{4r}^{+}(x_{i},0)}|\nabla u|^{2}\,dxdy\,.
\]
Using the inclusion $\Omega_{1}\subset\bigcup_{i\in Z_{1}}Q_{2r}^{+}%
(x_{ i},0)$, we obtain $u\in H^{2}(\Omega_{1})$ and
\[
\int_{\Omega_{1}}|\nabla^{2}u|^{2}\,dxdy\leq\frac{ C}{4r^{2}}\sum
_{i\in Z_{1}}\int_{Q_{4r}^{+}(x_{ i},0)}|\nabla u|^{2}\,dxdy\,.
\]
Since each rectangle $Q_{4r}^{+}(x_{ i},0)$ intersects at most $7$ rectangles
of the form $Q_{4r}^{+}(x_{\hat{\imath}},0)$, from the previous inequality
and from the inclusion $Q_{4r}(x_{2i},0)^{+}\subset\Omega_{h}$ we obtain
\begin{equation}
\int_{\Omega_{1}}|\nabla^{2}u|^{2}\,dxdy\leq\frac{ C}{4r^{2}}\int%
_{\Omega_{h}}|\nabla u|^{2}\,dxdy\,. \label{second derivative estimate 7}%
\end{equation}

To prove the estimate for $\Omega_{2}$,  we set $\rho:=r/(3+3L_{0})$ 
and we consider the set $Z_{2}$ of all pairs  of integers  $(i,j)$ such that $Q_{\rho}( i\rho %
, j\rho )\cap\Omega_{2}\neq\emptyset$.
 We claim that  $Q_{2\rho}( i\rho , j\rho )\subset\Omega_{h}$ for
every $(i,j)\in Z_{ 2}$.   Indeed, if $(i,j)\in Z_{ 2}$, then there exists $(x_{0},y_{0})\in Q_{\rho}(i\rho,j\rho)\cap\Omega_{2}$. Hence $\alpha_{r}<x_{0}<\beta_{r}$ and $r<y_{0}<h(x_{0})-r$. If $(x,y)\in Q_{2\rho}(i\rho,j\rho)$, we have $|x-x_{0}|<3\rho$ and $|y-y_{0}|<3\rho$. Recalling the Lipschitz estimate for $h$, this implies that $0<r-3\rho<y_{0}-3\rho< y < y_{0}+3\rho<h(x_{0})-r +3\rho\leq h(x)+L_{0}|x-x_{0}|-r+3\rho< h(x)+3L_{0}\rho-r+3\rho=h(x)$, which gives $(x,y)\in\Omega_{h}$ and concludes the proof of the inclusion $Q_{2\rho}(i\rho,j\rho)\subset\Omega_{h}$ for $(i,j)\in Z_{2}$. 
Therefore $u\in H^{2}(Q_{\rho}( i\rho , j\rho ))$ and by
\eqref{second derivative estimate 4},
\[
\int_{Q_{\rho}( i\rho , j\rho )}|\nabla^{2}u|^{2}\,dxdy\leq\frac{C}{r^{2}}%
\int_{Q_{2\rho}( i\rho , j\rho )}|\nabla u|^{2}\,dxdy\,.
\]
Using the inclusion $\Omega_{2}\subset\bigcup_{(i,j)\in Z_{2}}Q_{\rho}( i\rho , j\rho )$
we obtain $u\in H^{2}(\Omega_{2})$ and
\[
\int_{\Omega_{2}}|\nabla^{2}u|^{2}\,dxdy\leq\frac{C}{r^{2}}\sum_{(i,j)\in
Z_{2}}\int_{Q_{2\rho}( i\rho , j\rho )}|\nabla u|^{2}\,dxdy\,.
\]
Since each rectangle $Q_{2\rho}( i\rho , j\rho )$ intersects at most $49$ rectangles
of the form $Q_{2\rho}(\hat{\imath}\rho,\hat{\jmath}\rho)$, from the previous
inequality and from the inclusion $Q_{2\rho}( i\rho , j\rho )\subset\Omega_{h}$ we
obtain
\begin{equation}
\int_{\Omega_{2}}|\nabla^{2}u|^{2}\,dxdy\leq\frac{ C}{r^{2}}\int_{\Omega_{h}%
}|\nabla u|^{2}\,dxdy\,. \label{second derivative estimate 9}%
\end{equation}

To prove the estimate for $\Omega_{3}$, we consider 
 $x_{i}$ and $Z_{1}$ as for  $\Omega_{1}$.
 By Theorem \ref{theorem regularity away} we obtain that $u\in H^{2}(\Omega_{h,2r}^{ x_{i}%
-2r,x_{i}+2r})$ and that
\[
\int_{\Omega_{h,2r}^{ x_{i}-2r,x_{2i}+2r}}|\nabla^{2}u|^{2}\,dxdy\leq
\frac{c_{5}^{2}}{ r^{2}}\int_{\Omega_{h,4r}^{ x_{i}-4r,x_{i}+4r}}|\nabla
u|^{2}\,dxdy\,.
\]
Using the inclusion $\Omega_{3}\subset\bigcup_{i \in Z_{1}}%
\Omega_{h,2r}^{ x_{i}-2r,x_{i}+2r}$ we obtain $u\in H^{2}(\Omega_{3})$ and
\begin{equation}
\int_{\Omega_{3}}|\nabla^{2}u|^{2}\,dxdy\leq\frac{c_{5}^{2}}{ r^{2}}%
\sum_{i\in Z_{1}}\int_{\Omega_{h,4r}^{ x_{i}-4r,x_{i}+4r}%
}|\nabla u|^{2}\,dxdy\,. \label{7777}%
\end{equation}
 
Since each set $\Omega_{h,4r}^{x_{i}-4r,x_{i}+4r}$ intersects at most $7$ sets
of the form $\Omega_{h,4r}^{x_{\hat{\imath}}-4r,x_{\hat{\imath}}+4r}$, from the previous inequality
and from the inclusion $\Omega_{h,4r}^{x_{i}-4r,x_{i}+4r}\subset\Omega_{h}$, we obtain
\begin{equation}
\int_{\Omega_{3}}|\nabla^{2}u|^{2}\,dxdy\leq\frac{7c_{5}^{2}}{r^{2}}%
\int_{\Omega_{h}}|\nabla u|^{2}\,dxdy\,, \label{second derivative estimate 10}%
\end{equation}
which concludes the proof.
\end{proof}

\bigskip

\begin{theorem}
\label{theorem trace gradient}Under the hypotheses of Theorem \ref{theorem C3}%
, let $2\leq p\leq p_{0}/(2-p_{0})$. Then there exists a constant $c_{8}%
=c_{8}(\eta_{0},\eta_{1},M,p)>0$ such that%
\[
\Vert\nabla u\Vert_{L^{p}(\Gamma_{h})}\leq c_{8}B_{\tau}(H,{H^{0}},\alpha
,{\alpha^{0}},\beta,{\beta^{0}})^{1-1/p}.
\]
\end{theorem}

\begin{proof}
By \eqref{trace estimate alpha tau},%
\begin{equation}
\int_{\alpha}^{\alpha+r_{1}}|\nabla u(x,h(x))|^{p}\sqrt{1+(h^{\prime}(x))^{2}%
}dx\leq c_{4}^{p}B_{\tau}(H,{H^{0}},\alpha,{\alpha^{0}},\beta,{\beta^{0}})^{p-1}.
\label{re 45}%
\end{equation}
A similar estimate holds in $(\beta-r_{1},\beta)$. It remains to estimate
$\nabla u$ over $(\alpha+r_{1},\beta-r_{1})$. Let $ r:=\eta_{0} r_{1}/4$ 
 and for every  integer  $i$ we define
$x_{i}:= ir$.
Let  $Z$ be the set of integers $i$ such that $\alpha+r_{1}<x_i<\beta-r_{1}$. Since 
$(\alpha+r_{1},\beta-r_{1})\subset\bigcup_{i\in Z}%
I_{ r}(x_{i})$, we have 
\begin{equation}
\int_{\alpha+r_{1}}^{\beta-r_{1}}|\nabla u(x,h(x))|^{p}\sqrt{1+(h^{\prime
}(x))^{2}}dx\,\leq\sum_{i\in Z}\int_{I_{ r}(x_{i})}|\nabla
u(x,h(x))|^{p}\sqrt{1+(h^{\prime}(x))^{2}}dx\,. \label{re 46}%
\end{equation}
 Recalling that $r<r_{1}<\delta_{0}$, we see that \eqref{x0 interval} is satisfied with $r$ replaced by $r/2$. Moreover, 
from \eqref{h'' bound} and  the definitions of $r_{1}$ and $r$, we obtain 
\[
 r\Vert h^{\prime\prime}\Vert_{ L^{\infty}((\alpha,\beta))}\leq c_{0} r%
B_{\tau}(H,{H^{0}},\alpha,{\alpha^{0}},\beta,{\beta^{0}})\leq  c_{1}^{1/2}\eta_{0}/8<1\,.
\]
Hence we can apply Theorem \ref{theorem regularity away}, with $r$ replaced by $r/2$,  to obtain \eqref{trace estimate away} in each
interval $I_{ r }(x_{i})$.

By \eqref{trace estimate away},%
\[
\Vert\nabla u\Vert_{L^{p}(\Gamma_{h}^{x_{i}- r ,x_{i}+ r })}\leq
\frac{c_{6}}{ r ^{1-1/p}}\Vert\nabla u\Vert_{L^{2}(\Omega_{h,2 r }%
^{x_{i}-2 r ,x_{i}+2 r })}\,.
\]
Combining this inequality with \eqref{re 46} gives%
\begin{align*}
\int_{\alpha+r_{1}}^{\beta-r_{1}}|\nabla u(x,h(x))|^{p}\sqrt{1+(h^{\prime
}(x))^{2}}dx\,  &  \leq\frac{C}{ r ^{p-1}}\sum_{i\in Z}\Vert\nabla
u\Vert_{L^{2}(\Omega_{h,2 r }^{x_{i}-2 r ,x_{i}+2 r })}^{p}\\
&  \leq\frac{C}{ r ^{p-1}}\left(  \sum_{i\in Z}\Vert\nabla
u\Vert_{L^{2}(\Omega_{h,2 r }^{x_{i}-2 r ,x_{i}+2 r })}^{2}\right)
^{p/2}\,,
\end{align*}
where we used the fact that $p\geq2$.
Since each set $\Omega_{h,2 r }^{x_{i}-2 r ,x_{i}+2 r }$ intersects at most $7$ sets
of the form $\Omega_{h,2 r }^{x_{\hat{\imath}}-2 r ,x_{\hat{\imath}}+2r{2}}$, from the previous inequality
and from the inclusion $\Omega_{h,2 r }^{x_{i}-2 r ,x_{i}+2 r }\subset\Omega_{h}$  we obtain
\begin{equation}\label{tyu1*}
\int_{\alpha+r_{1}}^{\beta-r_{1}}|\nabla u(x,h(x))|^{p}\sqrt{1+(h^{\prime
}(x))^{2}}dx\leq\frac{C}{ r ^{p-1}}\Vert\nabla u\Vert_{L^{2}(\Omega_{h}%
)}^{p}\leq \frac{C}{r_{1}^{p-1}}\,,
\end{equation}
 where in the last inequality we used \eqref{estimate in H1} and the equality $ r :=\eta_{0} r_{1}/4$. By \eqref{r one over B tau} we have 
\begin{align*}
\frac{1}{r_{1}}&=\max\Big\{  \frac{2c_{0}}{c_{1}^{1/2}}B_{\tau
}(H,{H^{0}},\alpha,{\alpha^{0}},\beta,{\beta^{0}}),\frac{ 2}{\delta_{0} },\frac{2}{\eta_{1}}\Big\}\\&\leq\max\Big\{  \frac{2c_{0}}{c_{1}^{1/2}},\frac{ 2}{\delta_{0} },\frac{2}{\eta_{1}}\Big\}B_{\tau
}(H,{H^{0}},\alpha,{\alpha^{0}},\beta,{\beta^{0}})
\,,
\end{align*}
where in the last inequality we used the fact that $1\leq B_{\tau
}(H,{H^{0}},\alpha,{\alpha^{0}},\beta,{\beta^{0}})$.
Therefore, \eqref{tyu1*} yields 
\[
\int_{\alpha+r_{1}}^{\beta-r_{1}}|\nabla u(x,h(x))|^{p}\sqrt{1+(h^{\prime
}(x))^{2}}dx\,\leq CB_{\tau}(H,{H^{0}},\alpha,{\alpha^{0}},\beta,{\beta^{0}}%
)^{p-1}.
\]
Summing this inequality to \eqref{re 45} and to the corresponding inequality
in $(\beta-r_{1},\beta)$ gives%
\[
\int_{\alpha}^{\beta}|\nabla u(x,h(x))|^{p}\sqrt{1+(h^{\prime}(x))^{2}%
}dx\,\leq CB_{\tau}(H,{H^{0}},\alpha,{\alpha^{0}},\beta,{\beta^{0}})^{p-1}\,,
\]
 which concludes the proof. 
\end{proof}

\begin{theorem}
\label{theorem h iv Lp}Under the hypotheses of Theorem \ref{theorem C3}, let
$1<p\leq\frac{p_{0}}{4-2p_{0}}$. Then we have $h^{(\operatorname*{iv})}\in
L^{p_{0}/(4-2p_{0})}((\alpha,\beta))$ and there exists a constant $c_{9}%
=c_{9}(\eta_{0},\eta_{1},M,p)>0$ such that%
\begin{equation}
\Vert h^{(\operatorname*{iv})}\Vert_{L^{p}((\alpha,\beta))}^{p}\leq
c_{9}B_{\tau}(H,{H^{0}},\alpha,{\alpha^{0}},\beta,{\beta^{0}})^{q},
\label{h iv Lp estimate}%
\end{equation}
where $B_{\tau}(H,{H^{0}},\alpha,{\alpha^{0}},\beta,{\beta^{0}})$ is defined in
\eqref{B tau alpha beta} and
\[
q:=\max\{p,5p/2-1,3p-2,2p-1\}\,.
\]
\end{theorem}

\begin{proof}
By \eqref{equation fourth},
\[
-\gamma\Big(\frac{h^{\prime}}{J}\Big)^{\prime}+\nu_{0}\Big(\frac
{h^{\prime\prime}}{J^{5}}\Big)^{\prime\prime}+\frac{5}{2}\nu_{0}\Big(\frac{h^{\prime
}(h^{\prime\prime})^{2}}{J^{7}}\Big)^{\prime}+\overline{W}-\frac{1}{\tau}\bar
{H}= m  %
\]
in $(\alpha,\beta)$, where we recall that $\overline{W}(x):=W(Eu(x,h(x)))$ and
\[
\overline{H}(x):=\int_{-\infty}^{x}(H(s)-H^{0}(s))\sqrt{1+((\check{h}^{0})^{\prime
}(s))^{2}}\,ds\,.
\]
Hence,
\begin{equation}
|h^{(\operatorname*{iv})}|^{p}\leq  C |h^{\prime\prime}|^{p}+
C|h^{\prime\prime\prime}|^{p}%
|h^{\prime\prime}|^{p}+C|h^{\prime\prime}|^{3p}+C\overline{W}^{p}+\frac{C}{\tau
^{p}}|\overline{H}|^{p}+C| m  |^{p} \label{re 42a}%
\end{equation}
 for some constant $C>0$ depending on $L_{0}$.

 By \eqref{h'' bound} and \eqref{beta-alpha bound} we have
\begin{equation}\label{re 42aaa}
\int_{\alpha}^{\beta}|h^{\prime\prime}|^pdx\leq
(\beta-\alpha)
\Vert h^{\prime\prime}\Vert_{L^{\infty}(\alpha,\beta)}^{p}\leq CB_{\tau}(H,{H^{0}},\alpha,{\alpha^{0}},\beta,{\beta^{0}})^{p}.
\end{equation}

If $1<p<2$,  by H\"{o}lder's inequality with
exponents $\frac{2}{p}$ and $\frac{2}{2-p}$,%
\begin{align*}
\int_{\alpha}^{\beta}|h^{\prime\prime\prime}|^{p}|h^{\prime\prime}|^{p}dx  
&\leq\Big(  \int_{\alpha}^{\beta}\!| h^{\prime\prime}|^{2}dx\Big)
^{p/2}\Big(  \int_{\alpha}^{\beta}\!| h^{\prime\prime\prime}|^{2p/(2-p)}%
dx\Big)  ^{(2-p)/2} \!\!\!\!
 \\&\leq M^{ p/2}\Vert h^{\prime\prime\prime}\Vert_{L^{\infty}((\alpha,\beta
))}^{(3p-2)/2}\Vert h^{\prime\prime\prime}\Vert_{L^{1}((\alpha,\beta
))}^{(2-p)/2},
\end{align*}
where  in the last inequality  we used \eqref{h'' L2 norm} and  the fact that $\big(2p/(2-p)-1\big)(2-p)/2=(3p-2)/2$. Using \eqref{h''' bound better} and \eqref{h''' bound L1}, it
follows from the previous inequality that%
\begin{equation}
\int_{\alpha}^{\beta}|h^{\prime\prime\prime}|^{p}|h^{\prime\prime}|^{p}dx 
\leq CB_{\tau}(H,{H^{0}},\alpha,{\alpha^{0}},\beta,{\beta^{0}})^{5p/2-1},\label{re 43}
\end{equation}
 where we used the fact that $3p-2+(2-p)/2=5p/2-1$. 
If $p\geq 2$, by \eqref{h'' L2 norm},  \eqref{h'' bound}, and \eqref{h''' bound better} we have
\begin{align}
\int_{\alpha}^{\beta}|h^{\prime\prime\prime}|^{p}|h^{\prime\prime}|^{p}dx  &
\leq \Vert h^{\prime\prime\prime}\Vert_{L^{\infty}((\alpha,\beta))}^{p}\Vert h^{\prime\prime}
\Vert_{L^{\infty}((\alpha,\beta))}^{p-2}%
\int_{\alpha}^{\beta}|h^{\prime\prime}|^{2}dx
\nonumber
\\
&  { \le}C(B_{\tau}(H,{H^{0}},\alpha,{\alpha^{0}},\beta,{\beta^{0}}))^{3p-2}. \label{re 43*}
\end{align}

On the other hand, by  \eqref{h'' L2 norm} and  \eqref{h'' bound},
\begin{equation}
\int_{\alpha}^{\beta}|h^{\prime\prime}|^{3p}dx   \leq\Vert h^{\prime\prime
}\Vert_{L^{\infty}((\alpha,\beta))}^{3p-2}\int_{\alpha}^{\beta}| h^{\prime
\prime}|^{2}dx
  \leq C(B_{\tau}(H,{H^{0}},\alpha,{\alpha^{0}},\beta,{\beta^{0}}))^{3p-2}%
M\,.\label{re 44}
\end{equation}

By the previous theorem%
\begin{align*}
\int_{\alpha}^{\beta}\overline{W}^{p}dx  &  \leq\int_{\alpha}^{\beta}%
(W(Eu(x,h(x))))^{p}\sqrt{1+(h^{\prime}(x))^{2}}dx\\
&  \leq C\int_{\alpha}^{\beta}|\nabla u(x,h(x))|^{2p}\sqrt{1+(h^{\prime
}(x))^{2}}dx\leq CB_{\tau}(H,{H^{0}},\alpha,{\alpha^{0}},\beta,{\beta^{0}})^{2p-1}.
\end{align*}

Since $\frac{1}{\tau}|\overline{H}|\leq  C B_{\tau}(H,{H^{0}},\alpha,{\alpha^{0}}%
,\beta,{\beta^{0}})$  and $| m  |\leq CB_{\tau}(H,{H^{0}},\alpha,\alpha
_{0},\beta,{\beta^{0}})$ by \eqref{B tau alpha beta}   and \eqref{r4}, respectively,
combining these two inequalities with \eqref{re 42a}--
\eqref{re 44} yields \eqref{h iv Lp estimate}.
\end{proof}

\begin{remark}
\label{remark q}Note that $5p/2-1\leq2$ for $p\leq\frac{6}{5}$, $3p-2\leq2$
for $p\leq\frac{4}{3}$, and $2p-1\leq2$ for $p\leq3/2$. Hence, by taking
$1<p\leq\min\{6/5,p_{0}/(4-2p_{0})\}$ we have that $q\leq2$. Moreover, as in
the proof of \eqref{integral on interval}, we have that $H-H^{0}=0$ for
$x\notin(\min\{\alpha,\alpha^{0}\},\max\{\beta,\beta^{0}\})$. Hence, by
H\"{o}lder's inequality%
\begin{align*}
\left(  \int_{\mathbb{R}}|H-H^{0}|\,dx\right)  ^{2}  &  =\left(  \int%
_{\min\{\alpha,\alpha^{0}\}}^{\max\{\beta,\beta^{0}\}}|H-H^{0}|\,dx\right)
^{2}\\
&  \leq(\max\{\beta,\beta^{0}\}-\min\{\alpha,\alpha^{0}\})\int_{\min
\{\alpha,\alpha^{0}\}}^{\max\{\beta,\beta^{0}\}}|H-H^{0}|^{2}\,dx\,.
\end{align*}
In turn, by \eqref{B tau alpha beta} for $q\leq2$,%
\begin{align}
B_{\tau}(H,{H^{0}},\alpha,{\alpha^{0}},\beta,{\beta^{0}})^{q}  &  \leq4+4\frac
{(\max\{\beta,\beta^{0}\}-\min\{\alpha,\alpha^{0}\})}{\tau^{2}}\int%
_{\mathbb{R}}|H-H^{0}|^{2}dx\label{B tau alpha square}\\
&  \quad+\frac{4}{\tau^{2}}|\alpha-{\alpha^{0}}|^{2}+\frac{4}{\tau^{2}}%
|\beta-{\beta^{0}}|^{2}.\nonumber
\end{align}
\end{remark}

\section{Discrete Time Approximation}

For a given function $f=f(x,t)$ we  set 
\[
\dot{f}(x,t):=\frac{\partial f}{\partial t}(x,t)\quad\text{and\quad}f^{\prime
}(x,t):=\frac{\partial f}{\partial x}(x,t)\,.
\]
Fix $(\alpha_{0},\beta_{0},h_{0},u_{0})\in\mathcal{A}$ with
$\operatorname*{Lip}h_{0}<L_{0}$ and define
\begin{equation}
H_{0}(x):=\int_{-\infty}^{x}\check{h}_{0}(\rho)\,d\rho\,, \label{H0}%
\end{equation}
where $\check{h}_{0}$ is the extension of $h_{0}$ by zero outside the interval
$(\alpha_{0},\beta_{0})$. For every $k\in\mathbb{N}$ we set $\tau_{k}%
:=\frac{1}{k}$ and $t_{k}^{i}:=i\tau_{k}$, $i\in\mathbb{N}\cup\{0\}$. We
define $(\alpha_{k}^{i},\beta_{k}^{i},h_{k}^{i},u_{k}^{i})\in\mathcal{A}$
inductively with respect to $i$ as follows:
\begin{equation}\label{hkzero}
(\alpha_{k}^{0},\beta_{k}^{0},h_{k}^{0},u_{k}^{0}):=(\alpha_{0},\beta
_{0},h_{0},u_{0})\,,
\end{equation}  
and given $(\alpha_{k}^{i-1},\beta_{k}^{i-1},h_{k}^{i-1},u_{k}^{i-1}%
)\in\mathcal{A}$ we let $(\alpha_{k}^{i},\beta_{k}^{i},h_{k}^{i},u_{k}^{i}%
)\in\mathcal{A}$ be a minimizer of the functional
\begin{equation}
\mathcal{F}_{k}^{i-1}(\alpha,\beta,h,u):=\mathcal{S}(\alpha,\beta
,h)+\mathcal{E}(\alpha,\beta,h,u)+\mathcal{T}_{\tau_{k}}(\alpha,\beta
,h;\alpha_{k}^{i-1},\beta_{k}^{i-1},h_{k}^{i-1})\,, \label{functional Fi}%
\end{equation}
whose existence is guaranteed by Theorem \ref{theorem existence}. We introduce
the linear interpolations (in time) for $\alpha$, $\beta$, $h$ given by%
\begin{align}
\alpha_{k}(t)  &  :=\alpha_{k}^{i-1}+\frac{t-t_{k}^{i-1}}{\tau_{k}}(\alpha
_{k}^{i}-\alpha_{k}^{i-1})\,,\label{alpha inter}\\
\beta_{k}(t)  &  :=\beta_{k}^{i-1}+\frac{t-t_{k}^{i-1}}{\tau_{k}}(\beta
_{k}^{i}-\beta_{k}^{i-1})\,,\label{beta inter}\\
h_{k}(t,x)  &  :=\check{h}_{k}^{i-1}(x)+\frac{t-t_{k}^{i-1}}{\tau_{k}}%
(\check{h}_{k}^{i}(x)-\check{h}_{k}^{i-1}(x))\,, \label{h inter}%
\end{align} 
for $t\in\lbrack t_{k}^{i-1},t_{k}^{i}],\,i\in\mathbb{N}$, and $x\in
\mathbb{R}$, where $\check{h}_{k}^{i}$ is the extension of $h_{k}^{i}$ by zero
outside the interval $[\alpha_{k}^{i},\beta_{k}^{i}]$.

We also introduce the piecewise constant interpolations (in time) for $\alpha$, $\beta$, $h$
given by%
\begin{equation}
\hat{\alpha}_{k}(t):=\alpha_{k}^{i}\,,\quad\hat{\beta}_{k}(t):=\beta_{k}%
^{i}\,,\quad\hat{h}_{k}(t,x):=\check{h}_{k}^{i}(x) \label{h inter pc}%
\end{equation}
for $t\in(t_{k}^{i-1},t_{k}^{i}],\,i\in\mathbb{N}$, and $x\in\mathbb{R}$. For
$t=0$ we set $\hat{\alpha}_{k}(0):=\alpha_{k}^{0}=\alpha_{0}$, $\hat{\beta
}_{k}(0):=\beta_{k}^{0}=\beta_{0}$, and $\hat{h}_{k}(0,x):=h_{k}^{0}%
(x)=h_{0}(x)$. Observe that,  since $(\alpha_{k}^{i},\beta_{k}^{i},h_{k}^{i}%
)\in\mathcal{A}_{s}$ for every $i$ and $k$, we have $\hat{h}_{k}(t,\cdot)\in
H^{2}((\hat{\alpha}_{k}(t),\hat{\beta}_{k}(t)))$  for every $t\geq0$ and, by Lemma \ref{lemma endpoints},
\begin{equation}\label{tyu87}
\hat{\beta}_{k}(t)-\hat{\alpha}_{k}(t)\geq\sqrt{\frac{2A_{0}}{L_{0}}}.%
\end{equation}

Note that since $\operatorname*{Lip}\check{h}_{k}^{i-1}\leq L_{0}$ and
$\operatorname*{Lip}\check{h}_{k}^{i}\leq L_{0}$, we have $\operatorname*{Lip}%
h_{k}(t,\cdot)\leq L_{0}$ and $\operatorname*{Lip}\hat{h}_{k}(t,\cdot)\leq
L_{0}$ for every $t\in\lbrack0,\infty)$. Moreover by \eqref{area constraint},%
\begin{equation}
\int_{\mathbb{R}}h_{k}(t,x)\,dx=\int_{\mathbb{R}}\hat{h}_{k}(t,x)\,dx=A_{0}\,
\label{average hk}%
\end{equation}
for every $t\in\lbrack0,\infty)$ and all $k$.

\begin{lemma}
\label{lemma energy bound}There exists a constant $M_{0}>0$ such that%
\[
\mathcal{S}(\alpha_{k}^{i},\beta_{k}^{i},h_{k}^{i})+\mathcal{E}(\alpha_{k}%
^{i},\beta_{k}^{i},h_{k}^{i},u_{k}^{i})+\sum_{j=1}^{i}\mathcal{T}_{\tau_{k}%
}(\alpha_{k}^{j},\beta_{k}^{j},h_{k}^{j};\alpha_{k}^{j-1},\beta_{k}%
^{j-1},h_{k}^{j-1})\leq M_{0}%
\]
for every $i$ and $k$.
\end{lemma}

\begin{proof}
Fix $i\in\mathbb{N}$ and let $1\leq j\leq i$. Since $(\alpha_{k}^{j},\beta
_{k}^{j},h_{k}^{j},u_{k}^{j})\in\mathcal{A}$ is a minimizer of the functional
$\mathcal{F}_{k}^{j-1}$ defined in \eqref{functional Fi}, we have%
\begin{align*}
\mathcal{S}  &  (\alpha_{k}^{j},\beta_{k}^{j},h_{k}^{j})+\mathcal{E}%
(\alpha_{k}^{j},\beta_{k}^{j},h_{k}^{j},u_{k}^{j})+\mathcal{T}_{\tau_{k}%
}(\alpha_{k}^{j},\beta_{k}^{j},h_{k}^{j};\alpha_{k}^{j-1},\beta_{k}%
^{j-1},h_{k}^{j-1})\\
&  \leq\mathcal{S}(\alpha_{k}^{j-1},\beta_{k}^{j-1},h_{k}^{j-1})+\mathcal{E}%
(\alpha_{k}^{j-1},\beta_{k}^{j-1},h_{k}^{j-1},u_{k}^{j-1})\,.
\end{align*}
Summing both sides of this inequality over $j=1,\ldots,i$, we obtain%
\begin{align*}
\mathcal{S}  &  (\alpha_{k}^{i},\beta_{k}^{i},h_{k}^{i})+\mathcal{E}%
(\alpha_{k}^{i},\beta_{k}^{i},h_{k}^{i},u_{k}^{i})+\sum_{j=1}^{i}%
\mathcal{T}_{\tau_{k}}(\alpha_{k}^{j},\beta_{k}^{j},h_{k}^{j};\alpha_{k}%
^{j-1},\beta_{k}^{j-1},h_{k}^{j-1})\\
&  \leq\mathcal{S}(\alpha_{0},\beta_{0},h_{0})+\mathcal{E}(\alpha_{0}%
,\beta_{0},h_{0},u_{0})=:M_{0}\,,
\end{align*}
and this concludes the proof.
\end{proof}

\begin{proposition}
\label{proposition alpha}There exists a constant $M_{1}>0$ such that%
\begin{equation}
\int_{0}^{\infty}(\dot{\alpha}_{k}(t))^{2}dt\leq M_{1}\quad\text{and}\quad
\int_{0}^{\infty}(\dot{\beta}_{k}(t))^{2}dt\leq M_{1}
\label{alha dot k bounds}%
\end{equation}
for every $k$. In particular,%
\begin{equation}
|\alpha_{k}(t_{2})-\alpha_{k}(t_{1})|\leq M_{1}^{1/2}|t_{2}-t_{1}|^{1/2}%
,\quad|\beta_{k}(t_{2})-\beta_{k}(t_{1})|\leq M_{1}^{1/2}|t_{2}-t_{1}|^{1/2}
\label{alpha k holder}%
\end{equation}
for every $t_{1,}t_{2}\in\lbrack0,\infty)$, and%
\begin{equation}
\alpha^{0}-(tM_{1})^{1/2}\leq\alpha_{k}(t)\leq\beta_{k}(t)\leq\beta
^{0}+(tM_{1})^{1/2} \label{support}%
\end{equation}
for every $t\in\lbrack0,\infty)$.
\end{proposition}

\begin{proof}
By \eqref{surface energy nonnegative}, \eqref{tau energy}, and Lemma
\ref{lemma energy bound} we obtain%
\[
\frac{1}{\tau_{k}}\sum_{j=1}^{\infty}(\alpha_{k}^{j}-\alpha_{k}^{j-1})^{2}%
\leq\frac{2M_{0}}{\sigma_{0}}=:M_{1}\,.
\]
By \eqref{alpha inter} the previous inequality can be written as%
\[
\int_{0}^{\infty}(\dot{\alpha}_{k}(t))^{2}dt=\sum_{j=1}^{\infty}\int%
_{t_{k}^{j-1}}^{t_{k}^{j}}(\dot{\alpha}_{k}(t))^{2}dt\leq M_{1}\,.
\]
The first inequality \eqref{alpha k holder} now follows from the fundamental
theorem of calculus and H\"{o}lder's inequality. Since $\alpha_{k}%
(0)=\alpha^{0}$ the first inequality \eqref{support} is a direct consequence of 
\eqref{alpha k holder}. Similar estimates hold for $\beta_{k}$.
\end{proof}

Define%
\begin{align}
H_{k}^{i}(x)  &  :=\int_{-\infty}^{x}\check{h}_{k}^{i}(\rho)\,d\rho
\,,\text{\quad}H_{k}(t,x):=\int_{-\infty}^{x}h_{k}(t,\rho)\,d\rho
\,.\label{H k}\\
\hat{H}_{k}(t,x)  &  :=\int_{-\infty}^{x}\hat{h}_{k}(t,\rho)\,d\rho\,.
\label{Hk pc}%
\end{align}
Observe that by \eqref{h inter},%
\begin{align}
H_{k}(t,x)  &  =H_{k}^{i-1}(x)+\frac{t-t_{k}^{i-1}}{\tau_{k}}(H_{k}%
^{i}(x)-H_{k}^{i-1}(x))\quad\text{for }t\in\lbrack t_{k}^{i-1},t_{k}%
^{i}]\,,\label{H inter}\\
\hat{H}_{k}(t,x)  &  =H_{k}^{i}(x)\quad\text{for }(t_{k}^{i-1},t_{k}^{i}]\,,
\label{H inter pc}%
\end{align}
for $i\in\mathbb{N}$ and $x\in\mathbb{R}$.

\begin{proposition}
\label{proposition Hk}For every $k$,
\begin{equation}
\int_{0}^{\infty}\Vert\dot{H}_{k}(t,\cdot)\Vert_{L^{2}(\mathbb{R})}^{2}%
dt\leq2M_{0}\,, \label{Hk dot bound}%
\end{equation}
where $M_{0}$ is the constant in Lemma \ref{lemma energy bound}. In
particular,%
\begin{align}
\Vert H_{k}(t_{2},\cdot)-H_{k}(t_{1},\cdot)\Vert_{L^{2}(\mathbb{R})}  &
\leq(2M_{0})^{1/2}|t_{2}-t_{1}|^{1/2}\label{Hk holder}\\
\Vert H_{k}(t,\cdot)\Vert_{L^{2}(\mathbb{R})}  &  \leq(2M_{0})^{1/2}%
t^{1/2}+\Vert H^{0}\Vert_{L^{2}(\mathbb{R})} \label{Hk bound L2}%
\end{align}
for every $t,t_{1},t_{2}\in\lbrack0,\infty)$ and for every $k$, where $H^{0}$
is defined in \eqref{H0}.
\end{proposition}

\begin{proof}
By \eqref{surface energy nonnegative}, \eqref{tau energy}, and Lemma
\ref{lemma energy bound},%
\[
\frac{1}{2\tau_{k}}\sum_{j=1}^{\infty}\int_{\mathbb{R}}(H_{k}^{j}%
(x)-H_{k}^{j-1}(x))^{2}dx\leq M_{0}\,.
\]
Using \eqref{H inter} we have%
\begin{align*}
\int_{0}^{\infty}\int_{\mathbb{R}}(\dot{H}_{k}(t,x))^{2}dxdt  &  =\sum
_{j=1}^{\infty}\int_{t_{k}^{j-1}}^{t_{k}^{j}}\int_{\mathbb{R}}(\dot{H}%
_{k}(t,x))^{2}dxdt\\
&  =\frac{1}{\tau_{k}}\sum_{j=1}^{\infty}\int_{\mathbb{R}}(H_{k}^{j}%
(x)-H_{k}^{j-1}(x))^{2}dx\leq2M_{0}\,,
\end{align*}
which gives \eqref{Hk dot bound}. The estimate \eqref{Hk holder} follows from
the fundamental theorem of calculus and H\"{o}lder's inequality. From
\eqref{H0} and \eqref{Hk holder} we obtain \eqref{Hk bound L2}.
\end{proof}

\bigskip

\begin{proposition}
\label{proposition holder}There exists a constant $M_{2}>0$ such that%
\begin{align}
\Vert h_{k}(t,\cdot)\Vert_{L^{2}(\mathbb{R})}  &  \leq M_{2}(t^{1/2}+{ (\beta_0-\alpha_0))}%
\,,\label{hk L2 bound}\\
\Vert h_{k}(t_{2},\cdot)-h_{k}(t_{1},\cdot)\Vert_{L^{2}(\mathbb{R})}  &  \leq
M_{2}|t_{2}-t_{1}|^{3/10}, \label{hk holder}%
\end{align}
for every $t,t_{1},t_{2}\in\lbrack0,\infty)$ and for every $k$.
\end{proposition}

\begin{proof}
By \eqref{h inter} and \eqref{support} we have that $h_{k}(t,x)=0$ for every
$t\in\lbrack0,\infty)$ and for $x\notin\lbrack\alpha^{0}-(tM_{1})^{1/2}%
,\beta^{0}+(tM_{1})^{1/2}]$. Since $\operatorname*{Lip}h_{k}(t,\cdot)\leq
L_{0}$ for every $t\in\lbrack0,\infty)$, we obtain \eqref{hk L2 bound}. To
prove \eqref{hk holder}, fix $t_{1},t_{2}\in\lbrack0,\infty)$ and $k$ and
apply Theorem 7.41 in \cite{leoni-book2} to the function $v(x):=H_{k}%
(t_{2},x)-H_{k}(t_{1},x)$, $x\in\mathbb{R}$, to get%
\begin{align*}
\Vert h_{k}(t_{2},\cdot)&-h_{k}(t_{1},\cdot)\Vert_{L^{2}(\mathbb{R})}  
=\Vert H_{k}^{\prime}(t_{2},\cdot)-H_{k}^{\prime}(t_{1},\cdot)\Vert
_{L^{2}(\mathbb{R})}\\
&  \leq C_{GN}\Vert H_{k}(t_{2},\cdot)-H_{k}(t_{1},\cdot)\Vert_{L^{2}%
(\mathbb{R})}^{3/5}\Vert H_{k}^{\prime\prime}(t_{2},\cdot)-H_{k}^{\prime
\prime}(t_{1},\cdot)\Vert_{L^{\infty}(\mathbb{R})}^{2/5}\\
&  \leq C_{GN}(2M_{0})^{3/10}|t_{2}-t_{1}|^{3/10}\Vert h_{k}^{\prime}%
(t_{2},\cdot)-h_{k}^{\prime}(t_{1},\cdot)\Vert_{L^{\infty}(\mathbb{R})}%
^{2/5}\\
&  \leq C_{GN}(2M_{0})^{3/10}(2L_{0})^{2/5}|t_{2}-t_{1}|^{3/10}%
\end{align*}
where in the last  inequalities  we used \eqref{Hk holder} and the fact that
$\operatorname*{Lip}h_{k}(t,\cdot)\leq L_{0}$ for every $t\in\lbrack0,\infty)$.
\end{proof}

\begin{proposition}
\label{proposition second i}There exists a constant $M_{3}>1$ such that for
every $i,k$ we have%
\begin{equation}
\int_{\alpha_{k}^{i}}^{\beta_{k}^{i}}|(h_{k}^{i})^{\prime\prime}(x)|^{2}dx\leq
M_{3}\,. \label{h'' i bound}%
\end{equation}
In particular,
\begin{equation}
|(h_{k}^{i})^{\prime}(x_{2})-(h_{k}^{i})^{\prime}(x_{1})|\leq M_{3}%
^{1/2}(x_{2}-x_{1})^{1/2} \label{h' i holder}%
\end{equation}
for all $\alpha_{k}^{i}\leq x_{1}\leq x_{2}\leq\beta_{k}^{i}$.
\end{proposition}

\begin{proof}
It follows from \eqref{surface energy 2}, \eqref{surface energy nonnegative},
Lemma \ref{lemma energy bound}, and the fact that $\operatorname*{Lip}%
h_{k}^{i}\leq L_{0}$ that%
\[
\int_{\alpha_{k}^{i}}^{\beta_{k}^{i}}((h_{k}^{i})^{\prime\prime}(x))^{2}%
dx\leq(1+L_{0}^{2})^{5/2}\int_{\alpha_{k}^{i}}^{\beta_{k}^{i}}\frac
{((h_{k}^{i})^{\prime\prime}(x))^{2}}{(1+((h_{k}^{i})^{\prime}(x))^{2})^{5/2}%
}dx\leq2\nu_{0}^{-1}(1+L_{0}^{2})^{5/2}M_{0}\,.
\]
By the fundamental theorem of calculus and H\"{o}lder's inequality this
implies \eqref{h' i holder}.
\end{proof}

\section{Convergence to the Evolution Problem}\label{section 9}

 Throughout this section we assume that \begin{gather}(\alpha_{0},\beta_{0},h_{0})\in\mathcal{A}_{s}\,,
\label{assumption 1}
\\
h_{0}(x)>0\quad\text{for every }x\in (\alpha_{0},\beta_{0})\,,\label{h0 positive}
\\
h_{0}^{\prime}(\alpha_{0})>0\quad\text{and\quad}h_{0}^{\prime}(\beta_{0})<0\,,
\label{derivatives 0 not zero}%
\\
\operatorname*{Lip}h_{0}<L_{0}\,.
\label{Lip < L0}
\end{gather}

\begin{proposition}\label{rtyu8}
Let $\alpha_{k}$,
$\beta_{k}$, $\hat{\alpha}_{k}$, and $\hat{\beta}_{k}$ be defined as in
\eqref{alpha inter}, \eqref{beta inter}, and \eqref{h inter pc}. Then there
exist functions $\alpha$, $\beta:(0,\infty)\rightarrow\mathbb{R}$ such that
$\alpha,\beta\in H^{1}((0,T))$ for every $T>0$ and, up to subsequences (not
relabeled),
\begin{align}
\alpha_{k}  &  \rightharpoonup\alpha\quad\text{and}\quad\beta_{k}%
\rightharpoonup\beta\quad\text{weakly in }H^{1}%
((0,T))\,,\label{alpha k converge}\\
\alpha_{k}  &  \rightarrow\alpha\quad\text{and}\quad\beta_{k}\rightarrow
\beta\quad\text{uniformly in }[0,T]\,,\label{alpha k converge uniformly}\\
\hat{\alpha}_{k}  &  \rightarrow\alpha\quad\text{and}\quad\hat{\beta}%
_{k}\rightarrow\beta\quad\text{uniformly in }[0,T]\,,
\label{alpha hat converge uniformly}%
\end{align}
for every $T>0$.  Moreover,
\begin{equation}\label{tyu88}
\alpha(0)=\alpha_{0}\,,\quad\beta(0)=\beta_{0}\,,\ \text{ and }\ \beta(t)-\alpha(t)\geq\sqrt{\frac{2A_{0}}{L_{0}}}%
\end{equation}
for every $t\geq0$.
\end{proposition}

\begin{proof}
Proposition \ref{proposition alpha}, together with the initial conditions $\alpha_k(0)=\alpha_{0}$ and $\beta_k(0)=\beta_{0}$,   implies \eqref{alpha k converge}, which in
turn yields \eqref{alpha k converge uniformly}  in view of the compact
embedding. By \eqref{alpha inter} and \eqref{h inter pc}, for every $k$ and
$t\in\lbrack0,T]$, there exists a $\hat{t}\in\lbrack0,T]$ with $t\leq\hat
{t}<t+\tau_{k}$ such that $\hat{\alpha}_{k}(t)=\alpha_{k}(\hat{t})$. By
\eqref{alpha k holder}, $|\hat{\alpha}_{k}(t)-\alpha_{k}(t)|\leq M_{1}%
^{1/2}\tau_{k}$. This implies that $\hat{\alpha}_{k}-\alpha_{k}\rightarrow0$
uniformly. Similarly, $\hat{\beta}_{k}-\beta_{k}\rightarrow0$ uniformly.
Hence, \eqref{alpha hat converge uniformly} follows from
\eqref{alpha k converge uniformly}. The equalities in \eqref{tyu88} follow from \eqref{alpha k converge uniformly}, since $\alpha_k(0)=\alpha_{0}$ and $\beta_k(0)=\beta_{0}$. 
The inequality in \eqref{tyu88} follows from \eqref{tyu87} and \eqref{alpha hat converge uniformly}.
\end{proof}

\begin{proposition}
\label{proposition existence h}Let $h_{k}$ be defined as in \eqref{h inter}. Then there
exist a subsequence (not relabeled) and a nonnegative function $h_{\ast}\in
C^{0,3/10}([0,\infty);L^{2}(\mathbb{R}))$ such that for every $t\in
\lbrack0,\infty)$,
\begin{align}
\operatorname*{Lip}h_{\ast}(t,\cdot)  &  \leq L_{0}\quad\text{and}\quad
h_{\ast}(t,x)=0\quad\text{for }x\notin(\alpha(t),\beta
(t))\label{h star support}\\
&  h_{k}(t,\cdot)\rightarrow h_{\ast}(t,\cdot)\text{\quad uniformly in
}\mathbb{R}\,,\label{hk uniformly}\\
&  h_{k}^{\prime}(t,\cdot)\overset{\ast}{\rightharpoonup}h_{\ast}^{\prime
}(t,\cdot)\quad\text{weakly star in }L^{\infty}(\mathbb{R})\,.
\label{hk prime star}%
\end{align}
Moreover, if we let $h(t,\cdot)$ be the restriction of $h^{\ast}(t,\cdot)$ to
$(\alpha(t),\beta(t))$, then $h(t,\cdot)\in H_{0}^{1}((\alpha(t),\beta
(t)))\cap H^{2}((\alpha(t),\beta(t)))$ and%
\begin{gather}
\int_{\alpha(t)}^{\beta(t)}h(t,x)\,dx=A_{0}\,\,,\label{h average}\\
\int_{\alpha(t)}^{\beta(t)}| h^{\prime\prime}(t,x)|^{2}dx\leq M_{3}%
\,,\label{h H2 bound}\\
|h^{\prime}(t,x_{2})-h^{\prime}(t,x_{1})|\leq M_{3}^{1/2}(x_{2}-x_{1}%
)^{1/2}\quad\text{for all }x_{1},x_{2}\text{ with }\alpha(t)\leq x_{1}\leq
x_{2}\leq\beta(t) \label{h prime holder}%
\end{gather}
for every $t\in\lbrack0,\infty)$.
\end{proposition}

\begin{remark}
It follows from the previous proposition that $(\alpha(t),\beta(t),h(t,\cdot
))\in\mathcal{A}_{s}$ for every $t\in\lbrack0,\infty)$.
\end{remark}

\begin{proof}
[Proof of Proposition \ref{proposition existence h}]By Proposition
\ref{proposition holder} and the Ascoli--Arzel\`{a} theorem there exist a
subsequence (not relabeled) and a nonnegative function $h_{\ast}\in
C^{0,3/18}([0,\infty);L^{2}(\mathbb{R}))$ such that for every $t\in
\lbrack0,\infty)$,
\begin{equation}
h_{k}(t,\cdot)\rightharpoonup h_{\ast}(t,\cdot)\quad\text{weakly in }%
L^{2}(\mathbb{R})\,. \label{hk weakly}%
\end{equation}
Since $\operatorname*{Lip}h_{k}(t,\cdot)\leq L_{0}$ for every $t\in
\lbrack0,\infty)$ and every $k$, we obtain that $\operatorname*{Lip}h_{\ast
}(t,\cdot)\leq L_{0}$ and that \eqref{hk prime star} holds.

Fix $t\in\lbrack0,\infty)$ and $k$ and find $i$ such $t_{k}^{i-1}\leq t\leq
t_{k}^{i}$. Since $\alpha_{k}(t_{k}^{i-1})=\alpha_{k}^{i-1}$ and $\alpha
_{k}(t_{k}^{i})=\alpha_{k}^{i}$, (see \eqref{alpha inter}), by
\eqref{alpha k holder} we have
\begin{equation}
\alpha_{k}(t)-(M_{1}\tau_{k})^{1/2}\leq\min\{\alpha_{k}^{i-1},\alpha_{k}%
^{i}\}\leq\max\{\alpha_{k}^{i-1},\alpha_{k}^{i}\}\leq\alpha_{k}(t)+(M_{1}%
\tau_{k})^{1/2}, \label{e1}%
\end{equation}
where we used the fact that $t_{k}^{i}-t_{k}^{i-1}=\tau_{k}$. Similarly we can
show that
\begin{equation}
\beta_{k}(t)-(M_{1}\tau_{k})^{1/2}\leq\min\{\beta_{k}^{i-1},\beta_{k}%
^{i}\}\leq\max\{\beta_{k}^{i-1},\beta_{k}^{i}\}\leq\beta_{k}(t)+(M_{1}\tau
_{k})^{1/2}\,. \label{e2}%
\end{equation}
Therefore by \eqref{h inter} we have
\begin{equation}
h_{k}(t,x)=0\quad\text{for }x\notin(\alpha_{k}(t)-(M_{1}\tau_{k})^{1/2}%
,\beta_{k}(t)+(M_{1}\tau_{k})^{1/2})\,. \label{support hk}%
\end{equation}
Hence, by \eqref{alpha k converge uniformly} and \eqref{hk weakly} we obtain
that $h_{\ast}(t,x)=0$ for $x\notin(\alpha(t),\beta(t))$. In turn, since
$\operatorname*{Lip}h_{k}(t,\cdot)\leq L_{0}$, we can apply the
Ascoli--Arzel\`{a} theorem in the $x$ variable and we obtain
\eqref{hk uniformly}. Properties \eqref{average hk},
\eqref{alpha k converge uniformly}, \eqref{hk uniformly}, and
\eqref{support hk} imply \eqref{h average}.

To prove \eqref{h H2 bound} observe that by \eqref{h inter}, \eqref{e1},\ and
\eqref{e2}, for every $x\in(\alpha_{k}(t)+(M_{1}\tau_{k})^{1/2},\beta
_{k}(t)-(M_{1}\tau_{k})^{1/2})$ we have%
\[
h_{k}(t,x)=h_{k}^{i-1}(x)+\frac{t-t_{k}^{i-1}}{\tau_{k}}(h_{k}^{i}%
(x)-h_{k}^{i-1}(x))\,.
\]
Fix $a<b$ such that $\alpha(t)<a<b<\beta(t)$. By
\eqref{alpha k converge uniformly} we have that $\alpha_{k}(t)+(M_{1}\tau
_{k})^{1/2}<a<b<\beta_{k}(t)-(M_{1}\tau_{k})^{1/2}$ for all $k$ sufficiently
large, and so $h_{k}^{\prime\prime}(t,\cdot)\in L^{2}((a,b))$ and by
\eqref{h'' i bound}, we have that
\[
\int_{b}^{a}| h_{k}^{\prime\prime}(t,x)|^{2}dx\leq\frac{t-t_{k}^{i-1}}{\tau
_{k}}\int_{b}^{a}|(h_{k}^{i})^{\prime\prime}(x)|^{2}dx+\frac{t_{k}^{i}-t}%
{\tau_{k}}\int_{b}^{a}|(h_{k}^{i-1})^{\prime\prime}(x)|^{2}dx\leq M_{3}\,.
\]
It follows from \eqref{hk uniformly} that $h^{\prime\prime}(t,\cdot)\in
L^{2}((a,b))$ and
\begin{equation}\label{rtyu0}
\int_{a}^{b}| h^{\prime\prime}(t,x)|^{2}dx\leq M_{3}\,.
\end{equation}
Taking the limit as $a\rightarrow\alpha(t)$ and $b\rightarrow\beta(t)$ we
conclude that $h(t,\cdot)\in H^{2}((\alpha(t),\beta(t)))$ and that
\eqref{h H2 bound} holds. In turn, by the fundamental theorem of calculus and
H\"{o}lder's inequality, we obtain \eqref{h prime holder}.
\end{proof}

\begin{proposition}
\label{proposition existence H}Let  $\{h_{k}\}_{k}$ and $h_{\ast}$ be the subsequence and
the function given in Proposition \ref{proposition existence h}, and let $\{H_{k}\}_{k}$ be defined by \eqref{H k}. Then up to a
further subsequence (not relabeled),
\begin{equation}
H_{k}\rightharpoonup H\quad\text{weakly in }H^{1}((0,T);L^{2}(\mathbb{R}%
))\,\text{ for every }T>0\,, \label{Hk weakly}%
\end{equation}
with
\begin{equation}
H(t,x):=\int_{-\infty}^{x}h_{\ast}(t,\rho)\,d\rho\,. \label{H star}%
\end{equation}

\end{proposition}

\begin{proof}
By \eqref{Hk dot bound} and \eqref{Hk bound L2} there exist a subsequence (not
relabeled) and a function $H\colon(0,\infty)\rightarrow L^{2}(\mathbb{R})$ such
that $H\in H^{1}((0,T);L^{2}(\mathbb{R}))$ for every $T>0$ and
\eqref{Hk weakly} holds. It suffices to prove that $H$ equals the right-hand
side of \eqref{H star}.

Let $\varphi\in C_{c}^{\infty}((0,\infty))$ and $\psi\in C_{c}^{\infty
}(\mathbb{R})$. Then by \eqref{Hk weakly},%
\[
\int_{0}^{\infty}\int_{\mathbb{R}}\varphi(t)\psi(x)H_{k}(t,x)\,dxdt\rightarrow
\int_{0}^{\infty}\int_{\mathbb{R}}\varphi(t)\psi(x)H(t,x)\,dxdt
\]
as $k\rightarrow\infty$. On the other hand,%
\begin{align*}
\int_{0}^{\infty}\int_{\mathbb{R}}\varphi(t)\psi(x)H_{k}(t,x)\,dxdt  &
=\int_{0}^{\infty}\int_{\mathbb{R}}\int_{-\infty}^{x}\varphi(t)\psi
(x)h_{k}(t,\rho)\,d\rho dxdt\\
&  \rightarrow\int_{0}^{\infty}\int_{\mathbb{R}}\int_{-\infty}^{x}%
\varphi(t)\psi(x)h(t,\rho)\,d\rho dxdt
\end{align*}
as $k\rightarrow\infty$, where we used the Lebesgue dominated convergence and
\eqref{hk L2 bound}, \eqref{alpha k converge uniformly},   \eqref{hk uniformly}, and \eqref{support hk}. Given the arbitrariness of
$\varphi$ and $\psi$, we obtain \eqref{H star}.
\end{proof}

Next we study the convergence of the piecewise constant interpolations
$\hat{H}_{k}$.

\begin{proposition}
Let  $\{H_{k}\}_{k}$
and $H$ be the subsequence and the function given in Proposition
\ref{proposition existence H}, and let $\{\hat{H}_{k}\}_{k}$ be defined by \eqref{Hk pc}. Then%
\begin{equation}
\hat{H}_{k}(t,\cdot)\rightharpoonup H(t,\cdot)\,\,\quad\text{weakly in }%
L^{2}(\mathbb{R})\text{ for every }t\geq0\,. \label{6001}%
\end{equation}

\end{proposition}

\begin{proof}
Let $\varphi\in L^{2}(\mathbb{R})$ and define%
\[
\psi_{k}(t):=\int_{\mathbb{R}}(H_{k}(t,x)-H(t,x))\varphi(x)\,dx\,.
\]
By Proposition \ref{proposition existence H}, for every $T>0$ the function
$\psi_{k}\in H^{1}((0,T))$ and $\psi_{k}\rightharpoonup0$ weakly in
$H^{1}((0,T))$. This implies that $\psi_{k}(t)\rightarrow0$ for every
$t\in(0,T)$. By the arbitrariness of $\varphi$ and $T$, we deduce that
\begin{equation}
H_{k}(t,\cdot)\rightharpoonup H(t,\cdot)\,\,\quad\text{weakly in }%
L^{2}(\mathbb{R})\text{ for every }t\geq0\,. \label{6000}%
\end{equation}
By \eqref{H inter} and \eqref{Hk holder}, for every $t\in(t_{k}^{i-1}%
,t_{k}^{i}]$,
\[
\Vert\hat{H}_{k}(t,\cdot)-H_{k}(t,\cdot)\Vert_{L^{2}(\mathbb{R})}=\Vert
H_{k}(t_{k}^{i},\cdot)-H_{k}(t,\cdot)\Vert_{L^{2}(\mathbb{R})}\leq
(2M_{0})^{1/2}\tau_{k}^{1/2}\rightarrow0
\]
as $k\rightarrow\infty$. Together with \eqref{6000}, this implies \eqref{6001}.
\end{proof}

\begin{corollary}
Let $h_{\ast
}$ be as in Proposition \ref{proposition existence h} and  let $\{H_{k}\}_{k}$
and $H$ be the subsequence and the function given in Proposition
\ref{proposition existence H}. Fix $t\geq0$. Then the corresponding
subsequence $\{\hat{h}_{k}\}_{k}$ satisfies%
\begin{align}
&  \hat{h}_{k}(t,\cdot)\rightarrow h_{\ast}(t,\cdot)\text{\quad uniformly in
}\mathbb{R}\,,\label{6002}\\
&  \hat{h}_{k}^{\prime}(t,\cdot)\overset{\ast}{\rightharpoonup}h_{\ast
}^{\prime}(t,\cdot)\quad\text{weakly star in }L^{\infty}(\mathbb{R})\,.
\label{6003}%
\end{align}
\end{corollary}

\begin{proof}
Since $\hat{h}_{k}(t,\cdot)=\hat{H}_{k}^{\prime}(t,\cdot)\rightharpoonup
H^{\prime}(t,\cdot)=h_{\ast}(t,\cdot)$ weakly in $H^{-1}(\mathbb{R})$ in view
of the previous proposition, and $\{\hat{h}_{k}(t,\cdot)\}_{k}$ is bounded in
$W^{1,\infty}(\mathbb{R})$, the conclusion follows.
\end{proof}

\bigskip

In the remaining of this section we always assume that the sequences
$\{\alpha_{k}\}_{k}$, $\{\beta_{k}\}_{k}$, $\{h_{k}\}_{k}$, $\{H_{k}\}_{k}$
satisfy \eqref{alpha k converge}, \eqref{alpha k converge uniformly},
\eqref{hk uniformly}, \eqref{hk prime star}, and \eqref{Hk weakly},
\eqref{6001}, \eqref{6002}, and \eqref{6003}, and that $h(t,\cdot)$ is the restriction of $h{ _{\ast}}(t,\cdot)$ to
$(\alpha(t),\beta(t))$.

\begin{lemma}
Let $\{t_{k}\}_{k}$ be a sequence of nonnegative numbers converging to some
$t_{0}\geq0$. Then
\begin{equation}{ h_{k}(t_{k},\cdot)\rightarrow h_{\ast}(t_{0},\cdot)\,,\quad
\hat{h}_{k}(t_{k},\cdot)\rightarrow h_{\ast}(t_{0},\cdot)\text{\quad uniformly
in }\mathbb{R}\,.} \label{hk tk uniformly}%
\end{equation}
Moreover, if $\alpha(t_{0})<a<b<\beta(t_{0})$, then
\begin{equation}
{ h_{k}^{\prime}(t_{k},\cdot)\rightarrow h^{\prime}(t_{0},\cdot
)\,,\quad\hat{h}_{k}^{\prime}(t_{k},\cdot)\rightarrow h^{\prime}(t_{0},\cdot
)\quad\text{uniformly on }[a,b]\,. }\label{h prime k tk uniformly}%
\end{equation}
Finally, if $\{x_{k}\}_{k}$ is a sequence in $\mathbb{R}$ converging to some
$x_{0}\in\mathbb{R}$ such that $\hat{\alpha}_{k}(t_{k})\leq x_{k}\leq
\hat{\beta}_{k}(t_{k})$, then $\alpha(t_{0})\leq x_{0}\leq\beta(t_{0})$ and
\begin{equation}
\hat{h}_{k}^{\prime}(t_{k},x_{k})\rightarrow h^{\prime}(t_{0},x_{0})
\label{h' hat converges}%
\end{equation}

\end{lemma}

\begin{proof}
Since $\operatorname*{Lip}h_{k}(t_{k},\cdot)\leq L_{0}$ for every $k$, by
\eqref{support hk} we can apply the Ascoli--Arzel\`{a} theorem to obtain that
up to a subsequence (not relabeled)%
\[
h_{k}(t_{k},\cdot)\rightarrow g(\cdot)\text{\quad uniformly in }\mathbb{R}\,
\]
for some Lipschitz continuous function $g$. On the other hand, by
\eqref{Hk holder},%
\[
\Vert H_{k}(t_{k},\cdot)-H_{k}(t_{0},\cdot)\Vert_{L^{2}(\mathbb{R})}%
\leq(2M_{0})^{1/2}|t_{k}-t_{0}|^{1/2}%
\]
and in view of \eqref{Hk weakly} we have that
\[
H_{k}(t_{0},\cdot)\rightharpoonup H(t_{0},\cdot)\quad\text{weakly in }%
L^{2}(\mathbb{R})\,.
\]
The last two properties imply that%
\[
H_{k}(t_{k},\cdot)\rightharpoonup H(t_{0},\cdot)\quad\text{weakly in }%
L^{2}(\mathbb{R})\,.
\]
Since $H_{k}^{\prime}(t_{k},\cdot)=h_{k}(t_{k},\cdot)$, it follows that
$g(\cdot)=H^{\prime}(t_{0},\cdot)=h_{\ast}(t_{0},\cdot)$. Since the limit does
not depend on the subsequence, this concludes the proof of
\eqref{hk tk uniformly}.

By \eqref{h'' i bound},%
\[
\int_{a}^{b}| h_{k}^{\prime\prime}(t_{k},x)|^{2}dx\leq M_{3}%
\]
for all $k$ sufficiently large. By H\"{o}lder's inequality this bound implies
that $h_{k}^{\prime}(t_{k},\cdot)$ are H\"{o}lder continuous of exponent
$\frac{1}{2}$ on $[a,b]$ uniformly with respect to $k$. By the
Ascoli--Arzel\`{a} theorem again and by \eqref{hk tk uniformly} we obtain
\eqref{h prime k tk uniformly}. In particular, if $\alpha(t_{0})<x<\beta
(t_{0})$, then%
\begin{equation}
h_{k}^{\prime}(t_{k},x)\rightarrow h^{\prime}(t_{0},x)\,.
\label{h prime k at x}%
\end{equation}

If $\alpha(t_{0})<x_{0}<\beta(t_{0})$, then we can choose $a$ and $b$ with
$\alpha(t_{0})<a<x_{0}<b<\beta(t_{0})$ and \eqref{h' hat converges} follows
from \eqref{h prime k tk uniformly}. It remains to consider the cases
$x_{0}=\alpha(t_{0})$ or $x_{0}=\beta(t_{0})$. We consider only the case
$x_{0}=\alpha(t_{0})$. Fix $0<\varepsilon<\beta(t_{0})-\alpha(t_{0})$ and
observe that $\hat{\alpha}_{k}(t_{k})<x_{k}+\varepsilon<\hat{\beta}_{k}%
(t_{k})$ for all $k$ sufficiently large. By what we just proved%
\begin{equation}
\hat{h}_{k}^{\prime}(t_{k},x_{k}+\varepsilon)\rightarrow h^{\prime}%
(t_{0},x_{0}+\varepsilon)\,. \label{7000}%
\end{equation}
By the fundamental theorem of calculus, H\"{o}lder's inequality,  and
\eqref{h'' i bound},%
\[
|\hat{h}_{k}^{\prime}(t_{k},x_{k}+\varepsilon)-\hat{h}_{k}^{\prime}%
(t_{k},x_{k})|\leq\varepsilon^{1/2}\left(  \int_{x_{k}}^{x_{k}+\varepsilon
}|\hat{h}_{k}^{\prime\prime}(t_{k},x)|^{2}\right)  ^{1/2}\leq\varepsilon
^{1/2}M_{3}^{1/2}\,.
\]
Similarly,%
\[
|h^{\prime}(t_{0},x_{0}+\varepsilon)-h^{\prime}(t_{0},x_{0})|\leq
\varepsilon^{1/2}M_{3}^{1/2}\,.
\]
Therefore, by \eqref{7000},%
\[
\limsup_{k\rightarrow\infty}|\hat{h}_{k}^{\prime}(t_{k},x_{k})-h^{\prime
}(t_{0},x_{0})|\leq2\varepsilon^{1/2}M_{3}^{1/2}.
\]
Letting $\varepsilon\rightarrow0^{+}$, we obtain \eqref{h' hat converges}.
\end{proof}

In what follows $C_{b}(\mathbb{R})$ is the space of bounded continuous
functions in $\mathbb{R}$ with the supremum norm.

\begin{lemma}
\label{lemma h star continuous in time}The function $t\mapsto h_{\ast}%
(t,\cdot)$ from $[0,\infty)$ into $C_{b}(\mathbb{R})$ is continuous.
\end{lemma}

\begin{proof}
Let $\{t_{n}\}_{n}$ be a sequence in $[0,\infty)$ converging to $t_{0}$. Since
$h_{\ast}\in C^{3/10}([0,\infty);L^{2}(\mathbb{R}))$ we have $h_{\ast}%
(t_{n},\cdot)\rightarrow h_{\ast}(t_{0},\cdot)$ in $L^{2}(\mathbb{R})$. On the
other hand, $\operatorname*{Lip}h_{\ast}(t_{n},\cdot)\leq L_{0}$ for every $n$
and by \eqref{h star support} the supports of the functions $h_{\ast}%
(t_{n},\cdot)$ are contained in a compact set independent of $n$. Hence, the
Ascoli--Arzel\`{a} theorem implies $h_{\ast}(t_{n},\cdot)\rightarrow h_{\ast
}(t_{0},\cdot)$ in $C_{b}(\mathbb{R})$.
\end{proof}

\begin{lemma}
\label{lemma h' continuous}Fix a bounded interval $I\subset\lbrack0,\infty)$
and let $a<b$ be such that
\[
\alpha(t)\leq a<b\leq\beta(t)
\]
for all $t\in I$. Then the function $t\mapsto h^{\prime}(t,\cdot)$ from $I$
into $ C^{0}([a,b])$ is continuous
\end{lemma}

\begin{proof}
Let $\{t_{n}\}_{n}$ be a sequence in $I$ converging to $t_{0}\in I$. By Lemma
\ref{lemma h star continuous in time}, $h(t_{n},\cdot)\rightarrow
h(t_{0},\cdot)$ in $ C^{0}([a,b])$. On the other hand, by \eqref{h H2 bound}, the
sequence $\{h(t_{n},\cdot)\}_{n}$ is bounded in $H^{2}((a,b))$ and it
converges weakly to $h(t_{0},\cdot)$ in $H^{2}((a,b))$ and so $h^{\prime
}(t_{n},\cdot)\rightarrow h^{\prime}(t_{0},\cdot)$ in $ C^{0}([a,b])$.
\end{proof}

Since $h(t,\cdot)$ is defined only in $[\alpha(t),\beta(t)]$, its space
derivatives at the endpoints are one-sided.

\begin{lemma}
\label{lemma h' continuous at endpoint}The functions $t\mapsto h^{\prime
}(t,\alpha(t))$ and $t\mapsto h^{\prime}(t,\beta(t))$ are continuous on
$[0,\infty)$.
\end{lemma}

\begin{proof}
It is enough to give the proof for $t\mapsto h^{\prime}(t,\alpha(t))$. Let
$\{t_{n}\}_{n}$ be a sequence in $[0,\infty)$ converging to $t_{0}$ and let
$x\in(\alpha(t_{0}),\beta(t_{0}))$. By continuity of $\alpha$ and $\beta$
there exist an open bounded interval $I\subset\lbrack0,\infty)$ and $a<b$ such
that%
\[
\alpha(t)\leq a<x<b\leq\beta(t)
\]
for all $t\in I$. By Lemma \ref{lemma h' continuous},
\[
h^{\prime}(t_{n},x)\rightarrow h^{\prime}(t_{0},x)
\]
as $n\rightarrow\infty$. By \eqref{h prime holder},%
\begin{align*}
|h^{\prime}  &  (t_{n},\alpha(t_{n}))-h^{\prime}(t_{0},\alpha(t_{0}%
))|\leq|h^{\prime}(t_{n},\alpha(t_{n}))-h^{\prime}(t_{n},x)|+|h^{\prime}%
(t_{n},x)-h^{\prime}(t_{0},x)|\\
&  \quad+|h^{\prime}(t_{0},x)-h^{\prime}(t_{0},\alpha(t_{0}))|\\
&  \leq M_{3}^{1/2}|\alpha(t_{n})-x|^{1/2}+|h^{\prime}(t_{n},x)-h^{\prime
}(t_{0},x)|+M_{3}^{1/2}|\alpha(t_{0})-x|^{1/2}.
\end{align*}
Letting $n\rightarrow\infty$ gives%
\[
\limsup_{n\rightarrow\infty}|h^{\prime}(t_{n},\alpha(t_{n}))-h^{\prime}%
(t_{0},\alpha(t_{0}))|\leq2M_{3}^{1/2}|\alpha(t_{0})-x|^{1/2}.
\]
Taking the limit as $x\rightarrow\alpha(t_{0})^{+}$ we conclude the proof.
\end{proof}

\begin{theorem}
\label{theorem positive} Under the assumptions \eqref{assumption 1}--\eqref{Lip < L0}, there exists $T_{0}>0$ such that for all $t\in\lbrack0,T_{0}]$,%
\begin{align}
h(t,x)  &  >0\quad\text{for all }x\in(\alpha(t),\beta(t))\,,\label{tyu85}\\
h^{\prime}(t,\alpha(t))  &  >0\quad\text{and\quad}h^{\prime}(t,\beta(t))<0\,.
\label{h' positive}%
\end{align}
\end{theorem}

\begin{proof}
Fix $0<\varepsilon<\min\{h_{0}^{\prime}(\alpha_{0}),-h_{0}^{\prime}(\beta
_{0})\}$. By Lemma \ref{lemma h' continuous at endpoint} there exists
$T_{1}>0$ such that
\[
h^{\prime}(t,\alpha(t))\geq\varepsilon\quad\text{and\quad}h^{\prime}%
(t,\beta(t))\leq-\varepsilon
\]
for all $t\in\lbrack0,T_{1}]$. Fix $\delta>0$ such that $M_{3}^{1/2}%
\delta^{1/2}<\varepsilon$. By \eqref{h prime holder},%
\[
h^{\prime}(t,x)\geq h^{\prime}(t,\alpha(t))-M_{3}^{1/2}|x-\alpha(t)|^{1/2}%
\geq\varepsilon-M_{3}^{1/2}\delta^{1/2}>0
\]
for all $\alpha(t)\leq x\leq\alpha(t)+\delta$. Since $h(t,\alpha(t))=0$ by the
previous inequality we have that
\begin{equation}
h(t,x)>0\quad\text{for all }\alpha(t){ <} x\leq\alpha(t)+\delta\,\text{\ and
for all }0\leq t\leq T_{1}\,. \label{p1}%
\end{equation}
{ Moreover, }
\begin{equation}
h(t,x)>0\quad\text{for all }\beta(t)-\delta\leq x{ <}\beta(t)\,\text{\ and for
all }0\leq t\leq T_{1}\,. \label{p2}%
\end{equation}
Fix $a<b$ such that $\alpha_{0}<a<\alpha_{0}+\delta$ and $\beta_{0}%
-\delta<b<\beta_{0}$. Then there exists $0<T_{2}\leq T_{1}$ such that%
\begin{equation}
\alpha(t)<a<\alpha(t)+\delta\quad\text{and}\quad\beta(t)-\delta<b<\beta
(t)\quad\text{for every }0\leq t\leq T_{2}\,. \label{p2a}%
\end{equation}
Let $\eta:=\min_{[a,b]}h_{0}>0\,$. By Lemma
\ref{lemma h star continuous in time} there exists $0<T_{3}\leq T_{2}$ such
that%
\begin{equation}
h(t,x)\geq\frac{\eta}{2}\quad\text{for all }a\leq x\leq b\,\text{\ and for all
}0\leq t\leq T_{3}\,. \label{p3}%
\end{equation}
Combining \eqref{p1}--\eqref{p3}, we obtain $h(t,x)>0$
for all $x\in(\alpha(t),\beta(t))$ and for all $0\leq t\leq T_{3}$.
\end{proof}
 
\begin{proposition}
\label{proposition derivatives away} Under the assumptions of Theorem \ref{theorem positive}, 
there exist $k_{0}\in\mathbb{N}$  and $0<\eta_{0}<1$  such that
\begin{equation}
(h_{k}^{i})^{\prime}(\alpha_{k}^{i}) >2\eta_{0}\quad\text{and\quad}(h_{k}%
^{i})^{\prime}(\beta_{k}^{i})<-2\eta_{0}\label{h i prime positive}%
\end{equation}
for all $k\geq k_{0}$, and all $0\leq i\leq kT_{0}$.
\end{proposition}

\begin{proof}
Since the function $t\mapsto
h^{\prime}(t,\alpha(t))$ is continuous by Lemma
\ref{lemma h' continuous at endpoint}, Theorem \ref{theorem positive} implies
that there exists $ 0<\eta_{0}<1$ such that%
\begin{equation}
h^{\prime}(t,\alpha(t))>2\eta_{0}\quad\text{for every }t\in\lbrack0,T_{0}]\,.
\label{c000}%
\end{equation}
We claim that there exists $k_{0}\in\mathbb{N}$ such that
\begin{equation}
(h_{k}^{i})^{\prime}(\alpha_{k}^{i})>2\eta_{0}\quad\text{for  all }k\geq k_{0}\text{ and }0\leq i\leq kT_{0}\,. \label{c2}%
\end{equation}
If not, then for every $n\in\mathbb{N}$ there exist $k_{n}\geq n$ and $0\leq
i_{n}\leq k_{n}T_{0}$ such that
\begin{equation}
(h_{k_{n}}^{i_{n}})^{\prime}(\alpha_{k_{n}}^{i_{n}})\leq2\eta_{0}\,. \label{c0}%
\end{equation}
Define $t_{n}:=t_{k_{n}}^{i_{n}}=i_{n}/k_{n}$. Since $0\leq t_{n}\leq
T_{0}$, up to a subsequence, $t_{n}\rightarrow t_{0}$ for some $t_{0}%
\in\lbrack0,T_{0}]$, and by \eqref{alpha k converge uniformly},
\begin{equation}
\alpha_{k_{n}}^{i_{n}}=\alpha_{k_{n}}(t_{n})\rightarrow\alpha(t_{0})\,.
\label{c00}%
\end{equation}

By \eqref{h inter}, we have $h_{k_{n}}^{i_{n}}(x)=h_{k_{n}}(t_{n},x)$ for
every $x\in\lbrack\alpha_{k_{n}}^{i_{n}},\beta_{k_{n}}^{i_{n}}]$, and so
$(h_{k_{n}}^{i_{n}})^{\prime}(\alpha_{k_{n}}^{i_{n}})=h_{k_{n}}^{\prime}%
(t_{n},\alpha_{k_{n}}^{i_{n}})$, where $h_{k_{n}}^{\prime}(t_{n},\alpha
_{k_{n}}^{i_{n}})$ is the right derivative of $h_{k_{n}}^{\prime}(t_{n}%
,\cdot)$ at $\alpha_{k_{n}}^{i_{n}}$. Fix $\alpha(t_{0})<x<\beta(t_{0})$ and
let $a,b\in\mathbb{R}$ be such that $\alpha(t_{0})<a<x<b<\beta(t_{0})$. By
\eqref{h prime k tk uniformly},
\begin{equation}
h_{k_{n}}^{\prime}(t_{n},x)\rightarrow h^{\prime}(t_{0},x)\,. \label{c1}%
\end{equation}
By \eqref{h' i holder} and \eqref{h prime holder},%
\begin{gather*}
|(h_{k_{n}}^{i_{n}})^{\prime}    (\alpha_{k_{n}}^{i_{n}})-h^{\prime}%
(t_{0},\alpha(t_{0}))|\leq|(h_{k_{n}}^{i_{n}})^{\prime}(\alpha_{k_{n}}^{i_{n}%
})-(h_{k_{n}}^{i_{n}})^{\prime}(x)|
\\
+|h_{k_{n}}^{\prime}(t_{n},x)-h^{\prime
}(t_{0},x)|+|h^{\prime}(t_{0},x)-h^{\prime}(t_{0},\alpha(t_{0}))|\\
  \leq M_{3}^{1/2}|\alpha_{k_{n}}(t_{n})-x|^{1/2}+|h_{k_{n}}^{\prime}%
(t_{n},x)-h^{\prime}(t_{0},x)|+M_{3}^{1/2}|\alpha(t_{0})-x|^{1/2}.
\end{gather*}
Letting $n\rightarrow\infty$ and using \eqref{c00} and \eqref{c1}, we get%
\[
\limsup_{n\rightarrow\infty}|(h_{k_{n}}^{i_{n}})^{\prime}(\alpha_{k_{n}%
}^{i_{n}})-h^{\prime}(t_{0},\alpha(t_{0}))|\leq2M_{3}^{1/2}|\alpha
(t_{0})-x|^{1/2}.
\]
Taking the limit as $x\rightarrow\alpha(t_{0})^{+}$ we obtain that
\[
(h_{k_{n}}^{i_{n}})^{\prime}(\alpha_{k_{n}}^{i_{n}})\rightarrow h^{\prime
}(t_{0},\alpha(t_{0}))\,,
\]
contradicting \eqref{c000} and \eqref{c0}. This proves the claim. A similar
argument holds for $\beta_{k}^{i}$ and so \eqref{h i prime positive} is satisfied.
\end{proof}

\begin{proposition}
\label{proposition functions away} Under the assumptions of Theorem \ref{theorem positive}, let $\eta_{0}$ be as in Proposition \ref{proposition derivatives away}, and let  
\begin{equation}\label{tyu90}
0<\delta<\frac{1}{2}\min_{t\in\lbrack0,T_{0}]}(\beta(t)-\alpha(t))\,.
\end{equation}
 Then 
there exist $k_{1}\geq k_{0}$  and $0<\eta_{1}<1$ such that
\begin{align}
h_{k}^{i}(x)  &  \geq2\eta_{1}\quad\text{for all }x\in\lbrack\alpha_{k}%
^{i}+\delta,\beta_{k}^{i}-\delta]\,,\label{h i positive uniform}
\\
h_{k}^{i}(x)  &  >0\quad\text{for all }x\in(\alpha_{k}^{i},\beta_{k}%
^{i})\,,\label{h i positive}
\end{align}
for all $k\geq k_{1}$ and $0\leq i\leq kT_{0}$.
\end{proposition}

\begin{proof}
To prove \eqref{h i positive uniform}, we argue by contradiction and assume
that every $n\in\mathbb{N}$ there exist $k_{n}\geq n$, $0\leq i_{n}\leq
k_{n}T_{0}$, and $x_{n}\in\lbrack\alpha_{k_{n}}^{i_{n}}+\delta,\beta_{k_{n}%
}^{i_{n}}-\delta]$ such that $h_{k_{n}}^{i_{n}}(x_{n})\leq 1/n$. Define
$t_{n}:=t_{k_{n}}^{i_{n}}$. Since $0\leq t_{n}\leq T_{0}$, up to a subsequence
(not relabeled), $t_{n}\rightarrow t_{0}$ for some $0\leq t_{0}\leq T_{0}$. By
\eqref{alpha k converge uniformly},
\[
\alpha_{k_{n}}^{i_{n}}=\alpha_{k_{n}}(t_{n})\rightarrow\alpha(t_{0}%
)\,,\quad\beta_{k_{n}}^{i_{n}}=\beta_{k_{n}}(t_{n})\rightarrow\beta(t_{0})\,.
\]
Extracting a further subsequence (not relabeled), we have that $x_{n}%
\rightarrow x_{0}$ for some $x_{0}\in\lbrack\alpha(t_{0})+\delta,\beta
(t_{0})-\delta]$. Since%
\[
1/n\geq h_{k_{n}}^{i_{n}}(x_{n})=h_{k_{n}}(t_{n},x_{n})\rightarrow
h_{\ast}(t_{0},x_{0})
\]
by \eqref{hk tk uniformly},  we obtain a contradiction by Theorem
\ref{theorem positive}.

To show \eqref{h i positive}, fix $0<\varepsilon<\frac{1}{2}\min_{t\in\lbrack0,T_{0}]}(\beta(t)-\alpha(t))$ such that $M_{3}%
^{1/2}\varepsilon^{1/2}<2\eta_{0}$. By \eqref{h' i holder} and \eqref{c2},%
\[
(h_{k}^{i})^{\prime}(x)\geq(h_{k}^{i})^{\prime}(\alpha_{k}^{i})-M_{3}%
^{1/2}|x-\alpha_{k}^{i}|^{1/2}\geq2\eta_{0}-M_{3}^{1/2}\varepsilon^{1/2}>0
\]
for all $\alpha_{k}^{i}\leq x\leq\alpha_{k}^{i}+\varepsilon$ and $k\geq k_{0}%
$. Hence,
\begin{equation}
h_{k}^{i}(x)>0\quad\text{for }\alpha_{k}^{i}<x<\alpha_{k}^{i}+\varepsilon
\text{ and }k\geq k_{0}\,, \label{c3}%
\end{equation}
and in the same way we can show that
\begin{equation}
h_{k}^{i}(x)>0\quad\text{for }\beta_{k}^{i}-\varepsilon<x<\beta_{k}^{i}\text{
and }k\geq k_{0}\,. \label{c4}%
\end{equation}
 The positivity of $h_{k}^{i}(x)$ for $x\in[\alpha_{k}^{i}+\varepsilon,\beta_{k}^{i}-\varepsilon]$ is a consequence of \eqref{h i positive uniform} with $\delta=\varepsilon$.
\end{proof}

\begin{theorem}
\label{theorem lipschitz}
 Under the assumptions \eqref{assumption 1}--\eqref{Lip < L0}, there exists $0<T_{1}\leq
T_{0}$ such that
\begin{equation}\label{lipliplip}
\operatorname*{Lip}h_{\ast}(t,\cdot)< L_{0}%
\end{equation}
for every $t\in\lbrack0,T_{1}]$.
\end{theorem}

\begin{proof}
Fix $L_{1}$ with $\operatorname*{Lip}h_{0}<L_{1}< L_{0}$. By Lemma
\ref{lemma h' continuous at endpoint} the function $t\mapsto h^{\prime
}(t,\alpha(t))$ is continuous and since $h^{\prime}(0,\alpha(0))=h_{0}%
^{\prime}(\alpha_{0})<L_{1}$, there exist $T_{1}>0$ such that
\[
h^{\prime}(t,\alpha(t))<L_{1}\quad\text{for all }t\in\lbrack0,T_{1}]\,.
\]
Fix $\delta>0$ such that $M_{3}^{1/2}\delta^{1/2}< L_{0}-L_{1}$. By
\eqref{h prime holder},
\[
|h^{\prime}(t,x)|\leq h^{\prime}(t,\alpha(t))+M_{3}^{1/2}\delta^{1/2}%
<L_{1}+M_{3}^{1/2}\delta^{1/2}< L_{0}%
\]
for every $t\in\lbrack0,T_{1}]$ and every $\alpha(t)\leq x\leq\alpha
(t)+\delta$. Similarly, taking $T_{1}$ smaller, if  needed, we obtain
\[
|h^{\prime}(t,x)|< L_{0}\quad\text{for every }t\in\lbrack0,T_{1}]\text{
and every }\beta(t)-\delta\leq x\leq\beta(t)\,.
\]
It remains to prove that%
\begin{equation}
\max_{x\in\lbrack\alpha(t)+\delta,\beta(t)-\delta]}|h^{\prime}(t,x)|<
L_{0}\,. \label{b2}%
\end{equation}
Let $g(t)$ denote the left-hand side of \eqref{b2}. To prove that $g$ is
continuous, fix $0\leq t_{0}\leq T_{1}$ and $a,b$ with $\alpha(t_{0}%
)<a<\alpha(t_{0})+\delta$ and $\beta(t_{0})-\delta<b<\beta(t_{0})$. By
continuity of $\alpha$ and $\beta$ there exists an open interval $I$
containing $t_{0}$ such that $a<\alpha(t)+\delta$ and $\beta(t)-\delta<b$ for
all $t\in I$. By Lemma \ref{lemma h' continuous}, $t\mapsto h^{\prime}%
(t,\cdot)$ from $I$ into $ C^{0}([a,b])$ is continuous. Since $\alpha$ and $\beta$
are continuous, it follows that $g$ is continuous in $I$. By the arbitrariness of
$t_{0}$, we conclude that $g$ is continuous in $[0,T_{1}]$. Using the fact
that $g(0)\leq\operatorname*{Lip}h_{0}< L_{0}$, taking $T_{1}$ even
smaller, if  needed, we obtain that $g(t)< L_{0}$ for all $t\in
\lbrack0,T_{1}]$. This concludes the proof.
\end{proof}

\begin{proposition}\label{proposition 9.15}
Let $T_{1}$ be as in Theorem \ref{theorem lipschitz} and $k_1$ be as in Proposition \ref{proposition functions away}. Then
there exists $ k_{2}\geq k_{1}$ such that
\begin{equation}
\operatorname*{Lip}\hat{h}_{k}(t,\cdot)<L_{0} \label{h i k Lipshitz}%
\end{equation}
for all $k\geq  k_{2}$, and all $t\in\lbrack0,T_{1}]$.
\end{proposition}

\begin{proof}
Assume by contradiction that \eqref{h i k Lipshitz} does not hold. Recalling
that $\operatorname*{Lip}\hat{h}_{k}(t,\cdot)\leq L_{0}$ and that $\hat{h}%
_{k}(t,\cdot)\in C^{1,1/2}((\hat{\alpha}_{k}(t),\hat{\beta}_{k}(t)){ )}$ for every
$t$, for every $n\in\mathbb{N}$ there exist $k_{n}\geq n$, $t_{n}\in
\lbrack0,T_{1}]\ $and $x_{n}\in\lbrack\hat{\alpha}_{k_{n}}(t_{n}),\hat{\beta
}_{k_{n}}(t_{n})]$ such that
\begin{equation}
|\hat{h}_{k_{n}}^{\prime}(t_{n},x_{n})|=L_{0}\,. \label{b5}%
\end{equation}
Up to a subsequence, $t_{n}\rightarrow t_{0}$ for some $t_{0}\in\lbrack
0,T_{1}]$, and by \eqref{alpha hat converge uniformly},
\begin{equation}
\hat{\alpha}_{k_{n}}(t_{n})\rightarrow\alpha(t_{0})\quad\text{and\quad}%
\hat{\beta}_{k_{n}}(t_{n})\rightarrow\beta(t_{0})\,. \label{b3}%
\end{equation}
Hence, up to a subsequence (not  relabeled), $x_{n}\rightarrow x_{0}$ for some
$x_{0}\in\lbrack\alpha(t_{0}),\beta(t_{0})]$. By \eqref{h' hat converges} we
have that $\hat{h}_{k_{n}}^{\prime}(t_{n},x_{n})\rightarrow h^{\prime}%
(t_{0},x_{0})$. By \eqref{b5} and Theorem \ref{theorem lipschitz} we obtain a contradiction.
\end{proof}

Let $T_{1}$ be as in Theorem \ref{theorem lipschitz}. For every $t\in
\lbrack0,T_{1}]$, let $u(t,\cdot,\cdot)$ be the unique minimizer of the
problem
\begin{equation}\label{wer1}
\min\left\{  \int_{\Omega_{h(t,\cdot)}}W(Ev(x,y))\,dxdy:\,v\in\mathcal{A}%
_{e}(\alpha(t),\beta(t),h(t,\cdot))\right\}  \,,
\end{equation}
where $\mathcal{A}_{e}$ is defined in \eqref{class A h}.

\begin{proposition}
\label{proposition convergence u}Let  $T_{1}$ be as in Theorem \ref{theorem lipschitz} and let  
$\{t_{k}\}_{k}$ be a sequence in
$[0,T_{1}]$ converging to some $t_{0}$. Assume that for every $k$ there exists
$i_{k}\in\mathbb{N}\cup\{0\}$ such that $t_{k}=t_{k}^{i_{k}}$. Then
$\{u_{k}^{i_{k}}\}_{k}$ converges to $u(t_{0},\cdot,\cdot)$ weakly in
$H^{1}(\tilde{\Omega};\mathbb{R}^{2})$ for every open set $\tilde{\Omega
}\subset\Omega_{h(t_{0},\cdot)}$ with $\operatorname*{dist}(\tilde{\Omega
},\operatorname*{graph}(h(t_{0},\cdot))>0$.
\end{proposition}

\begin{proof}
Let $u_{k}:=u_{k}^{i_{k}}$. By minimality,
\begin{equation}
\int_{\Omega_{h_{k}(t_{k},\cdot)}}W(Eu_{k}(x,y))\,dxdy\leq\int_{\Omega
_{h_{k}(t_{k},\cdot)}}W(e_{0}I)\,dxdy=W(e_{0}I)A_{0}\,, \label{e10}%
\end{equation}
where $I$ is the $2\times2$ identity matrix and we used \eqref{area constraint}%
. Let $\alpha(t_{0})<a<b<\beta(t_{0})$ and let $\tilde{\Omega}$ be an open
set with boundary of class $C^{\infty}$ such that $[a,b]\times\{0\}\subset
\partial\tilde{\Omega}$ and $\operatorname*{dist}(\tilde{\Omega}%
,\operatorname*{graph}(h(t_{0},\cdot)){ )}>0$. 
 By \eqref{alpha k converge uniformly} and \eqref{hk tk uniformly} we have $\operatorname*{dist}(\tilde{\Omega}%
,\operatorname*{graph}(h_k(t_{k},\cdot)){ )}>0$ and
$u_{k}(x,0)=(e_{0}x,0)$ for all $x\in\lbrack a,b]$
and all $k$
sufficiently large. 
By Korn's inequality (see Lemma \ref{lemma korn})  we have that
$\{u_{k}\}_{k}$ is bounded in $H^{1}(\tilde{\Omega};\mathbb{R}^{2})$. Then
there exist a subsequence (not relabeled) and a function $v\in H^{1}%
(\tilde{\Omega};\mathbb{R}^{2})$ such that $u_{k}\rightharpoonup v$ weakly in
$H^{1}(\tilde{\Omega};\mathbb{R}^{2})$.

Take an increasing sequence of domains $\{\tilde{\Omega}_{n}\}_{n}$ as above
such that their union is $\Omega_{h(t_{0},\cdot)}$. By a diagonal argument, we
can extract a further subsequence (not relabeled) and construct a function
$v\in H_{\operatorname*{loc}}^{1}(\Omega_{h(t_{0},\cdot)};\mathbb{R}^{2})$
such that $u_{k}\rightharpoonup v$ weakly in $H^{1}(\tilde{\Omega}%
_{n};\mathbb{R}^{2})$ for every $\tilde{\Omega}_{n}$. By \eqref{W convex},
\eqref{e10}, and a lower semicontinuity argument we have that%
\[
\int_{\tilde{\Omega}_{n}}|Ev(x,y)|^{2}dxdy\leq C_{W}W(e_{0}I)A_{0}%
\]
for every $n$. By letting $n\rightarrow\infty$ we obtain%
\[
\int_{\Omega_{h(t_{0},\cdot)}}|Ev(x,y)|^{2}dxdy\leq C_{W}W(e_{0}I)A_{0}\,.
\]
In view of  \eqref{tyu85},  \eqref{h' positive},  and Theorem \ref{theorem lipschitz}, the set
$\Omega_{h(t_{0},\cdot)}$ has Lipschitz continuous boundary and since, by
construction, $v(x,0)=(e_{0}x,0)$ for all $x\in(a_{n},b_{n})$ for all $n$, we
can apply Korn's inequality (see Lemma \ref{lemma korn}) to conclude that
$v\in H^{1}(\Omega_{h(t_{0},\cdot)};\mathbb{R}^{2})$.

It remains to show that $v=u(t_{0},\cdot,\cdot)$. Let $w\in\mathcal{A}%
_{e}(\alpha(t_{0}),\beta(t_{0}),h(t_{0},\cdot))$. Since $h(t_{0},\cdot)\in
C^{1}([\alpha(t_{0}),\beta(t_{0})])$ by Proposition
\ref{proposition existence h} and \eqref{h' positive} holds we can argue as  at the end of Step 1 in  
the proof of Theorem \ref{theorem C3} to extend $w$ to
a function $\tilde{w}\in H^{1}(U;\mathbb{R}^{2})$, where $U:=(\alpha
(t_{0})-1,\beta(t_{0})+1)\times(0,\infty)$ and $\tilde{w}(x,0)=(e_{0}x,0)$ for
all $x\in(\alpha(t_{0})-1,\beta(t_{0})+1)$. By the minimality of $u_{k}$ in
$\Omega_{h_{k}(t_{k},\cdot)}$ we have%
\[
\int_{\Omega_{h_{k}(t_{k},\cdot)}}W(Eu_{k}(x,y))\,dxdy\leq\int_{\Omega
_{h_{k}(t_{k},\cdot)}}W(E\tilde{w}(x,y))\,dxdy\,.
\]
Letting $k\rightarrow\infty$ and using \eqref{hk tk uniformly} we obtain%
\[
\lim_{k\rightarrow\infty}\int_{\Omega_{h_{k}(t_{k},\cdot)}}W(E\tilde
{w}(x,y))\,dxdy=\int_{\Omega_{h(t_{0},\cdot)}}W(Ew(x,y))\,dxdy\,.
\]
On the other, by lower semicontinuity%
\begin{align*}
\int_{\tilde{\Omega}_{n}}W(Ev(x,y))\,dxdy&\leq\liminf_{k\rightarrow\infty}%
\int_{\tilde{\Omega}_{n}}W(Eu_{k}(x,y))\,dxdy\\&\leq\liminf_{k\rightarrow\infty
}\int_{\Omega_{h_{k}(t_{k},\cdot)}}W(Eu_{k}(x,y))\,dxdy
\end{align*}
for every $n$. Taking the limit as $n\rightarrow\infty$ we obtain%
\[
\int_{\Omega_{h(t_{0},\cdot)}}W(Ev(x,y))\,dxdy\leq\int_{\Omega_{h(t_{0}%
,\cdot)}}W(Ew(x,y))\,dxdy\,,
\]
which proves the minimality of $v$. By uniqueness, $v=u(t_{0},\cdot,\cdot)$,
and since the limit is independent of the subsequence, the entire sequence
$\{u_{k}\}_{k}$ converges to $u(t_{0},\cdot,\cdot)$ as in the statement.
\end{proof}

In what follows%
\begin{equation}
p_{1}:=\min\{6/5,p_{0}/(4-2p_{0})\}\quad\text{and}\quad q_{1}:=4p_{1}/(2+p_{1})\,,
\label{p1 and q1}%
\end{equation}
where $p_{0}$ is defined in \eqref{p0}. We observe that $1<p_{1}\leq6/5$ and
$4/3<q_{1}\leq3/2$.

\begin{theorem}\label{theorem 9.19}.  Under the assumptions \eqref{assumption 1}--\eqref{Lip < L0}, 
let $T_{1}$ be as in Theorem \ref{theorem lipschitz} and let $k_{2}$ be as in Proposition \ref{proposition 9.15}. Then 
there  exists  $M_{4}>0$ such that%
\begin{gather}
\int_{0}^{T_{1}}\Vert\hat{h}_{k}^{\prime\prime\prime}(t,\cdot)\Vert_{L^{q_{1}%
}((\hat{\alpha}_{k}(t),\hat{\beta}_{k}(t)))}^{p_{1}}dt\leq  M_{4}\Big(  \int%
_{0}^{T_{1}}(\hat{\beta}_{k}(t)-\hat{\alpha}_{k}(t))^{(2-5p_{1})/(4-2p_{1}%
)}\Big)  ^{1-p_{1}/2}+M_{4}\,, \label{h''' hat bound}%
\\
\int_{0}^{T_{1}}\int_{\hat{\alpha}_{k}(t)}^{\hat{\beta}_{k}(t)}|\hat{h}%
_{k}^{(\operatorname*{iv})}(t,x)|^{p_{1}}dxdt\leq  M_{4}\,, \label{bound bar h iv}%
\end{gather}
for all $k\geq  k_{2}$.
\end{theorem}

\begin{proof}
\textbf{Step 1. } In this proof $C$ denotes a constant, independent of $k$ and $i$, whose value can change from formula to formula.  Let $\hat{q}:=\max\{p_{1},5p_{1}/2-1,3p_{1}-2,2p_{1}-1\}$.
 Since $p_1\leq 6/5$, by Remark \ref{remark q} we have $\hat{q}\leq 2$, hence, by \eqref{B tau alpha square} and  \eqref{support},%
\begin{align*}
B_{\tau_{k}}(H_{k}^{i},H_{k}^{i-1},\alpha_{k}^{i},\alpha_{k}^{i-1},\beta
_{k}^{i},\beta_{k}^{i-1})^{\hat{q}}  &  \leq4+ C\Vert\dot{H}_{k}(t,\cdot)\Vert_{L^{2}(\mathbb{R})}^{2}
+4|\dot{\alpha}_{k}(t)|^{2}+4|\dot{\beta}_{k}(t)|^{2}\,.
\end{align*}

 By Proposition \ref{proposition second i}, $(\alpha_{k}^{i}, \beta_{k}^{i}, h_{k}^{i})$ satisfies \eqref{h'' L2 norm} with $M=M_3>1$. By Proposition \ref{proposition derivatives away} there exists $0<\eta_{0}<1$ such that $(\alpha_{k}^{i}, \beta_{k}^{i}, h_{k}^{i})$ satisfies \eqref{h' alpha and beta} for all $k\geq k_{0}$ and all $0\leq i\leq kT_{0}$. By \eqref{rtyu0} and \eqref{c000} for $t\in[0,T_{0}]$ we can apply Lemma \ref{qwe1} to $(\alpha(t), \beta(t), h(t,\cdot))$  obtaining $\beta(t)-\alpha(t)\geq 16\eta_{0}^2/M_{3}$, hence \eqref{tyu90} is satisfied by $\delta_{0}=\eta_{0}^2/(4M_{3})$. Therefore, by Proposition \ref{proposition functions away} there exists $0<\eta_{1}<1$ such that $(\alpha_{k}^{i}, \beta_{k}^{i}, h_{k}^{i})$ satisfies 
\eqref{h bounded from zero} and the { second} inequality in \eqref{rtyu2} for all $k\geq k_{1}$ and all $0\leq i\leq kT_{0}$.
Moreover, Proposition \ref{proposition 9.15}
implies that $h_{k}^{i}$ satisfies the { first} inequality in \eqref{rtyu2} for all $k\geq k_{2}$ and all $0\leq i\leq kT_{1}$. We conclude that all
 assumptions of Theorem \ref{theorem C3} are satisfied
 for all $k\geq k_{2}$ and all $0\leq i\leq kT_{1}$.
 
Hence, by Theorem \ref{theorem h iv Lp}, we have%
\begin{align*}
\int_{\alpha_{k}^{i}}^{\beta_{k}^{i}}|(h_{k}^{i})^{(\operatorname*{iv}%
)}(x)|^{p_{1}}dx  &  \leq c_{9}B_{\tau_{k}}(H_{k}^{i},H_{k}^{i-1},\alpha
_{k}^{i},\alpha_{k}^{i-1},\beta_{k}^{i},\beta_{k}^{i-1})^{\hat{q}}\\
&  \leq C+C\Vert\dot{H}_{k}(t,\cdot)\Vert_{L^{2}(\mathbb{R})}^{2}%
+C|\dot{\alpha}_{k}(t)|^{2}+{ C}|\dot{\beta}_{k}(t)|^{2}\,.
\end{align*}
Integrating in time over $[t_{k}^{i-1},t_{k}^{i}]$ and summing over all $1\leq
i\leq kT_{1}$ we obtain%
\begin{align}
\int_{0}^{T_{1}}\int_{\hat{\alpha}_{k}(t)}^{\hat{\beta}_{k}(t)}|\hat{h}%
_{k}^{(\operatorname*{iv})}(t,x)|^{p_{1}}dxdt  &  \leq CT_{1}+C\int_{0}^{T_{1}%
}\Vert\dot{H}_{k}(t,\cdot)\Vert_{L^{2}(\mathbb{R})}^{2}dt\label{m88}\\
&  \quad+C\int_{0}^{T_{1}}|\dot{\alpha}_{k}(t)|^{2}dt+C\int_{0}^{T_{1}}%
|\dot{\beta}_{k}(t)|^{2}dt\leq C\,,\nonumber
\end{align}
where in the last inequality we used \eqref{alha dot k bounds} and
\eqref{Hk dot bound}. This proves \eqref{bound bar h iv}.

\textbf{Step 2: }By standard interpolation results (\cite[Theorem
7.41]{leoni-book2}) for every $t\in\lbrack0,T_{1}]$,%
\begin{align*}
\Vert\hat{h}_{k}^{\prime\prime\prime}(t,\cdot)  &  \Vert_{L^{q_{1}}%
((\hat{\alpha}_{k}(t),\hat{\beta}_{k}(t)))}\leq C(\hat{\beta}_{k}%
(t))-\hat{\alpha}_{k}(t))^{1/q_{1}-3/2}\Vert\hat{h}_{k}^{\prime\prime}%
(t,\cdot)\Vert_{L^{2}((\hat{\alpha}_{k}(t),\hat{\beta}_{k}(t)))}\\
&  \quad+C\Vert\hat{h}_{k}^{\prime\prime}(t,\cdot)\Vert_{L^{2}((\hat{\alpha
}_{k}(t),\hat{\beta}_{k}(t)))}+C\Vert \hat{h}_{k}^{(\operatorname*{iv})}(t,\cdot
)\Vert_{L^{p_{1}}((\hat{\alpha}_{k}(t),\hat{\beta}_{k}(t)))}\,.
\end{align*}
In turn,%
\begin{align*}
\int_{0}^{T_{1}}  &  \Vert\hat{h}_{k}^{\prime\prime\prime}(t,\cdot
)\Vert_{L^{q_{1}}((\hat{\alpha}_{k}(t),\hat{\beta}_{k}(t)))}^{p_{1}}dt\\&\leq
C\int_{0}^{T_{1}}(\hat{\beta}_{k}(t))-\hat{\alpha}_{k}(t))^{1/2-5p_{1}/4}%
\Vert\hat{h}_{k}^{\prime\prime}(t,\cdot)\Vert_{L^{2}((\hat{\alpha}_{k}%
(t),\hat{\beta}_{k}(t)))}^{p_{1}}dt\\
&  \quad+C\int_{0}^{T_{1}}\Vert\hat{h}_{k}^{\prime\prime}(t,\cdot)\Vert
_{L^{2}((\hat{\alpha}_{k}(t),\hat{\beta}_{k}(t)))}^{p_{1}}dt+C\int_{0}^{T_{1}%
}\Vert \hat{h}_{k}^{(\operatorname*{iv})}(t,\cdot)\Vert_{L^{p_{1}}((\hat{\alpha}%
_{k}(t),\hat{\beta}_{k}(t)))}^{p_{1}}dt\,.
\end{align*}
By H\"{o}lder's inequality, \eqref{h'' i bound}, and \eqref{bound bar h iv},%
\begin{align*}
&\int_{0}^{T_{1}}  \Vert\hat{h}_{k}^{\prime\prime\prime}(t,\cdot
)\Vert_{L^{q_{1}}((\hat{\alpha}_{k}(t),\hat{\beta}_{k}(t)))}^{p_{1}}dt\\
&  \leq C\left(  \int_{0}^{T_{1}}(\hat{\beta}_{k}(t))-\hat{\alpha}%
_{k}(t))^{(2-5p_{1})/(4-2p_{1})}\right)  ^{1-p_{1}/2}\left(  \int_{0}^{T_{1}%
}\Vert\hat{h}_{k}^{\prime\prime}(t,\cdot)\Vert_{L^{2}((\hat{\alpha}%
_{k}(t),\hat{\beta}_{k}(t)))}^{2}dt\right)  ^{p_{1}/2}\\
&  \quad+CT_{1}^{1-p_{1}/2}\left(  \int_{0}^{T_{1}}\Vert\hat{h}_{k}%
^{\prime\prime}(t,\cdot)\Vert_{L^{2}((\hat{\alpha}_{k}(t),\hat{\beta}%
_{k}(t)))}^{2}dt\right)  ^{p_{1}/2}
\\&\quad+C\int_{0}^{T_{1}}\Vert
\hat{h}_{k}^{(\operatorname*{iv})}(t,\cdot)\Vert_{L^{p_{1}}((\hat{\alpha}_{k}%
(t),\hat{\beta}_{k}(t)))}^{p_{1}}dt\\
&  \leq C(T_{1}M_{3})^{p_{1}/2}\left(  \int_{0}^{T_{1}}(\hat{\beta}%
_{k}(t))-\hat{\alpha}_{k}(t))^{(2-5p_{1})/(4-2p_{1})}\right)  ^{1-p_{1}%
/2}+C T_{1} M_{3}^{p_{1}/2}+M_{4}\,.
\end{align*}
This proves \eqref{h''' hat bound}.
\end{proof}

\bigskip

\begin{theorem}\label{uyt01} Under the assumptions \eqref{assumption 1}--\eqref{Lip < L0},  let $T_{1}$ be as in Theorem \ref{theorem lipschitz}. Then 
for a.e. $t\in(0,T_{1})$ we have%
\begin{align}
\sigma_{0}\dot{\alpha}(t) & =\frac{\gamma}{J(t,\alpha(t))}- \gamma_{0}+\nu_{0}\frac{h^{\prime}(t,\alpha(t))}{(J(t,\alpha(t)))^{2}}\Big(\frac
{h^{\prime\prime}(t,\cdot)}{(J(t,\cdot))^{3}}\Big)^{\prime}(\alpha(t))\,,
\label{equation alpha dot}%
\\
\sigma_{0}\dot{\beta}(t)  &  =-\frac{\gamma}{J(t,\beta(t))}%
+\gamma_{0}-\nu_{0}\frac{h^{\prime}(t,\beta(t))}{J(t,\beta(t))^{2}}\Big(\frac
{h^{\prime\prime}(t,\cdot)}{J(t,\cdot)^{3}}\Big)^{\prime}(\beta(t))\,, \label{equation beta dot}%
\end{align}
 where $J(t,\cdot)$ is defined in \eqref{J and J0} using $h(t,\cdot)$. 
\end{theorem}

\begin{remark}\label{remark intrinsic boundary} To express \eqref{equation alpha dot} and \eqref{equation beta dot} in an intrinsic way, for every $t\in [0,T_{1}]$  we define $s(t,x)$ for $\alpha(t)\leq x\leq\beta(t)$ by
\begin{equation}\label{s(x)}
s(t,x):=\int_{\alpha(t)}^{x}\sqrt{1+(h^{\prime}(t,\rho))^{2}}d\rho
\end{equation}
The inverse of $s(t,\cdot)$, defined for 
$0=s(t,\alpha(t))\leq s\leq s(t,\beta(t))$, is denoted by $x(t,\cdot)$. Let $\kappa(t,\cdot):\big(s(t,\alpha(t)),s(t,\beta(t))\big)\rightarrow
\mathbb{R}$ be the signed curvature of the graph of $h(t,\cdot)$, considered as a
function of arclength. To be precise, we have%
\begin{equation}
\kappa(t,s):=\frac{h^{\prime\prime}(t,x(t,s))}{\Big(1+\big(h^{\prime}(t,x(t,s))\big)^{2}\Big)^{3/2}}\,,\label{kappa}
\end{equation}
hence
\[
\kappa(t,s(t,x))=\frac{h^{\prime\prime}(t,x)}{J(t,x)^{3}}
\]

Since
\[
\Big(\frac{h^{\prime\prime}(t,\cdot)}{(J(t,\cdot))^{3}}\Big)^{\prime}(x)=\partial_{s}\kappa%
\big(t,s(t,x)\big)J(t,x)\,,
\]
we can rewrite \eqref{equation alpha dot} and \eqref{equation beta dot} as%
\begin{align*}
& \sigma_{0}\dot{\alpha}(t)=\gamma\cos\theta_{\alpha}(t)-\gamma_{0}%
+\nu_{0}\partial_{s}\kappa
\big(t,s(t,\alpha(t))\big)\sin\theta_{\alpha}(t)\,,\\
& \sigma_{0}\dot{\beta}(t)=-\gamma\cos\theta_{\beta}(t)+\gamma_{0}
-\nu_{0}\partial_{s}\kappa
\big(t,s(t,\beta(t))\big)\sin\theta_{\beta}(t)\,,
\end{align*}
where
\begin{equation}
\theta_{\alpha}(t):=\arcsin\frac{h^{\prime}(t,\alpha(t))}{\sqrt{1+(h^{\prime}%
(t,\alpha(t)))^{2}}}\quad\text{and}\quad\theta_{\beta}(t):=\arcsin\frac{h^{\prime
}(t,\beta(t))}{\sqrt{1+(h^{\prime}(t,\beta(t)))^{2}}}\,, \label{theta1}%
\end{equation}
are the oriented angles between the oriented $x$-axis and the tangent to the
graph of $h(t,\cdot)$ at $(\alpha(t),0)$ and $(\beta(t),0)$, respectively. 
\end{remark}

\begin{proof}[Proof of Theorem \ref{uyt01}]
Fix $t\in\lbrack0,T_{1}]$ and $k\geq k_{2}$, where $k_{2}$ is as in Proposition \ref{proposition 9.15},  and find $i$ such $t_{k}^{i-1}\leq
t\leq t_{k}^{i}$. By \eqref{equation x l1},
\[
\sigma_{0}\frac{\alpha_{k}^{i}-\alpha_{k}^{i-1}}{\tau_{k}}=\frac{\gamma}%
{J_{k}^{i}(\alpha_{k}^{i})}-\gamma_{0}+\nu_{0}\frac{(h_{k}^{i})^{\prime
}(\alpha_{k}^{i})}{(J_{k}^{i}(\alpha_{k}^{i}))^{2}}\Big(\frac{(h_{k}%
^{i})^{\prime\prime}}{(J_{k}^{i})^{3}}\Big)^{\prime}(\alpha_{k}^{i})\,,
\]
 where $J_{k}^{i}(x):=(1+((h_{k}^{i})^{\prime}(x))^{2})^{1/2}$. Introducing
$\hat{J}_{k}(t,x):=(1+(\hat{h}_{k}^{\prime}(t,x))^{2})^{1/2}$, for every
$0\leq t\leq T_{1}$ we have

\begin{equation}
\sigma_{0}\dot{\alpha}_{k}(t-)=\frac{\gamma}{\hat{J}_{k}(t,\hat{\alpha}%
_{k}(t))}-\gamma_{0}+\nu_{0}\frac{\hat{h}_{k}^{\prime}(t,\hat{\alpha}%
_{k}(t))}{(\hat{J}_{k}(t,\hat{\alpha}_{k}(t)))^{2}}\Big(\frac{\hat{h}%
_{k}^{\prime\prime}(t,\cdot)}{(\hat{J}_{k}(t,\cdot))^{3}}\Big)^{\prime}%
(\hat{\alpha}_{k}(t))\,, \label{800}%
\end{equation}
where $\dot{\alpha}_{k}(t-)$ is the left derivative. We claim that%
\begin{equation}
\frac{\gamma}{\hat{J}_{k}(\cdot,\hat{\alpha}_{k}(\cdot))}\rightarrow
\frac{\gamma}{(1+(h^{\prime}(\cdot,\alpha(\cdot)))^{2})^{1/2}}\quad
\text{strongly in }L^{2}((0,T_{1}))\,. \label{801}%
\end{equation}
Since the function $s\mapsto 1/(1+s^{2})^{1/2}$ is 
$1$-Lipschitz 
continuous,  by \eqref{h' hat converges} we have 
\[
\left\vert \frac{1}{\hat{J}_{k}(t,\hat{\alpha}_{k}(t))}-\frac{1}%
{(1+(h^{\prime}(t,\alpha(t)))^{2})^{1/2}}\right\vert \leq|\hat{h}_{k}^{\prime
}(t,\hat{\alpha}_{k}(t))-h^{\prime}(t,\alpha(t))|\rightarrow0\,,
\]
 which gives  \eqref{801} by the dominated convergence theorem. 

To study the last term in \eqref{800} we observe that 
\begin{align*}
\frac{\hat{h}_{k}^{\prime}(t,\hat{\alpha}_{k}(t))}{\hat{J}_{k}(t,\hat{\alpha
}_{k}(t))^{2}}&\Big(\frac{\hat{h}_{k}^{\prime\prime}(t,\cdot)}{(\hat{J}%
_{k}(t,\cdot))^{3}}\Big)^{\prime}(\hat{\alpha}_{k}(t))    =\frac{\hat{h}%
_{k}^{\prime}(t,\hat{\alpha}_{k}(t))}{(\hat{J}_{k}(t,\hat{\alpha}_{k}%
(t)))^{5}}\hat{h}_{k}^{\prime\prime\prime}(t,\hat{\alpha}_{k}(t))\\&-3\frac
{(\hat{h}_{k}^{\prime}(t,\hat{\alpha}_{k}(t)))^{2}}{(\hat{J}_{k}(t,\hat
{\alpha}_{k}(t)))^{7}}(\hat{h}_{k}^{\prime\prime}(t,\hat{\alpha}_{k}%
(t)))^{2}  =\frac{\hat{h}_{k}^{\prime}(t,\hat{\alpha}_{k}(t))}{(\hat{J}_{k}%
(t,\hat{\alpha}_{k}(t)))^{5}}\hat{h}_{k}^{\prime\prime\prime}(t,\hat{\alpha
}_{k}(t))\,,
\end{align*}
where we used \eqref{h'' zero at endpoints}. We claim that%
\begin{equation}
\hat{h}_{k}^{\prime\prime\prime}(\cdot,\hat{\alpha}_{k}(\cdot))\rightharpoonup
h^{\prime\prime\prime}(\cdot,\alpha(\cdot))\quad\text{weakly in }L^{p_{1}%
}((0,T_{1}))\,. \label{weak convergence}%
\end{equation}
By the fundamental theorem of calculus,%
\[
\hat{h}_{k}^{\prime\prime\prime}(t,\hat{\alpha}_{k}(t))=\hat{h}_{k}%
^{\prime\prime\prime}(t,x)-\int_{\hat{\alpha}_{k}(t)}^{x}\hat{h}%
_{k}^{(\operatorname*{iv})}(t,\rho)\,d\rho\,,
\]
where $\hat{\alpha}_{k}(t)<x<\hat{\beta}_{k}(t)$. By Lemma
\ref{lemma endpoints} and the uniform continuity of $\alpha$ and $\beta$ (see
\eqref{alpha k converge}), we can subdivide $[0,T_{1}]$ into a finite  number 
of intervals $I$ such that for each of them there exist $a,b\in\mathbb{R}$
such that $\alpha(t)<a<b<\beta(t)$ for all $t\in I$. To prove
\eqref{weak convergence} it suffices to prove weak convergence in $L^{p_{1}%
}(I)$ for every such $I$. Fix $I=[t_{1},t_{2}]$ and the corresponding $a$,
$b$. By H\"{o}lder's inequality%
\[
|\hat{h}_{k}^{\prime\prime\prime}(t,\hat{\alpha}_{k}(t))|^{p_{1}}\leq
C|\hat{h}_{k}^{\prime\prime\prime}(t,x)|^{p_{1}}+C\int_{\hat{\alpha}_{k}%
(t)}^{\hat{\beta}_{k}(t)}|\hat{h}_{k}^{(\operatorname*{iv})}(t,\rho)|^{p_{1}%
}d\rho\,.
\]
Averaging in $x$ over $(a,b)$ and integrating in $t$ over $(t_{1},t_{2})$
gives%
\[
\int_{t_{1}}^{t_{2}}|\hat{h}_{k}^{\prime\prime\prime}(t,\hat{\alpha}%
_{k}(t))|^{p_{1}}dt\leq\frac{C}{b-a}\int_{t_{1}}^{t_{2}}\int_{a}^{b}|\hat
{h}_{k}^{\prime\prime\prime}(t,x)|^{p_{1}}dxdt+C\int_{t_{1}}^{t_{2}}\int%
_{\hat{\alpha}_{k}(t)}^{\hat{\beta}_{k}(t)}|\hat{h}_{k}^{(\operatorname*{iv}%
)}(t,\rho)|^{p_{1}}d\rho dt\,.
\]
Note that the first integral on the right-hand side is bounded because
$p_{1}\leq q_{1}$ and in view of \eqref{h''' hat bound}, while the second integral is bounded by \eqref{bound bar h iv}.

Let $\phi\in C_{c}^{\infty}((a,b))$ with $\int_{a}^{b}\phi(x)\,dx\neq0$ and
$\psi\in C_{c}^{\infty}((t_{1},t_{2}))$. Then%
\begin{align*}
\int_{a}^{b}\phi(x)\,dx&\int_{t_{1}}^{t_{2}}\hat{h}_{k}^{\prime\prime\prime
}(t,\hat{\alpha}_{k}(t))\psi(t)\,dt   =\int_{t_{1}}^{t_{2}}\int_{b}^{a}%
\hat{h}_{k}^{\prime\prime\prime}(t,x)\phi(x)\psi(t)\,dxdt\\
&  \quad-\int_{t_{1}}^{t_{2}}\int_{b}^{a}\int_{\mathbb{R}}\hat{h}%
_{k}^{(\operatorname*{iv})}(t,\rho)\chi_{(\hat{\alpha}_{k}(t),x)}(\rho
)\phi(x)\psi(t)\,d\rho dxdt,
\end{align*}
where $\hat{h}_{k}^{(\operatorname*{iv})}(t,\rho)\chi_{(\hat{\alpha}%
_{k}(t),x)}(\rho)$ is interpreted to be zero for $\rho\notin(\hat{\alpha}%
_{k}(t),x)$. By {\eqref{alpha hat converge uniformly}}, \eqref{6002},  and
\eqref{bound bar h iv}, we have%
\[
\hat{h}_{k}^{(\operatorname*{iv})}(t,\rho)\chi_{(\hat{\alpha}_{k}(t),x)}%
(\rho)\phi(x)\rightharpoonup h^{(\operatorname*{iv})}(t,\rho)\chi
_{(\alpha(t),x)}(\rho)\phi(x)\text{\quad weakly in }L^{p_{1}}({(t_{1}}%
,t_{2})\times\mathbb{R})
\]
for every fixed $x$. On the other hand, by \eqref{h''' hat bound},%
\[
\hat{h}_{k}^{\prime\prime\prime}(t,x)\rightharpoonup h^{\prime\prime\prime
}(t,x)\text{\quad weakly in }L^{p_{1}}((t_{1},t_{2});L^{q_{1}}((a,b))\,.
\]
Therefore,%
\begin{align*}
\int_{a}^{b}\phi(x)\,dx\int_{t_{1}}^{t_{2}}\hat{h}_{k}^{\prime\prime\prime
}(t,\hat{\alpha}_{k}(t))\psi(t)\,dt  &  \rightarrow\int_{t_{1}}^{t_{2}}%
\int_{b}^{a}h^{\prime\prime\prime}(t,x)\phi(x)\psi(t)\,dxdt\\
&  \quad-\int_{t_{1}}^{t_{2}}\int_{b}^{a}\int_{\alpha(t)}^{x}%
h^{(\operatorname*{iv})}(t,\rho)(\rho)\phi(x)\psi(t)\,d\rho dxdt\\
&  =\int_{a}^{b}\phi(x)dx\int_{t_{1}}^{t_{2}}h^{\prime\prime\prime}%
(t,\alpha(t))\psi(t)\,dt\,.
\end{align*}
Dividing by $\int_{a}^{b}\phi(x)\,dx$, we get
\[
\int_{t_{1}}^{t_{2}}\hat{h}_{k}^{\prime\prime\prime}(t,\hat{\alpha}%
_{k}(t))\psi(t)\,dt\rightarrow\int_{t_{1}}^{t_{2}}h^{\prime\prime\prime
}(t,\alpha(t))\psi(t)\,dt\,.
\]
By the arbitrariness of $\psi$ we obtain the weak convergence in $I$, which
suffices to prove \eqref{weak convergence}.

Arguing as in the first part of the proof 
\[\frac{\hat{h}_{k}^{\prime}%
(t,\hat{\alpha}_{k}(t))}{(\hat{J}_{k}(t,\hat{\alpha}_{k}(t)))^{5}}%
\rightarrow\frac{h^{\prime}(t,\alpha(t))}{(J(t,\alpha(t)))^{5}}
\]
pointwise and
is uniformly bounded. Therefore
\begin{equation}
\frac{\hat{h}_{k}^{\prime}(\cdot,\hat{\alpha}_{k}(\cdot))}{(\hat{J}_{k}%
(\cdot,\hat{\alpha}_{k}(\cdot)))^{5}}\hat{h}_{k}^{\prime\prime\prime}%
(\cdot,\hat{\alpha}_{k}(\cdot))\rightharpoonup\frac{\hat{h}_{k}^{\prime
}(t,\hat{\alpha}_{k}(t))}{(\hat{J}_{k}(t,\hat{\alpha}_{k}(t)))^{5}}%
h^{\prime\prime\prime}(\cdot,\alpha(\cdot))\quad\text{weakly in }L^{p_{1}%
}((0,T_{1}))\,. \label{802}%
\end{equation}
Combining \eqref{800}, \eqref{801}, and \eqref{802}, from \eqref{alpha k converge}  we obtain
\eqref{equation alpha dot}.
\end{proof}

\bigskip

Next we introduce the time derivative of $h$.

\begin{proposition}
For a.e. $t\in\lbrack0,+\infty)$ there exists an element of $H^{-1}((\alpha
(t),\beta(t)))$, denoted by $\dot{h}(t,\cdot)$, such that%
\begin{equation}
\frac{h(s,\cdot)-h(t,\cdot)}{s-t}\rightarrow\dot{h}(t,\cdot)\quad
\text{strongly in }H^{-1}((a,b)) \label{00}%
\end{equation}
for every $\alpha(t)<a<b<\beta(t)$.
\end{proposition}

\begin{proof}
 Let us fix $T>0$.  By Proposition \ref{proposition existence H}, $H\in H^{1}((0,T);L^{2}%
(\mathbb{R}))$ and $H^{\prime}=h_{\ast}$. It follows that \begin{equation}\label{h* in H-1}
h_{\ast}\in
H^{1}((0,T);H^{-1}(\mathbb{R}))\,,
\end{equation}
and so for a.e. $t\in(0,T)$,%
\begin{equation}
\frac{h_{\ast}(s,\cdot)-h_{\ast}(t,\cdot)}{s-t}\rightarrow\dot{h}_{\ast
}(t,\cdot)\quad\text{strongly in }H^{-1}(\mathbb{R})\,. \label{000}%
\end{equation}
In particular, if $t\in\lbrack0,T]$ satisfies \eqref{000} and $\alpha
(t)<a<b<\beta(t)$, by the continuity of $\alpha$ and $\beta$ (see  Proposition \ref{rtyu8}), we
have $\alpha(s)<a<b<\beta(s)$, for all $s$ close to $t$. Therefore \eqref{000}
implies \eqref{00}.
\end{proof}

\begin{theorem}\label{final theorem} Under the assumptions \eqref{assumption 1}--\eqref{Lip < L0}, let $T_{1}$ be as in Theorem \ref{theorem lipschitz}. Then 
for a.e. $t\in(0,T_{1})$ we have
\begin{equation}
\dot{h}=\Big[-\gamma\frac{1}{J}\Big(\frac{h^{\prime}}{J}\Big)^{\prime\prime
}+\nu_{0}\frac{1}{J}\Big(\frac{h^{\prime\prime}}{J^{5}}\Big)^{\prime
\prime\prime}+\frac{5}{2}\nu_{0}\frac{1}{J}\Big(\frac{h^{\prime}(h^{\prime\prime})^{2}%
}{J^{7}}\Big)^{\prime\prime}+\frac{1}{J}\overline{W}^{\prime}\Big]^{\prime}
\label{weak equation h dot}%
\end{equation}
in $\mathcal{D}^{\prime}((\alpha(t),\beta(t)))$, where $\overline{W}$ is defined in 
\eqref{J and J0}.
\end{theorem}

\begin{remark}\label{final intrinsic} To express \eqref{weak equation h dot} in an intrinsic way, besides the functions $s(t,\cdot)$, $x(t,\cdot)$, and $\kappa(t,\cdot)$ considered in Remark \ref{remark intrinsic boundary}, for $0=s(t,\alpha(t))\leq s\leq s(t,\beta(t))$ we introduce the normal velocity $\widetilde{V}(t,s)$ of the time dependent curve
$\Gamma_{h(t,\cdot)}$ at the point corresponding to the arclength parameter $s$, given by
\[
\widetilde{V}(t,s):=\frac{\dot{h}(t,x(t,s))}
{\sqrt{1+\big(h^{\prime}(t,x(t,s))\big)^{2}}}\,.
\]
Moreover, we introduce the chemical potential
$\zeta(t,\cdot)\colon \big(s(t,\alpha(t)), s(t,\beta(t))\big)\rightarrow\mathbb{R}$ given by%
\begin{equation}\nonumber
\zeta(t,s):=-\gamma\kappa(t,s)+\nu_{0}\Bigl(\partial_{ss}\kappa(t,s)+\frac{\kappa(t,s)^{3}}%
{2}\Bigr)+\widetilde{W}(t,s)\,,\label{chemical potential}
\end{equation}
where
\[
\widetilde{W}(t,s)   :=W\big(Eu(t,x(t,s),h(t,x(t,s))\big)\,.
\]
By direct computations (see \cite[Remark 3.2 and Lemma 6.7]{fonseca2012motion}) we obtain that \eqref{weak equation h dot} is equivalent to
\begin{equation}
\widetilde{V}(t,\cdot)=\partial_{ss}\zeta(t,\cdot)
\nonumber%
\end{equation}
in $\mathcal{D}\big((s(t,\alpha(t)),s(t,\beta(t))\big)$.
\end{remark}

\begin{proof}
By the continuity of $\alpha$ and $\beta$ (see Proposition \ref{rtyu8}), it suffices to prove that,  given
$a<b$ and a time interval $[t_{1},t_{2}]$ such that $\alpha(t)<a<b<\beta(t)$
for all $t\in\lbrack t_{1},t_{2}]$, the equality \eqref{weak equation h dot}  holds 
in $\mathcal{D}^{\prime}((a,b))$ for a.e. $t\in\lbrack t_{1},t_{2}]$. Let
$\delta>0$ be such that $\alpha(t)+3\delta<a<b<\beta(t)-3\delta$ for all
$t\in\lbrack t_{1},t_{2}]$. By \eqref{alpha k converge uniformly}, $\alpha
_{k}(t)+2\delta<a<b<\beta_{k}(t)-2\delta$ for all $t\in\lbrack t_{1},t_{2}]$
and all $k\geq  k^{*}$ for some $ k^{*}$. Hence, by Proposition
\ref{proposition functions away}  there exist $\zeta>0$  and $\hat{k}\ge k^{*}$  such that
\[
h_{k}(t,x)\geq\zeta\quad\text{for all }t\in\lbrack t_{1},t_{2}]\,,\text{ }%
x\in\lbrack a-\delta,b+\delta]\,,\text{ and }k\geq \hat{k}\,.
\]
Letting $k\rightarrow\infty$, by
\eqref{hk uniformly} we obtain that%
\[
h(t,x)\geq\zeta\quad\text{for all }t\in\lbrack t_{1},t_{2}]\,,\text{ }%
x\in\lbrack a-\delta,b+\delta]\,.
\]
In view of these inequalities  and of \eqref{equation h}  we can repeat the proof of
\cite[Theorem 3.8]{fonseca2012motion} to prove \eqref{weak equation h dot}\ in
$\mathcal{D}^{\prime}((a,b))$ for a.e. $t\in\lbrack t_{1},t_{2}]$. We observe
that since here we do not have periodic boundary conditions in $x$, the
argument used in \cite{fonseca2012motion} needs to be modified accordingly. To
be precise, the inequality (3.16) in \cite{fonseca2012motion} must be replaced
by%
\begin{align*}
\Vert h_{k}^{\prime}(\tau_{2},\cdot)-h_{k}^{\prime}(\tau_{1},\cdot
)\Vert_{L^{\infty}((a,b))}  &  \leq C\Vert h_{k}^{\prime\prime}(\tau_{2}%
,\cdot)-h_{k}^{\prime\prime}(\tau_{1},\cdot)\Vert_{L^{2}((a,b))}^{3/4}\Vert
h_{k}(\tau_{2},\cdot)-h_{k}(\tau_{1},\cdot)\Vert_{L^{2}((a,b))}^{1/4}\\
&  \quad+C\Vert h_{k}(\tau_{2},\cdot)-h_{k}(\tau_{1},\cdot)\Vert
_{L^{2}((a,b))}\\
&  \leq C|\tau_{2}-\tau_{1}|^{1/32}+C|\tau_{2}-\tau_{1}|^{1/8}%
\end{align*}
for every $\tau_{1},\tau_{2}\in\lbrack t_{1},t_{2}]$. Similarly, in the proof
of \cite[Corollary 3.7]{fonseca2012motion} the second displayed inequality
must be replaced by%
\begin{align*}
\int_{t_{1}}^{t_{2}}\Vert h_{n}^{\prime\prime\prime}(t,\cdot)-h_{m}%
^{\prime\prime\prime}(t,\cdot)\Vert_{L^{\infty}((a,b))}^{12/5}dt  &  \leq
C\sup_{t\in\lbrack t_{1},t_{2}]}\Vert h_{n}^{\prime}(t,\cdot)-h_{m}^{\prime
}(t,\cdot)\Vert_{L^{\infty}((a,b))}^{2/5}\\
&  \quad+\sup_{t\in\lbrack t_{1},t_{2}]}\Vert h_{n}^{\prime}(t,\cdot
)-h_{m}^{\prime}(t,\cdot)\Vert_{L^{\infty}((a,b))}^{12/5}\rightarrow0
\end{align*}
as $n,m\rightarrow\infty$.
\end{proof}

The following theorem summarizes the main results obtained in this section.

\begin{theorem} Under the assumptions \eqref{assumption 1}--\eqref{Lip < L0},  there exists $T>0$ such that the following items hold:
\begin{itemize}
\item[(i)] there exist 
two functions $\alpha, \beta\in H^{1}((0,T))$, with $\alpha(t)<\beta(t)$ for every $t\in[0,T]$, such that 
$\alpha(0)=\alpha_{0}$, $\beta(0)=\beta_{0}$, and 
\begin{equation}
\beta(t)-\alpha(t)\geq\sqrt{2A_{0}/L_{0}}\quad\text{for every }t\in [0,T]\,.
\label{9.73}
\end{equation}
\item[(ii)]  There exist $1<p_{1}<6/5$ and  a continuous function $h\geq0$, defined for $t\in [0,T]$ and $x\in[\alpha(t),\beta(t)]$, such that $h(0,x)=h_{0}(x)$ for every $x\in [\alpha_{0},\beta_{0}]$, 
$h(t,x)>0$ for every $t\in[0,T]$ and every $x\in (\alpha(t),\beta(t))$, $h(t,\alpha(t))=h(t,\beta(t))=0$,   $h^{\prime}(t,\alpha(t))>0$, and $h^{\prime}(t,\beta(t))<0$ for every $t\in[0,T]$,
and
\begin{gather}
 h(t,\cdot)\in H^{2}((\alpha(t),\beta(t)))
\ \text{ for every }t\in[0,T]\,,
\label{9.74}
\\
h(t,\cdot)\in W^{4,p_{1}}((\alpha(t),\beta(t)))
\text{ for a.e.\ }t\in[0,T]\,,
\label{9.78}
\\
\int_{\alpha(t)}^{\beta(t)}h(t,x)\,dx=A_{0}
\quad\text{for every }t\in[0,T]\,,
\label{h average bis}
\\
\operatorname*{Lip}h(t,\cdot) < L_{0}\quad\text{for every } t\in[0,T]\,.
\label{9.77}
\end{gather}
Moreover, if we denote by $h_{\ast}$ the extension of $h$ obtained by setting $h_{\ast}(t,x):=0$ for $x\in\mathbb{R}\setminus(\alpha(t),\beta(t))$, then
\begin{equation}
h_{\ast}\in
C^{0,3/10}([0,T];L^{2}(\mathbb{R}))\cap H^{1}((0,T);H^{-1}(\mathbb{R}))\,.
\label{9.79}
\end{equation}
\item[(iii)] The function $u$ introduced in \eqref{wer1} is such that
\begin{equation}
u(t,\cdot,\cdot)\in C^{3,1-1/p_{1}}(\overline{\Omega}^{a,b}_{h})\ \text{ for a.e.\ }t\in[0,T]\text{ and }\alpha(t)<a<b<\beta(t)\,,
\label{9.84}
\end{equation}
and $u(t,\cdot,\cdot)$ solves the boundary value problem
\begin{equation}
\left\{
 {\setstretch{1.1}
\begin{array}
[c]{ll}%
-\operatorname{div}\mathbb{C}Eu(t,x,y)
=0 & \text{in }\Omega_{h(t,\cdot)}%
\,,\\
\mathbb{C}Eu(t,x,h(t,x)) \nu^h(t,x)=0 & \text{for }x\in(\alpha(t),\beta(t)
)\,,\\
u(t,x,0)=(e_{0}x,0) & \text{for }x\in(\alpha(t),\beta(t))\,,
\end{array}
}
\right.\label{el7777}
\end{equation}
for a.e. $t\in (0,T)$.
\item[(iv)] Considering the functions $s$, $\kappa$, $\theta_\alpha$, $\theta_\beta$, and $\zeta$ introduced in  \eqref{s(x)}, \eqref{kappa}, \eqref{theta1}, and \eqref{chemical potential},  respectively, we have
\begin{gather}
\sigma_{0}\dot{\alpha}(t)=\gamma\cos\theta_{\alpha}(t)-\gamma_{0}%
+\nu_{0}\partial_{s}\kappa
\big(t,s(t,\alpha(t))\big)\sin\theta_{\alpha}(t)\ \text{ for a.e.\ }t\in[0,T]\,,
\label{9.81}\\
\sigma_{0}\dot{\beta}(t)=-\gamma\cos\theta_{\beta}(t)+\gamma_{0}
-\nu_{0}\partial_{s}\kappa
\big(t,s(t,\beta(t))\big)\sin\theta_{\beta}(t)\ \text{ for a.e.\ }t\in[0,T]\,,
\label{9.82}
\\
\widetilde{V}(t,\cdot)=\partial_{ss}\zeta(t,\cdot)
\ \text{ in }\ \mathcal{D}\big((s(t,\alpha(t)),s(t,\beta(t))\big)\ \text{ for a.e.\ }t\in[0,T] \,.
\label{9.85}
\end{gather}

\end{itemize}

\end{theorem}

\begin{proof} Item (i) follows from Proposition \ref{rtyu8}. 

Let $h$ and $h_{\ast}$ be the functions introduced in Proposition \ref{proposition existence h} and let $T$ be the constant $T_{1}$ introduced in Theorem \ref{theorem lipschitz}. By Proposition \ref{proposition existence h} and Theorem \ref{theorem positive}, we have that $h(t,\cdot)$ is strictly positive in $(\alpha(t),\beta(t))$ and vanishes at the endpoints, with $h^{\prime}(t,\alpha(t))>0$, and $h^{\prime}(t,\beta(t))<0$, for every $t\in[0,T]$. Property  \eqref{9.74} follows from \eqref{h H2 bound}; \eqref{9.78} from \eqref{6002} and \eqref{bound bar h iv}; \eqref{9.77} from \eqref{lipliplip}; \eqref{h average bis} from \eqref{h average}; and 
\eqref{9.79} from Proposition \ref{proposition existence h}
and \eqref{h* in H-1}.

Let $u$ be the function introduced in \eqref{wer1}. Then \eqref{9.84} follows from elliptic regularity (\cite[Theorem 9.3]{agmon-douglis-nirenbergII}), since
$h(t,\cdot)\in C^{3,1-1/p_{1}}([\alpha(t),\beta(t)])$ for a.e.\ $t\in [0,T]$ by \eqref{9.78}. In turn, by taking variations in \eqref{wer1}, we obtain that \eqref{el7777} holds.

Finally, considering the functions $s$, $\kappa$, $\theta_\alpha$, $\theta_\beta$, and $\zeta$ introduced in  \eqref{s(x)}, \eqref{kappa}, \eqref{theta1}, and \eqref{chemical potential},  respectively, we have that
\eqref{9.81} and \eqref{9.82} follow from Theorem \ref{uyt01} and Remark \ref{remark intrinsic boundary} , while \eqref{9.85} follows from
Theorem \ref{final theorem} and Remark \ref{final intrinsic}.
\end{proof}

\backmatter

\bmhead{Acknowledgements}
The research of G. Dal Maso  was supported by the National Research Project (PRIN  2017BTM7SN) 
``Variational Methods for Stationary and Evolution Problems with Singularities and 
 Interfaces", funded by the Italian Ministry of University and Research, that  of I. Fonseca by the National Science
Foundation under grants No. DMS-2205627 and 2108784, and that of G. Leoni under grant No.
DMS-2108784. 
 G. Dal Maso  is member of the Gruppo Nazionale per 
l'Analisi Matematica, la Probabilit\`a e le loro Applicazioni (GNAMPA) of the 
Istituto Nazionale di Alta Matematica (INdAM).
G. Leoni would like to thank Ian Tice for useful conversations on
the subject of this paper.
\section*{Data availability}

This research article does not involve empirical data. All results and analyses presented herein are derived from theoretical developments, mathematical proofs, and computational simulations. Therefore, there are no datasets to be made publicly available.



\bibliography{dewetting-references}

\end{document}